%% file: main_stefano_2.tex
\documentclass[reqno]{amsart}
\usepackage[url=false,giveninits=true,maxnames=5,backend=biber]{biblatex}
\makeatletter
\bibliography{sources}
\renewcommand\part{%
	\if@noskipsec \leavevmode \fi
	\par
	\addvspace{4ex}%
	\@afterindentfalse
	\secdef\@part\@spart}

\def\@part[#1]#2{%
	\ifnum \c@secnumdepth >\m@ne
	\refstepcounter{part}%
	\addcontentsline{toc}{part}{\thepart\hspace{1em}#1}%
	\else
	\addcontentsline{toc}{part}{#1}%
	\fi
	{\parindent \z@ \raggedright
		\interlinepenalty \@M
		\normalfont
		\ifnum \c@secnumdepth >\m@ne
		\Large\bfseries \partname\nobreakspace\thepart
		\par\nobreak
		\fi
		\huge \bfseries #2%
		\par}%
	\nobreak
	\vskip 3ex
	\@afterheading}
\def\@spart#1{%
	{\parindent \z@ \raggedright
		\interlinepenalty \@M
		\normalfont
		\huge \bfseries #1\par}%
	\nobreak
	\vskip 3ex
	\@afterheading}
\makeatother
 
\usepackage{color}
\usepackage[dvipsnames]{xcolor}

\usepackage{ifpdf}
\ifpdf 
    \usepackage[pdftex]{graphicx}   
    \usepackage[pdftex,     
            plainpages=false,   
            breaklinks=true,    
            colorlinks=true,
            linkcolor=red,
            citecolor=green,
            pdftitle=My Document
            pdfauthor=My Good Self
           ]{hyperref} 
\else 
    \usepackage{graphicx}       
    \usepackage{hyperref}       
\fi 

\usepackage{graphicx}
\usepackage{aurical}
\usepackage{amsfonts,amsmath}
\usepackage{amssymb}
\usepackage{verbatim}
\usepackage{amsopn}
\usepackage[english]{babel}
\usepackage{amsthm}
\usepackage{enumerate}
\usepackage{mathrsfs}
\usepackage{enumitem}
\usepackage{mathtools}
\usepackage{esint}
\usepackage{bbm}
\usepackage{marginnote}
\usepackage[marginparwidth=2cm]{geometry}
\usepackage[normalem]{ulem}
\usepackage{glossaries}
\usepackage{glossary-mcols}
\usepackage{subcaption}
\usepackage{csquotes}

\date{\today}

\theoremstyle{definition} \newtheorem{definition}{Definition}[section]
\theoremstyle{definition} \newtheorem{remark}[definition]{Remark}
\theoremstyle{plain} \newtheorem{lemma}[definition]{Lemma}
\theoremstyle{plain} \newtheorem{proposition}[definition]{Proposition}
\theoremstyle{plain} \newtheorem{theorem}[definition]{Theorem}
\theoremstyle{plain} \newtheorem{corollary}[definition]{Corollary}
\theoremstyle{definition} 
\theoremstyle{plain} 
\theoremstyle{definition} 
\theoremstyle{definition} \newtheorem{notation}[definition]{Notation}

\DeclareMathOperator{\BV}{BV}

\DeclareMathOperator{\diag}{diag}

\DeclareMathOperator{\clos}{clos}

\DeclareMathOperator{\inter}{int}

\newcommand{\R}{\mathbb{R}}
\newcommand{\Q}{\mathbb{Q}}
\newcommand{\N}{\mathbb{N}}

\newcommand*{\comme}[1]{\textcolor{cyan}{#1}}

\newcommand{\loc}{\text{\rm loc}}

\newcommand{\ind}{1\!\!\mathrm{I}}

\newcommand{\rest}{\llcorner}

\numberwithin{equation}{section} 

\theoremstyle{plain} \newtheorem*{theorem*}{Theorem}
\theoremstyle{plain} 
\theoremstyle{plain} \newtheorem*{mthm*}{Main Theorem}
\theoremstyle{plain} \newtheorem*{conjecture*}{Conjecture}
\theoremstyle{plain} 
\theoremstyle{plain} \newtheorem*{problem*}{Problem}

\makeglossaries

\title{Traveling profiles and control cost for a PDE describing the evolution of invasive species}

\subjclass[2020]{35K57,49K15,92-10}

\author{Stefano Bianchini}
\address{S. Bianchini: S.I.S.S.A., via Bonomea 265, 34136 Trieste, Italy}
\email{bianchin@sissa.it}

\author{Chiara Trifone}
\address{Chiara Trifone: S.I.S.S.A., via Bonomea 265, 34136 Trieste, Italy}
\email{ctrifone@sissa.it}

\begin{document}

\begin{abstract}
We develop a detailed analysis of optimal traveling waves $U(t,x) = U(x - \beta t)$ for a model of invasive-species control proposed in \cite{BressanChiri22}: the relative density $U \in [0,1]$ of the invasive species satisfies the following reaction-diffusion equation with a positive control
\begin{equation}
\label{Equa:PDE_abstract}
U_t = U_{xx} + f(U) - \tilde \alpha(t,x) U, \quad U \in [0,1], \ \tilde \alpha \geq 0.
\end{equation}
The control $\tilde \alpha(t,x)$ represents the fraction of the population removed at $(t,x)$: the minimal control effort $E(\beta,f)$ required to sustain a traveling invasion front with prescribed speed $\beta$ is defined as
\begin{equation}
\label{Equa:effort_abstract}
E(\beta,f) = \inf \bigg\{ \int_{\R} \tilde \alpha(x) dx, U = U(x-\beta t) \ \text{solves (\ref{Equa:PDE_abstract}) with} \ U(-\infty) = 0, U(+\infty) = 1 \bigg\}.
\end{equation}
In order to study large scale dynamics $(t,x) \mapsto (\epsilon t,\epsilon x)$, a fundamental role is played by the structure of traveling waves and the convexity and regularity properties of $E$.

The main results of this paper are the following.
\begin{enumerate}
\item In the phase plane $(U,P=U_x)$, there exists a unique optimal profile $P_\beta(U)$ minimizing \eqref{Equa:effort_abstract}.
\item It satisfies explicit first-order conditions, which are both necessary and sufficient.
\item The associated control is acting on an open subset of the set $\{U : P_\beta(U) = \sqrt{U f(U)}\}$, in particular it is uniformly integrable, and it depends smoothly on $(\beta,f)$ on a dense open set. 
\item The effort function $E(\beta,f)$ is only $C^1$ w.r.t. $\beta$ and Lipschitz w.r.t. $f$ in the $C^2$-topology, and is asymptocally linear for $\beta \to \infty$.
\item $\beta \mapsto E(\beta,f)$ is in general neither convex nor subadditive.
\end{enumerate}
%
\end{abstract}

\keywords{Reaction-diffusion PDE, traveling profiles, optimal control}

\maketitle

\begin{center}
\emph{Dedicated to Alberto Bressan, on the occasion of his 70th birthday}
\end{center}

\begin{center}
Preprint SISSA 09/2026/MATE
\end{center}

\tableofcontents
%


\section{Introduction}
\label{S:intro}

This paper is devoted to a detailed analysis of the effort function $E$ associated with traveling waves for the parabolic reaction diffusion equation with control
\begin{equation}
\label{Equa:parab_contr_1}
U_t = \Delta U + f(U) - \tilde \alpha(t,x) U, \quad U \in [0,1], \ \tilde \alpha \geq 0.
\end{equation}
The \emph{effort function $E(\beta,f)$} is precisely the minimal $L^1$-norm of $\tilde \alpha$ for a traveling wave solution $U(x-\beta t)$ to \eqref{Equa:parab_contr_1}. The main results of this paper concern the existence and characterization of the optimal traveling wave profile for a given speed \newglossaryentry{beta}{name=\ensuremath{\beta},description={speed of the traveling profile}} \gls{beta} and the structure and regularity of the effort \newglossaryentry{Ebeta}{name=\ensuremath{E(\beta)},description={effort function for a traveling wave speed $\gls{beta} \geq \beta^*$}} $\gls{Ebeta} = \inf \|\tilde \alpha\|_{L^1(\R)}$.

%

The PDE with control \eqref{Equa:parab_contr_1} is motivated by a model of pest eradication, which fits into the large subject of optimal harvesting problems and control of parabolic equations \cite{anitacapasso2019,Ania2019GlobalEF,doi:10.1137/16M1061886,coclitegaraspino2018,coron2007,seirin2013,faucris.242204001,Lenhart2001OPTIMALCO}. In order to understand the main results of this paper (as listed in the abstract), we collect here some of the analysis of \cite{bressancetraro,BressanChiri22,BressanChiri23}.

\subsection{A model for the evolution of invasive species}
\label{Ss:model_invasiv}

The starting point is the parabolic reaction-diffusion equation
\begin{equation}
\label{eq_intro}
U_t = U_{xx} + f(U).
\end{equation}
The reaction term \newglossaryentry{fU}{name=\ensuremath{f(U)},description={nonlinear source function}} $\gls{fU}$ may take different forms. Here we consider the following ones:

\begin{enumerate}[label=\textbf{(H\arabic*)}, ref=H\arabic*]
\item\label{H1_intro} The function $f \in C^2([0,1])$ satisfies (see Fig. \ref{Fig:fU_Gammasu})
\begin{equation*}
f(0)=f(1)=0, \ f'(0)<0, \ f'(1)<0,
\end{equation*}
\begin{equation}
\label{Equa:positive_prev_intro}
\int_0^1 f(U) \, dU > 0.
\end{equation}
Moreover, there exists a unique point \newglossaryentry{Ustar}{name=\ensuremath{U^*},description={unique $0$ of the source $f$ in $(0,1)$}} $\gls{Ustar} \in (0,1)$ such that $f(U^*)=0$ and $f'(U^*)>0$. A picture of $f$ is in Fig. \ref{Fig:fU_Gammasu} left. The simplest example is the cubic polynomial $f(U) = U (U^* - U) (1 - U)$, $U^* \in (0,\frac{1}{2})$.
\end{enumerate}

The interpretation of \eqref{eq_intro} is that $\gls{Uvaria} \in [0,1]$ represents the density of an invasive species, $1$ being the saturated value, and $f(U)$ is the rate of reproduction: in the space-independent case $U = U(t)$, Assumption \ref{H1_intro} implies that if the density is $< U^*$, then $U$ converges to $0$ (the species becomes extinct), while if the density is $> U^*$ then $U \nearrow 1$ (it saturates). Adding a diffusion $\Delta U$ allows \eqref{eq_intro} to have traveling profile: in \cite[Chapter 4.4.a]{Fife} it is shown that there exists a unique value \newglossaryentry{betastar}{name=\ensuremath{\beta^*},description={speed for traveling profiles without control $\alpha$}} $\gls{betastar}$ (see Remark \ref{Rem:positiv_perv}) such that \eqref{eq_intro} admits a solution of the form
\begin{equation}
\label{Equa:trav_boundary_intro}
U(t,x) = U(x - \beta t), \quad \lim_{x \to - \infty} U(x) = 0, \quad \lim_{x \to + \infty} U(x) = 1.
\end{equation}
In terms of large scale dynamics (i.e. upon rescaling $(\frac{t}{\epsilon},\frac{x}{\epsilon})$), the value $\beta^*$ is the normal speed of the boundary of the infested region $\{U = 1\}$.

\begin{remark}
\label{Rem:positiv_perv}
The condition \eqref{Equa:positive_prev} implies that there are no traveling profiles for $\beta \geq 0$: indeed for $\beta = 0$ the system \eqref{sys} below is Hamiltonian with \newglossaryentry{Hamilt}{name=\ensuremath{H(P,U)},description={Hamiltonian for $\beta = 0$}} \newglossaryentry{GU}{name=\ensuremath{G(U)},description={primitive of the source $f(U)$}}
\begin{equation*}
\gls{Hamilt} = \frac{P^2}{2} + G(U), \quad \gls{GU} = \int_0^U f(U') \, dU'.
\end{equation*}
Hence the condition \eqref{Equa:positive_prev_intro} gives
\begin{equation*}
H(1,0) = \int_0^1 f(U) \, dU > 0,
\end{equation*}
so that without control the traveling wave solving \eqref{eq_intro}, \eqref{Equa:trav_boundary_intro} must have negative speed. This is essential in order to get non trivial solutions when studying large scale dynamics: indeed we expect that if condition \eqref{Equa:positive_prev_intro} is not satisfied, then an arbitrarily small control makes the solution converge to $0$ \cite{bia_tri:gamma_lim}.
\end{remark}

Adding a control $- \tilde \alpha(t,x) U$ means that we want to minimize some cost function associated with the solution $U$. Natural formulas contain terms of the form (see \cite{BressanChiri23}):

\begin{description}
\item[running cost] e.g. area where the invasive specie is present,
$$
\int_0^T \int U(t,x) dx dt;
$$
\item[end cost] e.g. area of the region which is still invaded,
$$
\int U(T,x) dx;
$$
\item[control cost] e.g. some increasing superlinear function $\phi : [0,+\infty) \to [0,+\infty)$ of the $L^1$-norm of $\alpha(t)$,
\begin{equation*}
\int_0^T \phi \bigg( \int \tilde \alpha(t,x) dx \bigg) dt.
\end{equation*}
\end{description}
The last term is the most tricky one, since upon rescaling one obtains
\begin{equation*}
\epsilon \int_0^{\frac{T}{\epsilon}} \phi \bigg( \epsilon^{d-1} \int \tilde \alpha(t,x) dx \bigg) dt.
\end{equation*}
The inner integral rescales as $\epsilon^{d-1}$, which suggests that the control costs concentrated on codimension-one fronts, which should be the boundaries of the infested region; moreover the reaction term becomes $\frac{f(U)}{\epsilon}$. The presence of the monotone superlinear $\phi$ is to prevent the control to be applied at a single time instant, e.g. $t=0$.

It is thus natural to expect that in the limit $U(t,x)$ becomes the characteristic function of a set $A(t) \subset \R^d$ with finite perimeter (the points $U = 0,1$ are the only stable solutions to $f(U) = 0$) and the control cost converges to
\begin{equation*}
\int_0^T \phi \bigg( \int_{\partial^* A(t)} E(\beta(t,x)) \mathscr H^{d-1} \bigg) dt,
\end{equation*}
where $\partial^* A(t)$ is the reduced boundary of $A(t)$ and $\beta(t, x)$ is the normal speed. By elementary computations, if $\mathbf n(t,x) = (n_t,\mathbf n_x) \in \mathbb S^d$ is the inner normal to $\{(t,x), x \in A(t)\}$, then
\begin{equation*}
\beta(t,x) = - \frac{n_t(t,x)}{\mathbf n_x(t,x)},
\end{equation*}
and then $\int_{\partial^* A(t)} E(\beta(t,x)) \mathscr \ind_{A(t)}(dx)$ corresponds to an anisotropic perimeter: its lower semi-continuity w.r.t. the $L^1$-convergence requires that $E$ satisfies the assumptions:
\begin{enumerate}
\item \label{Point1:intro_E} $E(\beta)$ is convex,
\item \label{Point2:intro_E} $E(\beta) - \beta \partial_\beta E(\beta) \geq 0$.
\end{enumerate}
These assumptions are the starting point of the analysis of \cite{BressanChiri23}.

The analysis of the cost $E(\beta)$ is mainly done in \cite{BressanChiri22}. From the discussion above, one can think that the cost of a normal speed $\beta(t,x)$ corresponds (by blowing up a small region near $(t,x)$) to a traveling wave for the parabolic equation \eqref{eq_intro}, i.e.
\begin{equation*}
U_{xx} + \beta U_x + f(U) - \tilde \alpha U = 0, \quad U(t,x) = U(x - \beta t), \ \lim_{x \to \pm \infty} U(t,x) = \frac{1 \pm 1}{2}. 
\end{equation*}
The corresponding effort function $E(\beta)$ is the optimal choice of $\alpha$ or $U$:
\begin{equation}
\label{Equa:min_cost_intro}
E(\beta) = \inf \bigg\{ \int_\R \tilde \alpha(x) dx, \tilde \alpha = \frac{U_{xx} + \beta U_x + f(U)}{U}, \ U = U(x-\beta t) \ \text{traveling wave with speed $\beta$} \bigg\}.
\end{equation}
In \cite{BressanChiri22} the authors carry out an explicit form of the optimal solution to \eqref{Equa:min_cost_intro} under a \emph{monotonicity assumption}, which assures a simplified structure for $E(\beta)$.

To understand this condition, one first rewrites the ODE for traveling profiles
\begin{equation*}
U'' + \beta U' + f(U) = \tilde \alpha U, \quad U(-\infty) = 0, \  U(+\infty) = 1
\end{equation*}
as a first-order system in the phase plane \newglossaryentry{Uvaria}{name=\ensuremath{U},description={variable $U$ in the system of ODEs \eqref{sys}}} \newglossaryentry{Pvaria}{name=\ensuremath{P},description={variable $P$ in the system of ODEs \eqref{sys}}} $(\gls{Uvaria},\gls{Pvaria})$
\begin{align*}
\begin{cases}
U' = P\\
P' = -\beta P - f(U) + \tilde \alpha U,
\end{cases}
\end{align*}
where a traveling profile corresponds to a trajectory in the $(U,P)$ phase plane connecting the equilibria $(0,0)$ and $(1,0)$ with $P > 0$. Being $P > 0$ for such a etheroclinic connection, one can use $U$ as independent variable and obtain
\begin{equation*}
\frac{dP}{dU} = - \frac{f(U)}{P} - \beta + \frac{\tilde \alpha(U) U}{P}.
\end{equation*}
Setting now 
\begin{equation*}
\alpha(U(x)) = \frac{\tilde \alpha(x)}{P(U)}, \quad \int \alpha(U) dU = \int \tilde \alpha(x) dx,
\end{equation*}
we obtain the autonomous ODE with control
\begin{equation}
\label{Equa:P_ODE_intro}
\frac{dP}{dU} = - \frac{f(U)}{P} - \beta + \alpha(U) U, \quad \alpha \geq 0.
\end{equation}

Integrating by parts the difference of two instantaneous costs $\alpha_1,\alpha_2$ for two different trajectories $P_1,P_2$ (neglecting boundary terms, integrability issues and allowing $\alpha$ to be negative, see more below)
\begin{equation*}
\begin{split}
\int_0^1 \big( \alpha_1(U) - \alpha_2(U) \big) dU &= \int_0^1 \frac{\frac{dP_{1}}{dU} + \frac{f(U)}{P_1} + \beta}{U} dU - \int_\R \frac{\frac{dP_{2}}{dU} + \frac{f(U)}{P_2} + \beta}{U} dU \\
&= \int_0^1 (P_1 - P_2) \bigg( 1 - \frac{Uf(U)}{P_1 P_2} \bigg) \frac{dU}{U^2}. 
\end{split}
\end{equation*}
If in particular \newglossaryentry{Pstar}{name=\ensuremath{P^*(U)},description={curve of minimality}}
\begin{equation*}
P_2 = \sqrt{U f(U)} = \gls{Pstar},
\end{equation*}
then
\begin{equation*}
\int_0^1 \big( \alpha_1(U) - \alpha_2(U) \big) dU = \int \frac{(P_1 - P_2)^2}{P_1} \frac{dU}{U^2} \geq 0.
\end{equation*}
Observing also that $\alpha \geq 0$ requires that \newglossaryentry{Gammau}{name=\ensuremath{\Gamma_u(\beta)},description={unstable manifold of the equilibrium $(U,P) = (0,0)$}} \newglossaryentry{Gammas}{name=\ensuremath{\Gamma_s(\beta)},description={stable manifold of the equilibrium $(U,P) = (1,0)$}}
\begin{itemize}
\item any trajectory $P$ of \eqref{Equa:P_ODE_intro} cannot be below the unstable manifold \gls{Gammau} of $(0,0)$ (otherwise $P < 0$ for $U \searrow 1$);
\item any trajectory $P$ of \eqref{Equa:P_ODE_intro} cannot be above the stable manifold \gls{Gammas} of $(1,0)$ (otherwise $P > 1$ for $U \nearrow 1$);
\end{itemize}
the \emph{monotonicity conditions} is a sufficient condition assuring that the control for $P_2 = \sqrt{U f(U)}$ is positive in the strip $\Gamma_u \leq P_2 \leq \Gamma_s$: \newglossaryentry{Fdef}{name=\ensuremath{F(U)},description={function describing the crossing of the flow $P'$ across $P^*$}}
\begin{equation*}
0 \leq U \alpha_2(U) = \frac{dP^*}{dU} + \frac{f(U)}{P^*(U)} + \beta = \gls{Fdef} + \beta.
\end{equation*}
We will present a more detailed and precise description of the main results of \cite{BressanChiri22} in Section \ref{S:stru_opt}, here it suffices to understand that the optimal solution would like to be as close as possible to \gls{Pstar}, and in this case the structure of the optimal profile is the simplest: the unstable manifold \gls{Gammau} of $(0,0)$, an arc of \gls{Pstar} and finally the stable manifold \gls{Gammas}, see Fig. \ref{Fig:optproof_1}.

\subsection{Main results}
\label{Ss:main_resu_intro}

Aim of this paper is to address the following questions.
\begin{enumerate}
\item Is the \emph{monotonicity assumption} relevant or there are important cases where it has to be dropped?
\item If the monotonicity condition does not hold, can we describe the minima of \eqref{Equa:min_cost_intro}?
\item Does the effort $E(\beta)$ satisfy Points (\ref{Point1:intro_E}), (\ref{Point2:intro_E}) above, i.e. it is convex and $E(\beta) - \beta \partial_\beta E(\beta) \geq 0$, necessary for the $L^1$-l.s.c. in the control problem for set evolution?
\end{enumerate}

The answers to these questions are presented below.

\subsubsection{Monotonicity assumption, and more generally the structure of optimal solutions}
\label{Sss:mono_intro}

While in \cite{BressanChiri22} it is shown that there are functions satisfying the \emph{monotonicity assumption} (in particular the cubic $f(U) = U (U - U^*) (1-U)$ if $U^*$ is close to $1$), we show by direct computation that if $U^* < 0.4$ then the monotonicity condition does not hold. We observe here that for the $\Gamma$-limit problem for large scale dynamics the interesting case is when the control free speed $\beta^*$ is less than $0$, otherwise the $\Gamma$-limit is $0$ \cite{bia_tri:gamma_lim}: for the cubic case this conditions is equivalent to $U^* < \frac{1}{2}$.

\begin{theorem}
\label{Theo:nono_intro}
{\rm [{\bf Example of Remark \ref{Rem:existence_multiple_inter}}.]} For general source functions $f(U)$, the monotonicity assumption does not hold, and there are sources $f$ and speeds $\beta$ for which the function
\begin{equation*}
F(U) + \beta = \frac{dP^*}{dU} + \frac{f(U)}{P^*(U)} + \beta
\end{equation*}
has countably many crossings of $F(U) = 0$.

{\rm [{\bf Definition \ref{Def:trans_calT} and Theorem \ref{Theo:transv_generic}, Points (1,2).}]} There exists a dense open set \newglossaryentry{Tfrak}{name=\ensuremath{\mathfrak T},description={set of functions for which $\R \setminus \mathcal T$ is finite}} \gls{Tfrak} of admissible sources w.r.t. the $C^2$-topology for which the following holds: for every $f \in \mathfrak F$ there exists a $\mathscr L^1$-negligible set \newglossaryentry{Tcalf}{name=\ensuremath{\mathcal T_f},description={set of speeds $\beta$ such that the roots of $F(U) + \beta$, $P^*(U) - \Gamma_u(U)$ are finite and not degenerate}} \gls{Tcalf} of speeds $\beta > \beta^*(f)$ such that $\{F(U) + \beta = 0\}$ is finite and $\frac{d}{dU} F \not= 0$ on that set. Moreover the intersection of the manifolds \gls{Gammau}, \gls{Gammas} with \gls{Pstar} are finite and transversal too.

{\rm [{\bf Point (3) of Theorem \ref{Theo:transv_generic}}.]} If $f$ is a polynomial of order $k$, then the set \gls{Tfrak} is closed $\mathscr L^{k-1}$-negligible and \gls{Tcalf} is finite.
\end{theorem}

The above results gives a fairly complete picture of the structure of the vector field $- \frac{f(U)}{P} - \beta$ on the control curve \gls{Pstar}: its main use is in the possibility to approximate any admissible source with functions $f$ for which the finiteness of the sets $\{F = 0\} \cup \{\Gamma_u = P^*\} \cup \{\Gamma_s = P^*\}$ yields that the control set \newglossaryentry{OcalPopt}{name=\ensuremath{\mathcal O(\beta)},description={control set, open set where the control $\alpha$ is acting for the optimal profile $P_\beta$, i.e. $\alpha > 0$}} \gls{OcalPopt} is smothly depending on $\beta$ (Proposition \ref{Prop:reguls_Ibeta}). This will be a key ingredient for computing the regularity of the effort \gls{Ebeta}, since it will allow the explicit computations of the derivative.

\subsubsection{Existence, uniqueness and structure of the optimal trajectories.} In \cite{BressanChiri22}, after observing that when $\alpha \in L^1$ then $P$ solving \eqref{Equa:P_ODE_intro} is BV, by a compactness argument the authors conclude that a minimum \newglossaryentry{Pbeta}{name=\ensuremath{P_\beta},description={the unique optimal profile for a given velocity $\beta$}} \gls{Pbeta} for $E(\beta)$ exists (but no uniqueness is stated).

In this paper we prove the following properties of the optimal trajectory \gls{Pbeta}.

\begin{theorem}
\label{Theo:regul_optP_intro}
{\rm [{\bf Theorem \ref{betterprofile} and Definition \ref{candidate}.}]} Any optimal traveling profile \gls{Pbeta} is a \emph{candidate minimizer}: the control $\alpha$ is applied only when $\gls{Pbeta}(U) = \gls{Pstar}$. In particular, \gls{Pbeta} is uniformly a.c. and Lipschitz if $U \not= U^*$.

{\rm [\bf Definition \ref{Def:free_control_set}, Proposition \ref{nec2} and Corollary \ref{unique}.]} There exists an open set (the \emph{control set})
$$
\gls{OcalPopt} \subset \big\{ \Gamma_u < P^* < \Gamma_s \big\} \cap \inter \{P_\beta = P^*\}
$$
where $\alpha > 0$. If $I = (U^-,U^+)$ is a connected component of the \emph{free set} \newglossaryentry{IcalPopt}{name=\ensuremath{\mathcal I(\beta)},description={free set, open set where $\alpha = 0$ for  the optimal profile $P_\beta$}}
$$
\gls{IcalPopt} = \inter [0,1] \setminus \gls{OcalPopt},
$$
then any optimal solution \gls{Pbeta} must satisfy the \emph{first-order conditions}
\begin{subequations}
\label{Equa:nec2_gen_pro_intro}
\begin{equation}
\label{Equa:nec2-_intro}
\text{if $P(U^-) > \Gamma_u(U^-)$} \quad \text{then} \quad \int_{U^-}^{\tilde U} \frac{Uf(U)-P^2(U)}{UP^2(U)} \, dU \ge 0 \quad \forall \tilde U \in I,
\end{equation}
\begin{equation}
\label{Equa:nec2+_intro}
\text{if $P(U^+) < \Gamma_s(U^+)$} \quad \text{then} \quad \int_{\tilde U}^{U^+} \frac{Uf(U)-P^2(U)}{UP^2(U)} \, dU \le 0 \quad \forall \tilde U\in I.
\end{equation}
\end{subequations}
Finally, there exists a unique candidate minimizer \gls{Pbeta} satisfying the above first-order conditions.

{\rm [{\bf Corollary \ref{monotone} and Proposition \ref{Prop:reguls_Ibeta}.}]} The control set \gls{OcalPopt} is monotone increasing w.r.t. $\beta$, and on a open dense set of $\beta,f$ it is consists of finitely many intervals whose boundary points depend smoothly on $\beta,f$.
\end{theorem}

\begin{figure}
\centering
\resizebox{.75\textwidth}{!}{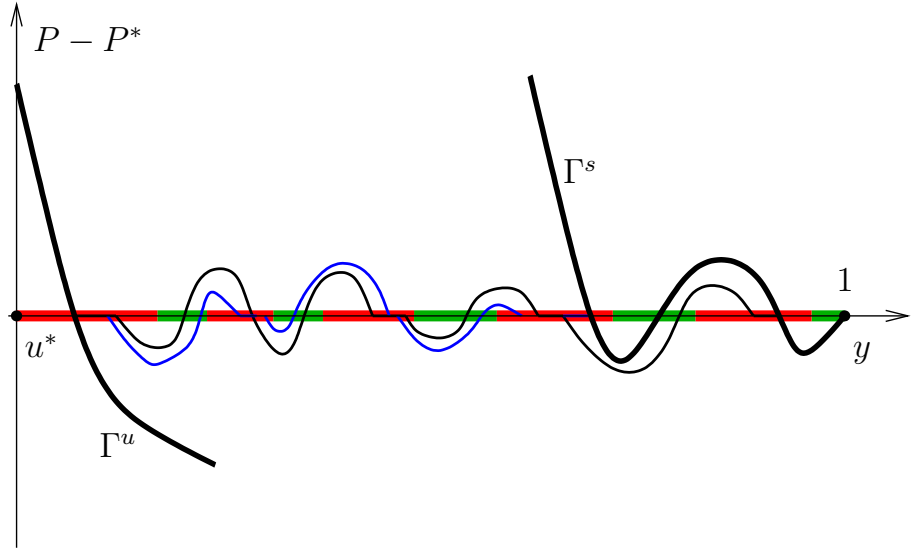}
\caption{A simplified model for the dynamics of \gls{Pbeta}: the control curve \gls{Pstar} is the horizontal axis, the set where $F(U) + \beta > 0$ are in red, and the two bold curves are the manifolds \gls{Gammau}, \gls{Gammas}. The black curve is \gls{Pbeta} and the blue one is another candidate minimizer: considering for example the first interval where $P_\beta \not= P^*$, the black perturbation is obtained by stopping the control $\alpha$ a little bit before and applying it later. Conditions \ref{Equa:nec2_gen_pro_intro} mean exactly that doing this operation the effort function is increasing.}
\label{fig:placeholder}
\end{figure}

The explanation of \eqref{Equa:nec2_gen_pro_intro} is straightforward if we consider Fig. \ref{Rem:existence_multiple_inter}. As explained in the caption, the condition \eqref{Equa:nec2-_intro} is the relative variation in the control $\alpha$ when the control is turned off  before $U^-$ and then applying it later at $\tilde{U}$: the positivity of the integral implies that the effort $E$ is increasing. Similarly, the condition \eqref{Equa:nec2+_intro} is when we apply the control before. Observe that the conditions $P(U^-) > \Gamma_u(U^-)$, $P(U^+) < \Gamma_s(U^+)$ are saying that on $\Gamma_u, \Gamma_s$ only one of \eqref{Equa:nec2_gen_pro_intro} can be imposed, because $\Gamma_u \leq P_\beta < \Gamma_s$ and hence only one perturbation can be considered.

The uniqueness is very useful, because it allows to identify the optimal trajectory \gls{Pbeta} by means of \eqref{Equa:nec2_gen_pro_intro}: the first application is exact in the third part of the statement above: indeed, to prove that \gls{OcalPopt} is increasing, one solves a constrained minimization problem in each connected component of \gls{IcalPopt}, and shows that gluing together these solution the inequalities \eqref{Equa:nec2_gen_pro_intro} are verified.

\subsubsection{Regularity properties of the effort function $E(\beta,f)$}
\label{Ss:regul_E_intro}

Having a precise description of the optimal profile \gls{Pbeta}, we can now answer the questions about the functional form of the effort \gls{Ebeta}. To precisely describe the statements, we need to introduce the first-order perturbations of \gls{Pbeta}: w.r.t. $\beta$
\begin{align*}
\begin{cases}
\frac{d}{dU}\left( \frac{\partial P_u(U)}{U} \right) = \frac{Uf(U)-P(U)^2}{UP(U)^2} \left(\frac{\partial P_u(U)}{U} \right) - \frac{1}{U}, \\ 
\partial P_u(U_u(\beta)) = \partial_\beta \Gamma_u(\beta,U_u(\beta)),
\end{cases}
\end{align*}
\begin{align*}
\begin{cases}
\frac{d}{dU}\left( \frac{\partial P_s(U)}{U} \right) = \frac{Uf(U)-P(U)^2}{UP(U)^2} \left(\frac{\partial P_s(U)}{U} \right) - \frac{1}{U}, \\ 
\partial P_s(U_s(\beta)) = \partial_\beta \Gamma_s(\beta,U_s(\beta)).
\end{cases}
\end{align*}
The function $\partial P_u$ is the perturbation of the solution by varying $\beta$ but keeping the control the same, with initial condition $\partial P_u(U = 0) = 0$; the same holds for $\partial P_s(U)$, with the end condition $\partial P_s(U=1) = 0$. The existence of these functions and their properties are studied in Lemma \ref{Lem:unique}.

In addition of $\beta^*$, there is an additional critical speed \newglossaryentry{betastarstar}{name=\ensuremath{\beta^{**}},description={value of $\beta$ for which $U_{uf}(\beta) = U^*$}} $\gls{betastarstar} > 0 > \beta^*$: this speed corresponds to the first speed such that $\Gamma_u$ ends in $(U^*,0)$. After this value, the optimal profile is obtained by joining (at infinite distance) the traveling wave that connects $0$ to $U^*$ and the one connecting $U^*$ to $1$.

\begin{theorem}
\label{Theo:Ebeta_intro}
[{\rm{\bf Propositions \ref{Prop:depe_E_f}, \ref{Prop:Lipschitz_below} and Corollary \ref{Cor:Lipsc_f}, Proposition \ref{Prop:monotone_f}.}}]
The map $(\beta,f) \mapsto E(\beta,f)$ is Lipschitz in the domain \newglossaryentry{Dcal}{name=\ensuremath{\mathcal D},description={set of admissible sources and admissible speed speeds $\beta$}}
\begin{equation*}
\gls{Dcal} = \Big\{ (\beta,f), f \ \text{satisfies Assumption \eqref{H1}}, \ \beta > \beta^*(f) \Big\} \subset \R \times C^2([0,1]).
\end{equation*}
Moreover
\begin{equation*}
\beta_1 \leq \beta_2, \ f_1 \leq f_2 \quad \Rightarrow \quad E(\beta_1,f_1) \leq E(\beta_2,f_2).
\end{equation*}

[{\rm{\bf Proposition \ref{Prop:E_Lipschitz_beta}, Theorem \ref{Theo:derivative_E}.}}] The function $\beta \mapsto E(\beta)$ is $C^1$, strictly increasing: the explicit expression of the derivative is 
\begin{equation*}
\frac{d}{d\beta} E(\beta) = 
\frac{\partial P_s(U,\beta) - \partial P_u(U, \beta)}{U} \quad \forall U \in \clos(\gls{OcalPopt}). \\
\end{equation*}
Moreover
\begin{equation*}
\lim_{\beta \to \infty} \bigg[ E(\beta) - \ln \bigg( \frac{1}{U^*} \bigg) \beta \bigg] = \int_{U^*}^U \frac{F(U)}{U} dU = \int_{U^*}^1 \frac{2 \sqrt{f(U)}}{U^{3/2}} dU > 0.
\end{equation*}

[{\rm {\bf Theorem \ref{2der}, Corollary \ref{Cor:blowup_E''}}}] The function \gls{Ebeta} is in general only $C^{1,1}_\loc$ in the open set
\begin{equation*}
(\beta^*,\beta^{**}) \cup (\beta^{**},+\infty).
\end{equation*}
Moreover, it is convex for $\beta > \beta^{**}$ , and in general
\begin{equation*}
\lim_{\beta \nearrow \beta^{**}} \partial^2_\beta E(\beta) = - \infty.
\end{equation*}
\end{theorem}

The second part of the statement gives that if $E$ is convex, then $E$ is subadditive: extending \gls{Ebeta} as
\begin{equation*}
E(\beta) = 0 \quad \text{if} \ \beta < \beta^*,
\end{equation*}
then
\begin{equation*}
E(\beta_1 + \beta_2) \leq E(\beta_1) + E(\beta_2). 
\end{equation*}
Moreover, it is fairly easy to see that (again under the assumption that $E$ is convex), then 
\begin{equation*}
E(\beta) - \beta E'(\beta) \geq 0,
\end{equation*}
which implies that the second condition (Point \eqref{Point2:intro_E} of page \pageref{Point2:intro_E}) is not needed and it is a consequence of the convexity.

However, the third part of the above theorem shows that in general the cost $E$ is not convex. This leaves the problem of the correct $\Gamma$-limit of the cost functional open \cite{bia_tri:gamma_lim}.

We conclude by presenting two explicit examples of non-convexity, one showing that also for the cubic source one cannot expect convexity (Examples \ref{Sss:E_not_convex} and \ref{Ss:anisotr_nonconvex}). One last example shows that \gls{Ebeta} is only $C^{1,1}$: the discontinuities of the second derivative are due exactly to the splitting of the connected components of \gls{IcalPopt}.

\subsection{Outline of the paper}
\label{Ss:outline_intro}

The paper is organized as follows.

Section \ref{S:struct_opt} contains definitions and preliminary material. In \eqref{eq} we introduce the PDE modeling the spread of the invasive species in absence of control. According to the definition given in \cite{BressanChiri22}, we assume that the source $f$ is \emph{bistable} (see (\ref{H1})) and, as in \cite{bressancetraro}, \emph{positively predominant} (see \eqref{Equa:positive_prev}).  Classically (\cite{Fife}, \cite{Kanel}), the PDE is solved looking for \emph{traveling waves}, that is solutions of the form $U(t,x) = U(x-\beta t)$ for some velocity $\gls{beta}$. By such a priori assumption, the equation is rewritten in terms of the first-order system \eqref{sys}, whose dynamics can be studied by linearization near the equilibria. The results of this analysis are listed in Points (1-6) of page \pageref{Point_5:eigen_equi}: regardless of the value of $\beta$, the points $(0,0)$ and $(1,0)$ are saddles, so that the portion $(U,P) \in [0,1] \times [0,\infty)$ of the phase plane under consideration contains an unstable manifold $\gls{Gammau}=\Gamma_u(\beta, f)$ through $(0,0)$ and a stable manifold $\gls{Gammas}=\Gamma_s(\beta, f)$ through $(1,0)$. We write $\Gamma_u(\beta)$ as a function $\Gamma_u(\beta,U)$ in the region $[0,\gls{Uufbeta} = U_{uf}(\beta,f)]$ where $\Gamma_u(\beta) \geq 0$, with $\Gamma_u(\beta,0) = \Gamma_u(\beta,U_{uf}(\beta)) = 0$ and $\Gamma_u(\beta)>0$ inside the interval $(0,U_{uf})$. Similarly, we write $\Gamma_s(\beta) = \Gamma_s(\beta,U)$ for all $U \in [0,1]$, as it is easy to see that the stable manifold is always a function for the values of $\beta$ where a traveling profile exists. The functions $\Gamma_u$ and $\Gamma_s$ are respectively decreasing and increasing with respect to $\beta$, so that there exists a unique value $\gls{betastar}$ for which they coincide and correspond to a traveling wave solution for \eqref{eq}. \\
Conversely, the nature of equilibrium $(U^*, 0)$ depends on the value of $\beta$: for $\beta > 2 \sqrt{f'(U^*)}$ it is a stable node with two different eigenvalues. We prove here the existence of a second critical speed $\gls{betastarstar} =\beta^{**}(f)\ge 2\sqrt{f'(U^*)}$, for which there holds $U_{uf}(\beta)=U^*$ for al $\beta \ge \beta^{**}$: this value corresponds to $\Gamma_u$ coinciding with the strongly stable manifolds corresponding to the most negative eigenvalue of $U^*$ \cite[Theorem 6.2.8]{katokhass}. We again remark that from a biological point of view, we are interested in those models corresponding to a negative value of $\beta^*$, meaning that spontaneously the invasive population tends to grow and it makes sense for a controller to intervene. It is straightforward to verify that $\beta^*$ is negative if and only if $f$ is positively predominant. \\
Next, in \eqref{eqwithcontrol}, we consider the PDE with the addition of a control term \gls{alphat}: in the paper, we always assume that the control is non-negative, its effect is proportional to the density itself (as in \cite{BressanChiri22}, Problem (P2)) and its cost is given by the $L^1$ norm of \gls{alphat}. Following the definition given in \cite{BressanChiri22}, we introduce the family of \emph{admissible profiles}  $\gls{Acalbeta}= \mathcal A_{\beta,f}$, generated by all Lipschitz solutions to the controlled system \eqref{syswithcontrol}, for some non-negative control $\tilde{\alpha}$ (see Definition \ref{admcurves}). For each curve in $\gls{Acalbeta}$, we define the associated cost as the $L^1$ norm of $\tilde{\alpha}$, and, finally, in Definition \ref{Def:effort_funvc},
we set the \emph{effort function} $\gls{Ebeta}=E(\beta, f)$ as the infimum among all costs associated to admissible trajectoriesfor $\beta$ and $f$.  \\
For the curves in $\gls{Acalbeta}$, being $P=U'$, if $P > 0$,  we can express $P$ as a function of $U$, solving \newglossaryentry{alphaU}{name=\ensuremath{\alpha(U)},description={control $\alpha$ when parametrized by $U$}}
\begin{equation}
\label{Equa:Pcontrol_structure}
\frac{dP}{dU} = - \beta - \frac{f(U)}{P} + \alpha(U) U, 
\end{equation}
with, in particular,  $\|\alpha\|_1 = \|\tilde \alpha\|_1$, so that the two minimization problems are equivalent.

Section \ref{Ss:stab_sti_Gamma_us} is devoted to the analysis of the stability of $\Gamma_u$ and $\Gamma_s$ both with respect to the velocity $\beta$ and the source $f$. By non-negativity of the control, an admissible profile always lies in the region enclosed by the unstable and stable manifold: it is thus to be expected that the regularity of $\Gamma_u$ and $\Gamma_s$ will have consequences on that of optimal profiles and, thus, of the effort. \\
In \eqref{Equa:Gammau_eq0} and \eqref{Equa:Gammas_eq0}, we recall the equation solved by \gls{Gammau} and \gls{Gammas} and their asymptotics as $U$ tends, respectively, to $0$ and $1$. Moreover, in Lemma \ref{Lem:lim_U_to_0}, we show that admissible profiles (i.e. solutions to the ODE with control \eqref{Equa:Pcontrol_structure}) satisfy the same tangency relation near the origin as \gls{Gammau}: such property will come useful in the following in order to justify integration by parts of the control. \\
Borrowing the notation from \cite{BressanChiri22}, we introduce the critical curve $\gls{Pstar} = \sqrt{U f(U)}$ and define \gls{Uubeta} as the first intersection point of \gls{Gammau} and \gls{Pstar} and, similarly, \gls{Usbeta} as the last intersection point of \gls{Gammas} and \gls{Pstar}. After the technical Lemma \ref{Lem:large_beta_Gammau}, whose proof is based on the comparison principle and states that $\Gamma_u(\beta,U) \sim f(U)$ for large $\beta$'s, in Theorem \ref{manifoldbound}, we provide explicit bounds for the derivative of the manifolds, both with respect to $\beta$ (Estimates \eqref{Equa:Gammau_beta}, \eqref{Equa:Gammas_beta}) and with respect to $f$ (Estimates \eqref{Equa:estimat_f_uGG}, \eqref{Equa:estimat_f_sGG}).
As noted in Remark \ref{Rem:precise_partial_Gamma}, while the smooth dependence of the manifolds with respect to the vector field is well-known near the equilibrium and for $\Gamma_u > 0$, the proof is particularly involved and the result is (to our knowledge) original when dealing with $\beta \sim \gls{betastarstar}$ and $U\sim \gls{Ustar}$: in such case, the equation becomes singular and one cannot refer to classical results. In Corollary \ref{Cor:Frechet_Gamma}, we sum up what we have proved above to show Fr\'echet  differentiability of the maps $(\beta,f,U) \mapsto \Gamma_u(\beta,f,U)$ and $(\beta,f,U) \mapsto \Gamma_s(\beta,f,U)$ outside the critical value $\beta^{**}$. \\
Next, in Theorem \ref{Theo:second_der_Gammaus}, we prove that the second derivative $\partial_\beta^2 \Gamma_u(\beta,U)$ exists in $[0, \gls{Uubeta}]$ when $\beta \not= \beta^{**}$, while $\partial_\beta^2 \Gamma_s(\beta,U)$ exists in $[0,1]$. Moreover, defining \gls{Ubarubeta} as the last intersection of \gls{Gammau} and \gls{Pstar}, in Corollary \ref{Cor:bound_hatu} we extend the estimate for $\partial_\beta^2\Gamma_u$ to the whole interval $[0,\overline U_u(\beta)] $. We remark that in order to bound the second derivative of $\Gamma_u$ with respect to $\beta$, we must make the further assumption that $\beta$ is uniformly far from the critical speed $\beta^{**}$: as shown in Proposition \ref{Prop:blowupsecond},  the result is sharp and, in general, it may happen that $\partial_\beta^2 \Gamma_u(\beta,U_u(\beta))$ blows up to $-\infty$ as $\beta \to \gls{betastarstar}$. Similarly, Proposition \comme{\ref{Prop:Lipechit_U_u}} has the purpose of ensuring the optimality of the results proved in Theorem \ref{Theo:second_der_Gammaus}. We show that the maps $ f\mapsto \partial_\beta \Gamma_u(\beta,f,U_u(\beta,f)), \Gamma_u(\beta,f,U_u(\beta,f))$ are, respectively, continuous and Lipschitz. Regularity is trivial for $\beta\neq \beta^{**}$ and we only need to check what happens as $\beta \nearrow \beta^{**}$. Continuity for the first map holds as $\partial_\beta \Gamma_u(\beta,f,U_u(\beta,f)) \to 0$ as $\beta \to \beta^{**}(f)$ both from the left and the right. Conversely, it is shown that $\partial_f \Gamma_u(\beta, f, U_u(\beta, f)) \not \to 0$ as $\beta \nearrow \beta^{**}$. Finally, in Remark \ref{Rem:const_U*}, we stress that the discontinuity of the derivative with respect to $f$ is due to the variation of the equilibrium point $U^*=U^*(f)$ and cannot occur when we perturb $f$ keeping $U^*$ fixed: in such case, one could prove that $f \mapsto \partial_\beta \Gamma_u$ is smooth.

In Section \ref{S:stru_opt}, we briefly go through the analysis of optimal profiles for \eqref{Equa:funct_eq_UP} presented in \cite{BressanChiri22}. Here the authors provide an explicit construction of the optimal profile $\gamma_\beta$, given by the concatenation of three functions (Equation \eqref{optgraph}): the curve $\gls{Gammau}$ from $0$ until it intersects $\gls{Pstar}$ at $\gls{Uubeta}$, the curve \gls{Gammas} from $1$ until it intersects \gls{Pstar} at \gls{Usbeta} and, finally, the arc on the graph of \gls{Pstar} from $U_u(\beta)$ to $U_s(\beta)$.
In \eqref{Equa:alpha_beta} we define $\alpha_\beta$ as the control associated to $\gamma_\beta$ and, subsequently, in \eqref{F}, we introduce the function
\begin{equation*}
\gls{Fdef}:=(P^*)'(U)+\frac{f(U)}{P^*(U)}, \quad U \in (U^*, 1),  
\end{equation*}
also noting that $\alpha_\beta(U)=\frac{F(U)+\beta}{U}$ for $U\in [\gls{Uubeta}, \gls{Usbeta}]$ and $\alpha_\beta=0$ elsewhere. \\
It is fundamental to remark that admissibility of the curve $\gamma_\beta$ is granted as the authors make a further \emph{monotonicity assumption} on the non-linearity $f$ (see (\ref{mon})). Such additional hypothesis implies, in particular, that \gls{Uubeta} and $\gls{Usbeta}$ are the only intersections of \gls{Pstar} with the manifolds and the positivity region for $\alpha_\beta$, namely the set
\begin{equation*}
\Big\{U\in (U^*,1): \  F(U)+\beta>0\Big\},
\end{equation*}
is an interval for all $\beta$.\\
In Lemma \ref{Lem:first_opt_1} and Proposition \ref{improv}, we recall how one can exploit Stokes' Theorem to compare the costs associated to different traveling waves. Heuristically, the obtained formulas suggest that, in order to reduce the cost associated to a given admissible trajectory $P$, one should always raise the profile in the regions where $P(U)<\gls{Pstar}$ and lower it if $P(U)>\gls{Pstar}$. In light of the above observation, in Corollary \ref{Cor:monotone_1}, we provide a slight generalization of the construction obtained in \cite{BressanChiri22}. \\
The section ends with Remark \ref{Rem:monoton_cond}, where we stress that the monotonicity assumption does not apply in several cases of biological relevance: we check that for a cubic source, that is $f(U) = U (U^* - U) (1 - U)$, Assumption (\ref{mon}) is not satisfied for $U^*<0.4$. Conversely, of course, the function is positively predominant if and only if $U^*< 0.5$. 

Section \ref{S:transver_flux} is still of technical nature, providing preliminary results regarding the transversality properties of $F$ and the manifolds $\Gamma_u$ and $\Gamma_s$, which will be  useful for the analysis of the differentiability of the effort, carried out in Sections \ref{S:regula_E_beta_f} and \ref{S:diffe_E_beta}.\\
First, in Proposition \ref{zeros}, via an almost immediate application of Sard's Lemma, we show  that the set of \emph{critical velocities}, that is those values of $\beta$ that do not satisfy the transversality condition 
\begin{equation*}
\gls{Fdef} + \beta = 0 \quad \Rightarrow \quad F'(U) \not= 0, 
\end{equation*}
is compact with zero Lebesgue measure for every smooth source $f$. Subsequently, with Proposition \ref{Prop:F_transv_poly}, we make a substantial improvement of the above result by restricting ourselves to the family of polynomial sources. It is easy to check that if $f$ is a polynomial, its associated set of critical velocities is at most finite. Moreover, we prove the existence of a sub-family \gls{NTFk} of condimension $1$ in the family of polynomials of order $k\ge 3$ such that if $f \not\in \gls{NTFk}$, all its associated critical velocities are local extrema of $F$, or, more precisely,
\begin{equation*}
F(U) + \beta = F'(U) = 0 \quad \Rightarrow \quad F''(U) \not= 0.
\end{equation*}
Finally, as third part of the statement, we prove that the same property remains true when we restrict ourselves to polynomials taking the value $0$ at a fixed $U^*$: this last result is connected to the differentiability property suggested in Remark \ref{Rem:const_U*}. \\
Next we address the problem of transversality of \gls{Gammau}, \gls{Gammas} and \gls{Pstar}. We define the sets of points where where $\Gamma_u$ and $\Gamma_s$ cross $P^*$, respectively \gls{Cs}$=C_s(\beta, f)$ and \gls{Cu}$=C_u(\beta, f)$.
In the proof of Proposition \ref{Prop:transv_Gamma_gen}, we exploit the Implicit Function Theorem to show that for $\beta$ outside of a compact set of zero Lebesgue measure, \gls{Cu} and \gls{Cs} are finite and the intersections are always \emph{transversal}, that is
\begin{equation*}
U \in C_u\cup C_s \quad \Rightarrow \quad F(U)+\beta \neq 0.
\end{equation*} 
As in the above, the result can be strengthened considering polynomial sources. In Proposition \ref{Prop:Gamma_trans_poly} we prove that if $f$ is a polynomial, the cardinality of $C_u(\beta)$ and $C_s(\beta)$ are uniformly bounded in $\beta$, and, in this case, the family of velocities for which a tangential crossing with $P^*$ happens either for $\Gamma_u$ or for $\Gamma_s$ is at most finite.\\
In Definition \ref{Def:trans_calT}, we introduce the family of \emph{transversal velocities} for a source $f$, that is \gls{Tcal}$=\mathcal{T}_f$. We say that $\beta$ is \emph{transversal} for $f$ if it is non-critical for $f$ and the associated sets $C_u(\beta, f)$ and $C_s(\beta, f)$ are finite and only consist of transversal intersection points. Collecting the results previously obtained throughout this section, we are able to prove Theorem \ref{Theo:transv_generic}. Here, we exhibit a dense open set \gls{Tfrak} in the space of admissible sources satisfying:
\begin{enumerate}
\item $\mathfrak T$  contains a dense set of polynomials,
\item if $f \in \mathfrak T$, outside a finite set the speeds $\beta's$ are transversal for $f$,
\item the set
\begin{equation*}
\bigcup_{f \in \mathfrak T} \mathcal T_f \times \{f\}
\end{equation*}
is open and dense in the family of admissible couples $(\beta, f)\in \gls{Dcal}$.
\end{enumerate} 
As previously mentioned, when $f$ satisfies the \emph{monotonicty condition}, the region $\{F+\beta>0\}$ is an interval for all $\beta$. In Remark \ref{Rem:existence_multiple_inter} we provide the explicit example of a source for which the same set consists of countably many connected components, thus also implying that $\mathfrak T$ is strictly contained the family of all admissible sources.

Motivated by the previous observations, starting from Section \ref{S:exist_regu_one_opt}, we begin our study of the minimization problem defining the effort \gls{Ebeta} when  assumption (\ref{mon}) is removed. We define the set \gls{calUset} where the critical curve \gls{Pstar} lies in the region enclosed by \gls{Gammau} and \gls{Gammas}, but its associated control \gls{Fdef}$+\beta$ is negative. We decompose \gls{calUset} as the union of its connected components, denoted by $\gls{Iomega}(\beta) =(\gls{Vomega}(\beta), \gls{Womega}(\beta))$, where the index \gls{Omegapar} is varying in a mostly countable family. For every $\varpi$, we denote by \gls{Qomega}and \gls{Romega} the solutions to ODE \eqref{functeq} coinciding with $P^*$ respectively at $V_\varpi^-$ and $W_\varpi^+$. In Proposition \ref{mft}, we introduce the \emph{minimal forward trajectory} \gls{Pbar} from the point $U^-$ with $\Gamma_u(U^-) \leq P^*(U^-) \leq \Gamma_s(U^-)$, which satisfies the following properties:
\begin{itemize}
\item $\overline{P}$ coincides with \gls{Pstar} at $U^-$ and is always above both $\Gamma_u$ and $P^*$ and below $\Gamma_s$, 
\item the control associated to $\overline{P}$ is always non-negative (\emph{admissibility}) and it is $0$ where $\overline{P}\neq P^*$,
\item for any other admissible trajectory $P$ such that $P>P^*$, then $P>\overline{P}$ (\emph{minimality}).
\end{itemize}
Similarly, in Proposition \ref{Mbt} we define the \emph{maximal backward trajectory} \gls{Ptilde} starting from $U^+$ with $\Gamma_u(U^+) \leq P^*(U^+) \leq \Gamma_s(U^+)$: the main difference here is that \gls{Ptilde} is defined for $U \leq U^+$ and lies below $P^*$. \\
Next, in Definition \ref{Def:candid_min}, we introduce the family  of \emph{candidate minimizers}, made of admissible trajectories whose associated control only acts in the region $\{P=P^*\}$.  The main result of this section is contained in Theorem \ref{betterprofile}, where we prove that given any admissible trajectory it is always possible to produce a candidate minimizer achieving a lower cost. In simple terms, the proof consists in modifying the given trajectory in each region where it does not coincide with $P^*$ by replacing it with either \gls{Pbar} or \gls{Ptilde}. Such operation decreases the cost because of the formula obtained in Proposition \ref{improv}. As stated in Corollary \ref{Cor:restrict_admissible}, the previous result implies that it is possible to restrict the class of optimizing sequences defining the effort \gls{Ebeta} and only minimize among candidate minimizers. \\
Next, we introduce the family \gls{Acaltildebeta}, made of candidate minimizers defined as functions of $U$, which will facilitate the computations carried out in the following steps: indeed the requirement of applying the control only on the critical curve yields strong regularity properties to the candidate minimizers. With the aid of the preliminary regularity Lemma \ref{Lem:regul_P_cand}, in Proposition \ref{corr} we provide an explicit integral expression for the cost associated to profiles in \gls{Acaltildebeta}. 
In Definition \ref{Def:free_control_set}, for each $P \in \tilde {\mathcal{A}}(\beta)$, we introduce the \emph{critical set} \gls{KcalP}, where $P$ equals $P^*$, the \emph{control set} \gls{OcalP},  where the control associated to $P$ is active, and the \emph{free set} \gls{IcalP}, where $P$ is a solution to \eqref{functeq}. In Lemma \ref{costexpression}, we refine the identity proved in Proposition \ref{corr}, showing that the cost associated to a candidate minimizer  can be computed integrating the control only over the control set. Finally, in Theorem \ref{exist}, we prove that the family of candidate minimizers is compact with respect to uniform convergence and the cost function restricted to such family is continuous, so that a simple application of the direct method of the Calculus of Variations allows us to conclude that a minimum exists.

After proving the existence of an optimal profile, in Section \ref{S:first_cond_opt} we develop our analysis providing first-order necessary conditions for minimality, which is the content of Proposition \ref{nec2}. We fix an optimal profile $P$ and $(U^-, U^+)\subseteq \gls{IcalP}$ connected component such that $\Gamma_u(U^-)<P(U^-)$ and $\Gamma_s(U^+)>P(U^+)$, and we prove that for all $\tilde{U} \in (U^-, U^+)$ the following integral inequalities must hold:
\begin{equation}
\label{Equa:kernel_intro_1}
\int_{U^-}^{\tilde U} \frac{Uf(U)-P^2(U)}{UP^2(U)} \, dU \ge 0 \quad \text{and}\quad  \int_{\tilde U}^{U^+} \frac{Uf(U)-P^2(U)}{UP^2(U)} \, dU \le 0.
\end{equation}
The proof is based on the comparison of the optimal cost associated to $P$ with that of a new admissible profile whose associated control is switched off slightly before $U^-$, and then turned back on at $\tilde{U}$, in order to reconnect with $P(\tilde{U})$. Once again by the Stokes' formula (see Proposition \ref{improv}), the ratio between the control applied at $\tilde{U}$ and the control not applied at $U^-$ can be computed as 
\begin{equation*}
e^{\int_{U^-}^{\tilde {U} }\frac{Uf(U)-P^2(U)}{UP^2(U)} \, dU},
\end{equation*}
so that, by optimality, the integral exponent has to be non-negative. The second inequality is obtained via the same technique, switching off the control slighlty after $U^+$. As noted in Remark \ref{nec2strong}, if $\Gamma_u(U^-)=P^*(U^-)$, we can only consider perturbations from the right so that only the second inequality holds. Similarly, when $\Gamma_s (U^+)=P(U^+)$, the only admissible perturbations are obtained moving backwards from $U^-$, which leads to the first inequality. 
The final part of this subsection contains some results regarding the first and last points where the optimal control is applied. In Definitions \ref{Def:u_intersection_set} and \ref{Def:s_intersection_set}, we recall the notation for the intersection points of \gls{Pstar} with \gls{Gammau} and \gls{Gammas}, respectively \gls{Cu} and $\gls{Cs}$. Also we define \gls{Chatubeta} as the set of intersection points $\overline{U}\in C_u(\beta)$ such that
\begin{equation*}
\gls{IuUbar}(U):=\displaystyle{\int_U^{\overline U} \frac{Wf(W)-\Gamma_u^2(W, \beta)}{W\Gamma_u^2(W, \beta)}\, dW} \le 0 \quad \forall U \in [U_u(\beta), \overline{U}]
\end{equation*}
and set \gls{Uhatubeta} as the maximum of \gls{Chatubeta}. Moreover, we provide analogous definitions for \gls{IsUbar}, \gls{Chatsbeta} and \gls{Uhatsbeta}.\\
Roughly speaking, in Propisition \ref{Prop:optimal_Gamma_us} we prove that if the optimal profile is a solution to \eqref{functeq} in a right neighborhood of the first intersection point \gls{Uubeta}, then it has to coincide with $\Gamma_u$ up to \gls{Uhatubeta}. Similarly, if \gls{Pbeta} is equal to $\Gamma_s$ in a left neighborhood of \gls{Usbeta}, then they coincide starting from \gls{Uhatsbeta}.\\
We end the section proving uniqueness of the optimal profile \gls{Pbeta}. In Theorem \ref{suff}, we show that only one candidate minizer can satisfy the integral conditions in Proposition \ref{nec2}. As an immediate consequence, we deduce that the optimal profile is unique and that the obtained necessary conditions are in fact sufficient (Corollary \ref{unique}).

Section \ref{S:regular_calO} begins with some additional notation. We denote by \gls{Pbeta} the unique optimal profile for \gls{Ebeta}, by \gls{alphabeta} its associated control and by \gls{KcalPopt}, \gls{OcalPopt}, and \gls{IcalPopt}, respectively, its critical, control and free sets. Aim of the first subsection is to show that control sets $\mathcal{O}(\beta)$ are strictly increasing with $\beta$. To such end, we fix $\beta_1 < \beta_2$ and $I$ connected component of $\mathcal{I}(\beta_1)$. In Equation \eqref{local}, we introduce a localized and constrained version of the optimization problem for $\beta_2$, obtained by minimizing the cost restricted to $I$ only among admissible profiles which coincide with \gls{Pstar} at the endpoints of $I$. Using techniques very similar to those developed in Sections \ref{S:exist_regu_one_opt} and \ref{S:first_cond_opt}, in Lemma \ref{Lem:restr_beta_i}, we prove that the problem in \eqref{local} admits a unique minimizer, satisfying necessary conditions analogous to those in Proposition \ref{nec2}. As noted in Remark \ref{Rem:beta12_manifold}, similar arguments can be exploited to deal with the localized problem in the first and last free intervals $(\gls{Uubeta},\gls{Uhatubeta})$ and $ (\gls{Uhatsbeta},\gls{Usbeta})$. The above results are crucial in the proof of Corollary \ref{monotone}, where we show $\mathcal{O}(\beta_1)\subseteq \mathcal{O}(\beta_2)$, which will also imply strict monotonicty of \gls{Ebeta}. For each $I \subseteq \mathcal{I}(\beta_1)$, the optimal profile for the local problem in \eqref{local} satisfies the sufficient conditions to be optimal (see Corollary \ref{unique}). Therefore, we are able to provide an explicit construction of $P_{\beta_2}$, which makes the inclusion in the statement evident. As noted in Remark \ref{Rem:tree_structure}, the optimal profile for $\beta_2$ satisfies the following property: each connected component of $\mathcal{I}(\beta_2)$ is strictly contained in a connected component of $\mathcal I(\beta_1)$. In fact, as $\beta$ increases, it can either happen that the connected components of $\mathcal{I}(\beta)$ simply shrink or they can split: such two different possibilities will have consequences on the regularity of the effort function $E$. \\
The following subsection is devoted to the analysis of the dependence of optimal profiles with respect to \gls{beta}. The main result is contained in Proposition \ref{Prop:reguls_Ibeta}. We fix a transversal source $f \in \gls{Tfrak}$ and we prove the existence of a set \gls{Ncal2} such that for all $\beta \not \in \mathcal{N}_2$ the following hold:
\begin{itemize}
    \item the number of connected components of \gls{IcalPopt} is finite and constant in $\beta$,
    \item  the set\begin{equation*}
\bigcup_{f \in \mathfrak T} \mathcal N_{2}(f)^c \times \{f\}
\end{equation*}
is open and dense in the family of admissible couples $(\beta,f) \in \gls{Dcal}$,
\item the boundary points of each connected component of $\mathcal I(\beta)$ depend smoothly on $\beta$ and $f$.
\end{itemize}
Finally, in Remark \ref{Rem:split_calI}, we stress that the proof of Proposition \ref{Prop:reguls_Ibeta} heavily relies on the assumption that for values $\beta\not\in \mathcal{N}_2$ no splitting happens in the connected components of $\mathcal{I}(\beta)$. The result is sharp: in Section \ref{S:counter_convex_C11}, we will see that when a splitting occurs, one also observes a decrease in regularity. 

The mentioned properties allow us to perform a detailed study of the dependence of the effort on both the velocity \gls{beta} and the source \gls{fU}, carried out in the subsequent sections.

We begin Section \ref{S:regula_E_beta_f} with the analysis of the dependence $\beta\mapsto \gls{Ebeta}$. In Proposition \ref{Prop:E_Lipschitz_beta}, we first recall that $E$ is strictly increasing with respect to $\beta$, as already proved in Corollary \ref{monotone}. Moreover, via a comparison argument, we show that $\beta\mapsto E(\beta)$ is Lipschitz continuous. Finally, noting that for $\beta\to +\infty$, the optimal profile \gls{Pbeta} approaches the critical curve \gls{Pstar}, we also conclude that $E$ is asymptotically linear.\\
Regularity with respect to $f$ is slightly more delicate. Once again, in Propositions \ref{Prop:depe_E_f} and \ref{Prop:Lipschitz_below}, we obtain Lipschitz dependence of $E$ on the source function by comparison with non-optimal profiles, but in this case the estimates are only local.  Conversely, the proof of monotonicity (Proposition \ref{Prop:monotone_f}) relies on the computation of the derivative of the effort with respect to $f$: to such end, we use for the first time an approximation technique which will be extensively exploited also in Section \ref{S:diffe_E_beta}.
Thanks to Theorem \ref{Theo:transv_generic} and Proposition \ref{Prop:reguls_Ibeta}, we are able to compute the derivative of the map
\begin{equation*}
    (a_1,\dots, a_k) \mapsto E(\beta, f) \quad \text{with} \ f(U)=\sum_{i=1}^k a_iU^i
\end{equation*}
for $(a_1,\dots,a_k)$ outside a closed Lebesgue-negligible set and with respect to variations satisfying 
\begin{itemize}
   \item $\displaystyle{\sum_{i=1}^ka_i=0}$, so that admissibilty is preserved,
   \item $\displaystyle{\sum_{i=1}^k a_iU^i \ge0}$, so that $\delta f \ge 0$.
\end{itemize}
Exploiting the first-order necessary conditions from Section \ref{S:first_cond_opt}, in Equation \eqref{Equa:deriv_expli_sour}, we provide an explicit formula for such derivative and note that it must be non-negative: if follows immediately that the effort function restricted to the family of polynomials is increasing. By density and recalling continuity of the effort with respect to $f$ (see Corollary \ref{Cor:Lipsc_f}), we conclude monotonicty for $f\mapsto E(f)$ for all admissible sources. We conclude the section with Remark \ref{Rem:deriva_f}, where we observe that via an approximation argument, the expression for $\partial_fE$ obtained in \eqref{Equa:deriv_expli_sour} can be extended to all admissible sources, provided $\beta$ is far from the critical speed \gls{betastarstar}.

In Section \ref{S:diffe_E_beta}, we investigate the differentiability properties of \gls{Ebeta}. First, we define the functions \gls{Ppartialu} and \gls{Ppartials}, respectively solving systems \eqref{Equa:delta_Pu} and \eqref{Equa:delta_Ps}: these functions are the first-order perturbations of the profile w.r.t. $\beta$, with \gls{Ppartialu} starting from $U=0$ and \gls{Ppartials} starting from $U=1$. In Lemma \ref{Lem:unique}, we provide a representation formula for $\frac{\gls{Ppartialu}}{U}$ and $\frac{\gls{Ppartials}}{U}$ and prove that they depend continuously on \gls{beta} and \gls{fU}.  Also, in Equation \eqref{k}, we analyze the kernel (already appearing in \eqref{Equa:kernel_intro_1})
\begin{equation*}
\gls{Kkernel} := \int_{U_1}^{U_2} \frac{U f(U) - P_\beta(U)^2}{U P_\beta(U)^2} \, dU, \quad U_u(\beta) \leq U_1,U_2 \leq U_s(\beta),
\end{equation*}
and we list its main properties, following from the necessary conditions proved in Section \ref{S:first_cond_opt}. \\
In Theorem \ref{Theo:derivative_E}, we prove that $E\in C^{1}((\beta^*, +\infty)$ and its first derivative can be expressed in terms of \gls{Ppartialu} and \gls{Ppartials} as follows:
\begin{equation*}
E'(\beta)= \frac{\partial P_s(U,\beta) - \partial P_u(U, \beta)}{U} \quad\text{for any} \ U \in \clos(\gls{OcalPopt}).
\end{equation*}
The main idea of the proof is to show that the previous formula holds in the case of a  polynomial source, and then pass to the general case by an approximation argument. We start by fixing $f \in \gls{Tfrak}$: by Proposition \ref{Prop:reguls_Ibeta}, there exists a finite set $\gls{Ncal2}=\mathcal{N}_2(f)$ and $N=N(f)\in \N$ such that the optimal free set \gls{IcalPopt} can be expressed as
\begin{equation*}
\gls{IcalPopt}=(\gls{Uubeta},\gls{Uhatubeta})  \cup \bigcup_{i=1}^N (U_i^-(\beta), U_i^+(\beta)) \cup (\gls{Uhatsbeta},\gls{Usbeta}) \quad \forall \beta \in \mathcal N_{2}(f)^c,
\end{equation*}
with all the extremals depending smoothly on $\beta$. It follows that the cost can be decomposed as
\begin{equation*}
\gls{Ebeta}=\int_{\gls{OcalPopt}}\frac{\gls{Fdef}+\beta}{U}\, dU= \int_{\hat{U}_u(\beta)}^{\hat{U}_s(\beta)}\frac{F(U)+\beta}{U}\, dU -\sum_{i=1}^N \int_{U_i^-(\beta)}^{U_i^+(\beta)} \frac{F(U)+\beta}{U}\, dU
\end{equation*}
and each term can be differentiated separately, which allows us to prove that the formula in the statement holds for $(\beta,f) \in \bigcup_{f \in \mathfrak T} \mathcal N_{2}(f)^c \times \{f\}$. \\
Next, in Theorem \ref{2der} we study second order differentiability of $E$. First, as in Theorem \ref{Theo:derivative_E}, we are able to compute explicitly $\partial_\beta^2 E(\beta,f)$ for $(\beta,f) \in \bigcup_{f \in \mathfrak T} \mathcal N_{2}(f)^c \times \{f\}$: in particular, since $\mathcal N_2(f)$ is finite for $f \in \gls{Tfrak}$ polynomial away from $\beta^{**}$, we deduce that the map $\beta \mapsto \partial_\beta E(\beta, f)$ is Lipschitz in such a case. 
Via approximation, we conclude that the same property holds for a generic source away from $\beta^{**}$. \\ The section ends with  Corollary \ref{Cor:blowup_E''}: exploiting the formula obtained in Theorem \ref{2der} and Proposition \ref{Prop:blowupsecond}, we prove that when Condition \eqref{Equa:cond_blow} holds, the second derivative of $E$ tends to $-\infty$ as $\beta \to \beta^{**}$.

Finally, in Section \ref{S:counter_convex_C11}, we collect three examples showing that the previously obtained results are sharp. The first example considers a perturbation of the Hamiltonian system when $\beta=0$, and in this case its associated manifolds can be explicitly computed, via Equations \eqref{unstable0} and \eqref{stable0}. Similarly, for $\beta=0$, we are able to provide formulas for $\frac{\partial \Gamma_u}{\partial \beta}$, $\frac{\partial \Gamma_s}{\partial \beta}$, $\frac{\partial ^2 \Gamma_u }{\partial \beta^2} $ and $\frac{\partial ^2 \Gamma_s }{\partial \beta^2}$. In Equation \eqref{Eq:nonconvsource}, we introduce a sequence $\{f_K\}_{K \in \mathbb{N}}$ of continuous sources for which we are able to show 
\begin{itemize}
    \item $\frac{\partial ^2 \Gamma_u}{\partial \beta^2}(0,U^*)>0$ and $ \frac{\partial ^2 \Gamma_u}{\partial \beta^2}(0,U_u(0))=\frac{\partial ^2 \Gamma_u}{\partial \beta^2}(0,U^*) + \mathcal O (\sqrt{K})$,
    \item $\frac{\partial^2\Gamma_s}{\partial \beta^2}(0, U_s(0))=\frac{1}{8\sqrt{K(1-2U^*)}} \to 0 \text{ as } K\to +\infty$,
\end{itemize}
so that, in particular, for $K \gg 1$,
\begin{equation*}
\frac{1}{U_s(0)} \frac{\partial^2 \Gamma_s}{\partial \beta^2}(0,U_s(0)) - \frac{1}{U_u(0)}\frac{\partial^2 \Gamma_u}{\partial \beta^2}(0,U_u(0))<0.
\end{equation*}
We remark that 
\begin{equation*}
    \int_0^1 f_K(U)\, dU \to +\infty \quad \text{for} \ K\to+\infty, 
\end{equation*}
meaning that requiring that $K$ is large enough corresponds  to asking that $f$ is highly positively predominant.\\
In the second example, we recall that in Proposition \ref{Prop:E_Lipschitz_beta}, we have proved that the effort is asymptotically linear:
\begin{equation*}
     E(\beta)\sim \ln \bigg( \frac{1}{U^*} \bigg) \beta +\int_{U^*}^1 \frac{2 \sqrt{f(U)}}{U^{3/2}} dU \quad \text{as}\ \beta\to+\infty.
\end{equation*}
Such property obviously has consequences on the convexity of $E$. In \cite{BressanChiri23}, in order to ensure lower semi-continuity of the cost functional, the authors introduce two main assumptions on the effort:
\begin{itemize}
     \item $E$ is convex, 
     \item $E(\beta)\ge \beta E'(\beta), \quad \forall \beta \ge 0.$
\end{itemize}
We note that if one assumes a priori that $E$ is convex, its asymptotic behavior automatically implies the second condition. Conversely, exploiting similar arguments and the Lipschitz continuity of the effort proved in Proposition \ref{Prop:E_Lipschitz_beta}, we are able to exhibit a cubic source whose associated effort is not convex. \\
The last example shows that the result obtained in Theorem \ref{2der} about $C^{1,1}$-regularity is sharp. As previously noted (see Remark \ref{Rem:split_calI}), the regularity of the optimal profile \gls{Pbeta}, and thus of \gls{Ebeta}, is strictly linked to the topological features of the free set \gls{IcalPopt}: such observation is crucial to construct a source whose associated effort is not of class $C^2$.  According to the localization property proved in Lemma \ref{Lem:restr_beta_i}, we can find $f \in \gls{Tfrak}$ and a velocity $\overline{\beta}$ for which a splitting occurs in $\mathcal{I}(\overline{\beta})$, in particular
\begin{equation*}
     \mathcal{I}(\beta)=\begin{cases}
         (U_1^-(\beta),U_1^+(\beta)) &\text{ for }   \beta<\overline{\beta}, \\
         (U_2^-(\beta), U_2^+(\beta)) \cup (U_3^-(\beta), U_3^+(\beta))&\text{ for }   \beta>\overline{\beta}.
     \end{cases}
\end{equation*}
Knowing that the optimal profile solves \eqref{functeq} in each connected component of $\mathcal{I}$, equals \gls{Pstar} at the endpoints and must satisfy the necessary condition \eqref{necint}, we are able to compute the right and left limits of $\partial^2_\beta E$ at $\bar \beta$ and prove that they do not coincide.

\subsection{Notation}
\label{Ss:notation}

A glossary of notations is at end of the paper: we have omitted standard notations for natural, real numbers, etc. Since the functions considered may depend on several variables, in order to avoid formulas too cumbersome we omit the dependence from the variable which do not play a role in the computation: this should not create confusion.

\section{Existence and structure of optimal profiles}
\label{S:struct_opt}

Aim of this section is to recall some basic facts about traveling profiles for parabolic reaction diffusion equations. For all classical results regarding traveling wave solutions for parabolic reaction diffusion equations we refer to \cite[Chapter 4]{Fife}.

\subsection{Traveling waves with no control}
\label{Ss:travelin_no_alpha}

Consider the PDE \eqref{Equa:parab_contr_1} without control
\begin{equation}
\label{eq}
U_t = U_{xx} + f(U).
\end{equation}
As in \cite{BressanChiri22} and \cite{BressanChiri23} we make the following assumptions on the \emph{source function}
$\gls{fU}$:
\begin{enumerate}[label=\textbf{(H\arabic*)}, ref=H\arabic*]
\item\label{H1} The function $f \in C^2([0,1])$ satisfies (see Fig. \ref{Fig:fU_Gammasu})
\begin{equation*}
f(0)=f(1)=0, \ f'(0)<0, \ f'(1)<0.
\end{equation*}
and there exists a unique point
$\gls{Ustar} \in (0,1)$ such that $f(U^*)=0$ and $f'(U^*)>0$.
\end{enumerate}
Moreover we require that $f$ is \emph{positively predominant}, that is
\begin{equation}
\label{Equa:positive_prev}
\int_0^1 f(U) \, dU > 0.
\end{equation}
A \emph{traveling wave solution with velocity
$\gls{beta}$} (or \emph{traveling profile}) for \eqref{eq} is a solution of the form
\begin{equation*}
U(t,x) = U(x-\beta t)
\end{equation*}
connecting $U(-\infty) = 0$ to $U(+\infty) = 1$. 
If we consider the  system in the phase plane
$(\gls{Uvaria},\gls{Pvaria})$
\begin{align}
\label{sys}
\begin{cases}
U' = P, \\
P' = - f(U) -\beta P,
\end{cases}
\end{align}
a traveling profile corresponds to a trajectory in the $(U,P)$-plane connecting the equilibria $(0,0)$ and $(1,0)$ with $P > 0$. 

Linearizing the system at a point $(U,P)$ we obtain the Jacobian matrix
\begin{equation*}
D(U,P)=\begin{pmatrix}
0 & 1 \\
- f'(U) & - \beta
\end{pmatrix},
\end{equation*}
whose eigenvalues are \newglossaryentry{lambdapm}{name=\ensuremath{\lambda_{\pm}(\beta,U)},description={eigenvalues of the linearized system \eqref{sys}}}
\begin{align*}
\gls{lambdapm} := - \frac{\beta}{2} \pm \sqrt{\frac{\beta^2}{4}-f'(U)}.
\end{align*}
We can therefore distinguish the following cases with respect to the stability of the fixed points $(0,0)$, $(U^*,0)$, and $(1,0)$.

\begin{enumerate}
\item When $f'(U)<0$, the eigenvalues are real and have opposite signs for all $\beta \in \R$: in particular the equilibria $(0,0)$ and $(1,0)$ are saddles.

\item If $f'(U)>0$ and $\beta <- 2 \sqrt{f'(U)}<0 $, the eigenvalues are both real and positive: the equilibrium $(U^*,0)$ is an unstable node.

\item If $f'(U)>0$ and $- 2\sqrt{f'(U)}<\beta < 0$, $\lambda_{\pm}(\beta, U)$ are complex conjugate, both with positive real part: the equilibrium $(U^*,0)$ is an unstable focus.

\item For $\beta = 0$ the system is Hamiltonian, so that $(U^*,0)$ is a center.

\item \label{Point_5:eigen_equi} If $f'(U)>0$ and $0 < \beta < 2\sqrt{f'(U)}$, $\lambda_{\pm}(\beta, U)$ are complex conjugate, both with negative real part: the equilibrium $(U^*,0)$ is a stable focus.

\item If $f'(U)>0$ and $\beta > 2\sqrt{f'(U)}>0 $, the eigenvalues are both real and negative: the equilibrium $(U^*,0)$ is a stable node.
\end{enumerate}

Therefore, for $P>0$ the phase portrait for the system \eqref{sys} contains an unstable manifold $\gls{Gammau} = \Gamma_u(\beta,f,U)$ of the point $(0,0)$ and a stable manifold $\gls{Gammas} = \Gamma(\beta,f,U)$ of the point $(1,0)$. 
Moreover (\cite[Chapter 4.4.a]{Fife}) there exists a unique value 
$\gls{betastar} = \beta^*(f)$ (see Remark \ref{Rem:positiv_perv}) such that $\Gamma_u(\beta^*)=\Gamma_s(\beta^*)$: thus, we can find a unique (up to translation) traveling profile $U$ with speed $\beta^*$. 

We write $\Gamma_u(\beta)$ as a function $\Gamma_u(\beta,U)$ in the region \newglossaryentry{Uufbeta}{name=\ensuremath{U_{uf}(\beta)},description={first point where $\Gamma_u(U,\beta) = 0$}} $[0,\gls{Uufbeta} = U_{uf}(\beta,f)]$ where $\Gamma_u(\beta) \geq 0$, with $\Gamma_u(\beta,0) = \Gamma_u(\beta,U_{uf}(\beta)) = 0$ and strictly positive inside the interval; similarly we write $\Gamma_s(\beta) = \Gamma_s(\beta,U)$, which it is easy to see that it is a function for $\beta \geq \beta^*$.

Point \eqref{Point_5:eigen_equi} above the fact that for $\beta \to \infty$ it holds $\Gamma_u(\beta,U) \to 0$ yields the existence of a value $\gls{betastarstar} = \beta^{**}(f) \geq 2\sqrt{f'(U^*)}$ such that for $U_{uf}(\beta^{**}) = U^*$. From this value on $U_{uf}(\beta) = U^*$.

\begin{remark}
\label{Rem:ether_1}
The value of $\beta^{**}$ corresponds to the "etheroclinic connection" where $\Gamma_u$ coincides with the "strongly" stable manifold defined by the eigenvalue $\lambda_-(\beta,U^*)$.
\end{remark}

See Fig. \ref{Fig:fU_Gammasu} for a graphical representation of the manifolds $\Gamma_u(\beta),\Gamma_s(\beta)$ for different values of $\beta$. 

\begin{figure}
\resizebox{\textwidth}{!}{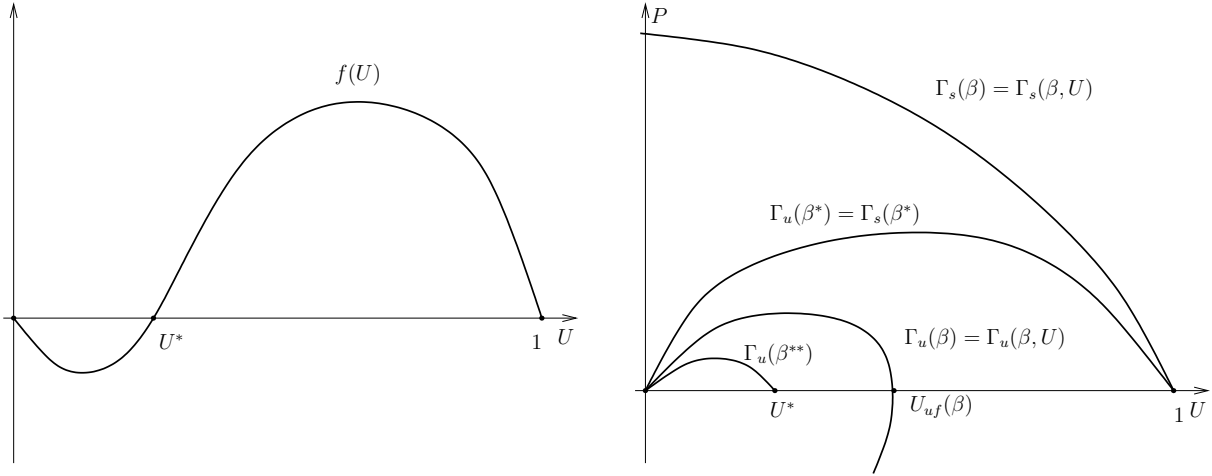}
\caption{The plot of the source $f(U)$ and of the stable and unstable manifolds for different speeds $\beta$.}
\label{Fig:fU_Gammasu}
\end{figure}

\subsection{Traveling waves with control}
\label{Ss:travelin_alpha}

Next, consider the model for PDE with control \newglossaryentry{alphat}{name=\ensuremath{\tilde \alpha(t,x)},description={control for the PDE \eqref{eqwithcontrol}}}
\begin{equation}
\label{eqwithcontrol}
U_t = U_{xx} + f(U) - \gls{alphat} U.
\end{equation}
The control $\tilde \alpha(t,x)$ is assumed to be positive, which means that our goal is to diminish $U$, and it allows to construct traveling profiles with speed $\beta > \beta^*$: for a given $\beta > \beta^*$, we look for a control $\tilde \alpha(s)$ and a function $U = U(s)$ such that $U(t,x) := U(x-\beta t), \ \tilde \alpha(t,x) = \tilde \alpha(x-\beta t)$ solve \eqref{eqwithcontrol}.

\begin{definition}[Admissible profiles]
\label{admcurves}
Given $\beta > \beta^*$, we denote by  $\gls{Acalbeta} = \mathcal A_{\beta,f}$ the set of all Lipschitz curves
$$
s \mapsto \gamma(s)=(U(s), P(s)) \in (0,1) \times (0,+\infty)
$$
which are graphs of solutions $(U(s),P(s))$ for the planar system
\begin{align}
\label{syswithcontrol}
\begin{cases}
U'= P, \\
P'= -\beta P - f(U) + \tilde \alpha U,
\end{cases} \quad (U,P)(-\infty) = (0,0), \ (U,P)(+\infty) = (1,0),
\end{align}
for some admissible control $\tilde \alpha \ge 0$.
\end{definition}

Given an admissible curve $\gamma \in \mathcal{A(\beta)}$, its associated cost \newglossaryentry{Etgamma}{name=\ensuremath{\tilde E(\gamma)},description={effort associated to an admissible curve $\gamma$}} \gls{Etgamma} can be computed as (see \cite{BressanChiri22})
\begin{equation*}
\gls{Etgamma} = \int_\R \alpha(x) dx = \int_{\gamma} \mathbf{v},
\end{equation*}
with
\begin{equation*}
\mathbf{v} := \left( \frac{f(U)}{UP} + \frac{\beta}{U}\right)dU + \left(\frac{1}{U}\right) dP,
\end{equation*}
i.e. a path integral along $\gamma$.

\begin{definition}[Effort function]
\label{Def:effort_funvc}
We define the \emph{effort function} $\gls{Ebeta} = E(\beta,f)$ as
\begin{equation}
\label{effort}
\begin{array}{ccccc}
E &\colon& [\beta^*, +\infty) &\to& \R^+ \\
&& \beta &\mapsto& E(\beta) := \inf \{ \tilde E(\gamma), \gamma \in \mathcal{A}_\beta\}
\end{array}
\end{equation}
\end{definition}

We can assume that $E(\beta) = +\infty$ for $\beta < \beta^*$, in the sense that there are no admissible solutions.

Being $P > 0$ if $\gamma \in \mathcal{A}_\beta$, it follows that $s \mapsto U(s)$ is invertible in $I$ and we can express $P$ as a function of $U$, solving 
\begin{equation*}
\frac{dP}{dU} = - \beta - \frac{f(U)}{P} + \alpha(U) U, \quad \gls{alphaU} = \frac{\tilde \alpha(s(U))}{P(U)},
\end{equation*}
with
\begin{equation*}
\int_I \tilde \alpha(s) \, ds = \int_{U(I)} \frac{\tilde \alpha(s(U))}{P(U)} \, dU=\int_{U(I)} \alpha(U) \, dU.
\end{equation*}
We therefore can use the cost $\alpha(U)$ for traveling profiles when computing the cost: this will remove the degeneracy of the invariance of the problem w.r.t. translations. If $\alpha = 0$ we obtain the ODE
\begin{equation}
\label{functeq}
\frac{dP}{dU} = - \frac{f(U)}{P(U)} - \beta.
\end{equation}

%


\section{Stability estimates for the manifolds $\Gamma_u$, $\Gamma_s$}
\label{Ss:stab_sti_Gamma_us}

In this section, we prove some useful results on the dependence of the unstable manifold $\Gamma_u$ and the stable manifold $\Gamma_s$ with respect to $\beta$ and the source $f$. It is well known that the stable and unstable manifolds depend smoothly on the vector field; our aim is to provide short proofs of some explicit estimates which will be used throughout the rest of the paper. A critical case is the "etheroclinic" connection at $\beta^{**}$, where the dependence may fail to be smooth.

In the variable $(U,P)$, the manifolds $\Gamma_u,\Gamma_s$ solve the ODE
\begin{equation}
\label{Equa:Gammau_eq0}
\frac{d \Gamma_u}{dU} = - \frac{f(U)}{\Gamma_u} - \beta, \quad \lim_{U \to 0} \frac{\Gamma_u(U)}{U} = \lambda_+(\beta,0),
\end{equation}
\begin{equation}
\label{Equa:Gammas_eq0}
\frac{d \Gamma_s}{dU} = - \frac{f(U)}{\Gamma_s(U)} - \beta, \quad \lim_{U \to 1} \frac{\Gamma_s(U)}{U} = \lambda_-(\beta,1).
\end{equation}
The limits as $U \to 0$, $U \to 1$ are required by the property of being tangent to the unstable, stable eigenspace of the equilibrium $U=0$, $U=1$ respectively.

The first result we prove is that the same tangency property holds as $U \to 0$ for solutions to
\begin{equation}
\label{Equa:funct_eq_UP}
\frac{dP}{dU} = - \frac{f(U)}{P} - \beta + \alpha(U) U, \quad \alpha(U) \geq 0,
\end{equation}
with the constraints
\begin{equation}
\label{Equa:functio_eq_0_constr}
\lim_{U \to 0} P(U) = 0, \quad \lim_{U \to 1} P(U) = 1, \quad \int_{(0,1)} \alpha(dU) < +\infty.
\end{equation}

\begin{lemma}
\label{Lem:lim_U_to_0}
If $P$ is positive solution to \eqref{Equa:funct_eq_UP}, \eqref{Equa:functio_eq_0_constr}, then
\begin{equation*}
\lim_{U \to 0} \frac{P(U)}{U} = \lim_{U \to 0} \frac{\Gamma_u(\beta,U)}{U} = \lambda_+(\beta,0).
\end{equation*}
\end{lemma}

The same result can be proved when $\alpha$ is a bounded measure on $(0,1)$, not necessarily positive.

\begin{proof}
Write
\begin{equation}
\label{Equa:ODE_scee}
\begin{split}
\frac{d}{dU} \bigg( \frac{P(U) - \Gamma_u(\beta,U)}{U} \bigg) &= - \frac{P(U) - \Gamma_u(\beta,U)}{U^2} + \frac{f(U)}{U \Gamma_u(\beta,U)} - \frac{f(U)}{U P} + \alpha(U) \\
&= - \frac{1 - \frac{f(U) U}{P \Gamma_u(\beta,U)}}{U} \frac{P(U) - \Gamma_u(\beta,U)}{U} + \alpha(U).
\end{split}
\end{equation}
The kernel $\kappa = - \frac{1 - \frac{f U}{P \Gamma_u}}{U}$ near $U = 0$ is estimated by
\begin{equation*}
\kappa(\beta,U) = - \frac{1 - \frac{f(U) U}{P(U) \Gamma_u(\beta,U)}}{U} \underset{f \leq 0}{\leq} - \frac{1}{U},
\end{equation*}
so that in particular
\begin{equation}
\label{Equa:kernel_estt}
e^{\int_\epsilon^U \kappa(\beta,U') dU'} \leq \frac{\epsilon}{U}, \quad 0 < \epsilon \leq U.
\end{equation}
For the solution of \eqref{Equa:ODE_scee} we obtain
\begin{equation*}
\begin{split}
\bigg| \frac{P(U) - \Gamma_u(\beta,U)}{U} \bigg| &\leq e^{\int_\epsilon^U \kappa(\beta,U') dU'} \bigg| \frac{P(U) - \Gamma_u(\beta,\epsilon)}{\epsilon} \bigg| + \int_\epsilon^U e^{\int_{U''}^U \kappa(\beta,U') dU'} \alpha(dU'') \\
\big[ (\ref{Equa:kernel_estt}) \big] \quad &\leq \frac{\epsilon}{U} \bigg| \frac{\Gamma_u(\beta,\epsilon) - P(\epsilon)}{\epsilon} \bigg| + \int_\epsilon^U \frac{U''}{U} \alpha(dU'') \\
&= \bigg| \frac{\Gamma_u(\beta,\epsilon) - P(\epsilon)}{U} \bigg| + \int_\epsilon^U \frac{U''}{U} \alpha(dU'').
\end{split}
\end{equation*}
Letting $\epsilon \to 0$ we obtain the bound
\begin{equation*}
\bigg| \frac{\Gamma_u(\beta,U) - P(U)}{U} \bigg| \leq \int_0^U \frac{U''}{U} \alpha(dU'') \leq \alpha((0,U)).
\end{equation*}
Hence we conclude that
\begin{equation*}
\lim_{U \to 0} \frac{P(U)}{U} = \lim_{U \to 0} \frac{\Gamma_u(\beta,U)}{U} = \lambda_+(\beta,0). \qedhere
\end{equation*}
\end{proof}

Now we address the dependence w.r.t. $\beta,f$ of the manifolds $\Gamma_u,(\beta,U)$, $\Gamma_s(\beta,U)$. It is elementary from \eqref{Equa:Gammau_eq0}, \eqref{Equa:Gammas_eq0} to see that the functions
$$
\beta,f \mapsto \Gamma_u(\beta,f,U), \ U_{uf}(\beta,f)
$$
are monotonically decreasing, as well as
$$
\beta,f \mapsto \Gamma_s(\beta,f,U)
$$
is monotonically increasing and always defined in $[0,1]$ if $\beta \geq \beta^*(f)$. 

For the control problem we are consider, the intervals of interests of $\Gamma_u,\Gamma_s$ are $[0,\gls{Uubeta}]$, $[\gls{Usbeta},1]$, where $U_u(\beta)$ ($U_s(\beta)$) is the first (last) intersection points of $\Gamma_u(\beta)$ ($\Gamma_s(\beta)$) with the curve
\begin{equation*}
\gls{Pstar} = \sqrt{Uf(U) \ind_{[U^*,1]}(U)}.
\end{equation*}
We recall that for $\beta \geq \beta^{**}$ it holds $U_u(\beta) = U^*$; observe also that for $\beta < \beta^{**}$ $U_u(\beta) > U^*$ and $\Gamma_u(U_u) > 0$. Moreover, since the derivative of $P^*$ in $U=1$ is $-\infty$, the point $U_s(\beta)$ is well defined. When we need to emphasize the dependence on $f$ we will write $U_u(\beta,f)$, $U_s(\beta,f)$, or for shortness $U_u(f)$, $U_s(f)$ when $\beta$ is kept constant.

Theorem \ref{manifoldbound} will be used repeatedly throughout the paper: its main point is that the derivatives are uniform even if $\beta$ is unbounded and $\Gamma_u(\beta,U)$ is close to $U^*$, the latter neighborhood being critical when the connection of Remark \ref{Rem:ether_1} happens.

We will need the following lemma.

\begin{lemma}
\label{Lem:large_beta_Gammau}
Define the constants
\begin{equation}
\label{Equa:def_M0_MU*}
M_0^+ = \max_{[0,U^*]} \bigg\{ \frac{f'(U)}{f'(0)} \bigg\} \geq 1, \quad M_{0}^- = \max_{[0,U^*]} \bigg\{ \frac{- f'(U)}{f'(U^*)} \bigg\}.
\end{equation}
Then it holds for $U \in [0,U^*]$
\begin{equation}
\label{Equa:lower_bdd_a}
\Gamma_u(\beta,U) \geq \bigg( - \frac{\beta}{2M_0^+} + \sqrt{ \frac{\beta^2}{4 (M_0^+)^2} - \frac{f'(0)}{M_0^+}} \bigg) \frac{f(U)}{f'(0)}, \quad \beta \geq \beta^*,
\end{equation}
and
\begin{equation*}
\Gamma_u(\beta,U) \leq \bigg( \frac{\beta}{2M_0^-} - \sqrt{\frac{\beta^2}{4 (M_0^-)^2} + \frac{f'(0)}{M_0^-}} \bigg) \frac{f(U)}{f'(0)},
\end{equation*}
for
\begin{equation*}
\beta \geq 2 \sqrt{- M_{0}^- f'(0)} = 2 \sqrt{\max_{[0,U^*]} \{f'(U)\}}.
\end{equation*}
\end{lemma}

\begin{proof}
Consider the function
\begin{equation*}
P_a(U) = a \frac{f(U)}{f'(0)}, \quad a > 0, \ U \in [0,U^*].
\end{equation*}
If
\begin{equation}
\label{Equa:a_subsol}
a \frac{f'(U)}{f'(0)} < - \frac{f'(0)}{a} - \beta,
\end{equation}
then $P_a$ is a subsolution to \eqref{functeq}. Recalling $M_0^+$ given by \eqref{Equa:def_M0_MU*}, the inequality \eqref{Equa:a_subsol} requires that
\begin{equation*}
M_0^+ a^2 + \beta a + f'(0) < 0, \quad 0 < a < - \frac{\beta}{2M_0^+} + \sqrt{ \frac{\beta^2}{4 (M_0^+)^2} - \frac{f'(0)}{M_0^+}}.
\end{equation*}
We need to check that the inequality holds near $U = 0$. Since $M_0^+ \geq 1$, then
\begin{equation*}
- \frac{\beta}{2M_0^+} + \sqrt{ \frac{\beta^2}{4 (M_0^+)^2} - \frac{f'(0)}{M_0^+}} = \frac{\lambda_+(\frac{\beta}{\sqrt{M_0^+}},0)}{\sqrt{M_0^+}} \leq - \frac{\beta}{2} + \sqrt{ \frac{\beta^2}{4} - f'(0)} = \lambda_+(\beta,0)
\end{equation*}
because
\begin{equation*}
\partial_{M_0^+} \bigg( - \frac{\beta}{2M_0^+} + \sqrt{\frac{\beta^2}{4 (M_0^+)^2} - \frac{f'(0)}{M_0^+}} \bigg) < 0, 
\end{equation*}
and we conclude that $\Gamma_u \geq a \frac{f}{f'(0)}$ for all $a < \frac{\lambda_+(\frac{\beta}{\sqrt{M_0^+}},0)}{\sqrt{M_0^+}}$. Letting $a \nearrow \frac{\lambda_+(\frac{\beta}{\sqrt{M_0^+}},0)}{\sqrt{M_0^+}}$ we obtain \eqref{Equa:lower_bdd_a}.

In a similar way, $P(U) = a \frac{f(U)}{f'(0)}$ is a supersolution if
\begin{equation*}
a \frac{f'(U)}{f'(0)} > - \frac{f'(0)}{a} - \beta,
\end{equation*}
which requires
\begin{equation*}
- a^2 M_0^- + \beta a + f'(0) > 0, \quad M_0^- = \max_{[0,U^*]} \bigg\{ \frac{-f'(u)}{f'(0)} \bigg\}.
\end{equation*}
Hence we obtain
\begin{equation*}
\frac{\beta}{2M^-_0} - \sqrt{ \frac{\beta^2}{4 (M^-_0)^2} + \frac{f'(0)}{M^-_0}} < a < \frac{\beta}{2M^-_0} + \sqrt{ \frac{\beta^2}{4 (M^-_0)^2} + \frac{f'(0)}{M^-_0}}.
\end{equation*}
For $\beta > 2 \sqrt{- M^-_0 f'(0)}$ it holds
\begin{equation*}
\begin{split}
\lambda_+(\beta,0) &\leq \frac{\beta}{2M^-_0} - \sqrt{ \frac{\beta^2}{4 (M^-_0)^2} + \frac{f'(0)}{M^-_0}}
\end{split}
\end{equation*}
because
\begin{equation*}
- \lambda_+(\beta,0)^2 M_0^- + \lambda_+(\beta,0) \beta + f'(0) = - \lambda_+(\beta,0)^2 (M_0^- + 1) < 0.
\end{equation*}
In particular for $\beta > 2 \sqrt{-M^-_0 f'(0)}$ we have
\begin{equation*}
- \frac{\beta}{2M^+_0} + \sqrt{ \frac{\beta^2}{4 (M^+_0)^2} - \frac{f'(0)}{M^+_0}} \leq \lambda_+(\beta,0) \leq \frac{\beta}{2M^-_0} - \sqrt{ \frac{\beta^2}{4 (M^-_0)^2} + \frac{f'(0)}{M^-_0}}. \qedhere
\end{equation*}
\end{proof}

We can now prove the main theorem of this section.

\begin{theorem}
\label{manifoldbound}
Let $\beta \ge \beta^*$ and $f \in C^2([0,1])$ satisfying hypothesis $(\ref{H1})$.

\begin{enumerate}
\item Derivative w.r.t. $\beta$: it holds
\begin{equation}
\label{Equa:Gammau_beta}
- 1 - \bigg[ \ln \bigg(\frac{U}{U^*} \bigg) \bigg]^+ < \frac{\partial_\beta \Gamma_u(\beta,U)}{U} < 0, \quad U \in (0,U_u(\beta)),
\end{equation}
\begin{equation}
\label{Equa:Gammas_beta}
0 < \frac{\partial_\beta \Gamma_s(\beta,U)}{1 - U} < 1, \quad U \in [U^*,1).
\end{equation}

\item 
Derivative w.r.t. $f$: there are two constants \newglossaryentry{rbar}{name=\ensuremath{\underline{r}},description={radius of the ball around the equilibria $(0,0)$, $U^*,0)$, $(1,0)$ where the perturbative analysis of the ODE is performed}} \newglossaryentry{betaunder}{name=\ensuremath{\underline{\beta}},description={speed where $U_u$ is sufficiently close to $U^*$}}
\begin{equation*}
\gls{rbar} = \underline{r}(\|f\|_{C^2}), \quad \gls{betaunder} = \underline{\beta}(\|f\|_{C^2}) < \beta^{**}(f)
\end{equation*}
and a constant \newglossaryentry{Cfbeta}{name=\ensuremath{C(f)},description={Lipschitz constant for the dependence of $\Gamma_u,\Gamma_s$ w.r.t. $f \in C^1([0,1])$}} $\gls{Cfbeta} > 0$ depending only on
\begin{equation*}
\begin{split}
&\|f\|_{C_2}, \quad \max \Big\{ \Gamma_u(U)^{-1}, \beta \in [\beta^*,\underline{\beta}], U \in [\bar r,U_u] \Big\}, \\
&\max \bigg\{ \Gamma_u(U)^{-1}, \beta \in \Big[ \beta^{*},2 \sqrt{\max_{[0,U^*]} f'(U)} \Big], U \in [\underline{r},U_u - \underline{r}] \bigg\},
\end{split}
\end{equation*}
such that if $\Gamma_u(U,f + \epsilon \delta f)$, $\Gamma_s(U,f+\epsilon \delta f)$ are the unstable, stable manifold for the source $f + \epsilon \delta f$ satisfying \eqref{H1} with $\delta f \geq 0$, then 
\begin{equation}
\label{Equa:estimat_f_uGG}
0 \le - \lim_{\epsilon \to 0} \frac{\Gamma_u(U,f+\epsilon \delta f) - \Gamma_u(U,f)}{U} \le C(f) \|\delta f\|_{C^1},
\end{equation}
for all $U \leq U_u(f)$, and
\begin{equation}
\label{Equa:estimat_f_sGG}
0 \le \lim_{\epsilon \to 0} \frac{\Gamma_s(U,f+\epsilon \delta f) - \Gamma_s(U,f)}{1 - U} \le C(f) \|\delta f\|_{C^1},
\end{equation}
for all $U \in (U^*,1)$.
\end{enumerate}
\end{theorem}


\begin{proof}
The proof is done in several steps. Step 1.1 and 1.2 concern the dependence w.r.t. $\beta$ of $\Gamma_u$ and $\Gamma_s$ respectively. Steps 2.1 to 2.6 analyze the dependence of $\Gamma_u$ w.r.t. $f$: the most difficult part is for $\beta$ close to the critical value $\beta^{**}$, where the asymptotic behavior of $\Gamma_u$ changes abruptly. Finally the analysis of $\Gamma_s$ concludes the proof.

\medskip

\noindent {\it Step 1: dependence w.r.t. $\beta$.} We first study the dependence w.r.t. $\beta$ for $\Gamma_u(\beta)$, $\Gamma_s(\beta)$. 

\smallskip

{\it Step 1.1: analysis of $\Gamma_u(\beta)$.} Recall that the ODE for $\Gamma_u$ is
\begin{equation*}
\dot \Gamma_u = - \frac{f(U)}{\Gamma_u} - \beta, \quad \lim_{U \to 0} \frac{\Gamma_u}{U} = \lambda_+(\beta,0) = - \frac{\beta}{2} + \sqrt{ \frac{\beta^2}{4} - f'(0)}.
\end{equation*}
%
Thus the perturbation $\delta \Gamma_u = \partial_\beta\Gamma_u(\beta,U)$ solves
\begin{equation*}
\dot{\delta \Gamma_u} = \frac{f(U)}{\Gamma_u^2} \delta \Gamma_u - 1, \quad \lim_{U \to 0} \frac{\delta \Gamma_u}{U} = \partial_U \lambda_+(\beta,0) = \frac{\beta}{4 \sqrt{ \frac{\beta^2}{4} - f'(0)}} - \frac{1}{2}.
\end{equation*}
The solution is given by the formula
\begin{equation}
\label{Equa:formula_Gammabeta}
\delta \Gamma_u(\beta,U) = \delta \Gamma_u(\beta,\epsilon) e^{\int_\epsilon^U \frac{f(W)}{\Gamma_u(\beta,W)^2} dW} - \int_\epsilon^U e^{\int_V^U \frac{f(W)}{\Gamma_u(\beta,W)^2} dW} dV,
\end{equation}
and since 
\begin{equation*}
\frac{f(W)}{\Gamma_u(\beta,W)^2} \sim\frac{f'(0)}{\lambda_+(\beta,0)^2 W} \quad \text{as $W \to 0$},
\end{equation*}
we deduce that
\begin{equation}
\label{Equa:sol_der_beta_u}
\delta \Gamma_u(\beta,U) = - \int_0^U e^{\int_V^U \frac{f(W)}{\Gamma_u(\beta,W)^2} dW} dV = - \int_0^U \frac{\Gamma_u(V)}{\Gamma_u(U)} e^{- \int_V^U \frac{\beta}{\Gamma_u(W)} dW} dV.
\end{equation}
Since $f(U) < 0$ in $(0,U^*)$, for $U \in (0,U^*)$ we obtain the estimate
\begin{equation}
\label{Equa:delta_Gamma_u}
0 > \delta \Gamma_u(\beta,U) > - \int_0^{U} 1\, dV = - U. 
\end{equation}
Using
\begin{equation*}
\frac{d}{dU} \frac{\delta \Gamma_u}{U} = \frac{U f(U) - \Gamma_u^2}{U \Gamma_u^2} \frac{\delta \Gamma_u}{U} - \frac{1}{U},
\end{equation*}
we obtain also for $U \in [U^*,U_u]$
\begin{equation*}
\begin{split}
\frac{\delta \Gamma_u(U)}{U} &= \frac{\delta \Gamma_u(U^*)}{U^*} e^{\int_{U^*}^U \frac{W f(W) - \Gamma_u(W)^2}{W \Gamma_u(W)^2} dW} - \int_{U^*}^U \frac{1}{V} e^{\int_V^U \frac{W f(W) - \Gamma_u(W)^2}{W \Gamma_u(W)^2} dW} dV \\
&\geq - 1 - \ln \bigg( \frac{U}{U^*} \bigg). 
\end{split}
\end{equation*}
being $\Gamma_u \geq P^* = \sqrt{Uf(U)}$ in the interval under consideration. Hence \eqref{Equa:Gammau_beta} is proved.

\smallskip

{\it Step 1.2: analysis of $\Gamma_s(\beta)$.} The analysis for $\Gamma_s(\beta,U)$ is entirely similar. The ODEs are the same, with limits
\begin{equation*}
\lim_{U \to 1} \frac{\Gamma_s}{1 - U} = - \lambda_-(\beta,1), \quad \lim_{U \to 1} \frac{\delta \Gamma_s}{1 - U} = - \partial_\beta \lambda_-(\beta,1).
\end{equation*}
The formula for the solution is
\begin{equation*}
\delta \Gamma_s(\beta,U) = - \int_1^U e^{\int_V^U \frac{f(W)}{\Gamma_u(\beta,w)^2} dW} dV = \int_U^1 e^{- \int_U^V \frac{f(W)}{\Gamma_u(\beta,w)^2} dW} dV,
\end{equation*}
and since $f(U) > 0$ in $(U^*,1)$ we obtain
\begin{equation*}
0 < \delta \Gamma_s(\beta,U) \leq 1 - U, 
\end{equation*}
which is \eqref{Equa:Gammas_beta}.

\medskip

\noindent{\it Step 2: dependence w.r.t. $f$.} 
First we study $\Gamma_u(\beta,f,U)$. The equation for the perturbation $\delta \Gamma_u(\beta,U)$ of $\Gamma_u(\beta,u)$ when $f \to f + \delta f$ satisfies
\begin{equation*}
\frac{d \delta \Gamma_u}{dU} = \frac{f(U)}{\Gamma_u^2} \delta \Gamma_u - \frac{\delta f(U)}{\Gamma_u}, \quad \lim_{U \to 0} \frac{\delta \Gamma_u(\beta,U)}{U} = \partial_{f'(0)} \lambda_+(\beta,0).
\end{equation*}
As in the previous proof, noticing that
\begin{equation*}
\frac{f(U)}{\Gamma_u^2} \underset{U \to 0}{\sim} \frac{f'(0)}{\lambda_+(\beta,0)^2 U},
\end{equation*}
we obtain the representation
\begin{equation}
\label{Equa:deltaGamma_uu}
\delta \Gamma_u(U) = - \int_0^U \frac{\delta f(V)}{\Gamma_u(V)} e^{\int_V^U \frac{f(W)}{\Gamma_u(W)^2} dW} dV.
\end{equation}
Using the ODE \eqref{Equa:Gammau_eq0} for $\Gamma_u(\beta,U)$ we can write
\begin{equation*}
\int_V^U \frac{f(W)}{\Gamma_u(W)^2} dW = \int_V^U \bigg( - \frac{\frac{d\Gamma_u}{dW}}{\Gamma_u(W)} - \frac{\beta}{\Gamma_u(W)} \bigg) dW = \ln \bigg( \frac{\Gamma_u(V)}{\Gamma_u(U)} \bigg) - \int_V^U \frac{\beta}{\Gamma_u(W)} dW,
\end{equation*}
so that we also obtain
\begin{equation}
\label{Equa:deltaGamma_u}
\frac{\delta \Gamma_u(U)}{U} = - \int_0^U \frac{\delta f(V)}{U \Gamma_u(U)} e^{- \int_V^U \frac{\beta}{\Gamma_u(W)} dW} dV.
\end{equation}

In both cases $\delta \Gamma_u \leq 0$, so that the sign in \eqref{Equa:estimat_f_uGG} is correct.

\smallskip

{\it Step 2.1: estimates for $\beta \gg 2\sqrt{\max_{[0,U^*]} f'(U)}$.} 
Using Lemma \ref{Lem:large_beta_Gammau}, for $\beta \geq 2 \sqrt{\max_{[0,U^*]} f'(U)}$ we estimate \eqref{Equa:deltaGamma_uu} as 
\begin{equation*}
\begin{split}
\bigg| \frac{\delta \Gamma_u(U)}{U} \bigg| &\leq \frac{\|\delta f\|_{C^0([0,U])}}{U} \int_0^U \frac{f'(0)}{\Big( - \frac{\beta}{2M_0^+} + \sqrt{\frac{\beta^2}{4 (M_0^+)^2} - \frac{f'(0)}{M_0^+}} \Big) f(V)} e^{\int_V^U \frac{f'(0)^2}{\Big( \frac{\beta}{2M_0^-} - \sqrt{\frac{\beta^2}{4 (M_0^-)^2} + \frac{f'(0)}{M_0^-}} \Big)^2 f(W)} dW} dV \\
&= \frac{\|\delta f\|_{C^0([0,U])}}{U} \frac{\Big( \frac{\beta}{2M_0^-} - \sqrt{\frac{\beta^2}{4 (M_0^-)^2} + \frac{f'(0)}{M_0^-}} \Big)^2}{\Big( - \frac{\beta}{2M_0^+} + \sqrt{\frac{\beta^2}{4 (M_0^+)^2} - \frac{f'(0)}{M_0^+}} \Big) f'(0)} \\
& \quad \quad \cdot \int_0^U \frac{f'(0)^2}{\Big( \frac{\beta}{2M_0^-} - \sqrt{\frac{\beta^2}{4 (M_0^-)^2} + \frac{f'(0)}{M_0^-}} \Big)^2 f(V)} e^{\int_V^U \frac{f'(0)^2}{\Big( \frac{\beta}{2M_{U^*}} - \sqrt{\frac{\beta^2}{4 M_{U^*}^2} - \frac{f'(U^*)}{M_{U^*}}} \Big)^2 f(W)} dW} dV \\
&= \frac{\|\delta f\|_{C^0([0,U])}}{U} \frac{\Big( \frac{\beta}{2M_0^-} - \sqrt{\frac{\beta^2}{4 (M_0^-)^2} + \frac{f'(0)}{M_0^-}} \Big)^2}{\Big( - \frac{\beta}{2M_0^+} + \sqrt{\frac{\beta^2}{4 (M_0^+)^2} - \frac{f'(0)}{M_0^+}} \Big) f'(0)} \cdot \left. - e^{\int_V^U \frac{f'(U^*)^2}{\Big( \frac{\beta}{2M_0^-} - \sqrt{\frac{\beta^2}{4 (M_0^-)^2} + \frac{f'(0)}{M_0^-}} \Big)^2 f(W)} dW} \right|_0^U \\
&\leq C \frac{\|\delta f\|_{C^1}}{\beta},
\end{split}
\end{equation*}
where the constant $C$ depends only on $f'(0),M_0^-,M_0^+$, and we observed that as $\beta \to \infty$
\begin{equation*}
- \frac{\beta}{2M_0^+} + \sqrt{\frac{\beta^2}{4 (M_0^+)^2} - \frac{f'(0)}{M_0^+}}, \frac{\beta}{2M_0^-} - \sqrt{\frac{\beta^2}{4 (M_0^-)^2} + \frac{f'(0)}{M_0^-}} \sim \frac{1}{\beta}.
\end{equation*}

In subsequent points we need only to consider the case for $\beta$ bounded, i.e. $\beta^* \leq \beta \leq 2 \sqrt{\displaystyle{\max_{[0,U^*]} f'(U)}}$. 

\smallskip

{\it Step 2.2: estimates for $U$ close to $0$.} By the stability of the unstable manifold \cite[Theorem 6.2.3]{katokhass}, there is a neighborhood $U = 0$ such that
\begin{equation}
\label{Equa:expans_U0}
\big| \Gamma_u(\beta,U) - \lambda_+(\beta,0) U \big| \leq \mathcal O(\|f\|_{C^2}) U^2,
\end{equation}
and the radius $\gls{rbar} = r(\|f\|_{C^2})$ depends only on the $C^2$-norm of $f$ for $\beta \in [\beta^*,2 \sqrt{\displaystyle{\max_{[0,U^*]} f'(U)}}]$. 
By \eqref{Equa:expans_U0}
\begin{equation*}
    \Gamma_u(\beta, U) \ge \lambda_+(\beta, 0)U-\mathcal{O}(\|f\|_{C^2}) U^2, \quad \forall U\in [0, \underline{r}],
\end{equation*}
so that for all $0\le V\le U \le \underline{r}$
\begin{align*}
\int_{V}^U \frac{1}{\Gamma_u(W)} \, dW &\le \int_{V}^U \frac{1}{\lambda_+(\beta,0) W - \mathcal O(\|f\|_{C^2}) W^2} \\
&= \frac{1}{\lambda_+(\beta,0)} \ln \bigg( \frac{U}{V} \bigg) + \frac{1}{\lambda_+(\beta,0)} \ln \bigg( \frac{1 - \mathcal O(\|f\|_{C^2}) V}{1 - \mathcal O(\|f\|_{C^2}) U} \bigg),
\end{align*}
which implies, for $\beta <0$,
\begin{align*}
e^{-\int_V^U \frac{\beta}{\Gamma_u(W)} \, dW} \le C(\|f\|_{C^2}) \bigg( \frac{V}{U} \bigg)^{\frac{\beta}{\lambda_+(\beta,0)}}.
\end{align*}
Substituting into \eqref{Equa:deltaGamma_u} and observing that
\begin{equation}
\label{Equa:lambda_+0_beta}
- \lambda_+(\beta,0) < \beta,
\end{equation}
we obtain 
\begin{equation*}
\begin{split}
\bigg| \frac{\delta \Gamma_u(U)}{U} \bigg| &\le \frac{C(\|f\|_{C^2}) \|\delta f\|_{C^0([0,U])}}{\Gamma_u(U) U} \int_0^U \bigg( \frac{V}{U} \bigg)^{\frac{\beta}{\lambda_+(\beta,0)}} \, dV \\
&= \frac{C(\|f\|_{C^2}) \|\delta f\|_{C^0([0,U])}}{\beta + \lambda_+(\beta,0)} \frac{\lambda_+(\beta,0)}{\Gamma_u(U)} \\
&\leq \frac{2 C(\|f\|_{C^2})}{\frac{\beta}{2} + \sqrt{\frac{\beta^2}{4} - f'(0)}} \frac{\|\delta f\|_{C^0([0,U])}}{U} \leq \frac{2 C(\|f\|_{C^2})}{\frac{\beta^*}{2} + \sqrt{\frac{\beta^*}{4} - f'(0)}} \|\delta f\|_{C^{0,1}}. 
\end{split}
\end{equation*}

The case $\beta\geq0$ is simpler as there holds
\begin{align*}
e^{-\int_V^U \frac{\beta}{\Gamma_u(W)} \, dW} \le 1 \quad \forall 0\leq V\leq U \leq \underline{r},
\end{align*}
and, thus, 
\begin{equation*}
\begin{split}
\bigg| \frac{\delta \Gamma_u(U)}{U} \bigg| &\le \frac{\|\delta f\|_{C^0([0,U])}}{\Gamma_u(U)}\\
    &\le \frac{1}{\lambda_+(\beta, 0)-\mathcal{O}(\|f\|_{C^2} )}\frac{\|\delta f\|_{C^0([0,U])}}{U}\\
    &\le C(\|f\|_{C^2}) \|\delta f\|_{C^{1}}.
\end{split}
\end{equation*}

In both cases
\begin{equation*}
\bigg| \frac{\delta \Gamma_u(U)}{U} \bigg| \le C(\|f\|_{C^2}) \|\delta f\|_{C^{1}}.
\end{equation*}

%

\smallskip

{\it Step 2.3: estimates for $\beta^* \leq \beta < \beta^{**}$ and $U \geq \underline{r}$.} In this case, by \eqref{Equa:expans_U0} we can choose $a > 0$ such that
\begin{equation*}
\Gamma_u(\beta,U) \geq 
\begin{cases}
(\lambda_+(\beta,0) - a) U & 0 \leq U < \underline{r}, \\
\displaystyle{\min_{[\underline{r},U_u(\beta)]}} \{\Gamma_u\} & \underline{r} \leq U \leq U_u(\beta),
\end{cases} \quad \text{with} \ \frac{-\beta}{\lambda_+(\beta,0) - a} < 1.
\end{equation*}
Indeed, for $U < \underline{r}$ we can use the previous point and \eqref{Equa:lambda_+0_beta}, while $\Gamma_u$ is uniformly positive for $U \geq \underline{r}$ and $\beta$ bounded. 
%

Hence, for $\beta < 0$ we can estimate the integral in the exponential in \eqref{Equa:deltaGamma_uu} as
\begin{equation*}
\begin{split}
\int_V^U \frac{-\beta}{\Gamma_u(W)} dW \leq \frac{-\beta}{\displaystyle{\min_{[\underline{r},U_u(\beta)]} \{\Gamma_u\}}} + \frac{-\beta}{\lambda_+(\beta,0) - a} \ln \bigg( \frac{\underline{r}}{V} \bigg) \ind_{[0,\underline{r})}(V).
\end{split}
\end{equation*}
Using the above estimate we obtain that
\begin{equation*}
\begin{split}
\bigg| \frac{\delta \Gamma_u(U)}{U} \bigg| &\leq \frac{\|\delta f\|_{C^0([0,U])}}{U \Gamma_u(U)} \int_0^U e^{- \int_V^U \frac{\beta}{\Gamma_u(W)} dW} dV \\
&\leq \frac{\|\delta f\|_{C^0([0,U])}}{U \displaystyle{\min_{[\underline{r},U_u(\beta)]} \{\Gamma_u\}}} e^{\frac{-\beta}{\displaystyle{\min_{[\underline{r},U_u(\beta)]} \{\Gamma_u\}}}} \bigg[ (U - \underline{r}) +  \int_0^{\underline{r}} \bigg( \frac{\underline{r}}{V} \bigg)^{\frac{-\beta}{\lambda_+(\beta,0) - a}} dV \bigg] \\
&= \frac{\|\delta f\|_{C^0([0,U])}}{U \displaystyle{\min_{[\underline{r},U_u(\beta)]} \{\Gamma_u\}}} e^{\frac{-\beta}{\displaystyle{\min_{[\underline{r},U_u(\beta)]} \{\Gamma_u\}}}} \bigg[ (U - \underline{r}) + \frac{\underline{r}}{1 + \frac{\beta}{\lambda_+(\beta,0) - a}} \bigg].
\end{split}
\end{equation*}

For $0 \leq \beta < \beta^{**}$ we can simply observe that
\begin{equation*}
e^{\int_V^U \frac{-\beta}{\Gamma_u(W)} dW} \leq 1,
\end{equation*}
and then
\begin{equation*}
\begin{split}
\bigg| \frac{\delta \Gamma_u(U)}{U} \bigg| &\leq \frac{\|\delta f\|_{C^0([0,U])}}{U \Gamma_u(U)} \int_0^U e^{- \int_V^U \frac{\beta}{\Gamma_u(W)} dW} dV \leq \frac{\|\delta f\|_{C^0([0,U])}}{\displaystyle{\min_{[\underline{r},U_u(\beta)]} \{\Gamma_u\}}}. \\ 
\end{split}
\end{equation*}

Hence we conclude that 
\begin{equation*}
\bigg| \frac{\delta \Gamma_u(U)}{U} \bigg| \leq C \Big( \max_{[\underline{r},U_u(\beta)]} \{\Gamma_u^{-1}\} \Big) \|\delta f\|_{C_0}.
\end{equation*}
The quantity $\displaystyle{\min_{[r(\|f\|_{C^2}),U_u(\beta)]} \Gamma_u}$ is uniformly bounded when we will fix the maximal speed $\underline{\beta} < \beta^{**}$, see Step 2.6 below.

\smallskip

{\it Step 2.4: estimates for $\beta^{**}  < \beta < 2 \sqrt{\max_{[0,U^*]} f'(U)}$ and $U$ close to $U^*$.} 
%
%
For $\beta > \beta^{**}$ we use the fact that at $\beta^{**}$ we are on the stable-stable manifold corresponding to $\lambda_-(\beta,U^*)$ and for $\beta > \beta^{**}$ we approach along the eigenvector corresponding to $\lambda_+(\beta,U^*)$, so that by \cite[Theorem 6.2.8]{katokhass} 
\begin{equation*}
- \lambda_+(\beta,U^*) (U^* - U) - \mathcal O(\|f\|_{C^2}) (U^* - U)^2 \leq \Gamma_u(U) \leq - \lambda_-(\beta,U^*) (U^* - U) + \mathcal O(\|f\|_{C^2}) (U^* - U)^2,
\end{equation*}
i.e. $\Gamma_u$ lies in between the invariant manifold corresponding to $\lambda_-(\beta,U^*)$ and the one corresponding to $\lambda_+(\beta,U^*)$. Then estimate the exponential in \eqref{Equa:deltaGamma_u} as
\begin{equation*}
\begin{split}
- \int_V^U \frac{\beta}{\Gamma_u(W)} dW &\leq - \int_{\max\{V,U^* - \underline{r}\}}^{U} \frac{\beta}{\Gamma_u(W)} dV \\
&\leq \mathcal O(\|f\|_{C^2}) + \frac{\beta}{\lambda_-(\beta,U^*)} \ln \bigg( \frac{\min\{U^* - V,\underline{r}\}}{U^* - U} \bigg).
\end{split}
\end{equation*}
Hence
\begin{equation*}
\begin{split}
e^{- \int_V^U \frac{\beta}{\Gamma_u(W)} dW} 
= \mathcal O(\|f\|_{C^2}) \bigg( \frac{\min\{U^*-V,\underline{r}\}}{U^*-U} \bigg)^\frac{\beta}{\lambda_-(\beta,U^*)}.
\end{split}
\end{equation*}
Substituting into \eqref{Equa:deltaGamma_u} we get for $0 < U^* - U < r(\|f\|_{C^2})$
\begin{equation*}
\begin{split}
\bigg| \frac{\delta \Gamma_u}{U} \bigg| &= \frac{C(\|f\|_{C^2}) \|\delta f\|_{C^0([0,U^*]}}{U \Gamma_u(U)} \int_0^U \bigg( \frac{\min\{U^* - V,\underline{r}\}}{U^* - U} \bigg)^\frac{\beta}{\lambda_-(\beta,U^*)} dV \\
&= \frac{C(\|f\|_{C^2}) \|\delta f\|_{C^0([0,U^*]}}{U (U^* - U)} \int_0^{U^* - \underline{r}} \bigg( \frac{U^*-U}{\underline{r}} \bigg)^\frac{\beta}{-\lambda_-(\beta,U^*)} dV \\
&\quad + \frac{C(\|f\|_{C^2}) \|\delta f\|_{C^0([0,U^*]}}{U (U^* - U)} \int_{U^* - \underline{r}}^U \bigg( \frac{U^*-V}{U^*-U} \bigg)^\frac{\beta}{\lambda_-(\beta,U^*)} dV \\
&\leq \frac{C(\|f\|_{C^2}) \|\delta f\|_{C^0([0,U^*]}}{U} \frac{U^* - \underline{r}}{\underline{r}^\frac{\beta}{- \lambda_-(\beta,U^*)}} (U^* - U)^{- 1 + \frac{\beta}{- \lambda_-(\beta,U^*)}} \\
& \quad + \frac{C(\|f\|_{C^2}) \|\delta f\|_{C^0([0,U^*]}}{U} \frac{(U^* - U)^{- 1 + \frac{\beta}{-\lambda_-(\beta,U^)}}}{- 1 + \frac{\beta}{-\lambda_-(\beta,U^*)}} \Big( (U^* - U)^{1 + \frac{\beta}{\lambda_-(\beta,U^*)}} - \underline{r}^{1 + \frac{\beta}{\lambda_-(\beta,U^*)}} \Big).
\end{split}
\end{equation*}
We have used the fact that $\beta > - \lambda_-(\beta,U^*)$. Since
$$
-1 - \frac{\beta}{\lambda_-(\beta,U^*)} > 0,
$$
it follows that the above term is uniformly bounded for $0 < U^* - U < \underline{r}$.

\smallskip

{\it Step 2.5: estimates for $0 \leq \beta \leq 2 \sqrt{\max_{[0,U^*]} f'(U)}$ and $\underline{r} \leq U \leq U^* - \underline{r}$.} In this case it holds simply
\begin{equation*}
\bigg| \frac{\delta \Gamma_u}{U} \bigg| \leq \frac{\|\delta f\|_{C^0([0,U^*]}}{\Gamma_u(\beta,U)} \leq C \Big( \max_{[\underline{r},U^* - \underline{r}]} \{\Gamma_u^{-1}\} \Big) \|\delta f\|_{C^0([0,U^*])}.
\end{equation*}
\smallskip

{\it Step 2.6: estimates for $\beta$ close to $\beta^{**}$ and $U > U^* - \underline{r}(\|f\|_{C^2})$.} Choose $0 < \underline{\beta} < \beta^{**}$ sufficiently close to $\beta^{**}$ such that for $\underline{\beta} < \beta < \beta^{**}$ the point $U_u(\beta)$ where $\Gamma_u$ intersects $P^*$ satisfies
$$
|(U_u(\beta)-U^*,\Gamma_u(U^*))| \leq \gls{rbar}. 
$$
Using \eqref{Equa:delta_Gamma_u} for $U  = U^* - \underline{r}$, it is sufficient to require that 
\begin{equation*} 
(- \lambda_-(\beta^{**},U^*) + C(\|f\|_{C^2}) \underline{r}) \underline{r} + \beta^{**} - \underline{\beta} < \underline{r}, \quad \beta^{**} - \underline{\beta} = \mathcal O(\underline{r}) = \mathcal O(\|f\|_{C^2}).
\end{equation*}
Taking a smaller $\underline{r}$ if necessary, from the ODE \eqref{Equa:Gammau_eq0} it follows that $\Gamma_u$ is decreasing because $\Gamma_u$ is above the manifold corresponding to the eigenvalues $\lambda_-(\beta,U^*)$,
\begin{equation*}
\Gamma_u(U) \geq \big( - \lambda_-(\beta,U^*) - C(\|f\|_{C^2})(U^* - U) \big) (U^* - U),
\end{equation*}
and for $U - \underline{r} \leq U < U^*$
\begin{equation*}
\begin{split}
- \frac{f(U)}{\Gamma_u} - \beta &\leq \frac{-f(U)}{\big( - \lambda_-(\beta,U^*) - C(\|f\|_{C^2})(U^* - U) \big) (U^* - U)} - \beta \\
&\leq - \frac{f'(U^*)}{\lambda_-(\beta,U^*)} - \beta + C(\|f\|_{C^2}) \underline{r} \\
&= \lambda_-(\beta,U^*) + C(\|f\|_{C^2}) \underline{r} < 0.
\end{split}
\end{equation*}
Moreover
\begin{equation}
\label{Equa:local_Lip_U**}
\frac{d\Gamma_u}{dU} \geq - \frac{f(U_u)}{\Gamma_u(U_u)} - \beta = - \sqrt{\frac{f(U_u)}{U_u}} - \beta > - \beta - \mathcal O(\underline{r}^{\frac{1}{2}}),
\end{equation}
which means that it is uniformly Lipschitz in the region $U < U_u(\beta)$. 
We also observe that by the bound \eqref{Equa:local_Lip_U**} 
the trajectories are crossing the curve $- (\beta + \frac{1}{C(\|f\|_{C^2})}) (U - U^*)$ in a point \newglossaryentry{Uunder}{name=\ensuremath{\underline{U}},description={intersection of $\Gamma_u$ with the line $(\beta + \frac{1}{C(\|f\|_{C^2})}) (U^* - U)$}} $(\gls{Uunder},\Gamma_u(\gls{Uunder}))$ such that
\begin{equation}
\label{Equa:rel_barU_U*}
\lim_{U_u \to U^*} \frac{U_u - \underline{U}}{\Gamma_u(U_u)} = C(\|f\|_{C^2}). 
\end{equation}
%
We can thus estimate $\Gamma_u(V)$ from above as
\begin{equation*}
\Gamma_u(V) \leq 
\Gamma_u(U) - (\beta + \mathcal O(\underline{r}^{\frac{1}{2}})) (V-U). 
\end{equation*}

Using \eqref{Equa:deltaGamma_uu}, the quantity we have to estimate is
\begin{equation*}
\begin{split}
\bigg| \int_{0}^{U} \frac{\delta f(V)}{U \Gamma_u(V)} e^{\int_V^U \frac{f(W)}{\Gamma_u(W)^2} dW} dV \bigg| &\leq \bigg| \int_{0}^{U^* - \underline{r}} \frac{\delta f(V)}{U \Gamma_u(V)} e^{\int_V^U \frac{f(W)}{\Gamma_u(W)^2} dW} dV \bigg| + \frac{\|\delta f\|_{C_0}}{U} \int_{U^* - \underline{r}}^U \frac{e^{\int_V^U \frac{f(W)}{\Gamma_u(W)^2} dW}}{\Gamma_u(V)} dV \\
&= \bigg| \frac{\delta \Gamma_u(U^* - \underline{r})}{U} \bigg| e^{\int_{U^*-\underline{r}}^U \frac{f(W)}{\Gamma_u(W)^2} dW} + \frac{\|\delta f\|_{C_0}}{U} \int_{U^* - \underline{r}}^U \frac{e^{\int_V^U \frac{f(W)}{\Gamma_u(W)^2} dW}}{\Gamma_u(V)} dV.
\end{split}
\end{equation*}
For $W\le U\le U_u(\beta)$
\begin{equation}
\label{Equa:min_ration_22}
\Gamma_u(\beta, W)^2 \ge f(W)W, 
\end{equation}
so that
\begin{align*}
\bigg| \frac{\delta \Gamma_u(U^* - \underline{r})}{U} \bigg| e^{\int_{U^*-\underline{r}}^U \frac{f(W)}{\Gamma_u(W)^2} dW} &\leq \bigg| \frac{\delta \Gamma_u(U^* - \underline{r})}{U} \bigg| e^{\int_{U^*}^U \frac{f(W)}{\Gamma_u(W)^2} dW} \\
&\leq \bigg| \frac{\delta \Gamma_u(U^* - \underline{r})}{U} \bigg| e^{\int_{U^*}^U \frac{1}{W} dW} \leq \bigg| \frac{\delta \Gamma_u(U^* - \underline{r})}{U^* - \underline{r}} \bigg|.
\end{align*}
The quantity $\frac{\delta \Gamma_u(U^* - \underline{r})}{U^*-\underline{r}}$ is already estimated in Step 2.5.

We are left with
\begin{equation}
\label{Equa:estim_last_stab}
\int_{U^* - \underline{r}}^U \frac{e^{\int_V^U \frac{f(W)}{\Gamma_u(W)^2} dW}}{\Gamma_u(V)} dV.
\end{equation}
Assume first that $U^* - \underline{r} < U \leq \underline{U}$: then, also $U^* - \underline{r} \leq W \leq \underline{U}$ and 
\begin{equation*}
\begin{split}
\frac{- f(W)}{\Gamma_u(W)} &\geq \frac{(f'(U^*) - C(\|f\|_{C^2}) \underline{r}) (U^* - W)}{\Gamma_u(\underline{U}) + (\beta + \frac{\underline{r}}{C(\|f\|_{C^2})}) (\underline{U} - W)} \\
&= \frac{(f'(U^*) - C(\|f\|_{C^2}) \underline{r}) (U^* - W)}{(\beta + \frac{\underline{r}}{C(\|f\|_{C^2})}) (U^* - \underline{U}) + (\beta + \frac{\underline{r}}{C(\|f\|_{C^2})}) (\underline{U} - W)} \\
&= \frac{(f'(U^*) - C(\|f\|_{C^2}) \underline{r}) (U^* - W)}{(\beta + \frac{\underline{r}}{C(\|f\|_{C^2})})(U^*-W)} \\
&\geq \frac{f'(U^*)}{\beta} - C(\|f\|_{C^2}) \underline{r} \geq \frac{1}{C(\|f\|_{C^2})}.
\end{split}
\end{equation*}
Then, observing that $f < 0$ in this region, \eqref{Equa:estim_last_stab} becomes for $U^* - \underline{r} \leq U \leq \underline{U}$
\begin{equation*}
\begin{split}
\int_{U^* - \underline{r}}^{U} \frac{e^{\int_V^U \frac{f(W)}{\Gamma_u(W)^2} dW}}{\Gamma_u(V)} dV &\leq \int_{U^* - \underline{r}}^U \frac{e^{- \int_V^U \frac{1}{C(\|f\|_{C^2})) \Gamma_u(W)} dW}}{\Gamma_u(V)} dV \leq C(\|f\|_{C^2}). 
\end{split}
\end{equation*}

If $\underline{U} \leq U \leq U_u$, then using \eqref{Equa:min_ration_22}
we obtain
\begin{equation*}
\begin{split}
\int_{\underline{U}}^U \frac{1}{\Gamma_u(V)} e^{\int_V^U \frac{f(W)}{\Gamma_u(W)^2} dW} dV \leq \frac{\underline{U} - U_u}{\Gamma_u(U_u)} \frac{U}{U^*} \underset{(\ref{Equa:rel_barU_U*})}{\leq} C(\|f\|_{C^2}).
\end{split}
\end{equation*}

We thus conclude that there exists $\underline{\beta} = \underline{\beta}(\|f\|_{C^2})$ such that for $\underline{\beta} < \beta < \beta^{**}$ and $U^* - \underline{r} < U < U_u$ it holds
\begin{equation*}
\bigg| \frac{\delta \Gamma}{U} \bigg| \leq C \Big( \|f\|_{C^2},\max_{[\underline{r},U^* - \underline{r}], \beta \leq \beta^{**}} \{\Gamma_u^{-1}\} \Big) \|\delta f\|_{C^1([0,U^*])}.
\end{equation*}

\smallskip

{\it Final estimate for $\Gamma_u$.} There exists $\underline{r} = \underline{r}(\|f\|_{C^2})$, $\underline{\beta}(\|f\|_{C^2})$ such that for $U \in (0,U_u)$
\begin{equation*}
\bigg| \frac{\delta \Gamma_u}{U} \bigg| \leq C(f) \|\delta f\|_{C^{1}}, 
\end{equation*}
with the constant $C(f)$ depending on
\begin{equation*}
\|f\|_{C_2}, \ \max \Big\{ \Gamma_u(U)^{-1}, \beta \in [\beta^*,\underline{\beta}], U \in [\underline{r},U_u] \Big\}, \ \max \Big\{ \Gamma_u(U)^{-1}, \beta \in \Big[ \beta^{*},2 \sqrt{\max_{[0,U^*]} f'(U)} \Big], U \in [\underline{r},U_u - \underline{r}] \Big\}.
\end{equation*}
Thus, recalling that $\delta \Gamma_u \leq 0$, we have proved \eqref{Equa:estimat_f_uGG}.

\smallskip

\noindent{\it Study of the perturbation of $\Gamma_s$.} The analysis of the dependence of $\Gamma_s$ w.r.t. the source $f$ follows the same lines, with the simplification that $\Gamma_s(U)$ is positive away from $1$ and decreasing as $\lambda_-(\beta,1) (U-1)$ for $U \nearrow 1$.

The linearized ODE is
\begin{equation*}
\frac{d \delta \Gamma_s}{dU} = \frac{f(U)}{\Gamma_s^2} \delta \Gamma_s - \frac{\delta f(U)}{\Gamma_s},
\end{equation*}
so that, noticing that
$$
- \frac{f(U)}{\Gamma_s(U)^2} \sim - \frac{f'(1)}{\lambda_-(\beta,1)^2} \frac{1}{1 - U} \quad \text{as $U \to 1$}
$$
and following the analysis leading to \eqref{Equa:deltaGamma_uu}, the solution can be written as
\begin{equation*}
\delta \Gamma(U) = \int_U^1 \frac{\delta f(V)}{\Gamma(V)} e^{- \int_U^V \frac{f(W)}{\Gamma_s(W)^2} dW} dV = \int_U^1 \frac{\delta f(V)}{\Gamma_s(U)} e^{\int_U^V \frac{\beta}{\Gamma_u(W)} dW} dV.
\end{equation*}
In particular $\delta \Gamma_s \geq 0$.

Observing that
$$
\Gamma_s(U) \geq \frac{(1- U)}{C(\|f\|_{C^2})}
$$
for some constant $C = C(\|f\|_{C^2})$ and $f \geq 0$ in $[U^*,1]$, we can estimate
\begin{equation*}
\begin{split}
|\delta \Gamma(U)| &= \bigg| \int_U^1 \frac{\delta f(V)}{1-V} \frac{1-V}{\Gamma_s(V)} e^{- \int_U^V \frac{f(W)}{\Gamma(W)^2} dW} dV \bigg| \\
&\leq \|\delta f\|_{C^1([U,1])} \int_U^1 \frac{1 - V}{\Gamma_s(V)} dV \leq C(\|f\|_{C^2}) \|\delta f\|_{C^1([U,1])} (1-U).
\end{split}
\end{equation*}
This concludes the proof of \eqref{Equa:estimat_f_sGG}.
\end{proof}

\begin{remark}
\label{Rem:precise_partial_Gamma}
The analysis of  $\Gamma_s$ is much easier since $\Gamma_s$ is uniformly positive and we know the asymptotic behavior near $U = 1$.

The main difficulty in studying $\Gamma_u$ arises when $\beta \sim \beta^{**}$: in this case the trajectory $\Gamma_u$ switches from the stable-stable manifold corresponding to $\lambda_-(\beta,U^*)$ to the stable invariant manifold corresponding to $\lambda_+(\beta,U^*)$. We will see that this has consequences for the analysis of $\partial^2_\beta \Gamma_u$ and the $C^1$-regularity of of $f \mapsto \Gamma_u(\beta,f,U)$.
\end{remark}

Using the regularity of the unstable/stable manifold w.r.t. the paramenter of the ODE, we can prove Fr\'echet differentiability of the maps
\begin{equation*}
(\beta,f,U) \mapsto \Gamma_u(\beta,f,U), \Gamma_s(\beta,f,U)
\end{equation*}
w.r.t. the product topology of $\R \times C^2([0,1]) \times \R$.

\begin{corollary}
\label{Cor:Frechet_Gamma}
Let $f$ satisfy Assumption \eqref{H1} and $U \in [0,U_{uf}(\beta,U))$: then the map
\begin{equation*}
(\beta,f,U) \mapsto \Gamma_u(\beta,f,U)
\end{equation*}
is Fr\'echet differentiable w.r.t. the product topology of $\R \times C^2([0,1]) \times \R$, with derivative
\begin{equation*}
\begin{split}
\langle \delta \Gamma_u(\beta,f,U),(\delta \beta,\delta f,\delta U) \rangle &= - \delta \beta \int_0^U e^{\int_V^U \frac{f(W)}{\Gamma_u(\beta,f,W)^2} dW} dV \\
& \quad - \int_0^U \frac{\delta f(V)}{\Gamma_u(\beta,f,V)} e^{\int_V^U \frac{f(W)}{\Gamma_u(\beta,f,W)^2} dW} dV - \delta U \bigg( \frac{f(U)}{\Gamma_u(\beta,f,U)} + \beta \bigg).
\end{split}
\end{equation*}

Similarly, for $f$ satisfying Assumption \eqref{H1} and $U \in (U^*,1]$, the map
\begin{equation*}
(\beta,f,U) \mapsto \Gamma_s(\beta,f,U)
\end{equation*}
is Fr\'echet differentiable w.r.t. the product topology of $\R \times C^2([0,1]) \times \R$, with derivative
\begin{equation*}
\begin{split}
\langle \delta \Gamma_s(\beta,f,U),(\delta \beta,\delta f,\delta U) \rangle &= \delta \beta \int_0^U e^{-\int_U^V \frac{f(W)}{\Gamma_s(\beta,f,W)^2} dW} dV \\
& \quad + \int_0^U \frac{\delta f(V)}{\Gamma_s(\beta,f,V)} e^{-\int_U^V \frac{f(W)}{\Gamma_u(\beta,f,W)^2} dW} dV - \delta U \bigg( \frac{f(U)}{\Gamma_s(\beta,f,U)} + \beta \bigg).
\end{split}
\end{equation*}
\end{corollary}

\begin{proof}
We present the proof only for $\Gamma_u$, and only for the derivative w.r.t. $\beta,f$, being the differentiability w.r.t. $U$ obtained directly from the ODE.

Recall that:
\begin{itemize}
\item near $(U,P) = (0,0)$ there is a neighborhood whose size $(0,\underline{r})$ is depending only on the $C^2$-norm of $f$, such that $\Gamma_u$ is the unique fixed point a a contraction map, depending smoothly on $\beta,f$ (\cite[Theorem 8.2]{coudeneerne_ergodic}),
\item the ODE depends smoothly on $\beta,f$.
\end{itemize}
The formula for the derivative is obtained by using \eqref{Equa:sol_der_beta_u} and \eqref{Equa:deltaGamma_uu}.
\end{proof}

In the next statement we will prove the second derivatives
\begin{equation*}
\partial_\beta^2 \Gamma_u(\beta,U) = \lim_{\delta \beta \searrow 0} \frac{\partial_\beta \Gamma_u(\beta + \delta \beta,U) - \partial_\beta \Gamma_u(\beta,U)}{\delta \beta}, \quad U \in [0,U_u(\beta)]
\end{equation*}
\begin{equation*}
\partial_\beta^2 \Gamma_s(\beta,U) = \lim_{\delta \beta \searrow 0} \frac{\partial_\beta \Gamma_s(\beta + \delta \beta,U) - \partial_\beta \Gamma_s(\beta,U)}{\delta \beta}, \quad U \in [0,1],
\end{equation*}
exist and are bounded when $\beta$ is away from the critical speed $\beta^{**}$. 
%

\begin{theorem}
\label{Theo:second_der_Gammaus}
The following holds.
\begin{enumerate}
\item For all $\beta \leq \underline{\beta} < \beta^{**}$ there is a constant $C(f,\underline{\beta})$ depending only on
\begin{equation*}
\begin{split}
&\|f\|_{C_2}, \quad \max \Big\{ \Gamma_u(U)^{-1}, \beta \in [\beta^*,\underline{\beta}], U \in [\bar r,U_u(\beta)] \Big\} 
\end{split}
\end{equation*}
such that
\begin{equation}
\label{Equa:beta2_Gamma_u}
\big| \partial^2_\beta \Gamma_u(\beta,U) \big| \leq C(f,\underline{\beta}), \quad U \in [0,U_u(\beta)].
\end{equation}

\item There is a constant depending only on $\|f\|_{C^2}$ such that
\begin{equation}
\label{Equa:beta2_Gamma_s}
0 \leq \partial^2_\beta \Gamma_s(U) \leq C(\|f\|_{C^2}), \quad U \in [0,1].
\end{equation}
\end{enumerate}
\end{theorem}

The fact that $\partial_\beta^2 \Gamma_u$ may be positive will have implications in the analysis of the convexity of the cost functional $E(\beta)$ also in the regions where $\beta \not= \beta^{**}$, see Example \ref{Sss:E_not_convex}.

\begin{proof}
We study separately $\Gamma_u$ and $\Gamma_s$, and we will do the estimate also for the endpoints $U_u(\beta,U_s(\beta)$.

\smallskip

\noindent{\it Step 1: analysis of $\Gamma_u$.} The equation for $\partial^2_\beta \Gamma_u$ is
\begin{equation*}
\frac{d}{dU} \partial_\beta^2 \Gamma_u = \frac{f(U)}{\Gamma_u^2} \partial_\beta^2 \Gamma_u - \frac{2 f(U)}{\Gamma_u^3} (\partial_\beta \Gamma_u)^2, \quad \lim_{U \to 0} \frac{\partial_\beta^2 \Gamma_u}{U} = \partial^2_\beta \lambda_+(\beta,0).
\end{equation*}
As in the case of $\partial_\beta \Gamma_u$ (Step 1.1 of the proof of Theorem \ref{manifoldbound}), we obtain the representation
\begin{equation}
\label{Equa:formula_beta2u}
\partial_\beta^2 \Gamma_u(U) = - \int_0^{U} \frac{2 f(V)}{\Gamma_u(V)^3} \big( \partial_\beta \Gamma_u(V) \big)^2 e^{\int_V^U \frac{f(W)}{\Gamma_u(W)^2} dW} dV.
\end{equation}

Using \eqref{Equa:Gammau_beta} we have
\begin{equation*}
\frac{f(V) \big( \partial_\beta \Gamma_u(V) \big)^2 }{\Gamma_u(V)^3} \leq C(f),
\end{equation*}
where $C(f)$ depends only on
$$
\|f\|_{C^2} \quad  \text{and} \quad \max \Big\{ \Gamma_u(U)^{-1}, \beta \in [\beta^*,\underline{\beta}], U \in [\bar r,U_u(\beta)] \Big\}.
$$
Since $f \leq 0$ for $U \leq U^*$, we then obtain
\begin{equation*}
\begin{split}
|\partial^2_\beta \Gamma_u(U_u(\beta)| &\leq C(f) \int_0^{U_u(\beta)} e^{\int_{U^*}^{U_u(\beta)} \frac{f(W)}{\Gamma_u(W)^2} dW} dV \leq C(f). 
\end{split}
\end{equation*}
which gives \eqref{Equa:beta2_Gamma_u} because of \eqref{Equa:rel_barU_U*}.

\smallskip

\noindent{\it Step 2: analysis of $\Gamma_s$.} The computations are similar, with the simplification that from \eqref{Equa:Gammas_beta} and $\Gamma_s(U) \geq \frac{1 - U}{C(\|f\|_{C^2})}$ we have
\begin{equation*}
0 \leq \frac{\partial_\beta \Gamma_s(U)}{\Gamma_s(U)}, \frac{f(U)}{\Gamma_s(U)} \leq C(\|f\|_{C^2}), \quad U \in (U^*,1).
\end{equation*}
The solution formula for $\partial_\beta^2 \Gamma_s$ is
\begin{equation*}
\partial_\beta^2 \Gamma_s(U_s) = \int_{U_s}^1 \frac{2 f(V)}{\Gamma_u(V)^3} \big( \partial_\beta \Gamma_u(V) \big)^2 e^{- \int_{U_s}^V \frac{f(W)}{\Gamma_u(W)^2} dW} dV,
\end{equation*}
and using the previous bounds one immediately deduces \eqref{Equa:beta2_Gamma_s}.
\end{proof}

Let \newglossaryentry{Ubarubeta}{name=\ensuremath{\overline{U}_u(\beta)},description={last point of intersection of $\gamma_u(\beta)$ with $P^*$}} \gls{Ubarubeta} be the last intersection of $\Gamma_u$ with $P^*$, and \newglossaryentry{Ubarsbeta}{name=\ensuremath{\overline{U}_s(\beta)},description={first point of intersection of $\gamma_u(\beta)$ with $P^*$}} \gls{Ubarsbeta} be the first intersection of $\Gamma_s$ with $P^*$. Then we have also the following

\begin{corollary}
\label{Cor:bound_hatu}
The estimate \eqref{Equa:beta2_Gamma_u} holds in the interval $[0,\overline {U}_u(\beta)]$ and $\beta^* \leq \beta \leq \underline{\beta} < \beta^{**}$.
\end{corollary}

\begin{proof}
It suffices to prove that if $\overline U_u(\beta) > U_u(\beta)$, then $\Gamma_u(\beta,U)$ is uniformly positive in $[U_u(\beta),\overline U_u(\beta)]$. We observe first that being $\Gamma_u$ decreasing for $U > U^*$ and $P^*(U)$ increasing in $(U^*,U^* + \underline{r})$, then if $U_u(\beta) < U^* + \underline{r}$ then there are no further intersections. Hence if $\overline U_u(\beta) > U_u(\beta)$ we must have
\begin{equation*}
\Gamma_u(U) > \min_{[U^* + \bar r,1-\bar r]} f(U) > 0. \qedhere
\end{equation*}
\end{proof}

We conclude this section with the following two results, which show the the regularity results obtained above are optimal. The result below anticipates the loss of regularity of $E(\beta)$ at $\beta^{**}$ of Section \ref{S:diffe_E_beta}. 

\begin{proposition}
\label{Prop:blowupsecond}
If
\begin{equation}
\label{Equa:cond_blow}
\frac{\lambda_+(\beta^{**},U^*)}{\lambda_-(\beta^{**},U^*)} < \frac{1}{2},
\end{equation}
then
\begin{equation*}
\lim_{\beta \nearrow \beta^{**}} \partial^2_\beta \Gamma_u(\beta,U_u(\beta)) = +\infty.
\end{equation*}
\end{proposition}

\begin{figure}
\resizebox{.75\textwidth}{!}{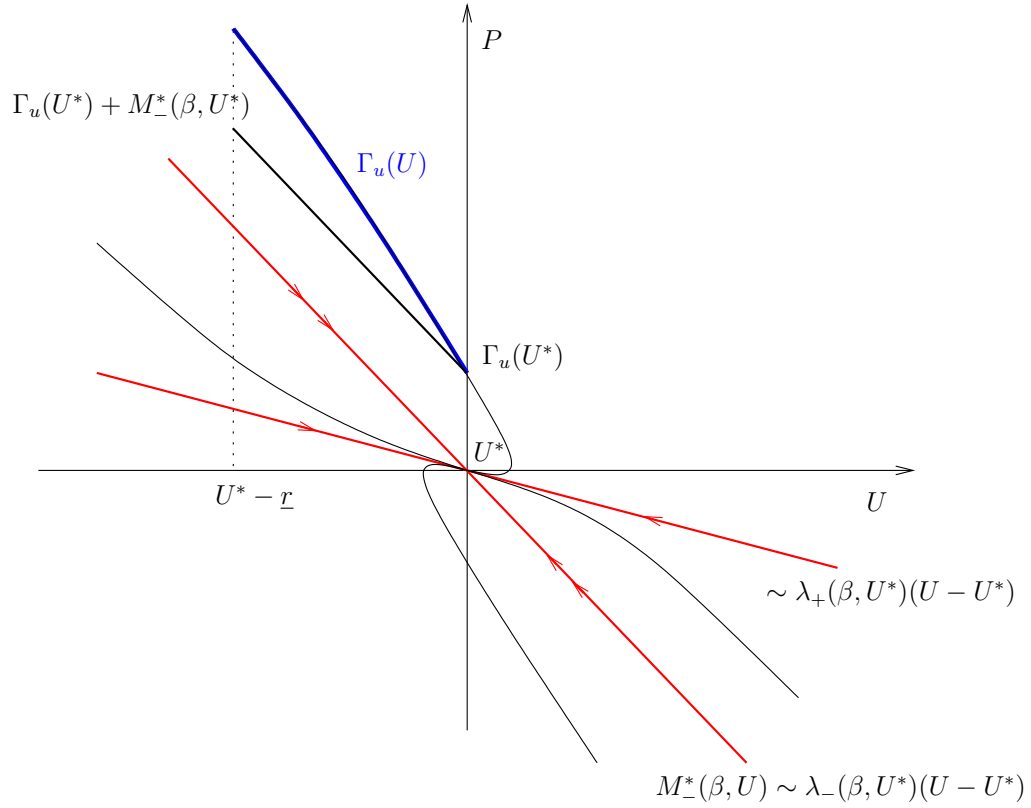}
\caption{The phase-plane picture near $(U^*,0)$ for $\beta \sim \beta^{**}$.}
\label{fig:nearUstar}
\end{figure}

\begin{proof}
If
$$
M^*_-(\beta,U) = \lambda_-(\beta^{**},U^*) (U - U^*) + C(\|f\|_{C^2}) (U - U^*)^2
$$
is the invariant manifold corresponding to $\lambda_-(\beta,U^*)$ (see \cite[Theorem 6.2.8]{katokhass}), then writing
$$
P(U) = M^*_-(\beta,U) + g(U), \quad g(U) > 0,
$$
we obtain
\begin{equation*}
\frac{dg}{dU} = - \frac{f(U)}{M^*_-(\beta,U) + g(U)} - \beta + \frac{d}{dU} M^*_-(\beta,U) = \frac{f(U) g(U)}{M^*_-(\beta,U) (M^*_-(\beta,U) + g(U))} < 0
\end{equation*}
for $U \in (U^* - \underline{r},U^*)$, which gives
\begin{equation}
\label{Equa:Gammau_bound}
\Gamma_u(U^*) + M^*_-(\beta,U) \leq \Gamma_u(U) \leq \Gamma_u(U^* - \underline{r}) + \big( M^*_-(\beta,U) - M^*_-(\beta,U^* - \underline{r}) \big).
\end{equation}
A picture of trajectories near the point $(U^*,0)$ for $\beta$ close to $\beta^{**}$ is given in Fig. \ref{fig:nearUstar}.

Using the assumption that $f(U) < 0$ in $(0,U^*)$, we can estimate the integral for $U^* - \underline{r} \leq V' \leq V \leq U^*$ as
\begin{equation*}
\begin{split}
\int_{V'}^{V} \frac{f(W)}{\Gamma_u(\beta,W)^2} dW &\geq \int_{U^*-\underline{r}}^{V} \frac{f(W)}{\Gamma_u(\beta,W)^2} dW \\
&\geq \int_{U^* - \underline{r}}^{V} \frac{f'(U^*) (W - U^*)}{(\Gamma_u(U^*) + (\lambda_-(\beta,U^*) - C(\|f\|_{C^2}) \underline{r}) (W - U^*))^2} dW \\
&\geq - \frac{f'(U^*)}{(\lambda_-(\beta,U^*) - C(\|f\|_{C^2}) \underline{r})^2} \ln \bigg( 1 + \frac{- \lambda_-(\beta,U^*) + C(\|f\|_{C^2}) \underline{r}}{\Gamma_u(U^*)} (V - U^* + \underline{r}) \bigg).
\end{split}
\end{equation*}
Hence by \eqref{Equa:formula_Gammabeta}
\begin{equation}
\label{Equa:more_perc_der_beta}
\begin{split}
\partial_\beta \Gamma_u(V) &= \partial_\beta \Gamma_u(U^* - \underline{r}) e^{\int_{U^* - \underline{r}}^{V} \frac{f(W)}{\Gamma_u(\beta,W)^2} dW} - \int_{U^* - \underline{r}}^{V} e^{\int_{V'}^{V} \frac{f(W)}{\Gamma_u(\beta,W)^2} dW} dV' \\
&\leq - \mathcal O(1) \bigg( 1 + \frac{- \lambda_-(\beta,U^*) + C(\|f\|_{C^2}) \underline{r}}{\Gamma_u(U^*)} (V - U^* + \underline{r}) \bigg)^{- \frac{f'(U^*)}{(\lambda_-(\beta,U^*) - C(\|f\|_{C^2}) \underline{r})^2}} \\
&\sim - \Gamma_u(U^*)^{\frac{f'(U^*)}{(\lambda_-(\beta,U^*) - C(\|f\|_{C^2}) \underline{r})^2}}.
\end{split}
\end{equation}
and \eqref{Equa:formula_beta2u} gives (being the second derivative positive)
\begin{equation*}
\begin{split}
\partial^2_\beta \Gamma_u(U^*) &= - \int_0^{U^*} \frac{2 f(V)}{\Gamma_u(V)^3} \big( \partial_\beta \Gamma_u(V) \big)^2 e^{\int_V^{U^*} \frac{f(W)}{\Gamma_u(W)^2} dW} dV \\
&= \partial^2_\beta \Gamma_u(U^* - \underline{r}) e^{\int_{U^* - \underline{r}}^{U^*} \frac{f(W)}{\Gamma_u(W)^2} dW} - \int_{U^* - \underline{r}}^{U^*} \frac{2 f(V)}{\Gamma_u(V)^3} \big( \partial_\beta \Gamma_u(V) \big)^2 e^{\int_V^{U^*} \frac{f(W)}{\Gamma_u(W)^2} dW} dV \\
&\gtrsim \partial^2_\beta \Gamma_u(U^* - \underline{r}) \Gamma_u(U^*)^{\frac{f'(U^*)}{(\lambda_-(\beta,U^*) + C(\|f\|_{C^2}) \underline{r})^2}} \\
& \quad + \Gamma_u(U^*)^{\frac{2 f'(U^*)}{(\lambda_-(\beta,U^*) + C(\|f\|_{C^2}) \underline{r})^2}} \int_{U^* - \underline{r}}^{U^*} \frac{- 2 f(V)}{\Gamma_u(V)^3} e^{\int_V^{U^*} \frac{f(W)}{\Gamma_u(W)^2} dW} dV.
\end{split}
\end{equation*}
The latter integral can be estimated by integrating by parts
\begin{equation*}
\begin{split}
\int_{U^* - \underline{r}}^{U^*} \frac{- 2 f(V)}{\Gamma_u(V)^3} e^{\int_V^{U^*} \frac{f(W)}{\Gamma_u(W)^2} dW} dV &= \int_{U^* - \underline{r}}^{U^*} \frac{- f(V)}{\Gamma_u(V)^3} e^{\int_V^{U^*} \frac{f(W)}{\Gamma_u(W)^2} dW} dV \\
& \quad + \frac{e^{\int_V^{U^*} \frac{f(W)}{\Gamma_u(W)^2} dW}}{\Gamma_u(V)} \bigg|_{U^*-\underline{r}}^{U^*} - \int_{U^* - \underline{r}}^{U^*} \frac{\frac{d\Gamma(V)}{dV}}{\Gamma_u(V)^2} e^{\int_V^{U^*} \frac{f(W)}{\Gamma_u(W)^2} dW} dV \\
&= \frac{1}{\Gamma_u(U^*)} - \frac{e^{\int_{U^* - \underline{r}}^{U^*} \frac{f(W)}{\Gamma_u(W)^2} dW}}{\Gamma_u(U^* - \underline{r})} + \int_{U^* - \underline{r}}^{U^*} \frac{\beta}{\Gamma_u(V)^2} e^{\int_V^{U^*} \frac{f(W)}{\Gamma_u(W)^2} dW} dV \\
&= \frac{1 - e^{- \int_{U^* - \underline{r}}^{U^*} \frac{\beta}{\Gamma_u(W)} dW}}{\Gamma_u(U^*)} + \frac{1}{\Gamma_u(U^*)} \int_{U^* - \underline{r}}^{U^*} \frac{\beta}{\Gamma_u(V)} e^{- \int_V^{U^*} \frac{\beta}{\Gamma_u(W)} dW} dV \\
&= 2 \frac{1 - e^{- \int_{U^* - \underline{r}}^{U^*} \frac{\beta}{\Gamma_u(W)} dW}}{\Gamma_u(U^*)}. 
\end{split}
\end{equation*}
Using that the term
\begin{equation*}
\int_{U^*}^{U_u} \frac{f(V)}{\Gamma_u(V)^3} (\partial_\beta \Gamma_u(V))^2 e^{\int_{V}^{U_u} \frac{f(W)}{\Gamma_u(W)^2} dW} dV \leq C(f) \frac{(U_u - U^*)^2}{\Gamma_u(U_u)^3} \sim \sqrt{U_u - U^*},
\end{equation*}
we obtain that as $\beta \to \beta^{**}$ 
\begin{equation*}
\partial_\beta^2 \Gamma_u(U_u) \sim \Gamma_u(U^*)^{\frac{2 f'(U^*)}{\lambda_-(\beta,U^*)^2} - 1}.
\end{equation*}
In particular, if
\begin{equation*}
\frac{f'(U^*)}{\lambda_-(\beta^{**},U^*)^2} = \frac{\lambda_+(\beta^{**},U^*)}{\lambda_-(\beta^{**},U^*)} < \frac{1}{2}
\end{equation*}
then the second derivative approaches $+\infty$ as $\beta \nearrow \beta^{**}$.
\end{proof}

This second result addresses the dependence of the derivatives w.r.t. $\beta$ and the dependence of $f \mapsto \Gamma_u(f)$ when $\beta \sim \beta^*$.

\begin{proposition}
\label{Prop:Lipechit_U_u}
The map
\begin{equation*}
f \mapsto \partial_\beta \Gamma_u(\beta,f,U_u(\beta,f))
\end{equation*}
is continuous in the $C^2$-topology. The map
\begin{equation*}
f \mapsto \Gamma_u(\beta,f,U_u(\beta,f))
\end{equation*}
is only Lipschitz.
\end{proposition}

\begin{proof}
For $\beta < \beta^{**}(f)$, we can apply the previous Corollary \ref{Cor:Frechet_Gamma} in both cases, observing that the derivatives are bounded and continuous w.r.t. $\beta,f$; for $\beta > \beta^{**}(f)$, we have $\Gamma_u(U_u = U^*) = 0$. We thus need only to the case $\beta \nearrow \beta^{**}$.

The derivative w.r.t. $\beta$ is estimated similarly to \eqref{Equa:more_perc_der_beta}: from \eqref{Equa:Gammau_bound} we obtain
\begin{equation*}
\begin{split}
\int_{V}^{U^*} \frac{f(W)}{\Gamma_u(W)^2} dW &\gtrsim \int_{V}^{U^*} \frac{f'(U^*) (W - U^*)}{(\Gamma_u(\beta,U^* - \underline{r}) - M^*_-(\beta,U^*- \underline{r})) + \lambda_-(\beta,U^*) (W - U^*))^2} dW \\
&\sim - \frac{f'(U^*)}{\lambda_-(\beta,U^*)^2} \ln \bigg( 1 + \frac{-\lambda_-(\beta,U^*) (U^* - V)}{\Gamma_u(\beta,U^* - \underline{r}) - M^*_i(\beta,U^*- \underline{r})} \bigg),
\end{split}
\end{equation*}
and then for $\underline{\beta} < \beta < \beta^{**}$
\begin{equation*}
\begin{split}
\partial_\beta \Gamma(\beta,f, U_u(\beta)) &= \partial_\beta \Gamma_u(\beta,f,U^*-\underline{r}) e^{\int_{U^*-r}^{U^*} \frac{f(W)}{\Gamma_u(\beta,W)^2} dW} - \int_{U^*-\underline{r}}^{U^*} e^{\int_{V}^{U^*} \frac{f(W)}{\Gamma_u(\beta,W)^2} dW} dV - \int_{U^*}^{U_u(\beta)} \bigg( \frac{f(V)}{\Gamma_u(V)} + \beta \bigg) dV \\
&\sim \big( \Gamma_u(\beta,U^* - \underline{r}) - M^*_-(\beta,U^*- \underline{r}) \big)^{- \frac{f'(U^*)}{\lambda-(\beta,U^*)^2}} + \mathcal O(U_u - U^*).
\end{split}
\end{equation*}
Since $\Gamma_u(\beta,U^* - \underline{r}) \to M^*_-(\beta,U^* - \underline{r})$ as $\beta \nearrow \beta^{**}$, then $\partial_\beta \Gamma_u(\beta,f,U_u(\beta,f)) \to 0$ as $\beta \to \beta^{**}(f)$. 

The analysis of $f \mapsto \Gamma_u(\beta,f,U^*(\beta,U))$ follows the same line: write with the notation of Step 2 of Theorem \ref{manifoldbound}
\begin{equation*}
\begin{split}
\delta \Gamma(U_u) &= - \int_0^{U_u} \frac{\delta f}{\Gamma_u(V)} e^{\int_V^{U_u} \frac{f(W)}{\Gamma_u(W)^2} dW} dV \\
&= \delta \Gamma(U^* - \underline{r}) e^{\int_{U^* - \underline{r}}^{U_u} \frac{f(W)}{\Gamma_u(W)^2} dW} - \int_{U^* - \underline{r}}^{U^*} \frac{\delta f}{\Gamma_u(V)} e^{\int_V^{U_u} \frac{f(W)}{\Gamma_u(W)^2} dW} dV - \int_{U^*}^{U_u} \frac{\delta f}{\Gamma_u(V)}e^{\int_V^{U_u} \frac{f(W)}{\Gamma_u(W)^2} dW} dV \\
&\sim \delta \Gamma(U^* - \underline{r}) e^{\int_{U^* - \underline{r}}^{U_u} \frac{f(W)}{\Gamma_u(W)^2} dW} - \int_{U^* - \underline{r}}^{U^*} \frac{\delta f}{\Gamma_u(V)} e^{\int_V^{U_u} \frac{f(W)}{\Gamma_u(W)^2} dW} dV - \mathcal O(\sqrt{U_u - U^*}) \|\delta f\|_{C^0}.
\end{split}
\end{equation*}
The difference is in the second integral: using that
\begin{equation*}
\Gamma_u(\beta,U) \sim \lambda_-(\beta,U^*(f)) (U - U^*(f)) 
\end{equation*}
by the linear analysis of Proposition \ref{Prop:blowupsecond}, 
\begin{equation*}
\begin{split}
\int_{U^*- \underline{r}}^{U^*} \frac{e^{\int_V^U \frac{f(W)}{\Gamma_u(W)^2} dW}}{\Gamma_u(V)} dV &\geq \int_{U^* - \underline{r}}^{U^*} e^{\frac{f'(U^*)}{\lambda_-(U^*)} \int_V^{U^*} \frac{dW}{\Gamma)u(W)}} = \frac{- \lambda_-(U^*)}{f'(U^*)} = \frac{1}{- \lambda_+(\beta,U^*)}. 
\end{split}
\end{equation*}
Hence the derivative w.r.t. the perturbation $\delta f$ is not continuous, unless $\delta f(U^*) = 0$, i.e. $U^*$ remains constant.
\end{proof}

\begin{remark}
\label{Rem:const_U*}
In the case $U^*(f)$ is constant, then it is possible to extend the previous analysis to show that the dependence is in fact smooth, not only Lipschitz, in a similar way as the dependence of $\Gamma_u(f)$ near $U=0$. We will not pursue this line in this work.
\end{remark}

\section{Some initial estimates and the monotone case}
\label{S:stru_opt}

In this section we begin the analysis of optimal profiles for \eqref{Equa:funct_eq_UP}, and recall the results developed in \cite{BressanChiri22}, which we briefly summarize. 
%

In \cite{BressanChiri22} the authors make the further \emph{monotonicity assumption} 
\begin{enumerate}[label=\textbf{(H\arabic*)}, ref=H\arabic*]
\setcounter{enumi}{1}
\item\label{mon}
for all $U \in [U^*,1]$, one has $(4-2\sqrt{3})f'(U) + 2U f''(U) \le 0$.
\end{enumerate}
Under this condition, the infimum in \eqref{effort} is achieved by the smooth curve $\gamma_\beta$ obtained via the concatenation of three parts:
\begin{itemize}
\item the unstable manifold $\Gamma_u=\Gamma_u(\beta)$ until it reaches a point \newglossaryentry{Uubeta}{name=\ensuremath{U_u(\beta)},description={(first) intersection of $\Gamma_u(\beta)$ with $P^*$}} $(\gls{Uubeta}, P^*(\gls{Uubeta}))$ on the critical curve $\gls{Pstar} = \sqrt{Uf(U)}$;
\item the stable manifold $\Gamma_s=\Gamma_s(\beta)$ until the point \newglossaryentry{Usbeta}{name=\ensuremath{U_s(\beta)},description={(last) intersection of $\Gamma_s(\beta)$ with $P^*$}} $(\gls{Usbeta}, P^*(U_s(\beta)))$ of intersection with the graph of $P^*$;
\item the segment of the graph of $P^*$ connecting the previous two curves.
\end{itemize}
Recall that \gls{Gammau} and \gls{Gammas} are the unstable manifold of $(0,0)$ and the stable manifold of $(1,0)$. Due to the monotonicity assumption, there are no other intersections of $\Gamma_u$ or $\Gamma_s$ with $P^*$ beyond $U_u,U_s$.

\begin{figure}
\centering
\resizebox{.75\textwidth}{!}{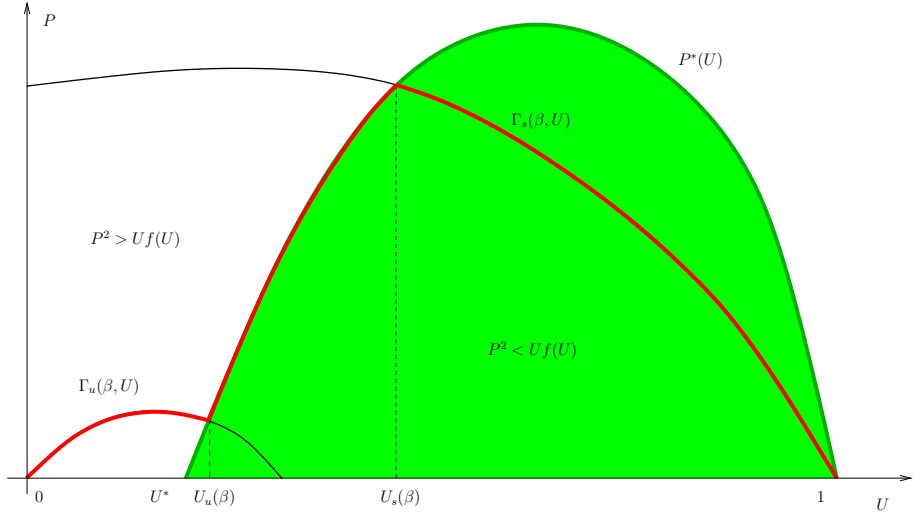} 
\caption{The structure of the optimal solution under Condition (\ref{mon}).}
\label{Fig:optproof_1}
\end{figure}

Moreover, when $\beta^* \le \beta \le \beta^{**}$, the optimal profile $\gamma_\beta$ is the graph of the function (see Fig. \ref{Fig:optproof_1})
\begin{equation}
\label{optgraph}
\gls{Pbeta}(U):=\begin{cases}
\Gamma_u(\beta, U) & U \in [0,U_u(\beta)], \\
P^*(U) & U \in [U_u(\beta), U_s(\beta)], \\
\Gamma_s(\beta, U) & U \in [U_s(\beta),1].
\end{cases}
\end{equation}
In particular, there is an explicit expression for the optimal control, obtained by Equation \eqref{Equa:funct_eq_UP} for $P(U)$:
\begin{equation}
\label{Equa:alpha_beta}
\alpha_\beta(U)= \begin{cases}
\frac{F(U)+\beta}{U} & U \in [U_u(\beta),U_s(\beta)],\\
0 & U \in [0,U_u(\beta)]\cup[U_s(\beta), 1],
\end{cases}
\end{equation}
with
\begin{equation}
\label{F}
\gls{Fdef} := (P^*)'(U) + \frac{f(U)}{P^*(U)} = \frac{3f(U)+Uf'(U)}{2\sqrt{Uf(U)}},  \quad \forall U \in (U^*,1).
\end{equation}
This function will also play a fundamental role in the following sections: indeed the level sets $\{F > - \beta\}$ describe the regions where the trajectories of the ODE \eqref{functeq} are entering $\{P < P^*\}$, and thus a control keeping $P(U) = P^*(U)$ can be applied.

Finally, even when $\beta>\beta^{**}$, a simple approximation argument \cite{bressancetraro} shows that the \emph{effort \gls{Ebeta}} 
\begin{align}
\label{minfunct}
\begin{split}
E(\beta)=\inf \bigg\{ \int_0^1 \frac{1}{U} \left(P' +\frac{f(U)}{P} + \beta \right) \, dU&, P \text{ Lipschitz such that } \\
&P(0)=P(1)=0 \text{ and } P' +\frac{f(U)}{P} + \beta \ge 0 \bigg\},
\end{split}
\end{align}
is equal to
\begin{equation*}
E(\beta) = \int_{U_u(\beta)}^{U_s(\beta)} \frac{F(U) + \beta}{U} dU.
\end{equation*}
We remark here that
\begin{equation}
\label{Equa:bound_integr}
F(U) \sim \begin{cases}
\frac{\sqrt{f'(U^*)}}{(U - U^*)^{\frac{1}{2}}} & U \searrow U^*, \\
- \frac{\sqrt{-f'(1)}}{(1 - U)^{\frac{1}{2}}} & U \nearrow 1,
\end{cases}
\end{equation}
so that integrability is assured. In particular
\begin{equation}
\label{Equa:limi_F}
\lim_{U \to U^*} F(U) = + \infty, \quad \lim_{U \to 1} F(U) = - \infty.
\end{equation}

The key role of Assumption (\ref{mon}) is that it implies a strict monotonicity condition for $F$ (i.e., it is strictly decreasing), which, in turn, guarantees that the upper level sets
$$
\{U\in [U^*,1]: \ F(U)+\beta>0\}
$$
are a single interval $(U^*,\check U)$. By observing Fig. \ref{Fig:optproof_1}, one sees that the intersection points of $\Gamma_u(\beta),\Gamma_s(\beta)$ with $P^*$ are unique, and that $\alpha_\beta$ as defined in \eqref{Equa:alpha_beta} is positive so that $\gamma_\beta$ is admissible.

Finally, optimality follows by the following result which we will use also in the next sections: consider two Lipschitz curves $\gamma_i : [0,1] \mapsto [0,1] \times [0,\infty)$, with
$$
\gamma_i(0) = (0,0), \ \gamma_i(1) = (1,0), \ i=1,2, 
$$
and let
\begin{equation*}
\gamma(s) = \begin{cases}
\gamma_1(s) & s \in [0,1], \\
\gamma_2(2-s) & s \in (1,2),
\end{cases}
\end{equation*}
be the closed curve obtained by joining them. Write also the bounded open region determined by $\gamma$ as the countable union of simply connected open sets $\{A_i\}_i$, to which we associate the rotation number \newglossaryentry{rhoi}{name=\ensuremath{\rho_i},description={orientation of a Jordan curve}}
\begin{equation*}
\gls{rhoi} = \frac{1}{2\pi} \oint \frac{\gamma - (U_i,P_i)}{|\gamma - (U_i,P_i)|^2} \wedge d\gamma, \quad (U_i,P_i) \in A_i. 
\end{equation*}

\begin{lemma}
\label{Lem:first_opt_1}
It holds
\begin{equation*}
\begin{split}
\int \bigg( \frac{1}{U}, \frac{f(U) + \beta}{UP} \bigg) \cdot (d\gamma_1 - d\gamma_2) &= \sum_i \rho_i \int_{A_i} \mathrm{curl} \bigg( \frac{1}{U}, \frac{f(U) + \beta}{UP} \bigg) \mathscr L^2 \\
&= \sum_i \rho_i \int_{A_i} \frac{P^2 - U f(U)}{P^2 U^2} \mathscr L^2,
\end{split}
\end{equation*}
whenever the series converges.  
\end{lemma}

We can observe that the solutions of the ODE with positive control are functions $P = P(U) \in \BV([0,1])$, so that the above lemma becomes:

\begin{proposition}
\label{improv}
Let $P_1, P_2: [0,1] \to [0,+\infty)$ be BV functions such that
$$
\lim_{U \to 0} P(U) = \lim_{U \to 1} P(U) = 0.
$$
Assume that the associated measure controls $\alpha_1,\alpha_2$,
\begin{equation*}
\alpha_i(dU) := \frac{1}{U}\left(DP_i(dU) + \bigg( \frac{f(U)}{P_i(U)} + \beta \bigg) dU \right), \quad i=1,2,
\end{equation*}
are positive and bounded. Then
\begin{equation}
\label{Equa:cost_expl}
\tilde \alpha_2([0,1]) - \tilde \alpha_1([0,1]) = \int_0^1 \int_{P_1(U)}^{P_2(U)} \frac{P^2 - U f(U)}{U^2 P^2} dU.
\end{equation}
\end{proposition}


\begin{proof}
We would like to integrate by parts
\begin{equation*}
\tilde \alpha_2([0,1]) - \tilde \alpha_1([0,1]) = \int_0^1 \frac{1}{U}\left(DP_i(dU) + \bigg( \frac{f(U)}{P_i(U)} + \beta \bigg) dU \right),
\end{equation*}
which is justified by Lemma \ref{Lem:lim_U_to_0}, which also yields
\begin{equation*}
\int_0^U \frac{U'}{P(U')} dU' \sim \frac{U}{\lambda_+(\beta,0)}< + \infty.
\end{equation*}
%
We can thus integrate by parts and write
\begin{align*}
\int_0^1 \left[\alpha_1(dU) -\alpha_2(dU) \right] &= \int_0^1 \frac{1}{U}\left(DP_1(dU) -DP_2(dU)\right) + f(U) \left(\frac{1}{P_1(U)}-\frac{1}{P_2(U)} \right) \,dU  \\
&= \frac{1}{U}\left(P_1(U)-P_2(U)\right)\Big|_0^1 \\
& \quad + \int_0^1 \left[\frac{1}{U^2}\left(P_1(U)-P_2(U)\right) +f(U)\left(\frac{1}{P_1(U)}-\frac{1}{P_2(U)}\right)\right] \, dU \\
&= \int_0^1 \int_{P_1(U)}^{P_2(U)} \left(-\frac{1}{U^2} +\frac{f(U)}{P^2}\right) \, dP \, dU \\
&= \int_0^1 \int_{P_1(U)}^{P_2(U)}  \frac{f(U)U-P^2}{U^2P^2}\, dP \, dU,
\end{align*}
which is the statement.
\end{proof}

Roughly speaking, the above proposition says that in order to minimize $E(\beta)$ we want the curve $\gamma$ to be as close as possible to $P^*$.

\begin{corollary}
\label{Cor:monotone_1}
Let $\beta > \beta^*$ be fixed and assume that
\begin{equation}
\label{adm}
F(U) + \beta \ge 0 \quad \forall U \in (U,1) \text{ s.t. } \Gamma_u(\beta, U)< P^*(U)<\Gamma_s(\beta, U).
\end{equation}
Then the minimization problem \eqref{effort} admits a unique optimal profile given by
\begin{equation}
\label{optgraphgen}
P_\beta(U) := \begin{cases}
\Gamma_u(\beta, U) & \Gamma_u(\beta, U)\ge P^*(U), \\
P^*(U) & \Gamma_u(\beta,U) < P^*(U) < \Gamma_s (\beta, U), \\
\Gamma_s(\beta, U) & \Gamma_s(\beta, U)\le P^*(U). \\
\end{cases}
\end{equation}
Moreover the sets
\begin{equation*}
I_1 = \big\{ \Gamma_u(\beta, U)\ge P^*(U) \big\}, \quad I_2 = \big\{ \Gamma_s(\beta, U)\le P^*(U) \big\}, \quad I_3 = \big\{ \Gamma_u(\beta, U) < P^*(U) < \Gamma_s(\beta,U) \big\},
\end{equation*}
are disjoint intervals.
\end{corollary}

\begin{proof}
Consider the three regions of the statement
\begin{equation*}
I_1 := \Big\{ U\in[0,1]: \Gamma_u(\beta, U) \ge P^*(U) \Big\}, \quad I_2 := \Big\{ U \in[0,1]: \Gamma_s(\beta, U)\le P^*(U) \Big\},
\end{equation*}
and
\begin{equation*}
I_3 := [0,1] \setminus (I_1\cup I_2) = \Big\{ U\in [0,1]: \Gamma_u(\beta, U)< P^*(U)<\Gamma_s(\beta, U) \Big\}.
\end{equation*}
Observe that $I_3$ is an interval: indeed, assume by contradiction that
\begin{equation*}
\Gamma_u(\beta,U) < P^*(U) < \Gamma_s(\beta,U), \quad \Gamma_u(\beta,\tilde U') = P^*(U') < \Gamma_s(\beta,U'),
\end{equation*}
for some $U^* < U < U' < 1$. Being the flow of the ODE smooth, then the trajectories with initial point $(U,P)$, $\Gamma_u(\beta,U) < P < P^*(U)$ must cross the curve $P^*(U)$ from below, thus contradicting \eqref{adm}. The argument $\Gamma_s(\beta)$ is completely analogous.

Since $\Gamma_u(\beta) < \Gamma_s(\beta)$, we thus deduce that $I_1$, $I_2$ are also intervals, more precisely
\begin{equation*}
I_1 = [0,\inf I_3], \quad I_2 = [\sup I_3,1].
\end{equation*}

Condition \eqref{adm} guarantees that the control \gls{alphabeta} associated to $P_\beta$, namely
\begin{equation*}
\alpha_\beta (U)= \begin{cases}
0 & U\in I_1\cup I_2, \\
\frac{F(U)+\beta}{U} \ & U \in I_3,
\end{cases}
\end{equation*}
is non-negative and, thus, $P_\beta$ is admissible.

Let $P$ any competitor in \ref{minfunct} and denote by $\alpha$ its associated control: by Lemma \ref{improv}
\begin{align*}
\int_0^1 \left[\alpha(U)-\alpha_\beta(U)\right] \, dU
&=\int_0^1 \int _{P(U)}^{P^\beta(U)} \frac{f(U)U-P^2}{U^2P^2} \, dU\\
&=\int_{I_1} \int _{P(U)}^{\Gamma_u(\beta, U)} \frac{f(U)U-P^2}{U^2P^2} \, dU\\
& \quad + \int_{I_2}\int _{P(U)}^{\Gamma_s(\beta, U)} \frac{f(U)U-P^2}{U^2P^2} \, dU \\
& \quad + \int_{I_3} \int _{P(U)}^{P^*(U)} \frac{f(U)U-P^2}{U^2P^2} \, dU.
\end{align*}
Recall that, by the Maximum Principle, the admissibility condition for $P$ implies that
\begin{equation*}
\Gamma_u(\beta, U)\le P(U)\le \Gamma_s(\beta, U) \ \forall U\in [0,1].
\end{equation*}
It follows that $\forall U \in I_1$ and $\Gamma_u(\beta,U)\le P\le P(U)$
\begin{equation*}
f(U)U-P^2 \begin{cases}
\leq 0 & U \in [0,U^*] \text{ being $f \leq 0$}, \\
= (P^*(U))^2 - P^2 \le 0 & \text{being $P^* \leq \Gamma_u(\beta)$}.
\end{cases}
\end{equation*}
Hence
\begin{equation*}
\int_{I_1} \int _{P(U)}^{\Gamma_u(\beta, U)} \frac{f(U)U-P^2}{U^2P^2} \, dP\, dU= -\int_{I_1} \int _{\Gamma_u(\beta, U)}^{P(U)} \frac{f(U)U-P^2}{U^2P^2} \, dP \, dU \ge 0.
\end{equation*}
Similarly we deduce
\begin{equation*}
\int_{I_2}\int _{P(U)}^{\Gamma_s(\beta, U)} \frac{f(U)U-P^2}{U^2P^2} \, dP\, dU\ge 0,
\end{equation*}
in this case because the integrand is positive. Finally,
\begin{equation*}
\int_{I_3} \int _{P(U)}^{P^*(U)} \frac{f(U)U-P^2}{U^2P^2} \, dU = \int_{I_3} \int _{P(U)}^{P^*(U)} \frac{(P^*(U))^2-P^2}{U^2P^2} \, dU\ge 0.
\end{equation*}
We thus conclude that the solution $P_\beta$ of \eqref{optgraphgen} is optimal.

The same argument clearly implies uniqueness: assume that $P$ is another optimal profile; then
\begin{equation*}
\begin{split}
\int_{I_1} \int _{P(U)}^{\Gamma_u(\beta, U)} \frac{f(U)U-P^2}{U^2P^2} \, dP\, dU &= \int_{I_2}\int _{P(U)}^{\Gamma_s(\beta, U)} \frac{f(U)U-P^2}{U^2P^2} \, dP\, dU \\
&= \int_{I_3} \int _{P(U)}^{P^*(U)} \frac{f(U)U-P^2}{U^2P^2} \, dU = 0
\end{split}
\end{equation*}
and thus, since all integrands have a sign, necessarily the integration intervals are empty, i.e.
\begin{equation*}
P(U):=\begin{cases}
\Gamma_u(\beta, U) & U \in I_1 \\
\Gamma_s(\beta, U)  & U \in I_2 \\
P^*(U) & U \in I_3. \\
\end{cases} \qedhere
\end{equation*}
\end{proof}

The above results suggest that if we wish to achieve a structural characterization of optimal profiles without condition (\ref{mon}), or even the more general (\ref{adm}), it might be useful to study the level sets of the function $F$. In general the regions where $F + \beta > 0$ may consist of countably many open intervals, which make the structure of the optimal solution difficult to study, in particular when we will allow $f$ and $\beta$ to vary.

In the next remark we show that the monotonicity condition \eqref{adm} is not valid even for the simplest function satisfying \eqref{H1}, i.e. a cubic polynomial.

\begin{remark}
\label{Rem:monoton_cond}
Consider the simple cubic source
\begin{equation}
\label{Equa:f_cubic}
f(U) = U (U - U^*) (1 - U),
\end{equation}
The function $F$ is represented in Fig. \ref{Fig:fcubicF} left as a function of $U^*,U$. Note that the function is not monotone for $U^* < 0.4$: for $U^* \ll 1$ there are two critical points, a local minimum and a local maximum. In particular, for $0 < -\beta \ll 1$ there are $3$ solution of $F(U) + \beta = 0$:
\begin{equation*}
\{F(U) + \beta > 0\} = (U^*,U_1) \cup (U_2,U_3), \quad U_1 < U_2,\ U_3 < 1.
\end{equation*}
We thus conclude that the Condition (\ref{mon}) fails even for the cubic source. For a deeper analysis of this case see also Remark \ref{Rem:cubic_case} in the next section.

\begin{figure}
\begin{subfigure}{.475\textwidth}
\resizebox{\textwidth}{!}{\includegraphics{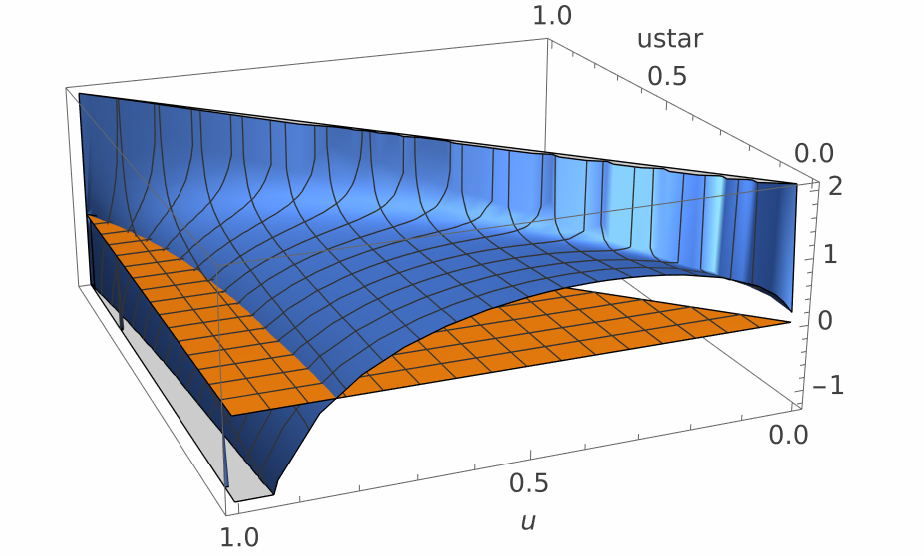}}
\end{subfigure} \hfill
\begin{subfigure}{.475\textwidth}
\resizebox{\textwidth}{!}{\includegraphics{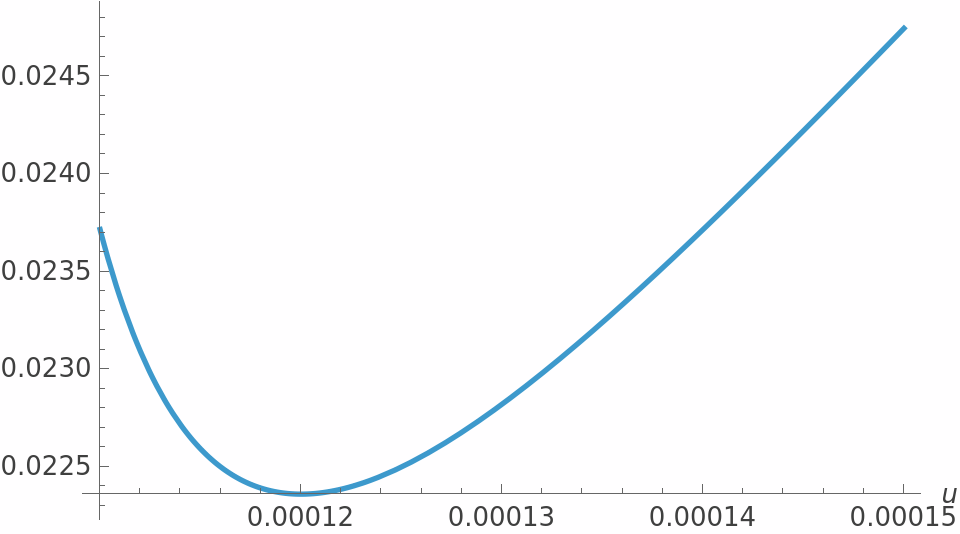}}
\end{subfigure}
\caption{The function $F$ for the cubic \eqref{Equa:f_cubic} in the region $U^* < U < 1$. 
On the right the plot of $F(U)$ for $U^* = 10^{-4}$: the minimum can be shown to be at $\sim \frac{6 U^*}{5}$. 
}
\label{Fig:fcubicF}
\end{figure}
\end{remark}

\section{Genericity of transversality properties}
\label{S:transver_flux}

Our aim is to show that there exists a dense family of polynomial sources $f$ satisfying Assumptions \eqref{H1} and enjoying the following additional transversality properties. 

\begin{enumerate}

\item \label{Point1:transv} The function \eqref{F}
\begin{equation*}
(U^*,1) \ni U \mapsto F(U) = \frac{3f(U) + U f'(U)}{2\sqrt{Uf(U)}}
\end{equation*}
has finitely many critical points and the second derivative does not vanish at these points (i.e. they are local maxima/minima).

\item \label{Point2:transv} For all $\beta$ the function \newglossaryentry{Fudef}{name=\ensuremath{F_u(\beta,U)},description={difference $\Gamma_u(\beta,U) - P^*(U)$}}
\begin{equation*}
[U^*,U_u(\beta)] \ni U \mapsto \gls{Fudef} = \Gamma_u(\beta,U) - P^*(U)
\end{equation*}
has finitely many critical points, these points are  contained in $(U^*,\gls{Uhatubeta})$, and the second derivative does not vanish at these points (i.e. they are local maxima/minima).

\item \label{Point3:transv} Similarly, for all $\beta$ the function \newglossaryentry{Fsdef}{name=\ensuremath{F_s(\beta,U)},description={difference $\Gamma_s(\beta,U) - P^*(U)$}}
\begin{equation*}
[U^*,1] \ni U \mapsto \gls{Fsdef} = \Gamma_s(\beta,U) - P^*(U)
\end{equation*}
has finitely many critical points, these points are  contained in $(\gls{Uhatsbeta},1)$, and the second derivative does not vanish at these points (i.e. they are local maxima/minima).

\end{enumerate}

We will also show that the above transversality conditions are generic, in the sense that the set of functions satisfying them is open and dense. Moreover when restricted to polynomials, this transversality has full Lebesgue measure.

This section will play a fundamental role when studying the differentiability properties of the effort \gls{Ebeta}: indeed in Section \ref{S:diffe_E_beta} we will show that $\beta \mapsto E(\beta,f)$ is differentiable by computing its derivative when $f$ belongs to this transversality set, and then verifying that the expression for the derivative $\frac{d}{d\beta} E(\beta)$ is continuous w.r.t. $\beta,f$.

\subsection{Transversality properties of $F$}
\label{Ss:transv_F}

We start with the function $F$, which does not depend on $\beta$.

A straightforward application of Sard's Lemma \cite{sard_origi} would let us conclude that the set of $\beta$'s not satisfying the transversality condition
\begin{equation*}
F'(U) \not= 0 \quad \text{when} \quad F(U) + \beta = 0
\end{equation*}
is compact with zero Lebesgue measure for every smooth source $f$, because of \eqref{Equa:bound_integr}, \eqref{Equa:limi_F}. 

%
%
%
%
%
%

The statement is presented here below.

\begin{proposition}
\label{zeros}
Let $F$ as defined in (\ref{F}). There exists  \newglossaryentry{Ncal0}{name=\ensuremath{\mathcal N_0},description={set of critical values for $F$}} $\gls{Ncal0} \subseteq \R$ compact set of negligible Lebesgue measure such that for all $\beta \in \R\setminus \mathcal{N}_0$ there is \newglossaryentry{Nbeta}{name=\ensuremath{N(\beta)},description={number of roots of $F(U) + \beta = 0$}} $N = \gls{Nbeta}$ and there are finitely many $U_0=U^*<U_1<...<U_{2N-1}<U_{2N}=1$ satisfying \newglossaryentry{D+i}{name=\ensuremath{D^+_i},description={intervals where the flow is entering the region $\{P \leq P^*\}$}}\newglossaryentry{D-i}{name=\ensuremath{D^-_i},description={intervals where the flow is exiting the region $\{P \leq P^*\}$}}
\begin{equation*}
F(U)+\beta \begin{cases} =0 & \text{ for }  U=U_i, \ i=1,...,2N-1, \\
>0  & \text{ for } U \in \gls{D+i} := (U_{2i},U_{2i+1}), i=0,...,N-1, \\
<0  & \text{ for } U \in \gls{D-i} := (U_{2i+1},U_{2i+2}), i=0,...,N-1,
\end{cases}
\end{equation*}
and
\begin{equation*}
(-1)^i F'(U_i) > 0 \quad \forall i=1,..., 2N-1.
\end{equation*}
\end{proposition}

\begin{proof}
%
First note that, by (\ref{H1}), $F \in C^1((U^*, 1))$ and there holds
\begin{equation*}
F'(U)=\frac{2U f(U) f'(U) + U^2 f(U) f(U)''-\frac{3}{2}f^2(U)-\frac{1}{2}U^2(f'(U))^2}{2(Uf(U))^\frac{3}{2}},
\end{equation*}
which implies in accordance with \eqref{Equa:limi_F}
\begin{equation*}
\lim_{U \to U^*, 1} F'(U)=-\infty.
\end{equation*}
In particular, we can find $I\subseteq(U^*,1)$ compact such that $F'(U)<0 \ \forall U \notin I$.

A direct application of Sard's Lemma \cite{sard_origi} then implies that the compact set
$$
\gls{Ncal0} := - F(\{U\in (U^*,1): \ F'(U)=0\})
$$
is Lebesgue negligible, i.e.
\begin{align*}
\mathcal L^1(\mathcal N_0) &= \mathcal{L}^1 \big( F \big( \big\{ U \in (U^*,1): F'(U) = 0 \big\} \big) \big) = 0 .
\end{align*}

Fix $\beta\in \R\setminus \mathcal{N}_0$. From the coercivity conditions \eqref{Equa:limi_F}
we can infer that the level set \newglossaryentry{Lbeta}{name=\ensuremath{L_\beta},description={level set $\{F(U) + \beta = 0\}$}}
$$
\gls{Lbeta}:= \big\{ U\in(U^*,1): \ F(U)+\beta=0 \big\}
$$
is compact.
Assume by contradiction that there is a sequence $\{U_n\}_{n\in\mathbf{N}}\subseteq L_\beta$ such that $U_n\neq U_k$ for $n\neq k$. By the previous observation we can assume that, up to subsequences, $U_n\to U \in L_\beta$. Since we can also assume that definitely $U_n\neq U$, the derivative of $F$ at $U$ can be computed as follows
\begin{equation*}
F'(U)=\lim_{n\to+\infty} \frac{F(U_n)-F(U)}{U_n-U}=0,
\end{equation*}
against the hypothesis $\beta \notin \mathcal{N}_0$. 

We conclude that the cardinality of $L_\beta$ is finite and, once again by \eqref{Equa:limi_F}, we can find $U_0 = U^* < U_1 < ... < U_{2N-1} < U_{2N} = 1$ as in the statement.
\end{proof}

\begin{remark}
\label{zerosp}
Clearly, if $K \subseteq\R\setminus \mathcal{N}_0$ is compact, we can find $N \ge 1$ such that $N(\beta)\le N \ \forall \beta \in K$. Indeed, first of all, by \eqref{Equa:limi_F}, we find $I \Subset (U^*,1)$ such that $-F(U) \notin K \ \forall U \notin I$, which implies
\begin{equation*}
\big\{ U\in(U^*,1): F(U)+\beta=0 \big\} = \big\{ U\in I: F(U)+\beta=0 \big\} \ \forall \beta \in K.
\end{equation*}
By contradiction, assume that for all $n \in \N$ there exists $\beta_n \in K$ such that
$$
\sharp \{U\in I: F(U)+\beta_n=0\} \geq n.
$$
By compactness, up to subsequences, we may assume $\beta_n \to \beta \in K$. 
By induction, construct a sequence $\{U_n\}_{n\in\N}\subseteq I$ such that
$$
F(U_n)+\beta_n=0 \ \forall n\in \N \quad \text{and} \quad U_n \neq U_k \ \forall n\neq k.
$$
With the same argument as the one employed in the proof of Lemma \ref{zeros}, we deduce that
$$
U_n \to U \in I \quad \text{such that} \quad F(U)+\beta=0 \ \text{and} \ F'(U)=0,
$$
which contradicts the hypothesis that $-\beta$ is not singular. The same proof shows that $N(\beta)$ is constant in any connected component of $\R \setminus \mathcal N_0$.
\end{remark}

In the case of polynomial $f(U)$, we can further improve the above proposition by showing that for all $\beta$ the number of critical points is finite and they are local minimizers or maximizers outside a codimension $1$ set in the space of polynomial with fixed degree.

\begin{proposition}
\label{Prop:F_transv_poly}
Assume that
$$
f(U) = \sum_{i=1}^k a_i U^i
$$
is a polynomial function satisfying condition (\ref{H1}) and define, as in \eqref{F},
\begin{equation*}
F(U) = \frac{3f(U)+Uf'(U)}{2\sqrt{Uf(U)}}, \quad U \in (U^*,1).
\end{equation*}
Then the following holds.
\begin{enumerate}
\item 
The set
\begin{equation*}
\big\{ U \in (U^*,1): \ F'(U)=0 \big\}  
\end{equation*}
is at most finite for all $\beta \geq \beta^*$.

\item There exists a closed set \newglossaryentry{NTFk}{name=\ensuremath{\mathcal{NT}_k},description={non-transversality set of polynomials of degree $k$}} \gls{NTFk} of codimension $1$ in the space of polynomials with degree $k$ satisfying Assumption \eqref{H1}, such that outside this set the second derivative of $F$ is not $0$ in the critical points, i.e. they are maxima or minima: this set \gls{NTFk} is an algebraic surface.

\item For all fixed $U^* \in (0,1)$, the set of non-transversality
\begin{equation*}
\begin{split}
&\bigg\{ (U,a_1,\dots,a_k), f(U) = \sum_{i=1}^k a_i U^i \ \text{satisfies (\ref{H1}), $f(U^*) = 0$ and} \ F'(U) = F''(U) = 0 \bigg\} \\
&\subset \bigg\{ (U,a_1,\dots,a_k), f(U) = \sum_{i=1}^k a_i U^i \ \text{satisfies (\ref{H1}), $f(U^*) = 0$} \bigg\} 
\end{split}
\end{equation*}
has codimensions $1$, and it is an algebraic surface. 
\end{enumerate}
\end{proposition}

Note that since $f(U)$ has $3$ zeros, clearly for $k \leq 2$ the sets are empty. Moreover, the non tranversality set has zero Lebesgue measure.


\begin{proof}
{\it Point (1).} Clearly by \eqref{Equa:limi_F} we have that $F'(U)$ cannot be identically $0$ 
: then, since
\begin{equation*}
F'(U)=\frac{2Uf(U) f'(U) + U^2 f(U) f''(U) - \frac{3}{2}f^2(U) - \frac{1}{2}U^2(f'(U))^2}{2(Uf(U))^\frac{3}{2}},
\end{equation*}
zeros of $F'$ are also zeros for the non-identically null polynomial
\begin{equation}
\label{Equa:Q_def}
U \mapsto Q(U) = 2U f(U) f'(U)+U^2 f(U) f''(U)-\frac{3}{2}f^2(U)-\frac{1}{2}U^2(f'(U))^2.
\end{equation}
We deduce that they are finite.
%

\medskip

\noindent{\it Point (2).} Consider the function
\begin{equation*}
(U,a_1,\dots,a_k) \mapsto Q(U,a_1,\dots,a_k) = 2 U f(U) f'(U) + U^2 f(U) f''(U)-\frac{3}{2}f^2(U)-\frac{1}{2}U^2(f'(U))^2,
\end{equation*}
where by Assunmption \eqref{H1}
\begin{equation*}
f(U) = \sum_{i=1}^k a_i U^i, \quad f(U^*) = \sum_{i=1}^k a_i (U^*)^i = 0, \quad f(1) = \sum_{i=1}^k a_i = 0.
\end{equation*}
In this part the second condition is not used.

The function $Q$ can be written explicitly as the quadratic form
\begin{equation*}
\begin{split}
Q(U,a_1,\dots,a_k) &= \frac{1}{2} \sum_{ij} a_i a_j U^{i+j} \big( i^2 + j^2 + i + j - ij - 3 \big) \\
&= \frac{1}{2} (a_i U^i)^T [\gamma_{ij}] (a_i U^i), \quad \gamma_{ij} = i^2 + j^2 + i + j - ij - 3.
\end{split}
\end{equation*}
If we define the $k \times k$-matrices
\begin{equation*}
L_1 = \left[ \begin{array}{ccccc}
1 & 1 & 1 & \dots & 1 \\
0 & 1 & 1 & \dots & 1 \\
0 & 0 & 1 & \dots & 1 \\
\vdots & \vdots & \vdots & \ddots & \vdots \\
0 & 0 & \dots & 0 & 1
\end{array} \right], \quad L^{-1}_1 = \left[ \begin{array}{ccccc}
1 & - 1 & 0 & \dots & 0 \\
0 & 1 & -1 & \dots & 0 \\
0 & 0 & 1 & \dots & 0 \\
\vdots & \vdots & \vdots & \ddots & \vdots \\
0 & 0 & \dots & 0 & 1
\end{array} \right],
\end{equation*}
\begin{equation*}
L_2 = \left[ \begin{array}{cccccc}
1 & 0 & 0 & 0 & \dots & 0 \\
0 & 1 & 1 & 1 & \dots & 1 \\
0 & 0 & 1 & 1 & \dots & 1 \\
0 & 0 & 0 & 1 & \dots & 1 \\
\vdots & \vdots & \vdots & \vdots &\ddots & \vdots \\
0 & 0 & 0 & 0 & \dots & 1
\end{array} \right], \quad L^{-1}_2 = \left[ \begin{array}{cccccc}
1 & 0 & 0 & 0 & \dots & 0 \\
0 & 1 & -1 & 0 & \dots & 0 \\
0 & 0 & 1 & -1 & \dots & 0 \\
0 & 0 & 0 & 1 & \dots & 0 \\
\vdots & \vdots & \vdots  & \vdots & \ddots & \vdots \\
0 & 0 & 0 & 0 & \dots & 1
\end{array} \right],
\end{equation*}
\begin{equation*}
L_3 = \left[ \begin{array}{ccccccc}
1 & 0 & 0 & 0 & 0 & \dots & 0 \\
0 & 1 & 0 & 0 & 0 &\dots & 0 \\
0 & 0 & 1 & 1 & 1 & \dots & 1 \\
0 & 0 & 0 & 1 & 1 & \dots & 1 \\
0 & 0 & 0 & 0 & 1 & \dots & 1 \\
\vdots & \vdots & \vdots & \vdots & \vdots & \ddots & \vdots \\
0 & 0 & 0 & 0 & 0 & \dots & 1
\end{array} \right], \quad L^{-1}_3 = \left[ \begin{array}{ccccccc}
1 & 0 & 0 & 0 & 0 & \dots & 0 \\
0 & 1 & 0 & 0 & 0 & \dots & 0 \\
0 & 0 & 1 & -1 & 0 & \dots & 0 \\
0 & 0 & 0 & 1 & -1 &\dots & 0 \\
0 & 0 & 0 & 0 & 1 &\dots & 0 \\
\vdots & \vdots & \vdots & \vdots & \vdots & \ddots & \vdots \\
0 & 0 & 0 & 0 & 0 & \dots & 1
\end{array} \right],
\end{equation*}
and
\begin{equation}
\label{Equa:L_matrix}
L = L_3 L_2 L_1 = \left[
\begin{array}{cccccc}
1 & 1 & 1 & 1 & \dots & 1 \\
0 & 1 & 2 & 3 & \dots & k-1 \\
0 & 0 & 1 & 3 & \dots & \frac{(k-1)(k-2)}{2} \\
0 & 0 & 0 & 1 & \dots & \frac{(k-2)(k-3)}{2} \\
\vdots & \vdots & \vdots & \vdots & \ddots & \vdots \\
0 & 0 & 0 & 0 & \dots & 1
\end{array} \right],
\end{equation}
then
\begin{equation}
\label{Equa:diag_gamma}
L^{-T} [\gamma_{ij}] L^{-1} = \left[ \begin{array}{cccccc}
0 & 3 & 2 & 0 & \dots & 0 \\
3 & -1 & 0 & 0 & \dots & 0 \\
2 & 0 & 0 & 0 & \dots & 0 \\
0 & 0 & 0 & 0 & \dots & 0 \\
\vdots & \vdots & \vdots & \vdots & \ddots & \vdots \\
0 & 0 & 0 & 0 & \dots & 0
\end{array} \right].
\end{equation}

Define the map
\begin{equation}
\label{Equa:map_tildea}
(U,a_2,\dots,a_k) \mapsto \tilde a(U,a_2,\dots,a_k) = L (a_i U^i), \quad a_1 = - \sum_i a_i.
\end{equation}
This map is a diffeoemorphism because $U \in (U^*,1)$ and the Jacobian matrix
\begin{equation*}
\left[ \begin{array}{ccccc}
- \sum_{i=2}^k a_i & - U & -U & \dots & - U \\
2 a_2 U & U^2 & 0 & \dots & 0 \\
3 a_3 U^2 & 0 & U^3 & \dots & 0 \\
\vdots & \vdots & \vdots & \ddots & \vdots \\
k a_k U^{k-1} & 0 & 0 & \dots & U^k
\end{array} \right],
\end{equation*}
has determinant
\begin{equation*}
U^{\frac{k(k+1)}{2}-1} \sum_{i=2}^k (i-1) a_i = U^{\frac{k(k+1)}{2}-1} f'(1) \not= 0,    
\end{equation*}

Using \eqref{Equa:map_tildea} we obtain
\begin{equation*}
Q(U,a_1,\dots,a_k) = \frac{1}{2} \tilde a^T L^{-T} [\gamma_{ij}] L^{-1} \tilde a.
\end{equation*}
Performing the same computations for the derivative we obtain
\begin{equation*}
\begin{split}
\partial_U Q(U,a_1,\dots,a_k) &= \frac{1}{2} \sum_{ij} a_i a_j U^{i+j-1} (i+j) \gamma_{ij} \\
&= \frac{1}{2U} \tilde a^T \Big( L^{-T} \diag(1,\dots,k) L^T L^{-T} [\gamma_{ij}] L^{-1} + L^{-T} [\gamma_{ij}] L^{-1} L \diag(1,\dots,k) L^{-1} \Big) \tilde a.
\end{split}
\end{equation*}

We will show that the conditions 
\begin{equation}
\label{Equa:QpartialQ_0}
Q(U,a_1,\dots,a_k) = 0, \quad \partial_U Q(U,a_1,\dots,a_k) = 0
\end{equation}
define a set of codimention at least $2$ in $(U,\tilde a) \in \R^{k+1}$.
%
%

For this purpose, it is enough to consider the first $3$ components of $\tilde a$: indeed because of \eqref{Equa:diag_gamma} we have
\begin{equation*}
L^{-T} \diag(1,\dots,k) L^T L^{-T} [\gamma_{ij}] L^{-1} + L^{-T} [\gamma_{ij}] L^{-1} L \diag(1,\dots,k) L^{-1} = \left[ \begin{array}{cc} \left[ \begin{array}{ccc}
0 & 9 & 14 \\
9 & 2 & 0 \\
14 & 0 & 0
\end{array} \right] & A_{12} \\
A_{12}^T & 0 \end{array} \right],
\end{equation*}
with $A_{12} \in \R^{3 \times (k-3)}$. The second condition in \eqref{Equa:QpartialQ_0} gives then
\begin{equation}
\label{Equa:der_tilde_a}
\left( \begin{array}{ccc}
\tilde a_1 & \tilde a_2 & \tilde a_3
\end{array} \right) \left[ \begin{array}{ccc}
0 & 9 & 14 \\
9 & 2 & 0 \\
14 & 0 & 0
\end{array} \right] \left( \begin{array}{c}
\tilde a_1 \\ \tilde a_2 \\ \tilde a_3
\end{array} \right) + 2 \left( \begin{array}{ccc}
\tilde a_4 & \dots & \tilde a_k
\end{array} \right) A_{12}^T\left( \begin{array}{c}
\tilde a_1 \\ \tilde a_2 \\ \tilde a_3
\end{array} \right) = 0.
\end{equation}
%

Assume that the linear term in $(\tilde a_4,\dots,\tilde a_k)$ of \eqref{Equa:der_tilde_a} is $0$: then by \eqref{Equa:QpartialQ_0} the first $3$ components of $\tilde a$ are required to satisfy 
\begin{equation*}
\begin{cases}
6 \tilde a_1 \tilde a_2 - \tilde a_2^2 + 4 \tilde a_1 \tilde a_3 = 0, \\
18 \tilde a_1 \tilde a_2 + 2 \tilde a_2^2 + 28 \tilde a_1 \tilde a_3 = 0,
\end{cases}
\end{equation*}
a system whose solutions are the three lines
\begin{equation}
\label{Equa:thrre_tildea}
\big\{ \tilde a_1 = \tilde a_2 = 0 \big\}, \quad \big\{ \tilde a_2 = \tilde a_3 = 0 \big\}, \quad \bigg\{ \tilde a_2 = \frac{8}{3} \tilde a_1, \tilde a_3 = - \frac{20}{9} \tilde a_1 \bigg\}.
\end{equation}
%
The first condition reads as
\begin{equation*}
\sum_i a_i U^i = f(U) = 0, \quad \sum_i a_i (i-1) U^i = U \frac{d}{dU} f(U) - f(U) = 0,
\end{equation*}
which is not possible under Assumption \eqref{H1}. The second condition reads as
\begin{equation*}
U \frac{df}{dU} - f = 0, \quad \sum_i a_i \frac{(i-1)(i-2)}{2} U^i = \frac{U^2}{2} \frac{d^2 f}{dU^2} - \bigg( U \frac{df}{dU} - f \bigg) = \frac{U^2}{2} \frac{d^2 f}{dU^2} = 0,
\end{equation*}
and substituting these conditions into \eqref{Equa:QpartialQ_0} we obtain
\begin{equation*}
\partial_U Q(U,a_1,\dots,a_k) = U^2 f \frac{d^3f}{dU^3} = 0.
\end{equation*}
Thus this case corresponds to
\begin{equation*}
U \frac{df}{dU} = f, \quad \frac{d^2 f}{dU^2} = 0, \quad \frac{d^3 f}{dU^3} = 0,
\end{equation*}
which has higher codimension (3).
%

If the second term of \eqref{Equa:der_tilde_a} is not $0$, then it defines a linear function over a codimension $1$ algebraic surface given by the first condition of \eqref{Equa:QpartialQ_0}, i.e. a codimention $2$ algebraic surface. This set corresponds to an algebraic surface in $\R^{k+1}$ by the change of variable \eqref{Equa:map_tildea}: its projection on $(a_1,\dots,a_k)$ is then an algebraic set \gls{NTFk}  of codimension $1$.

\medskip

\noindent{\it Point (3).} The final part is whether the codimension of the set
\begin{equation}
\label{Equa:critical1}
\begin{split}
&\bigg\{ (U,a_1,\dots,a_k), f(U) = \sum_{i=1}^k a_i U^i \ \text{satisfies (\ref{H1}), $f(U^*) = 0$ and} \ F'(U) = F''(U) = 0 \bigg\} \\
\end{split}
\end{equation}
is $1$ in the set
\begin{equation}
\label{Equa:critical2}
\begin{split}
\bigg\{ (U,a_1,\dots,a_k), f(U) = \sum_{i=1}^k a_i U^i \ \text{satisfies (\ref{H1}), $f(U^*) = 0$} \bigg\} 
\end{split}
\end{equation}
with $U^*$ fixed: this condition is equivalent to
\begin{equation*}
a_1 + a_2 U^* = - \sum_{i=3}^k a_i (U^*)^{i-1}, \quad a_1 + a_2 = - \sum_{i=3}^k a_i.
\end{equation*}
The space \eqref{Equa:critical1} is thus the section of the codimension 2 surface given by \eqref{Equa:QpartialQ_0} into the affine codimension $1$ sub-space \eqref{Equa:critical2}, and then it is of codimension $1$.
%
\end{proof}

\begin{remark}
\label{Rem:cubic_case}
For the cubic source
\begin{equation*}
f(U,U^*) = U (U - U^*) (1 - U)
\end{equation*}
one gets using \eqref{Equa:L_matrix}
\begin{equation*}
\tilde a_1 = f(U,U^*) = - U^*U + (1 + U^*) U^2 - U^3, \quad \tilde a_2 = U f'(U,U^*) - f(U,U^*) = (1 + U^*) U^2 - 2 U^3, \quad \tilde a_3 = - U^3,
\end{equation*}
so that the three solutions \eqref{Equa:thrre_tildea} yield
\begin{equation*}
\tilde a_1 = 0, \ \tilde a_2 = 0 \quad \Rightarrow \quad f(U,U^*) = 0, \ U \frac{d}{dU} f(U,U^*) = 0
\end{equation*}
\begin{equation*}
\tilde a_2 = 0, \ \tilde a_3 = 0 \quad \Rightarrow \quad U \frac{d}{dU} f(U,U^*) = f(U,U^*), \ U^3 = 0,
\end{equation*}
\begin{equation*}
\tilde a_2 = \frac{8}{3} \tilde a_1, \ \tilde a_3 = - \frac{20}{9} \tilde a_1 \quad \Rightarrow \quad U \frac{d}{dU} f(U) - f(U) = \frac{8}{3} f(U), \ U^3 = \frac{20}{9} f(U).
\end{equation*}
The only case with is not a $0$ of $f(U,U^*)$ is the third one, whose unique solution is
\begin{equation*}
U = \frac{2}{5}, \quad U^* = \frac{7}{25}.
\end{equation*}
%
%

\begin{figure}
\begin{subfigure}{.475\textwidth}
\resizebox{\textwidth}{!}{\includegraphics{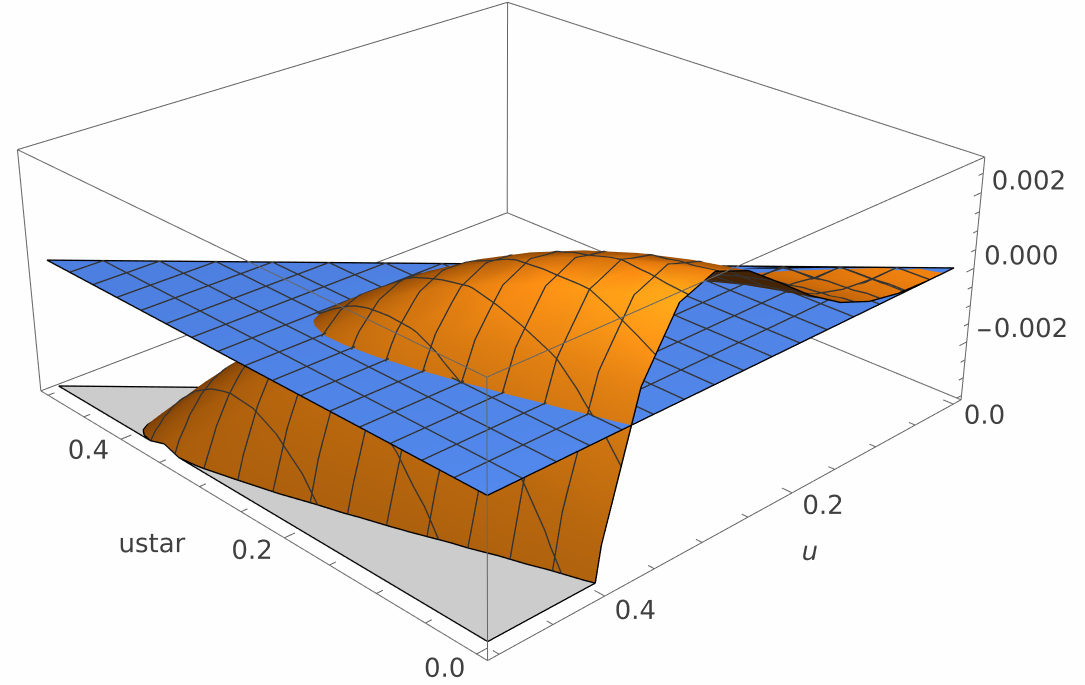}}
\end{subfigure} \hfill
\begin{subfigure}{.475\textwidth}
\resizebox{\textwidth}{!}{\includegraphics{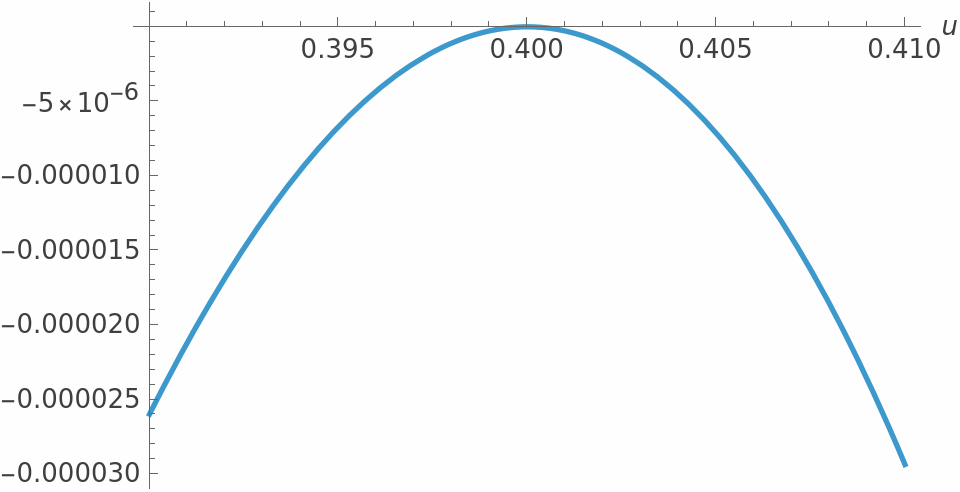}}
\end{subfigure}
\caption{The numerator $Q(U)$ of the function $\partial_U F(U,U^*)$ for the cubic case of Remark \ref{Rem:cubic_case} and its graph when $U^* = \frac{7}{25}$.}
\label{Fig:Fcubic_der}
\end{figure}
\end{remark}

\subsection{Transversality of $\Gamma_u(\beta,U)$, $\Gamma_s(\beta,U)$ and $P^*(U)$}
\label{Ss:transvers_Gammaus}

We turn now the attention to the transversality properties of $\Gamma_u,\Gamma_s$ and $P^*$.

\begin{proposition}
\label{Prop:transv_Gamma_gen}
There exists a compact set \newglossaryentry{Ncal1}{name=\ensuremath{\mathcal N_1},description={set of non-transversality of $\Gamma_s,\Gamma_u$ and the curve $P^*$}} $\gls{Ncal1} \subseteq \R$ of negligible Lebesgue measure such that for all $\beta \in [\beta^*,+\infty)\setminus \mathcal{N}_1$ the unstable and stable manifolds $\Gamma_u(\beta,U),\Gamma_s(\beta,U)$ given by \eqref{Equa:Gammau_eq0}, \eqref{Equa:Gammas_eq0} respectively, cross transversally the critical curve $P^*$ only at a finite number of points: in other words, the sets \newglossaryentry{Cs}{name=\ensuremath{C_s},description={intersection points of $\Gamma_s$ with $P^*$}} \newglossaryentry{Cu}{name=\ensuremath{C_u},description={intersection points of $\Gamma_u$ with $P^*$}}
\begin{equation*}
\gls{Cs} := \big\{ U\in (U^*, 1): \Gamma_s(U,\beta)=P^*(U) \big\}, \quad \gls{Cu} := \big\{ U\in (U^*, 1): \Gamma_u(U,\beta)=P^*(U) \big\}
\end{equation*}
are at most finite and 
\begin{equation}
\label{tang_10}
F(U) + \beta \neq 0 \quad \forall U \in C_s \cup C_u,
\end{equation}
if $\beta \notin \mathcal N_1$.
\end{proposition}

\begin{proof}{\it Study of $\Gamma_s$ and $P^*$.}
First of all, we wish to show that, for $\beta$ sufficiently large, the stable manifold associated to equation \eqref{functeq} and the critical curve $P^*$ cannot cross tangentially, i.e. the set \newglossaryentry{Ncal1s}{name=\ensuremath{\mathcal N_{1,s}},description={set of non-transversality of $\Gamma_s$ and the curve $P^*$}}
\begin{equation*}
\gls{Ncal1s} := \Big\{ \beta \in [\beta^*, +\infty): \Gamma_s(\beta, U)=P^*(U) \text{ and } F(U)+\beta=0 \text{ for some } U\in (U^*, 1) \Big\}
\end{equation*}
is bounded.

Note that, for $\beta>0$, the function $U \mapsto \beta(1-U)$ is a subsolution to \eqref{functeq} in  $[U^*,1]$, which implies
\begin{equation}
\label{lbound}
\Gamma_s(U,\beta)\ge \beta(1 - U) \quad \forall U \in [U^*,1].
\end{equation}
We can thus estimate for $U^* \le U\le 1$
\begin{align*}
\Gamma _s(\beta, U) - \beta(1-U) =\int_U^1 \frac{f(U')}{\Gamma_s(\beta, U')}\, dU' \le \frac{1}{\beta}\int_U^1 \frac{f(U')}{1-U'} \, dU'\le -\frac{f'(1)}{\beta}(1+C_1)(1-U),
\end{align*}
where $C_1>0$ is a constant only depending on $\|f\|_{C^2}$. When $U$ is sufficiently close to $1$, there holds
$$
3f(U)+Uf'(U) \underset{U \to 1}{\sim} 4 f'(U) U < 0,
$$
and thus, assuming also
$$
\beta^2 \ge 1+f'(1)(1+C_1),
$$
we conclude that
\begin{align*}
-\frac{3f(U)+Uf'(U)}{2\beta\Gamma_s(\beta, U)}\ge -\frac{3f(U)+Uf'(U)}{[\beta^2-f'(1)(1+C_1)](1-U)} \sim -\frac{f'(1)}{1-U} \to +\infty \quad \text{as} \ U \to 1.
\end{align*}
In particular, we can find $\overline U \in (U^*, 1)$ only depending on $f$ such that
\begin{equation}
\label{ge1}
- \frac{3f(U)+Uf'(U)}{2\beta\Gamma_s(\beta, U)} > 1 \quad \forall U \in [\overline U, 1] \text{ and } \forall \beta > \sqrt{1 + f'(1) (1 + C_1)}.
\end{equation}

Furthermore, note that, once again by \eqref{lbound}, the crossing condition
\begin{equation*}
\Gamma_s(\beta,U) = P^*(U)
\end{equation*}
implies, in particular,
\begin{equation*}
\beta (1-U) \le \Gamma_s(\beta,U)=\sqrt{Uf(U)}\le\sqrt{C_2(1-U)},
\end{equation*}
i.e.
\begin{equation}
\label{lb}
U \ge 1 - \frac{C_2}{\beta^2}, \quad \text{with } C_2 = C_2(\|f\|_{C^2}).
\end{equation}
We set
\begin{equation}
\label{betabar}
\overline {\beta}:=\max \bigg\{ \sqrt{1+f'(1)(1+C_1)}, \sqrt{\frac{C_2}{1-\overline{U}}} \bigg\}.
\end{equation}
By contradiction, let $\beta > \overline{\beta}$ and $U_T \in (U^*,1)$ such that
\begin{equation}
\label{tangu}
\Gamma_s(\beta, U_T)=\sqrt{U_Tf(U_T)} \quad \text{and} \quad F(U_T)+\beta=0.
\end{equation}
By \eqref{lb} and the choice of $\overline {\beta}$ in \eqref{betabar}, we have $U_T\ge \overline{U}$ and, thus, by \eqref{tangu} and \eqref{ge1},
\begin{align*}
1=-\frac{F(U_T)}{\beta}=-\frac{3f(U_T)+U_Tf'(U_T)}{2\beta\Gamma_s(\beta, U_T)}> 1,
\end{align*}
which clearly contradicts our assumption. We can thus conclude $\mathcal{N}_{1,s} \subseteq [\beta^*, \overline{\beta}]$.

By \eqref{Equa:limi_F}, there exists $I \subseteq (U^*,1)$ compact such that $F(U)+\beta\neq 0 \ \forall \beta \in [\beta^*, \overline {\beta}]$ and $\forall U \notin I$, which implies
\begin{equation*}
\mathcal{N}_{1,s} = \Big\{ \beta \in[\beta^*, \overline {\beta}]: \Gamma_s(U, \beta)=P^*(U) \text{ and } F(U)+\beta=0 \text{ for some } U\in I \Big\},
\end{equation*}
i.e. $\mathcal{N}_{1,s}$ is compact. By Proposition \ref{zeros}, it clearly holds
\begin{equation*}
\mathcal{L}^1(\mathcal{N}_{1,s}) = \mathcal{L}^1 \big( \mathcal{N}_{1,s} \setminus \mathcal{N}_0 \big).
\end{equation*}
Since the set $\R \setminus\mathcal{N}_0$ is open, it can be expressed as the countable union of compact intervals: by Remark \ref{zerosp} and the Implicit Function Theorem, if $I$ is one of these intervals, we can exhibit $N \ge 1$ and smooth functions \newglossaryentry{Uibeta}{name=\ensuremath{U_i(\beta)},description={roots of $F(U) + \beta = 0$}} $U_i = \gls{Uibeta}$, $i=1,...,N$, such that $\forall \beta \in I$ it holds
\begin{equation}
\label{Equa:intervals_Ui}
\big\{ U \in (U^*,1): \ F(U)+\beta=0 \big\} = \big\{ U_1(\beta) < ... < U_{N}(\beta) \big\}. 
\end{equation}

Fix $\beta_0 < \beta^*$ and define the function \newglossaryentry{Gs}{name=\ensuremath{G_s},description={transversality function for $\Gamma_s$ and $P^*$}}
\begin{equation*}
\begin{array}{ccccc}
\gls{Gs} &:& (U^*,1) \times (\beta_0, + \infty) &\to& \R^2, \\
&& (U,\beta) &\mapsto&  (P^*(U)-\Gamma_s(\beta,U), F(U)+\beta).
\end{array}
\end{equation*}
If $I$ is one of the intervals where \eqref{Equa:intervals_Ui} holds, in order to prove our thesis it will be sufficient to show that
\begin{equation*}
\mathcal{L}^1 \Big( \Big\{ \beta \in I: \ G_s(U, \beta)= 0 \text{ for some } U\in I \Big\} \Big) = 0.
\end{equation*}
This means that
\begin{equation*}
\Big\{ \beta \in K : \ G_s(U, \beta)=0 \text{ for some } U \in (U^*, 1) \Big\} = \bigcup_{i=1}^N \big\{ \beta \in I: \ G_s(U_i(\beta), \beta) = 0 \big\} =: \bigcup_{i=1}^N \mathcal N_{1,s,i}.
\end{equation*}
Computing the Jacobian matrix of $G$ we get
\begin{equation*}
DG_s(U,\beta) = \left[ \begin{array}{cc}
F(U)+\beta  & -\partial_{\beta} \Gamma_s(U,\beta) \\
F'(U) & 1
\end{array} \right],
\end{equation*}
and recalling that by \eqref{Equa:Gammas_beta}
\begin{equation*}
\partial_\beta \Gamma_s(U,\beta) > 0, 
\end{equation*}
we obtain 
$$
\det(DG_s(U_i(\beta), \beta)) = \partial_{\beta} \Gamma_s(U_i(\beta),\beta)F'(U_i(\beta))\neq 0 \quad \forall \beta \in I, i=1,...,N.
$$
We deduce that, given $\beta_0\in I$ such that $G_s(U_i(\beta_0), \beta_0)=0$, necessarily
\begin{equation*}
G_s(U_i(\beta), \beta) = (P^*(U_i(\beta)) - \Gamma_s(U_i(\beta),\beta),0) \neq 0 \quad \forall \beta \in (\beta_0-\delta_0, \beta_0+\delta_0) \setminus \beta_0,
\end{equation*}
i.e. points in $N_i$ are isolated and, thus, at most locally finite in $I$.

It remains to show that the cardinality of $C_s$ is finite. We exploit once again the argument in proof of Proposition \ref{zeros} and Remark \ref{zerosp}: by contradiction consider a sequence $U_n \subseteq C_s$ consists of distinct points and $U\in [U^*,1]$ such that $U_n \to U$. Clearly there holds $\Gamma_s(U, \beta)=P^*(U)$ and
\begin{equation*}
F(U)+\beta=(P^*-\Gamma_s(\cdot, \beta))'(U)=\lim_{n\to\infty}  \frac{(P^*(U_n)-\Gamma_s(U_n, \beta))-(P^*(U)-\Gamma_s(U, \beta))}{U_n-U}=0,
\end{equation*}
which is a contradiction both if $U\in \{U^*,1\}$ (by \eqref{Equa:limi_F}) and if $U\in (U^*,1)$ (by \eqref{tang_10}).

The proof for $C_u$ will follow the same line. Here, $\Gamma_u(U,\beta)=0 $ for all $\beta \ge \beta^{**}$, in which case the unstable manifold and $P^*$ never cross in the open interval $(U^*,1)$: in particular
\begin{align*}
\mathcal{N}_{1,u} :&= \Big\{ \beta \in [\beta^*, +\infty): \Gamma_u(\beta, U)=P^*(U) \text{ and } F(U)+\beta=0 \text{ for some } U \in (U^*, 1) \Big\} \\
&= \Big\{ \beta \in [\beta^*, \beta^{**}]: G_u(\beta, U)=(0,0)  \text{ for some } U \in (U^*, 1) \Big\},
\end{align*}
with \newglossaryentry{Gu}{name=\ensuremath{G_u},description={transversality function for $\Gamma_u$ and $P^*$}}
\begin{equation*}
\begin{array}{ccccc}
\gls{Gu} &:& (U^*,1) \times (\beta_0, +\infty) &\to& \R^2, \\
&& (U,\beta) &\mapsto& (P^*(U)-\Gamma_u(\beta,U), F(U)+\beta).
\end{array}
\end{equation*}

Finally, the thesis follows setting $\mathcal{N}_1 := \mathcal{N}_0\cup \mathcal{N}_{1,s} \cup \mathcal{N}_{1,u}$.
\end{proof}

\begin{remark}
\label{Rem:controll_inter}
The same computation as in Remark \ref{zerosp} gives that the number of intersections of $\Gamma_u(U),\Gamma_s(U)$ and $P^*(U)$ is constant in every connected component of $\R \setminus \mathcal N_1$.
\end{remark}

As in the case of $F(U) + \beta = 0$, by a considering polynomial source $f(U)$ we can improve Proposition \ref{Prop:transv_Gamma_gen} obtaining that \gls{Ncal1} is finite.

\begin{proposition}
\label{Prop:Gamma_trans_poly}
Assume that the source $f(U)$ is a polynomial satisfying Assumption \eqref{H1} with $U^*$ fixed. Then
\begin{enumerate}
\item for all $\beta$ the cardinality of $C_s(\beta)$, $C_u(\beta)$ is uniformly bounded;
\item the set \gls{Ncal1}, defined in Proposition \ref{Prop:transv_Gamma_gen} as the set of speeds such that $\Gamma_u(\beta,U)$, $\Gamma_s(\beta,U)$ do not cross $P^*(U)$ transversally, is finite.
\end{enumerate}
\end{proposition}

\begin{proof}
It is sufficient to prove the proposition for $\Gamma_u$, being the proof for $\Gamma_s$ entirely similar.

\smallskip

\noindent{\it Finiteness of \gls{Cs}, \gls{Cu}.} For every $\beta$, the set where $F(U) + \beta = 0$ is a subset of the zeros of the polynomial
\begin{equation*}
F(U)^2 = \beta^2 U f(U),
\end{equation*}
and then the number of zeros of $F(U) + \beta$ is uniformly bounded w.r.t. $\beta$.
Since there can be only one crossing point where of $P^*(U) - \Gamma_u(U) = 0$ in every connected component of $\{F(U) + \beta \not= 0\}$, it follows that
\begin{equation*}
\sharp \big\{ U : P^*(U) - \Gamma_u(\beta,U) = 0 \big\} \leq N < \infty
\end{equation*}
for some $N$. Hence the cardinality of $C_s(\beta),C_u(\beta)$ is uniformly bounded for all $\beta$.

\smallskip

\noindent{\it Finiteness of $\mathcal N_{1,u}$.} The set $\mathcal N_{1,u}$ is
\begin{equation*}
\mathcal N_{1,u} = \Big\{ \beta: \exists U \Big( P^*(U) - \Gamma_u(\beta,U) = 0 \ \wedge F'(U) = 0 \Big) \Big\}.
\end{equation*}
Observe that the second condition implies that $U$ is a zero of the polynomial $Q(U)$ defined in \eqref{Equa:Q_def}: if $\{\bar U_i\}_i$ are the roots of $Q(U)$, then the set $\mathcal N_{1,u}$ can be written as
\begin{equation*}
\mathcal N_{1,u} = \Big\{ \beta : \Gamma_u(\beta,\bar U_i) = P^*(\bar U_i) \ \text{for some $i$} \Big\}.
\end{equation*}
By \eqref{Equa:Gammau_beta} it follows that $\partial_\beta \Gamma_u < 0$, so that there can be only one $\beta_i$ such that $\Gamma_u(\beta_i,U_i) = P^*(U_i)$.
\end{proof}

\subsection{Transversal set $\mathcal T$ of speeds $\beta$}
\label{Ss:transc_calT}

In order to compute the first derivative of the effort function $E$ we need to select a specific family of velocities such that the ODE for trajectories satisfy some transversality conditions w.r.t. $P^*(U)$.

\begin{definition}
\label{Def:trans_calT}
We say that $\beta \ge \beta^*$ is \emph{transversal}, and write \newglossaryentry{Tcal}{name=\ensuremath{\mathcal T},description={set of speeds $\beta$ such that the roots of $F(U) + \beta$, $P^*(U) - \Gamma_u(U)$ are finite and not degenerate}} $\beta \in \gls{Tcal} = \mathcal T_f$, if
\begin{enumerate}
\item there exist finitely many $U_0=U^*<U_1<...<U_{2N}=1$ such that
\begin{equation*}
F(U)+\beta \begin{cases} =0 & \text{for }  U=U_i, \ i=1,...,2N, \\
>0  & \text{for } U\in P_i:=(U_{2i},U_{2i+1}), \ i=0,...,N-1, \\
<0  & \text{for } U\in N_i:=(U_{2i+1},U_{2i+2}), \ i=0,...,N-1;
\end{cases}
\end{equation*}
\item the tangency points are not degenerate:
\begin{equation*}
(-1)^i F'(U_i) <0 \quad \forall i=1,..., 2N-1;
\end{equation*}
\item the sets
\begin{equation*}
\gls{Cs} := \big\{ U \in (U^*, 1): \Gamma_s(U,\beta)=P^*(U) \big\} \quad \text{and} \quad \gls{Cu} := \big\{ U\in (U^*, 1): \Gamma_u(U,\beta)=P^*(U) \big\}
\end{equation*}
are finite;
\item the points in $C_s \cup C_u$ correspond to transversal crossings:
\begin{equation*}
F(U) + \beta \neq 0 \quad \forall U \in C_s \cup C_u.
\end{equation*}
\end{enumerate}
\end{definition}

Since the properties above are valid in an open set $(\beta,f) \in \R \times C^2([0,1])$ with the product topology because the functions involved are continuous, we can collect the results of this section into the following

\begin{theorem}
\label{Theo:transv_generic}
There is a dense open set \gls{Tfrak} in the space of admissible sources $f$ w.r.t. the norm $C^2([0,1])$, such that:
\begin{enumerate}
\item it contains a dense set of polynomials;
\item for all $f \in \mathfrak T$ there exists a finite set $\mathcal N_0 \cup \mathcal N_1$ given by Propositions \ref{zeros} and Proposition \ref{Prop:transv_Gamma_gen} such that the transversality set \gls{Tcalf} of $f$ is
$$
\gls{Tcalf} = [\beta^*,\infty) \setminus (\mathcal N_0 \cup \mathcal N_1);
$$
\item the set
\begin{equation*}
\bigcup_{f \in \mathfrak T} \mathcal T_f \times \{f\} \subset \gls{Dcal} = \Big\{ (\beta,f) \in \R \times C^2([0,1]), \beta > \beta^*(f), f \ \text{satisfying Assumption \eqref{H1}} \Big\}
\end{equation*}
is open and dense in the product topology, and of full measure in the subset \newglossaryentry{Dcalk}{name=\ensuremath{\mathcal D(k)},description={domain of admissible polynomials $f = \sum_{i=1}^k a_i U^i$ and admissible speeds $\beta$}}
\begin{equation*}
\gls{Dcalk} = \bigg\{ (\beta,a_1,\dots,a_k) \in \R^{k+1}, f = \sum_{i=1}^K a_i U^i,\ \beta \geq \beta^*(f), \ f \ \text{satisfies Assumpion (\ref{H1})} \bigg\}.
\end{equation*}
\end{enumerate}
\end{theorem}

\begin{proof}
Being the functions
\begin{equation*}
(\beta,U) \mapsto F(U) + \beta, \ \Gamma_u(\beta,U), \ \Gamma_s(\beta,U)
\end{equation*}
continuous in $C^2((U^*,1))$, the conditions listed in Definition \ref{Def:trans_calT} define an open set $\mathfrak T$ in $C^2$.

By Propositions \ref{Prop:F_transv_poly} and \ref{Prop:Gamma_trans_poly} this set $\mathfrak T$ contains a set of polynomials which is dense in the space of polynomial sources, and it is of full measure being its complement the union of a codimension $1$ algebraic surface (Proposition \ref{Prop:F_transv_poly}) and the graphs of finitely many continuous functions $\beta(f)$ (Proposition \ref{Prop:Gamma_trans_poly}).

By Weierstrass Theorem \cite[Section 12.3]{royden2010real} we conclude that $\mathfrak T$ must be dense in $C^1(U^*,1)$.
\end{proof}

A generic situation is presented in Fig. \ref{Fig:transv_generic}.

\begin{figure}
\resizebox{\textwidth}{!}{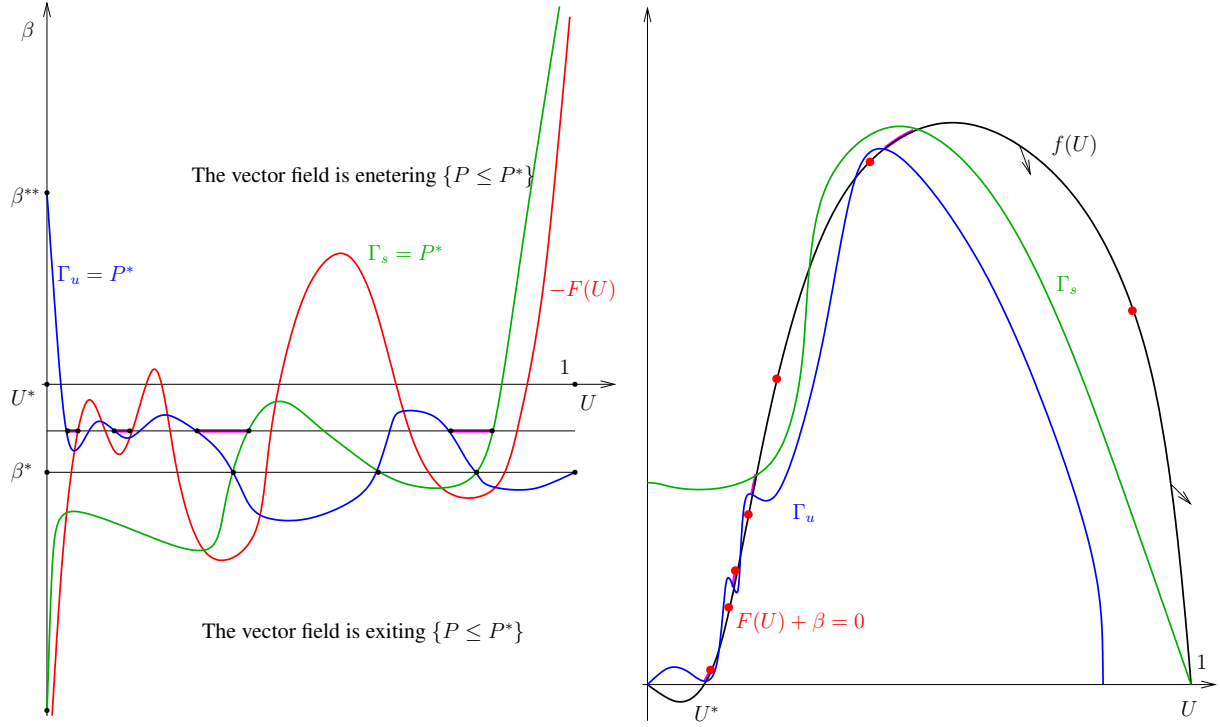}
\caption{A generic situation where Theorem \ref{Theo:transv_generic} holds. On the left the sets $\{F(U) + \beta = 0\}$ (red), $\{\Gamma_u(\beta,U) = P^*(U)\}$ (blue), $\{\Gamma_s(\beta,U) = P^*(U)\}$ (red) are depicted. Fixing $\beta$, the region where $- F(U) > \beta$ is the region where the vector field is exiting the region $\{P \leq P^*\}$, while $- F(U) < \beta$ the vector field is entering, so that the control can be applied here. In order to remain inside $\Gamma_u,\Gamma_s$, the only place where we can apply the control are the magenta segments. On the right the curve $P^*$ and the trajectories $\Gamma_u,\Gamma_s$ are represented in the phase space $(U,P$). The red dots are the zeros of $F(U) + \beta$.}
\label{Fig:transv_generic}
\end{figure}

\begin{remark}
\label{Rem:existence_multiple_inter}
We now show that it is possible to exhibit a source $f$ satisfying condition $(\ref{H1})$ and such that $F(U) + \beta$ has infinitely many oscillation. This proves that the statement of Theorem \ref{Theo:transv_generic} is not empty, i.e. $\mathfrak T$ does not contain all admissible sources $f$.

We start by considering a function $f \in C^{2,1}([0,1])$ defined as
\begin{align*}
f(U) := \begin{cases}
\text{smooth connection} & 0 \leq U < U^*, \\
\frac{a}{U^3}-\frac{a(U_1-U)^3}{(U^*)^3(U_1-U^*)^3}, & U^* \le U < U_1, \\
\frac{a}{U^3} & U_1 \le U < U_2, \\
\frac{a}{U^3}-\frac{a(U-U_2)^3}{(1-U_2)^3} & U_2 \le U \le 1,
\end{cases}
\end{align*}
where $U_1$, $U_2$ are points we will choose later. The main property of this function is that $F(U) = 0$ on the segment $[U_1,U_2]$: we will perturb it in order to achieve countably many oscillations of $F$.

In order to have
\begin{equation*}
f'(U^*) = \frac{3a(2U^* - U_1)}{(U^*)^4(U_1 - U^*)} > 0
\end{equation*}
we must assume $U_1 < 2 U^*$. Moreover, the same condition implies that 
\begin{equation*}
f(U) = \frac{a}{U^3} \bigg( 1 - \frac{U^3 (U_1 - U)^3}{(U^*)^3 (U_1 - U^*)^3} \bigg) > 0 \quad \forall U^* \le U \le U_1.
\end{equation*}
We also note that, for $U^* \le U \le U_1$,
\begin{align*}
F(U) &= \frac{3f(U)+Uf'(U)}{2\sqrt{Uf(U)}} \\
&= \frac{3}{2} \sqrt{\frac{a}{2(U^*)^3(U_1-U^*)^3}} \frac{U(U-U_1)^2(2U-U_1)}{\sqrt{(U^*)^3(U_1-U^*)^3 -U^3(U_1-U)^3}} > 0.
\end{align*}

By the previous choice of $U_1$, the map
\begin{equation*}
[U^*, U_1] \ni U \mapsto \sqrt{(U^*)^3(U_1 - U^*)^3 -U^3(U_1-U)^3}
\end{equation*}
is increasing. Viceversa, if we also assume $U_1<\frac{16}{7+\sqrt{17}} U^*$,
\begin{align*}
\frac{d}{dU} \bigg( U (U - U_1)^2 (2U - U_1) \bigg) = (U - U_1)(8U^2 -7U_1 U + U_1^2) \le 0 \quad \forall U \in [U^*, U_1].
\end{align*}
Both together imply that $F$ is decreasing in $[U^*, U_1]$.

Computing $F(U)$ for $U \in [U_1,U_2]$ we obtain
\begin{equation*}
3 f(U) + U f'(U) = \frac{3 a}{U^3} - \frac{3 a}{U^3} = 0,
\end{equation*}
so that $F(U) = 0$ for all $U \in [U_1, U_2]$.

Finally, in the interval $[U_2,1]$,
\begin{equation*}
F(U) := - \frac{3}{2}\sqrt{\frac{a}{(1-U_2)^3} }\frac{U (U-U_2)^2 (2U - U_2)}{\sqrt{(1 - U_2)^3 - U^3(U - U_2)^3}} \le 0,
\end{equation*}
and it is decreasing. 

We thus conclude that for the unperturbed $f(U)$ the set
\begin{equation*}
J^+ = \big\{ U \in (U^*, 1): \ F(U) > 0 \big\}
\end{equation*}
is an interval, and thus by Corollary \ref{Cor:monotone_1} (or Theorem $3.2$ in \cite{BressanChiri23}) the optimal profile is given by \eqref{optgraphgen}.

\begin{figure}
\resizebox{.65\textwidth}{!}{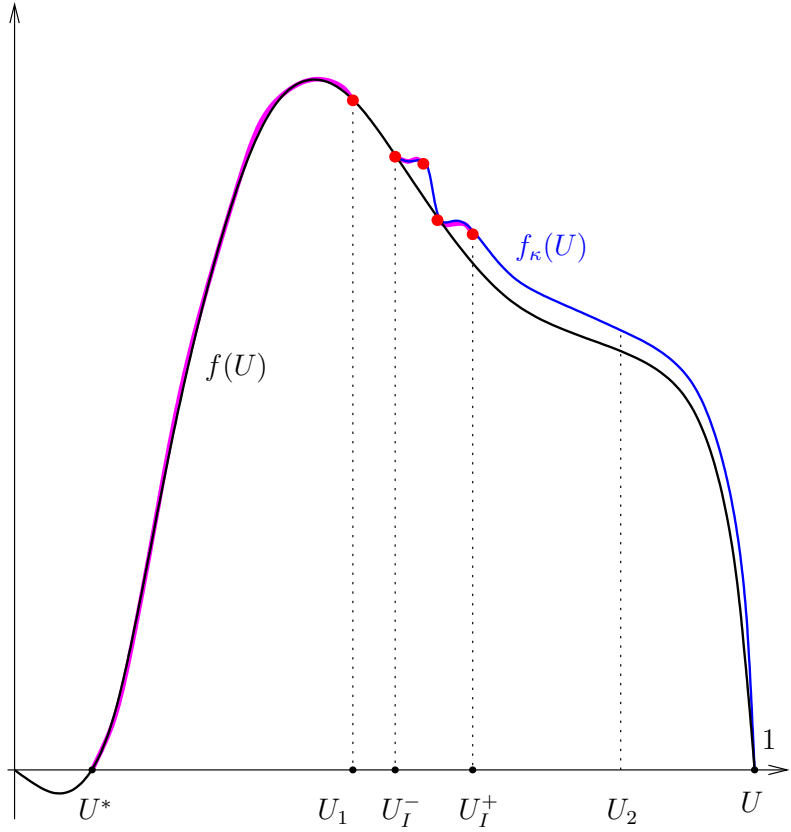}
\caption{The function $f(U)$ and its perturbation $f_\kappa(U)$, Remark \ref{Rem:existence_multiple_inter}. The magenta regions are where $F(U) + \beta > 0$.}
\label{Fig:cubic_hyperbol_perp_1}
\end{figure}

We will now perturb $f(U)$ into countably many disjoint intervals contained in $[U_1,U_2]$ so that $F(U)$ changes sign in each of these intervals. It is clear that we have only to give the perturbation in a single interval $I = (U^-_I,U^+_I) \subset [U_1,U_2]$. Set then
\begin{equation*}
f_\kappa(U) := \begin{cases}
\text{smooth connection} & 0 \leq U < U^*, \\
\frac{a}{U^3} - \frac{a(U_1-U)^3}{(U^*)^3(U_1-U^*)^3}, & \ U^* \le U < U_1, \\
\frac{a}{U^3} & \ U_1 \le U < U^-_I, \\
\frac{a}{U^3} + \kappa \sin \big( \frac{3\pi}{2|I|} (U - U^-) \big)^2 & \ U_I^- \le U < U^+_I, \\
\frac{a + \kappa (U^+_I)^3}{U^3} & \ U^-_I \le U < U^+_I, \\
\frac{a + \kappa (U^+_I)^3}{U^3}-\frac{(a + \kappa (U^+_I)^3)(U-U_2)^3}{(1-U_2)^3} & \ U_I^+ \le U \le 1,
\end{cases}
\end{equation*}
for some $\kappa > 0$ and compute in $I$
\begin{equation*}
\begin{split}
3 f_\kappa(U) + U \frac{d}{dU} f_\kappa(U) = 3 \kappa \sin \bigg( \frac{3\pi}{2|I|} (U - U^-) \bigg)^2 + \frac{3 \pi \kappa}{2|I|} U \sin \bigg( \frac{3\pi}{2I} (U - U^-) \bigg) \cos\bigg( \frac{3\pi}{2|I|} (U - U^-) \bigg).
\end{split}
\end{equation*}
It is fairly easy to see that this function is positive near $U^-_I,U^+_I$, while
\begin{equation*}
\begin{split}
3 f_\kappa \bigg( \frac{U^-_I + U^+_I}{2} \bigg) + \frac{U^-_I + U^+_I}{2} \frac{d}{dU} f_\kappa \bigg( \frac{U^- + U^+}{2} \bigg) &= 3 \kappa \sin \bigg( \frac{3\pi}{4} \bigg)^2 - \frac{3 \pi \kappa}{2 |I|} \frac{U^- + U^+}{2} \sin \bigg( \frac{3\pi}{4} \bigg) \cos \bigg( \frac{3\pi}{4} \bigg) \\
&= \frac{3\kappa}{2} \bigg( 1 - \frac{\pi}{2} \frac{U^+ + U^-}{2(U^+ - U^-)} \bigg) < 0.
\end{split}
\end{equation*}
Hence the function has an oscillation.

We observe that the perturbation is not $C^2$, but a slight variation will not alter the sign of $F(U)$.
%
\end{remark}

\section{Existence and regularity of an optimal profile}
\label{S:exist_regu_one_opt}

In this section we show that the optimal profile $P_\beta(f,U)$ exists and the optimal control $\alpha$ acts only when $P_\beta(f,U) = P^*(f,U)$.

The main tool is the comparison with trajectories which, starting from on e point $(U,P^*(U))$ remain above $P^*$ for $U' > U$ or below for $U' < U$ and are as close as possible to $P^*$: these \emph{minimal forward} and \emph{maximal backward} trajectories are uniquely defined, and every other trajectory for which the control acts when $P \not= P^*$ will be further from $P^*$ that one of those. Hence optimal trajectories must belong to a class of \emph{candidate minimizers $\tilde A(\beta)$} for which the controls acts on $\{P = P^*\}$, and this class of function is compact w.r.t. the $C^0$ convergence. Showing that the effort function $E(f,\beta)$ is continuous w.r.t. the $C^0$ convergence in $\tilde A(\beta)$, we obtain that the minimum exists by the standard method of the calculus of variations.

\subsection{Minimal forward and maximal backward trajectories}
\label{Ss:minim_bac_forw}

Fix $\beta > \beta^*$ and consider the set \newglossaryentry{calUset}{name=\ensuremath{\mathcal{U}(\beta)},description={open set where the control cannot act}}
\begin{equation*}
\gls{calUset} := \Big\{ U : \Gamma_u(\beta, U)<P^*(U)<\Gamma_s(\beta, U) \text{ and } F(U)+\beta<0 \Big\}.
\end{equation*}
The previous set can be decomposed as the union of its connected components, i.e. \newglossaryentry{Iomega}{name=\ensuremath{I_\varpi},description={decomposition of the open set $\mathcal{U}(\beta)$ into connected components}} \newglossaryentry{Vomega}{name=\ensuremath{V_\varpi^-},description={initial point of the interval $I_\varpi$}} \newglossaryentry{Womega}{name=\ensuremath{W_\varpi^+},description={final point of the interval $I_\varpi$}}
\begin{equation*}
\mathcal{U}(\beta) = \bigcup_{\varpi} \gls{Iomega}(\beta) = \bigcup_{\varpi} (\gls{Vomega}(\beta), \gls{Womega}(\beta)),
\end{equation*}
where the parameter set \newglossaryentry{Omegapar}{name=\ensuremath{\varpi},description={countable parameter set}} \gls{Omegapar} is at most countable and ordered so that $W_\varpi^+ < V_{\varpi'}^-$ for all $\varpi < \varpi'$. Recalling that 
\begin{equation*}
\lim_{U\to U^*} F(U)=+\infty,
\end{equation*}
we deduce 
\begin{equation*}
\inf\, \mathcal{U}(\beta)) \geq U_1 = \min F^{-1}(-\beta) > U^*,
\end{equation*}
where we recall that \newglossaryentry{U1}{name=\ensuremath{U_1},description={first root of $F(U) + \beta = 0$}} \gls{U1} is the first root of $F(U) + \beta = 0$.

\begin{notation}
\label{Not:Q_R_omeega_sol}
For every $\varpi$, let \newglossaryentry{Qomega}{name=\ensuremath{Q_\varpi},description={solution of the ODE \eqref{functeq} starting from $(V_\varpi^-,P^*(V_\varpi^-))$}} \newglossaryentry{Romega}{name=\ensuremath{R_\varpi},description={solution of the ODE \eqref{functeq} starting from $(W_\varpi^-,P^*(W_\varpi^-))$}} \gls{Qomega}, \gls{Romega} be the solutions to ODE \eqref{functeq} coinciding with $P^*$ respectively at $V_\varpi^-$ and $W_\varpi^+$, i.e.
\begin{equation*}
\begin{cases}
Q_\varpi '(U) + \frac{f(U)}{Q_\varpi(U)} + \beta = 0, \\
Q_\varpi(V_\varpi^-) = P^*(V_\varpi^-),
\end{cases} \quad \text{and} \quad \begin{cases}
R_\varpi '(U)+\frac{f(U)}{R_\varpi(U)}+\beta=0, \\
R_\varpi (W_\varpi^+)=P^*(W_\varpi^+),
\end{cases}
\end{equation*}
and let \newglossaryentry{Jomega-}{name=\ensuremath{J^-_\varpi},description={interval of existence of $Q_\varpi$}} \newglossaryentry{Jomega+}{name=\ensuremath{J^+_\varpi},description={interval of existence of $R_\varpi$}} \gls{Jomega-}, \gls{Jomega+} be their respective maximal intervals of existence.
\end{notation}
Since, by definition of $\mathcal{U}(\beta)$,
$$
\Gamma_u(\beta, V_\varpi^-)\le Q_\varpi(V_\varpi^-)=P^*(V_\varpi^-)\le \Gamma_s(\beta, V_\varpi^-)
$$
and
$$
\Gamma_u(\beta, W_\varpi^+) \le P^*(W_\varpi^+)=R_\varpi(W_\varpi^*+) \le \Gamma_s(\beta, W_\varpi^+),
$$
the smoothenss of the ODE \eqref{functeq} implies
\begin{equation*}
\Gamma_u(\beta, U) \le Q_\varpi(U) \le \Gamma_s(\beta, U) \quad \forall U \in J_\varpi^-
\end{equation*}
and
\begin{equation*}
\Gamma_u(\beta, U) \le R_\varpi(U) \le \Gamma_s(\beta, U) \quad \forall U \in J_\varpi^+.
\end{equation*}

\begin{proposition}
\label{Prop:QR_bounds}
For each $\varpi$ there exist \newglossaryentry{Vomega+}{name=\ensuremath{V^+_\varpi},description={first future point where $Q_\varpi = P^*$}} \gls{Vomega+} such that
\begin{equation*}
W_\varpi^+\le V_\varpi^+, 
\end{equation*}
\begin{equation*}
Q_\varpi(U) > P^*(U) \ \forall U \in (V_\varpi^-, V_\varpi^+) \quad \text{and} \quad Q_\varpi(V_\varpi^+) = P^*(V_\varpi^+) \leq \Gamma_s(V_\varpi^+).
\end{equation*}
Similarly, for each $\varpi$ there exist \newglossaryentry{Womega-}{name=\ensuremath{W^-_\varpi},description={first past point where $R_\varpi = P^*$}} \gls{Womega-} such that
\begin{equation*}
U^* <  W_\varpi^-\le V_\varpi^-,
\end{equation*}
\begin{equation*}
R_\varpi(U) < P^*(U) \ \forall U \in (W_\varpi^-, W_\varpi^+) \quad \text{and} \quad R_\varpi(W_\varpi^-) = P^*(W_\varpi^-) \geq \Gamma_u(W^-_\varpi).
\end{equation*}
\end{proposition}

\begin{proof}
First note that since $F(U)+\beta<0$ in a right neighborhood of $V_\varpi^-$, there also holds
\begin{equation*}
Q_\varpi(U) > P^*(U) \quad \text{for $U > V_{\varpi}^-$ close to $V_\varpi^-$}.
\end{equation*}
Assume that $Q_\varpi > P^*$ for all $U > V_\varpi^-$: in particular $Q_\varpi \geq P^*$ is well-defined up to $1$ and, by the Maximum Principle,
\begin{equation*}
Q_\varpi(U) \le \Gamma_s(\beta, U) \quad \forall U \in (V_\varpi^-, 1).
\end{equation*}
Since
%
\begin{equation*}
\Gamma_s(\beta, U) < P^*(U) \quad \text{for $U<1$ close to $1$},
\end{equation*}
we obtain a contradiction. We can therefore define
\begin{equation*}
V_\varpi^+: = \min \big\{ U>V_\varpi^-: Q_\varpi=P^* \big\}.
\end{equation*}
Note that 
we must have $F(V_\varpi^+)+\beta\ge 0$, and, hence, $V_\varpi^+ \ge W_\varpi^+$.

Similarly $R_\varpi(U) < P^*(U)$ in a left neighborhood of $W_\varpi^+$. We will prove that 
$$
\inf J_\varpi^+ < U^*.
$$
Indeed, assume $\inf J_\varpi^+ > U^*$: then, since $f>0$ in $(U^*, 1)$ and $R_\varpi(\inf J_\varpi^+)=0$,
\begin{equation*}
\lim_{\substack{U\to \inf J_\varpi^+\\
U>\inf J_\varpi}} R_\varpi'(U) = \lim_{\substack{U\to \inf J_\varpi^+\\
U>\inf J_\varpi}} -\frac{f(U)}{R_\varpi(U)}-\beta = -\infty,
\end{equation*}
which would imply $Q_\varpi<0$ in a right neighborhood of $\inf J_\varpi^+$.

Clearly, if $\inf J_\varpi^+ < U^*$, then $R_\varpi(U^*) > 0 = P^*(U^*)$ and thus
\begin{equation*}
W_\varpi^- := \max \big\{ U < W_\varpi^+: \ R_\varpi(U) = P^*(U) \big\}
\end{equation*}
is well-defined and strictly larger than $U^*$.

Finally, consider the case $\inf J_\varpi^+=U^*$. Under such hypothesis, $(U^*, 0)$ is a stable node, otherwise $0 < \Gamma_u(U^*) \leq R_\varpi(U^*)$ a contradiction: in particular it admits a neighborhood such that all trajectories entering this neighborhood cannot exit from it in the future. Hence each point in the semispace $U \ge \inf \mathcal{U}(\beta) > U^*$ cannot belong to a trajectory containing $(U^*,0)$. This last case is thus impossible.
\end{proof}

\begin{notation}
\label{Not:QR_omega*}
In the following, we adopt the following notation: \newglossaryentry{Qomegastar}{name=\ensuremath{Q_\varpi^*},description={prolongation of $Q_\varpi$ as $P^*$ outside $(V^-_\varpi,V^+_\varpi)$}} \newglossaryentry{Romegastar}{name=\ensuremath{R_\varpi^*},description={prolongation of $R_\varpi$ as $P^*$ outside $(W^-_\varpi,W^+_\varpi)$}}
\begin{equation*}
\gls{Qomegastar}(U) := \begin{cases}
Q_\varpi(U) & U \in (V_\varpi^-, V_\varpi^+), \\
\min\{P^*(U),\Gamma_s(\beta,U)\} & U \in (U^*, V_\varpi^-] \cup [V_\varpi^+, 1),
\end{cases}
\end{equation*}
\begin{equation*}
\gls{Romegastar}(U) := \begin{cases}
R_\varpi(U) & U \in (W_\varpi^-, W_\varpi^+), \\
\max\{\Gamma_u(\beta,U),P^*(U)\} & U \in (U^*, W_\varpi^-] \cup [W_\varpi^+, 1).
\end{cases}\end{equation*}
\end{notation}

Note that, by uniqueness of the ODE \eqref{functeq}, necessarily 
\begin{subequations}
\label{disjoint}
\begin{equation}
\label{disjoint_V}
(V_\varpi^-, V_\varpi^+) \cap (V_{\varpi'}^-, V_{\varpi'}^+) \not= \emptyset, \ \varpi \leq \varpi' \quad \Rightarrow \quad (V_{\varpi'}^-, V_{\varpi'}^+) \subset (V_{\varpi}^-, V_{\varpi}^+),
\end{equation}
\begin{equation}
\label{disjoint_W}
(W_\varpi^-, W_\varpi^+) \cap (W_{\varpi'}^-, W_{\varpi'}^+) \not= \emptyset, \ \varpi \leq \varpi' \quad \Rightarrow \quad (W_{\varpi}^-, W_{\varpi}^+) \subset (W_{\varpi'}^-, W_{\varpi'}^+),
\end{equation}
\end{subequations}
Also, we remark that, since $V_\varpi^- \ge U_1$, it also holds 
\begin{equation*}
Q_\varpi^*(U)=P^*(U) \quad \forall U \in (U^*, U_1) \text{ and } \forall \varpi.
\end{equation*}

\begin{proposition}[Minimal forward trajectory]
\label{mft}
Let $\beta \ge \beta^*$ and $U^- \in (U^*,1)$ such that $\Gamma_u(U^-) \leq P^*(U^-) \leq \Gamma_s(U^-)$. There exists \newglossaryentry{Pbar}{name=\ensuremath{\overline{P}},description={minimal forward trajectory}} $\gls{Pbar}:[U^-,1) \to \R$ Lipschitz continuous and such that
\begin{enumerate}
\item \label{1} $\overline P(U^-)=P^*(U^-)$ and 
$$
\min \big\{ P^*(U), \Gamma_s(\beta, U) \big\} \leq \bar P(U) \leq \Gamma_s(\beta,U) \quad \forall U \in [U^-,1);
$$
\item \label{2} $\overline{P}$ is admissible in $[U^-,1)$, i.e.
\begin{equation*}
\overline {\alpha}(U) := \frac{d}{dU} \overline{P}(U)+\frac{f(U)}{\overline P(U)}+\beta\ge 0 \quad \text{for a.e. } U \in (U^-,1),
\end{equation*}
and there holds
\begin{equation*}
\overline {\alpha} (U)= 0 \quad \text{for all } U \in \{\overline{P}>P^*\};
\end{equation*}
\item \label{3} for any other admissible trajectory $P:[U^-,1] \to \R$ such that $P > P^*$ in some interval $I=(U^-, U^+)$ and $P(U^\pm)=P^*(U^\pm)$, there also holds $\overline P \le P$ in $I$.
\end{enumerate}
\end{proposition}

\begin{proof}
First assume $F(U^-)+\beta \ge 0$: in particular, $U^- \notin N$. We consider several cases.

\medskip

\noindent {\it Case 1.} If $\mathcal{U}(\beta) \cap (U^-,1) = \emptyset$, we define as in Corollary \ref{Cor:monotone_1} 
\begin{equation*}
\overline{P}(U):=\begin{cases}
\Gamma_s(\beta, U) & \Gamma_s(\beta, U)<P^*(U), \\
P^*(U) & F(U)+\beta\ge 0,
\end{cases} \quad \text{with} \ U \in [U^-,1]. 
\end{equation*}
This function is Lipschitz and obviously satisfies Conditions \emph{(\ref{1})} and \emph{(\ref{2})} in the statement.

Let $P$ be admissible and such that $P > P^*$ in an interval $I$. By admissibility $\Gamma_s(\beta, U) \ge P(U) \geq \min\{P^*(U),\Gamma_s(\beta,U)\}$ for all $U \in I$: clearly $\overline{P}\le P$ in $I$, verifying Condition \eqref{3}.

\medskip

\noindent {\it Case 2.} Assume that $\mathcal{U}(\beta) \cap (U^-, 1) \neq \emptyset$ and, thus, there exists a non-empty subset of indexes $\{\varpi'\}$ such that
$$
(U^-,1) \cap \mathcal{U}(\beta) = \bigcup_{\varpi'} I_{\varpi'}.
$$
For each $U^-\le U$ with $P^*(U) \le \Gamma_s(\beta,U)$, define
\begin{equation*}
\overline{P}(U) := \sup_{\varpi'} \big\{ Q^*_{\varpi'} (U) \big\}.
\end{equation*}
Since $\min\{P^*,\Gamma_s\} \leq Q^*_\varpi \leq \Gamma_s$ for all $\varpi'$, we deduce that
$$
\min \big\{ P^*(U),\Gamma_s(\beta,U) \big\} \leq \overline{P}(U) \leq \Gamma_s(\beta,U) \quad \forall U \in (U^-,1).
$$
Moreover, for all $\varpi'$ it holds $Q_\varpi^*(U^-)=P^*(U^-)$, and thus $\overline{P}(U^-)=P^*(U^-)$. This proves Condition \ref{1}.

Consider now $U$ such that $\overline{P}(U) > P^*(U)$. By definition of supremum, there exists a maximizing sequence $\varpi_n' = \varpi_n' (U)$ such that $Q^*_{\varpi'_n}(U)\nearrow \overline P(U)$: in particular, definitely $Q^*_{\varpi'_n}(U) > P^*(U)$ and, thus, $U \in (V_{\varpi'_n}^-, V_{\varpi'_n}^+)$ and $Q^*_{\varpi'_n}(U) = Q_{\varpi'_n}(U)$. By \eqref{disjoint_V} we have that 
\begin{equation*}
(V_{\varpi'_n}^-, V_{\varpi'_n}^+) \subseteq (V_{\varpi'_m}^-, V_{\varpi'_n}^+) \quad \text{and} \quad {Q}_{\varpi'_n} \le {Q}_{\varpi'_m} \  \forall n \le m.
\end{equation*}
We deduce that the function $Q^U := \sup_{n \in \N} Q_{\varpi'_n}$ is a well-defined solution to \eqref{functeq} in the open interval
$$
I(U):= \bigcup_{n \in \N} (V_{\varpi'_n}^-, V_{\varpi'_n}^+) \quad \text{satisfying $Q^U>P^*$ and $Q^U=P^*$ on $\partial I(U)$}. 
$$
Moreover $U \in I(U)$ and $\overline{P} = Q^U$ in  $I(U)$. We also remark that $Q^U$ only depends on $U$ and not on the specific maximizing sequence: indeed given another $\varpi_n''$ satisfying the same approximating properties, there holds definitely $U \in (V_{\varpi'_n}^-, V_{\varpi'_n}^+)\cap (V_{\varpi_n''}^-, V_{\varpi_n''}^+) $ and thus, by (\ref{disjoint_V}), the intervals are ordered. The same reasoning gives that the if $I(U) \cap I(U') \not= \emptyset$ for $U \not= U'$, then there some common interval $(V^-_{\bar \varpi},V^+_{\bar \varpi})$, and from $\bar \varpi$ onward the two sequences construct the same intervals: hence $I(U) = I(U')$.

It is therefore possible to find countably many points $\{U_n\}_{n \in \N} \subseteq (U^-, U_s(\beta))$ such that $\overline{P}(U_n) > P^*(U_n)$ and, setting $I_n = I(U_n)$ and $Q_n = Q^{U_n}$,
\begin{equation*}
\big\{ U \in (\bar U,1): \ \overline{Q}(U) > P^*(U)\} = \bigcup_{n \in \N} I_n \quad \text{and} \quad \bar P(U) = Q_n(U) \ \forall U \in I_n.
\end{equation*}
Moreover, we remark that if $U \in \mathcal{U}(\beta)$, $U>U^-$, then in particular $U \in I_\varpi\subseteq (V_\varpi^-, V_\varpi^+)$ for some $\varpi'$, which implies $\overline{P}(U)\ge Q_\varpi^*(U)>P^*(U)$ and, thus $U \in I$.

%
%
By the construction of $\overline{P}$ and by the observation that
$$
(U^-, \gls{Usbeta}) \setminus \bigcup_{n\in\N} I_n \subseteq (U^-, U_s(\beta)) \setminus \mathcal{U}(\beta),
$$
we deduce that $\overline{P}$ is Lipschitz continuous and satisfies Conditions (\ref{1}) and (\ref{2}).

It remains to check Condition (\ref{3}). Let $P$ admissible and such that $P > P^*$ in $I=(U^-, U^+)$ and $P(U^\pm)=P^*(U^\pm)$. As for the case $\mathcal{U}(\beta) \cap(U^-, 1)=\emptyset$, we immediately deduce that $P\ge \overline{P} = P^*$ in $I\setminus \left(\cup_{n\in\N}I_n\right)$. We wish to show that $P(U)\ge Q_\varpi^*(U) \ \forall U \in I$ and $\forall \varpi$. We distinguish several cases.

If $U^+ \le V_\varpi^-$, then by the uniqueness of \eqref{functeq} it follows that $Q_\varpi^*(U) \leq P^*(U) < P(U)$ for all $U \in I$.

Next, assume $U^- < V_\varpi^- < U^+ \le V_\varpi^+ $: since $P(V_\varpi^-)>P^*(V_\varpi^-)=Q_\varpi(V_\varpi^-)$ and $P(U^+)=P^*(U^+)\le Q_\varpi(V_\varpi^-)$, there exists $\bar U \in (V_\varpi^-, U^+]$ such that $Q_\varpi(\bar U)=P(\bar U)$. By the admissibility of $P$, it would follow $P(U)\le Q_\varpi(U)$ for $V_\varpi^-\le U \le \bar U$, and, thus, in particular, $P(V_\varpi^-)\le Q_\varpi(V_\varpi^-)=P^*(V_\varpi^-)$, against our hypothesis.

Finally, let $U^-<V_\varpi^-<V_\varpi^+< U^+$. Clearly, for $U \in (U^-, V_\varpi^-]\cup[V_\varpi^+, U^+)$, $Q_\varpi^*(U) \leq P^*(U)<P(U)$. By contradiction, let $\bar U \in (V_\varpi^-, V_\varpi^+)$ such that $P(\bar U)=Q_\varpi(U)$: once again the positivity of the control associated to $P$ implies $P(U) \ge Q_\varpi(U)$ for all $\bar U\le U \le V_\varpi^+$, which contradicts our hypothesis $P(V_\varpi^+)>P^*(V_\varpi^+)=Q_\varpi(V_\varpi^+).$

We thus conclude conclude that $P(U) \ge \overline{P}(U) = \sup_\varpi Q^*_\varpi(U)$ for all $U \in I$, which is Condition \eqref{3}.

\medskip

{\it Case $F(U^-) + \beta < 0$.} In order to deal with the case $F(U^-)+\beta<0$, start by setting $\overline{P}$ equal to the solution to \eqref{functeq} satisfying $\overline P(U^-) = P^*(U^-)$. At first clearly $\overline P> P^*$, up to a point $W^{-}>U^-$ where $\overline P$ and $P^*$ cross again and $F(W-)+\beta \ge 0$. Starting from such point we repeat the construction of Case 2 above.
\end{proof}

A completely analog statement holds for the maximal backward trajectory $\tilde P$. We give the statement without proof.

\begin{proposition}[Maximal backward trajectory]
\label{Mbt}
Let $\beta \ge \beta^*$ and $\Gamma_u(U^+) \leq P^*(U^+) \leq \Gamma_s(U^+)$. There exists \newglossaryentry{Ptilde}{name=\ensuremath{\tilde P},description={maximal backward trajectory}} $\gls{Ptilde}:(U^*,U^+] \to \R$ Lipschitz continuous and such that
\begin{enumerate}
\item \label{1M} $\tilde P(U^+) = P^*(U^+)$ and
$$
\Gamma_u(\beta, U) \leq \tilde P(U) \leq \max \big\{ P^*(U),\Gamma_u(\beta,U) \big\} \quad \forall U \in (U^*,U^+];
$$
\item \label{2M} $\tilde P$ is admissible, i.e.
\begin{equation*}
\tilde \alpha(U) := \tilde P'(U) + \frac{f(U)}{\tilde P(U)} + \beta \ge 0 \quad \text{for a.e. } U \in (U^*,U^+],
\end{equation*}
and there holds
\begin{equation*}
\tilde \alpha(U) = 0 \quad \text{for all } U \in \{\tilde P < P^*\};
\end{equation*}
\item \label{3M} for any other admissible trajectory $P:[U^*,U^-] \to \R$ such that $P < P^*$ in some interval $I=(U^-, U^+)$ and $P(U^\pm) = P^*(U^\pm)$, there also holds $\tilde P \geq P$ in $I$;
\item \label{Point4:maxback} there exists \newglossaryentry{Ubar}{name=\ensuremath{\bar U},description={minimal value of $U$ before which the control is necessarily applied for admissible solutions}}
$\gls{Ubar} > U^*$ such that $U^- \geq \bar U$.
\end{enumerate}
\end{proposition}

\begin{proof}
The only additional point is Point \eqref{Point4:maxback}. If $\beta \leq 0$ observe that $U_u > U^*$ uniformly, so that any backward direction ends in $U^- > U_u$. If $\beta > 0$, the open set 
\begin{equation*}
\big\{ U \in (U^*,1), P < \min\{ P^*(U), \epsilon \} \big\}
\end{equation*}
contains $(U^*,1)$ and $-\frac{f}{P} - \beta < 0$ there. Hence trajectories there will cross the axis $\{P=0\}$ in the future. The constant $\epsilon > 0$ can be chosen as
\begin{equation*}
\epsilon = \min \big\{ P^*(U), U \in (U^* + \delta,1 - \delta) \big\},
\end{equation*}
for $\delta \ll 1$ because $f'(0) > 0, f'(1) < 0$. The conclusion follows by taking $\bar U = U^* + \delta$ so that $P^*(\bar U) < \epsilon$.
\end{proof}

\subsection{Candidate minimizers}
\label{Ss:candi_min}

Using the minimal forward trajectory \gls{Pbar} and the maximal backward trajectory \gls{Ptilde} we can restrict the class of minimizing sequences.

\begin{definition}
\label{Def:candid_min}
Let $\beta \ge \beta^{*}$. We say that an admissible profile \newglossaryentry{Acalbeta}{name=\ensuremath{\mathcal A(\beta)},description={family of candidate minimizers}} $\gamma \in \gls{Acalbeta}$ is a \emph{candidate minimizer} if it is the concatenation of the following curves:
\begin{enumerate}
\item the unstable manifold $\Gamma_u$ up to the first intersection point $(U_u(\beta),P^*(U_u(\beta)))$ with the graph of the critical curve $P^*$; 
\item if $\beta \geq \beta^{**}$, then $\gamma = P^*$ in the interval $[U^*,\bar U]$, where \gls{Ubar} is given by Point \eqref{Point4:maxback} of Proposition \ref{Mbt}.
\item a Lipschitz curve 
\begin{equation*}
\gamma=(U, P):[0,\overline{s}] \to \R^2
\end{equation*}
such that
$$
\gamma(0)=(U_u(\beta), P^*(U_u(\beta))) \in \Gamma_u(\beta), \quad \gamma(\overline{s})=(U_s(\beta), P^*(U_s(\beta))) \in \Gamma_s(U),
$$
and
\begin{align*}
\begin{cases}
U'(s)=P(s), \\
P'(s)+f(U(s))+\beta=0,
\end{cases} \quad \text{for every $s \in [0,\overline{s}]$, $P(s) \neq P^*(U(s))$};
\end{align*}
\item the stable manifold $\Gamma_s(\beta)$ starting from $(U(\bar s),P(\bar s)) = (U_s(\beta), P^*(U_s(\beta)))$ up to $(1,0)$, where $U_s(\beta)$ is the last intersection point of $\Gamma_s(\beta,U)$ with $P^*(U)$.
\end{enumerate}
\end{definition}

The following theorem is the key result on the regularity class for minimizers, which will be refined later on by giving necessary and sufficient conditions for optimality: these conditions however can be stated only on sufficiently regular solutions to \eqref{functeq}, i.e. the class of candidate minimizers. 

\begin{theorem}
\label{betterprofile}
Let $\beta\ge \beta^{*}$ and $\gamma\in \mathcal{A}(\beta)$ be an admissible profile for the minimization problem (\ref{minfunct}). Then there exists $\tilde \gamma$ candidate minimizer such that
\begin{equation*}
\gls{Etilde}(\tilde{\gamma}) \le \tilde E(\gamma).
\end{equation*}
\end{theorem}

\begin{figure}
\resizebox{.6\textwidth}{!}{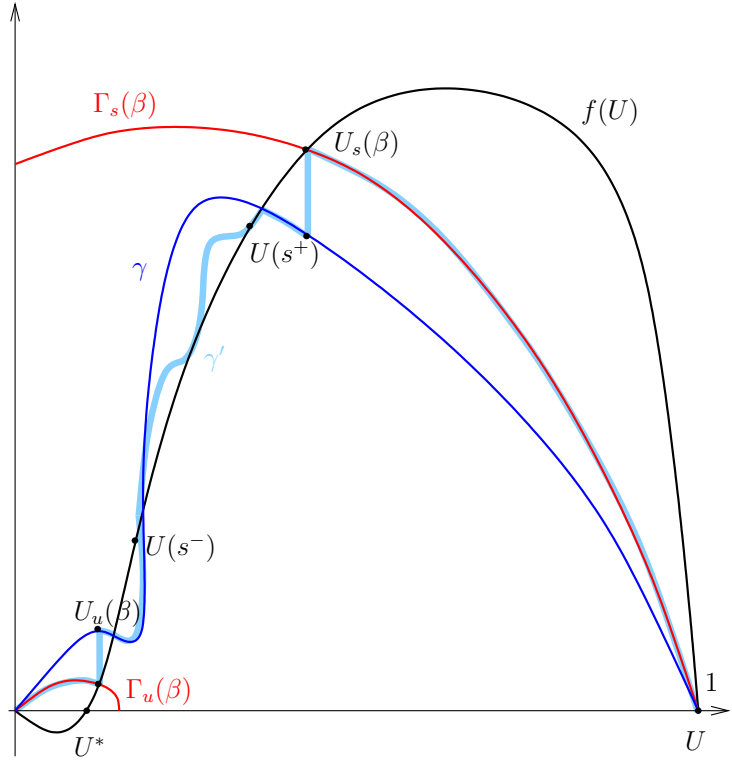}
\caption{The replacement operations in the proof of Theorem \ref{betterprofile}: the blue curve $\gamma$ is replaced with the light blue $\tilde \gamma$ lowering the cost.}
\label{Fig:rreplAcal}
\end{figure}

\begin{proof}
The proof is done by replacing arcs of $\gamma$ in order to lower the cost: at the end of the construction we will have a candidate minimizer. Recall that, by admissibility of the control, $\gamma(s) = (U(s),P(s))$ is always contained in the region enclosed by the unstable manifold $\Gamma_u$ and stable manifolds $\Gamma_s$, and that by \eqref{Equa:funct_eq_UP} it is the graph of a BV function $P(U)$ with the segments filling the jumps (which are only upward).

\medskip

\noindent{\it Step 1.} We first replace $\gamma$ with another admissible curve with lower cost, which we keep calling $\gamma$ with a slight abuse of notation: let \gls{Uubeta} (\gls{Usbeta}) be the first (last) intersection of $\Gamma_u$ ($\Gamma_s$) with $P^*$. The new curve is defined by concatenating (see Fig. \ref{Fig:rreplAcal})
\begin{itemize}
\item the unstable manifold until the point $(U_u(\beta), P^*(U_u(\beta)))$,
\item the upwards segment joining $(U_u(\beta), P^*(U_u(\beta)))$ and a point on $\gamma \cap \{U=U_u(\beta)\}$,
\item the admissible curve $\gamma$ from the previous point up to a point on $\gamma\cap \{U=U_s(\beta)\}$,
\item the upwards segment joining the previous point and $(U_s(\beta), P^*(U_s(\beta)))$,
\item the stable manifold up to $(1,0)$.
\end{itemize}
The fact that the cost $E$ is decreasing is already given in \cite{BressanChiri22}, and it is immediate from Proposition \ref{improv}. Denote by $(U,P):[0,\overline{s}]\to \R^2$ a Lipschitz parametrization for $\gamma$.

\medskip

\noindent{\it Step 2.} Let
$$
J := \big\{ s\in [0, \overline{s}]: \ U(s) \in [U_u(\beta), U_s(\beta)] \big\}
$$
and let $s \in J$ be such that $P(s) > P^*(U(s))$: consider the interval
\begin{equation}
\label{Equa:def_I}
I = (s^-, s^+), \quad s \in I \ \text{and} \ I \subset \big\{ P(s) > P^*(U(s)) \big\}.
\end{equation}
We show that we can assume
\begin{equation*}
U(s^-) > U^*.
\end{equation*}
This is clear for $\beta < \beta^{**}$ because $U(s^-) \ge U_u(\beta) > U^*$. If $\beta \geq \beta^{**}$, then from \eqref{Equa:bound_integr} 
\begin{equation*}
\gls{Ubar} = \inf \big\{ U > U^*: F(U) + \beta^{**} < 0 \big\} > U^*.
\end{equation*}
Hence if $U(s^+) \leq \bar U$, we can replace the whole interval $I$ with $P^*$, lowering the cost by Proposition \ref{improv}. If instead $U(s^-) < \bar U < U(s^+)$, we can consider a new curve equal to $P^*$ up to $\bar U$ and then a vertical segment connecting $P^*(\bar U)$ to $P(\bar U)$: again by Proposition \ref{improv} the cost is decreasing.

At this point we obtain a new curve $\gamma$ which is the concatenation of (see Fig. \ref{Fig:rreplAcal})
\begin{itemize}
\item the unstable manifold until the point $(U_u(\beta), P^*(U_u(\beta)))$,
\item the arc of $P^*$ for $U \in (U_u(\beta),U(0))$,
\item an admissible curve $\gamma(s)$, $s \in [0,\bar s]$, connecting $(U(0),P^*(U(0)))$ to $(U_s(\beta),P^*(U_s(\beta)))$, 
\item the stable manifold from $(U_s(\beta),P^*(U_s(\beta)))$ up to $(1,0)$.
\end{itemize}

\medskip

\noindent{\it Step 3.} Let $I = (s^-,s^+)$ be as in \eqref{Equa:def_I} and $\bar P$ be the minimal forward trajectory constructed in Proposition \ref{mft} starting from $(U(s^-),P^*(U(s^-)))$: by Conditions \eqref{1} and \eqref{2} of the same proposition it is admissible and by Condition \eqref{3} it holds $\bar P \leq P$.

An analog reasoning can be done in the intervals $I = (s^-,s^+)$ where $P(s) < P^*(U(s))$, in this case considering the maximal backward trajectory of Proposition \ref{Mbt}. The only difference is that we have to prove that the maximal backward trajectory coincide with $P^*$ near $U^*$. This is exactly what is proved in Point \eqref{Point4:maxback} of Proposition \ref{Mbt}.

We can thus replace the connected open intervals of
$$
\varpi^+ = \big\{ s : P(s) > P^*(U(s)) \big\} = \bigcup_j I^+_j, \quad \varpi^- = \big\{ s : P(s) < P^*(U(s)) \big\} = \bigcup_k I^-_k
$$
with the function $\bar P$ (if $P > P^*$) or $\tilde P$ (if $P < P^*$): the new curve is now a function $\check P(U)$ such that
\begin{equation*}
\check P(U) = \begin{cases}
\Gamma_u(\beta,U) & \text{for $U \leq U_u(\beta)$ by Step 1}, \\
P^*(U) & U \notin \varpi^+ \cup \varpi^-, \ \text{by Step 2 if $U_u(\beta) = U^*$ this contains an interval $(U^*,\bar U)$}, \\
\text{a solution to (\ref{functeq})} & \text{by Propositions \ref{mft}, \ref{Mbt}}, \\
\Gamma_s(\beta,U) & \text{for $U \geq U_s(\beta)$ by Step 1}.
\end{cases}
\end{equation*}
Hence the corresponding curve $\tilde \gamma$ is a candidate minimizer with lower cost, because the cost $E$ is decreasing at each step of the construction. The statement that it has finite length is elementary from the a.c. properties of $P^*$, the Lipschitz properties of $\Gamma_u,\Gamma_s$ and the fact that it is a solution to the ODE \eqref{functeq} when $P \not= 0$.
\end{proof}

The previous theorem immediately implies the following

\begin{corollary}
\label{Cor:restrict_admissible}
It holds
\begin{equation*}
\gls{Ebeta} = \inf \big\{ \gls{Etilde}(\gamma), \ \gamma \text{ candidate minimizer for } \beta \big\}.
\end{equation*}
Moreover, if the minimizer $\gamma$ exists, it must belong to the class of candidate minimizers.
\end{corollary}

\begin{proof}
We need to prove only the last part of the statement. If $\gamma$ is a minimizer, then the cost is not decreasing for all the operations performed in Step 1, 2, 3 of the proof of Theorem \ref{betterprofile}. We observe that since the integrand in \eqref{Equa:cost_expl}
\begin{equation*}
\frac{P^2 - U f(U)}{U^2 P^2} = 0 \quad \Longleftrightarrow \quad P = P^*(U),
\end{equation*}
in every step of the proof of Theorem \ref{betterprofile} the cost $\tilde E(\gamma)$ is decreasing unless $\gamma$ is unchanged. Thus $\gamma$ is already a candidate minimizer.
\end{proof}

Besides providing a necessary condition for minimality, Theorem \ref{betterprofile} can be combined with a classical compactness and continuity argument to guarantee the existence of an optimal profile.

\begin{lemma}
\label{lowerbound}
Let $\gamma: [s_1, s_2] \subseteq [0,\overline{s}] \to \R^2$ be a Lipschitz almost everywhere solution to the uncontrolled system
\begin{align*}
\begin{cases}
U'(s) = P(s), \\
P'(s)+f(U(s))+\beta P(s)=0,
\end{cases}
\end{align*}
and such that $P(s_2)=P^*(U(s_2))$. Assume that there exists $I$ open interval compactly contained in $(U^*,1)$ such that
$$
U(s) \in I \ \text{for all $s \in [s_1, s_2]$}.
$$
Set
\begin{align*}
m_f := \min_{I} f > 0, \quad m_{P^*} := \min_{I} P^* > 0, \quad \text{and} \quad P_m := \begin{cases}
m_{P^*} & \text{if} \ \beta \ge 0, \\
\min\{m_{P^*}, - \frac{m_f}{\beta}\} & \text{if} \ \beta < 0.
\end{cases}
\end{align*}
Then $P(s) \ge P_m > 0$ for all $s \in [s_1, s_2]$.
\end{lemma}

\begin{proof}
If $\beta\ge 0$, simply note that $P'(s)\le 0$ for a.e. $s\in[s_1, s_2]$, and thus $P(s)\ge P(s_2)\ge m_{P^*}=P_m$.

Viceversa, let $\beta<0$ and assume by contradiction that $P(\overline{s})<P_m$ for some $\overline{s}\in [s_1, s_2)$. Since $P(s_2) = P^*(U(s_2)) \ge P_m$, there exists $\delta>0$ such that $P(s)<P_m \ \forall s\in [\overline{s}, \overline{s}+\delta] $ and $P(\overline{s}+\delta)=P_m$. It follows that
\begin{align*}
P_m-P(\overline {s})&= \int_{\overline{s}}^{\overline{s}+\delta} P'(s) \, ds \\
& \le \int_{\overline{s}}^{\overline{s}+\delta} \left(-m_f-\beta P_m\right) \, ds \le 0,
\end{align*}
against our hypothesis.
\end{proof}

The next definition is analogous to Definition \ref{Def:candid_min}, in this case for functions $P(U)$.

\begin{definition}
\label{candidate}
Let $\beta \ge \beta^*$. We say that an a.c. profile $P: [0,1] \to [0,+\infty)$ is a \emph{candidate minimizer} for $\beta$, and write \newglossaryentry{Acaltildebeta}{name=\ensuremath{\tilde{\mathcal{A}}(\beta)},description={set of functions $P(U)$ which are candidate minimizers}} $P \in \gls{Acaltildebeta}$, if the following holds:
\begin{enumerate}
\item $P(U) = \Gamma_u(\beta,U)$ in $[0,U_u(\beta)]$ and $P(U) = \Gamma_s(\beta,U)$ in $[U_s(\beta),1]$;
\item if $\beta \geq \beta^{**}$, then $P(U) = P^*(U)$ in $[U^*,\gls{Ubar}]$;
\item $\Gamma_u(\beta,U) \leq P(U) \leq \Gamma_s(\beta,U)$;
\item for a.e. $U \in [U_u(\beta), U_s(\beta)]$
$$
\alpha(U) := \frac{1}{U} \left(P'+\frac{f(U)}{P}+\beta\right) \ge 0;
$$
\item $\alpha(U)=0$ for all $U \in \{U\in[U_u(\beta), U_s(\beta)]: P^*(U)\neq P(U)\}$.
\end{enumerate}
\end{definition}

\begin{lemma}
\label{Lem:regul_P_cand}
It holds
\begin{equation*}
0 \leq \alpha U \leq F(U) + \beta,
\end{equation*}
and if $P \in \tilde{\mathcal A}(\beta)$, then $P$ is Lipschitz outside the point $U^*$ for $\beta \geq \beta^{**}$, where is it $\frac{1}{2}$-H\"older continuous.
\end{lemma}

\begin{proof}
The first estimate on $\alpha$ is elementary since $\alpha \not= 0$ iff $P = P^*$.

The bounds
\begin{equation*}
P_m \underset{\text{Lemma (\ref{lowerbound})}}{\leq} P(U) \leq \Gamma_s(U,\beta) \quad \text{for} \ \min\{\bar U,U_u(\beta)\} \leq U \leq U_s(\beta)
\end{equation*}
gives that $P$ is Lipschitz in that interval, being the solution of a smooth ODE with an $L^\infty$ control. If $\beta \geq \beta^{**}$, then $P = P^*$ in the interval $[U^*,\bar U]$, where we can apply the definition of $P^*(U) = \sqrt{Uf(U)}$ for the $\frac{1}{2}$-H\"older continuity.
\end{proof}

\begin{proposition}
\label{corr}
Each $\gamma \in \mathcal{A}(\beta)$ uniquely corresponds to an admissible minimizer $P \in \tilde{\mathcal{A}}(\beta)$ and there holds
\begin{equation*}
\gls{Etilde}(\gamma) = \tilde{E}(P) = \int_{U_u(\beta)}^{U_s(\beta)}\frac{1}{U} \left(P'+\frac{f(U)}{P}+\beta\right)\, dU.   
\end{equation*}
In particular
\begin{equation*}
E(\beta)=\inf\big\{ \tilde{E}(P), P \in\tilde{\mathcal{A}}(\beta) \big\}.
\end{equation*}
\end{proposition}

\begin{proof}
First, assume that $\beta<\beta^{**}$, and thus $I=(U_u(\beta), U_s(\beta))$ is compactly contained in $(U^*, 1)$. Let $\gamma \in \mathcal{A}(\beta)$: comparing Definition \ref{Def:candid_min} and Definition \ref{candidate} we have only to verify that in the interval $[\min\{U_u(\beta),\bar U\},U_s(\beta)]$ the graph of the curve $\gamma$ corresponds to a function $P(U)$. This happens if the second component $P(s)$ of $\gamma$ remains bounded and above $0$ in $[\min\{U_u(\beta),\bar U\},U_s(\beta)]$, which is the case because
\begin{equation*}
P_m \leq P(s) \leq \Gamma_s(U(s))
\end{equation*}
by Lemma \ref{lowerbound} and the admissibility.

In the case $\beta \geq \beta^{**}$ we just use that $P(s) = P^*(U(s))$ for $U(s) \in [U^*,\bar U]$.
\end{proof}

\begin{remark}
\label{Rem:minimal_U}
By observing that $U_u(\beta) \searrow U^*$ as $\beta \nearrow \beta^{**}$, there exists \newglossaryentry{beta***}{name=\ensuremath{\beta^{***}},description={speed from which the control starts immediately from $U_u(\beta)$}} $0 < \gls{beta***} < \beta^{**}$ and a point \newglossaryentry{Ucheck}{name=\ensuremath{\check U},description={minimal point where the control is not acting}} \gls{Ucheck} such that 
\begin{equation*}
P(U) = P^*(U) \quad \forall U \in [U_u(\beta), \bar U], \beta > \beta^{***}.
\end{equation*}
Indeed, if $P(U) < P^*(U)$, $U > U^*$, for a candidate minimizer $P \in \tilde A(\beta)$, then from the ODE \eqref{functeq}
\begin{equation*}
\forall U' \geq U \Big( P'(U') = - \frac{f(U')}{P} - \beta < - \beta \quad \Rightarrow \quad P(U') \leq P(U) - \beta (U' - U) \Big),
\end{equation*}
which implies that if $P^*(U)$ is increasing in $(U^*,U^* + \underline{r})$ then
\begin{equation*}
P(U) > \beta (U^* + \underline{r} - U).
\end{equation*}
Thus the solution is equal to $P^*(U)$ in the connected component of
\begin{equation*}
\big\{ U : F(U) + \beta < 0, P^*(U) < \beta (U^* + \underline{r} - U) \big\},
\end{equation*}
containing $U_u(\beta)$.
%
\end{remark}

\subsection{Control set and existence of a minimizer}
\label{Ss:control_set_min}

Here we give a detailed description of the open sets $\gls{OcalP} \subset (U^*,1)$ where the control $\alpha$ is acting, and we prove that there exists a minimizing profile, showing that $\tilde E(P)$ is continuous w.r.t. the uniform convergence in the set of candidate minimizers.

\begin{definition}[Critical, control and free sets]
\label{Def:free_control_set}
Let $\beta \ge \beta^{*}$ and $P \in \tilde{\mathcal{A}}(\beta)$. In the following we denote the \emph{critical set} \newglossaryentry{KcalP}{name=\ensuremath{\mathcal K(P,\beta)},description={critical set, closed set where the function $P$ is equal to $P^*$, see also \eqref{Equa:KOIbeta}}} \newglossaryentry{OcalP}{name=\ensuremath{\mathcal O(P,\beta)},description={control set, open set where the control $\alpha$ is acting, i.e. $\alpha > 0$, see also \eqref{Equa:KOIbeta}}}
\begin{equation*}
\gls{KcalP} := \big\{ U \in [U_u(\beta), U_s(\beta)]: P(U)=P^*(U) \big\},
\end{equation*}
the \emph{control set}
\begin{equation*}
\gls{OcalP} := \inter \big( \mathcal{K}(P) \cap \{F(U)+\beta>0\} \big),
\end{equation*}
and the \emph{free set} \newglossaryentry{IcalP}{name=\ensuremath{\mathcal I(P,\beta)},description={free set, open set where $\alpha = 0$, see also \eqref{Equa:KOIbeta}}}
\begin{equation*}
\gls{IcalP} := \inter \big( [0,1] \setminus \mathcal{O}(P)^C \big).
\end{equation*}
See Fig. \ref{Fig:KOIdef}. The connected components of $\mathcal I(P,\beta)$ will be denoted by \newglossaryentry{Aomega}{name=\ensuremath{A_\varpi},description={connected components of $\mathcal I(P,\beta)$}} $\gls{Aomega} = (U^-_\varpi,U^+_\varpi)$, $\varpi$ belonging to a subset of $\Q$, ordered so that
\begin{equation*}
\varpi < \varpi' \quad \Longrightarrow \quad U^+_\varpi < U^-_{\varpi'}.
\end{equation*}
When we do not need the enumeration of $A_\varpi$, we will use the generic notation $I$ for an interval.
\end{definition}

\begin{figure}
\resizebox{.5\textwidth}{!}{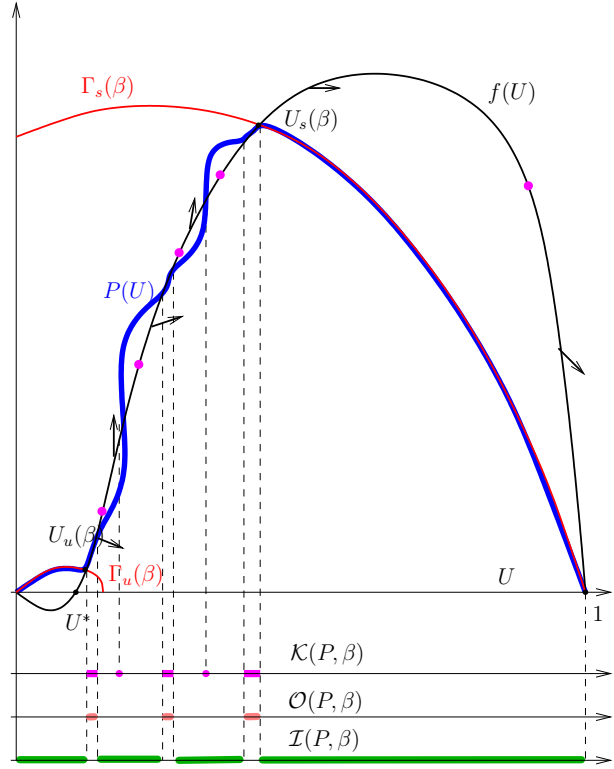}
\caption{The sets $\mathcal K(P,\beta)$, $\mathcal O(P,\beta)$, $\mathcal I(P,\beta)$ of Definition \ref{Def:free_control_set}: the arrow and the magenta dots denote the regions where $F(U) + \beta$ is positive or negative.}
\label{Fig:KOIdef}
\end{figure}

Note that, by the definition of candidate minimizer,
\begin{equation*}
\alpha(U)=\frac{1}{U} \left(P'+\frac{f(U)}{P}+\beta\right)=0 \quad \text{in} \ \mathcal{K}(P)^\mathsf{c},
\end{equation*}
whereas, clearly,
\begin{equation*}
\alpha(U)=\frac{F(U)+\beta}{U} \geq 0 \quad \text{is an admissible control if } U \in \mathcal{K}(P).
\end{equation*}
In particular, we can write
\begin{equation*}
\tilde{E}(P) = \int_{\mathcal{K}(P)} \alpha(U)\, dU= \int_{\inter \mathcal{K}(P)} \frac{F(U)+\beta}{U} \, dU+ \int_{\partial \mathcal{K}(P)}\alpha(U) \, dU.
\end{equation*}
The following lemma refines the identity.

\begin{lemma}
\label{costexpression}
Let $\beta \ge \beta^*$ and $P \in \gls{Acaltildebeta}$. Then
\begin{equation*}
\alpha(U) := \frac{1}{U} \left(P'+\frac{f(U)}{P}+\beta\right)=\frac{F(U)+\beta}{U}=0 \quad \text{for a.e. } U \in \partial \mathcal{K}(P).
\end{equation*}
In particular
\begin{equation*}
\tilde{E}(P)=\int_{\mathcal{O}(P)} \frac{F(U)+\beta}{U} \, dU.
\end{equation*}
\end{lemma}

Recall that being $P \in \tilde A(\beta)$ a.c., then $\alpha \in L^1$, and actually from Lemma \ref{Lem:regul_P_cand} we have uniform integral bounds.

\begin{proof}
We can write
\begin{equation*}
\mathcal{K}(P)^\mathsf{c} = \bigcup \big\{I \subseteq \mathcal{K}(P)^\mathsf{c} \text{ connected component} \big\}.
\end{equation*}
so that each $U \in \partial \mathcal{K}(P)$ is either an endpoint of some of the intervals $I$ (which are at most countable) or it is the limit of elements belonging to an accumulating families of such intervals. In the first case the set is $\mathscr L^1$-negligible.

If the second situation occurs, we can construct a sequence of connected components $I_{n}=(U_n^-, U_n^+)$ such that $I_n \neq I_m \ \forall n \neq m$ and $U_n^{\pm} \to U$ as $n \to +\infty$. In particular, since $P(U)=P^*(U)$, assuming that $U$ is a point of differentiability for $P$,
\begin{equation*}
P'(U) = \lim_{n \to +\infty} \frac{P(U_n^+) -P(U) }{U_n^+-U} = \lim_{n \to +\infty} \frac{P^*(U_n^+) -P^*(U) }{U_n^+-U} = (P^*)'(U)
\end{equation*}
and, thus,
\begin{equation*}
\alpha(U) U = P'(U) + \frac{f(U)}{P(U)} + \beta = F(U)+\beta.
\end{equation*}
Finally, note that if $I=(U^-, U^+)$ is a connected component of $\mathcal{K}(P)^ \mathsf{c}$, there holds $P(U^\pm)=P^*(U^\pm)$ and either $P<P^*$ or $P>P^*$ in $I$. In the first case $\pm(F(U^\pm)+\beta) \le 0$, in the latter $\pm\left(F(U^\pm)+\beta\right) \ge 0$.
Therefore, since $F(U_n^+)+\beta$ and $F(U_n^-)+\beta$ have opposite signs, necessarily \
\begin{equation*}
F(U)+\beta=\lim_{n\to +\infty}F(U_n^\pm)+\beta=0, 
\end{equation*}
which concludes the proof.
\end{proof}

We finally prove the existence of an optimal profile.

\begin{theorem}
\label{exist}
Let $\beta \ge \beta ^*$. The family $\tilde{\mathcal{A}}(\beta)$ is compact w.r.t uniform convergence and the map \newglossaryentry{Etilde}{name=\ensuremath{\tilde E},description={cost function of a profile $P \in \tilde A(\beta)$}}
\begin{equation*}
\begin{array}{ccccl}
\gls{Etilde} &:& \tilde{\mathcal{A}}(\beta) &\to& [0, +\infty) \\
&&P &\mapsto& \tilde{E}(P) := \displaystyle{\int_{U_u(\beta)}^{U_s(\beta)} \frac{1}{U} \bigg( P' + \frac{f(U)}{P} + \beta \bigg) \, dU}
\end{array}
\end{equation*}
is continuous. Hence, there exists $P_\beta = P(U)$ optimal profile for $E(\beta)$.
\end{theorem}

\begin{proof}
For simplicity, we distinguish the cases $\beta <\beta^{**}$ and $\beta\ge \beta^{**}$.

\medskip

\noindent{\it Case $\beta < \beta^{**}$.} By the same arguments exploited in the proof of Proposition \ref{corr}, if we assume $\beta<\beta^{**}$, there exists a positive constant $P_m$, only depending on $f$, such that $P(U)\ge P_m$ for all $U \in [U_u(\beta), U_s(\beta)]$ and $P \in\tilde{\mathcal{A}}(\beta)$.

First, we prove continuity w.r.t. uniform convergence. Let $\{P_n\} \subseteq \tilde{\mathcal{A}}(\beta)$ be a sequence of candidate minimizers converging uniformly to $P \ge P_m>0$. 
By definition of candidate minimizer and Lemma \ref{costexpression}, for a.e. $U \in [U_u(\beta), U_s(\beta)]$ and for all $n \in \N$, there holds
\begin{align}
\label{boundcont}
0 \le  \alpha_n(U)U := P_n'(U)+\frac{f(U)}{P_n(U)}+\beta =\begin{cases}
F(U)+\beta & U \in \mathcal{O}(P_n), \\
0 & \text{ otherwise}.
\end{cases}
\end{align}
Since $\mathcal{O}(P_n) \subseteq [U_u(\beta), U_s(\beta)]\subseteq [U^*, 1)$ and
\begin{equation*}
\lim_{U\to U^*} F(U) \sqrt{U-U^*} < +\infty,
\end{equation*}
we obtain that $\alpha_n$ are uniformly integrable and 
\begin{equation*}
\sup_{n\in\N} \int_{0}^1 \alpha_n(U)\, dU=\sup_{n\in\N} \int_{\mathcal{O}(P_n)} \frac{F(U)+\beta}{U}\, dU \le \int_{U^*}^1 \frac{|F(U)+\beta|}{U} \, dU < +\infty.
\end{equation*}
Therefore, up to subsequences, there exists $\alpha \in L^1([0,1])$, $\alpha \ge 0$ such that $\alpha_n \rightharpoonup \alpha$. By definition of a.e. derivative and by uniform convergence, for all $U_1, U_2 \in [U_u(\beta), U_s(\beta)]$
\begin{align*}
\int_{U_1}^{U_2} P'(U) \, dU & = P(U_2) - P(U_1)\\
&=\lim_{n\to \infty} P_n(U_2) - P_n(U_1)\\
&= \lim_{n\to \infty} \int_{U_1}^{U_2} P_n'(U) \, dU \\
&=\lim_{n\to \infty} \int_{U_1}^{U_2} \left(\alpha_n(U)U -\frac{f(U)}{P_n(U)} -\beta \right) \, dU\\
&=\int_{U_1}^{U_2} \left(\alpha(U)U - \frac{f(U)}{P(U)} -\beta \right) \, dU,
\end{align*}
which implies
\begin{align*}
\tilde{E}(P)=&\int_{U_u(\beta)}^{U_s(\beta)} \frac{1}{U} \left(P'(U)+\frac{f(U)}{P} +\beta)\right)\, dU \\
&=\int_{U_u(\beta)}^{U_s(\beta)} \alpha(U) \, dU\\
&=\lim_{n\to+\infty} \int_{U_u(\beta)}^{U_s(\beta)} \alpha_n(U) \, dU\\
&=\lim_{n\to +\infty} \tilde{E}(P_n).
\end{align*}

We now prove that $\tilde A(\beta)$ is a compact set of $C_0([0,1])$. 
Once again, by \eqref{boundcont} and since $\mathcal{O}(P_n) \subseteq [U_u(\beta), U_s(\beta)] \Subset (U^*, 1)$, we deduce that there exists a positive constant C, independent of $n$, such that
\begin{equation*}
0\le \alpha_n(U)U:= P_n'(U)+\frac{f(U)}{P_n(U)}+\beta \le C \quad \text{for a.e. $U \in [U_u(\beta), U_s(\beta)]$},
\end{equation*}
and, in particular, the sequence $P_n$ is equi-Lipschitz. Moreover, $P \leq \Gamma_s(\beta)$ 
i.e. $\{P_n\}$ is also equi-bounded. It is therefore possible to extract a subsequence $\{P_{n_k}\}$ uniformly converging in $[U_u(\beta), U_s(\beta)]$ to a Lipschitz profile $P$: it only remains to prove that $P \in \tilde{\mathcal{A}}(\beta)$. As shown above,
\begin{equation*}
\alpha_n \rightharpoonup \alpha:=\frac{1}{U} \bigg( P'+\frac{f(U)}{P}+\beta \bigg) \quad \text{in $L^1([U_u(\beta), U_s(\beta)])$}.
\end{equation*}
If $I \Subset \{P(U) \neq P^*(U)\}$, then by uniform convergence $I \subset \{P_n(U) \not= P^*(U)\}$ for $n \gg 1$, so that $\alpha_n = 0$ on $I$. In particular, $\alpha = 0$ in $I$.

\medskip

\noindent{\it Case $\beta\ge \beta^{**}$.} This case only requires a slightly more technical argument to deal with the singularity in $U_u(\beta) = U^*$. Recalling Lemma \ref{Lem:regul_P_cand} and Remark \ref{Rem:minimal_U}, the sequence $P_n$ is uniformly $\frac{1}{2}$-H\'older and $\alpha_n$ is uniformly integrable; moreover $P_n = P^*$, $\alpha_n = \frac{F(U) + \beta}{U}$ for $U \in (U^*,\gls{Ucheck})$ by Remark \ref{Rem:minimal_U}.
Hence
\begin{equation*}
\tilde{E}(P) = \int_{U^*}^{\check U} \frac{F(U)+\beta}{U}\, dU + \int_{\check U}^{U_s(\beta)} \frac{1}{U} \left(P'+\frac{f(U)}{P}+\beta\right)\, dU.
\end{equation*}
Therefore, we can repeat exactly the same arguments as before simply replacing the interval $[U_u(\beta), U_s(\beta)]$ with $[\check U, U_s(\beta)] \Subset (U^*, 1)$.
\end{proof}

\section{First-order conditions for optimality and uniqueness of optimal profiles}
\label{S:first_cond_opt}

A candidate minimizer is determined by the intervals where $P \not= P^*$. In this section we will give necessary conditions for minimality, and show that there is only one candidate minimizer satisfying these conditions: hence they are also sufficient in the set of candidate minimizers, and the minimizer is unique.


\subsection{First-order conditions}
\label{Ss:first_order}

Here we present the first-order conditions for optimality. The fundamental idea is to consider admissible infinitesimal perturbations of the optimal profile $P_\beta(U)$: the fact that these perturbations exist is a consequence of the definition of \gls{IcalP}.

\begin{proposition}
\label{nec2}
Let $\beta \ge \beta^*$, and let $P(U)$ be an optimal profile for $E(\beta)$ and $I = (U^-, U^+) \subseteq \gls{IcalP} \cap (U_u(\beta),U_s(\beta))$ be a connected component with $U_u(\beta) \leq U^- < U^+ \leq U_s(\beta)$. Then
\begin{subequations}
\label{Equa:nec2_gen_pro}
\begin{equation}
\label{Equa:nec2-}
\text{if $P(U^-) > \Gamma_u(U^-)$} \quad \text{then} \quad \int_{U^-}^{\tilde U} \frac{Uf(U)-P^2(U)}{UP^2(U)} \, dU \ge 0 \quad \forall \tilde U \in I,
\end{equation}
\begin{equation}
\label{Equa:nec2+}
\text{if $P(U^+) < \Gamma_s(U^+)$} \quad \text{then} \quad \int_{\tilde U}^{U^+} \frac{Uf(U)-P^2(U)}{UP^2(U)} \, dU \le 0 \quad \forall \tilde U\in I.
\end{equation}
\end{subequations}
\end{proposition}

The meaning of the above conditions is the following. Assume that instead of switch off the control at $U^-$ (so that $P$ becomes a solution to \eqref{functeq}), we decide to switch it off a little bit before $U^-$, and the to apply it later at $\tilde U$ in order to become again $P(\tilde U)$, see Fig. \eqref{Fig:linera_pert}. The quantity $e^{\int_{U^-}^{\tilde U} \frac{Uf(U)-P^2(U)}{UP^2(U)} \, dU}$ is the ratio between the control applied in $\tilde U$ and the control not applied in $U^-$. Saying that this ratio is greater than $1$ implies that the cost $\tilde E$ is increasing in this operation.
%

\begin{figure}
\resizebox{.6\textwidth}{!}{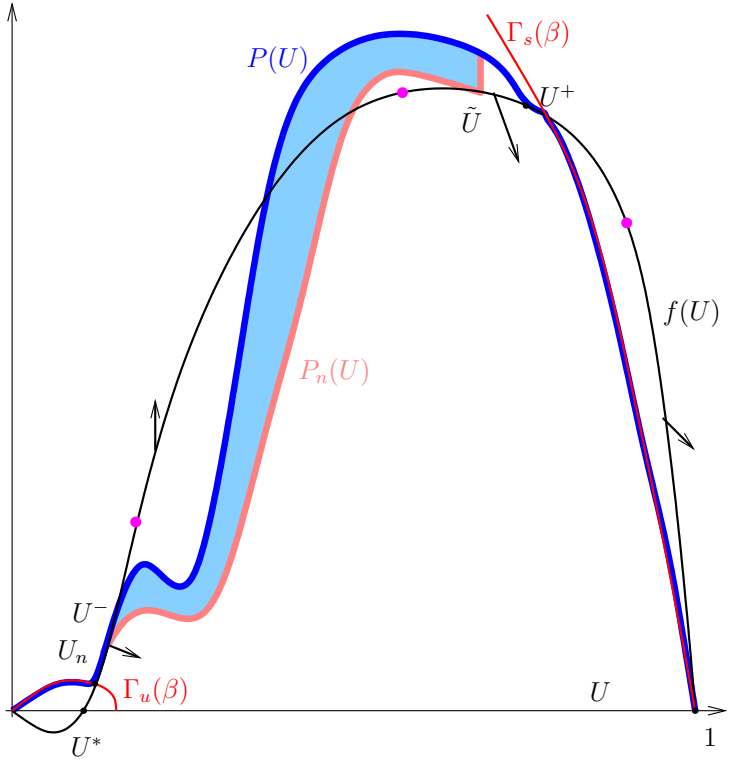}
\caption{The perturbation of Proposition  \ref{nec2}: if we turn off the control before $U^-$, then we need more control later in $\tilde U$}
\label{Fig:linera_pert}
\end{figure}

\begin{proof}
We consider only the case $\Gamma_u(\beta) < P(U^-)$ and we prove \eqref{Equa:nec2-}, the other case is analogous. 

%
Note first that, by definition of $\mathcal{I}(P)$, the infima of connected components of $\mathcal{I}(P)$ are exactly the points $U \in (U^+,1)$ that satisfy
\begin{enumerate}[label=(\roman*)]
\item $P(U)=P^*(U)$,
\item $\exists \delta > 0 \text{ s.t. } (U, U+\delta)\subseteq \mathcal{I}(P)$,
\item $\forall \delta>0$ it holds $(U-\delta, U)\cap\mathcal{O}(P) \neq \emptyset$.
\end{enumerate}
In particular, there exist a sequence $\{U_n^{-}\}_{n \in \N}$ such that
$$
U_n^-\nearrow U^-, 
\quad P(U_n^-) = P^*(U_n^-), \quad F(U_n^-) + \beta > 0.
$$
For the remainder of the proof, given $\bar U \in (U^*,1)$, we denote by $P(U,\overline 	U)$ the solution to Equation \eqref{functeq} with initial data $P(\bar U, \bar U) = P^*(\bar U)$: in particular, $P(U,U^-) = P(U)$ for $U \in [U^-,U^+]$.

Compute the ODE for the perturbation $\partial_{\bar U} P(U,\bar U)$:
\begin{equation*}
\frac{\partial}{\partial U}\left( \frac{\partial P}{\partial \bar U}(U, \bar U)\right) = \frac{f(U)}{P(U, \bar U)^2}\frac{\partial P}{\partial \bar U}(U,\bar U). 
\end{equation*}
obtaining 
\begin{equation*}
\frac{\partial}{\partial U} \bigg( \frac{1}{U}\frac{\partial P}{\partial \bar U}(U, \bar U) \bigg) = \frac{1}{U}\frac{\partial P}{\partial \bar U}(U, \bar U) \left( \frac{f(U)U-P(U, \bar U)^2}{UP(U, \bar U)^2}\right). 
\end{equation*}
The solution of the above ODE can be written as
\begin{equation}
\label{comput}
\frac{1}{U}\frac{\partial P}{\partial \bar U}(U, \bar U) = \frac{1}{\bar U} \frac{\partial P}{\partial \bar U}(\bar U,\bar U) \exp \bigg( {\int_{\bar U}^U \frac{Wf(W) - P^2(W, \bar U)}{P^2(W, \bar U)W}} \, dW \bigg).
\end{equation}

Observe next that the trajectories $P(U,\bar U)$ with initial condition $P_n(U_n^-) = P^*(U_n^-) = P(U_n^-)$ satisfy
$$
P_n(U) \le P(U) \quad \forall U \in (U^-,U^+)
$$
because $P$ is admissible so that the control $\alpha$ is positive with non-zero mass in $(U^-_n,U^-)$.
Therefore, for all $\tilde U \in I$ we can define an admissible curve $\gamma_n$ joining four arcs:
\begin{itemize}
\item the optimal curve $P$ up to $(U_n^-, P(U_n^-))=(U_n^-, P^*(U_n^-))$,
\item the trajectory corresponding to $P_n$ up to the point $(\tilde U,P_n(\tilde U))$,
\item the vertical upwards segment joining $(\tilde U,P_n(\tilde U))$ and $(\tilde U,P(\tilde U))$,
\item the optimal curve from $(\tilde U,P(\tilde U))$.
\end{itemize}

We exploit the Stokes' Theorem to compute the difference of costs and integrate by parts: 
\begin{align*}
\tilde E(\gamma_n) - \tilde E(\beta) &= \int_{U_n^-}^{\tilde U} \int_{P_n(U)}^{P(U)} \frac{Uf(U)-P^2}{U^2P^2} \, dP \, dU \\ 
&= \int_{U_n^-}^{\tilde U} \left[\frac{1}{U} \left( \frac{f(U)}{P_n(U)}- \frac{f(U)}{P(U)}\right)- \frac{1}{U^2} \left(P(U) - P_n(U)\right) \right] \, dU \\ 
\big[ \ \text{integrating by parts} \ \big] \quad &= - \int_{U_n^-}^{\tilde U} \frac{1}{U} \left[\left(P'(U) +\frac{f(U)}{P(U)}\right)- \left( P_n'(U) + \frac{f(U)}{P_n(U)} \right) \right] \, dU + \frac{P(\tilde U) - P_n(\tilde U)}{\tilde U} \\ 
&= \frac{P(\tilde U) - P_n(\tilde U)}{\tilde U} - \int_{U^-_n}^{U^-} \alpha(U) dU \\
&= \frac{P(\tilde U) - P_n(\tilde U)}{\tilde U} - \frac{P(U_n^-, U^-) - P(U_n^-, {U_n}^-)}{U_n^-} \\
& \quad - \int_{U^-_n}^{U^-} \int_{P(U)}^{P(U,U^-)} \frac{Uf(U)-P^2}{U^2P^2} \, dP \, dU \\
\big[ \text{Lipschitz trajectories} \big] \quad &= \frac{P(\tilde U, U^-) - P(\tilde U,U_n^-)}{\tilde U} - \frac{P(U_n^-, U^-) - P(U_n^-, {U_n}^-)}{U_n^-} + \frac{1}{U^-_n} \mathcal O(U^- - U^-_n)^2 \\
&= \frac{P(U_n^-, U^-) - P(U_n^-, {U_n}^-)}{U_n^-} \left[ \frac{\frac{P(\tilde U, U^-) - P(\tilde U,U_n^-)}{\tilde U}}{\frac{P(U_n^-, U^-) - P(U_n^-, {U_n}^-)}{U_n^-}} - 1 \right] + \frac{1}{U^-_n} \mathcal O(U^- - U^-_n)^2.
\end{align*}
Noting that $\frac{P(\tilde U, U^-) - P(\tilde U,U_n^-)}{\tilde U} > 0$, by the optimality of $P$ and Equation \eqref{comput} we conclude
\begin{align*}
0 \le \lim_{n \to \infty} \frac{E(\gamma_n)-E(\beta)}{\frac{P(\tilde U, U^-) - P(\tilde U,U_n^-)}{\tilde U}} &= \lim_{n \to \infty} \left[ \frac{\frac{P(\tilde U, U^-) - P(\tilde U,U_n^-)}{\tilde U}}{\frac{P(U_n^-, U^-) - P(U_n^-, U_n^-)}{U_n^-}} - 1 \right] \\
&= \exp \bigg( \int_{U^-}^{\tilde U} \frac{Uf(U)-P^2(U)}{UP^2(U)} \, dU \bigg) - 1,
\end{align*}
which immediately yields \eqref{Equa:nec2-}.
%
%
%
%
\end{proof}

\begin{remark}
\label{nec2strong}
If $\Gamma_u(U^-) < P(U^-)$, $P(U^+) < \Gamma_s(U^+)$, the above conditions read as
\begin{subequations}
\label{Equa:nec2_gen}
\begin{equation}
\label{nec2-}
\int_{U^-}^{\tilde U} \frac{Uf(U)-P^2(U)}{UP^2(U)} \, dU \ge 0 \quad \forall \tilde U \in I,
\end{equation}
\begin{equation}
\label{nec2+}
\int_{\tilde U}^{U^+} \frac{Uf(U)-P^2(U)}{UP^2(U)} \, dU \le 0 \quad \forall \tilde U\in I.
\end{equation}
\end{subequations}

If $\Gamma_u(U^-) = P(U^-)$, $P(U^+) < \Gamma_s(\beta)$, then we only have
\begin{equation}
\label{nec2-strong}
\int_{\tilde U}^{U^+} \frac{Uf(U)-P^2(U)}{UP^2(U)} \, dU \le 0 \quad \forall \tilde U\in I,
\end{equation}
and similarly for $\Gamma_u(U^+) < P(U^-)$, $P(U^+) = \Gamma_s(\beta)$
\begin{equation}
\label{nec2+strong}
\int_{U^-}^{\tilde U} \frac{Uf(U)-P^2(U)}{UP^2(U)} \, dU \ge 0 \quad \forall \tilde U \in I.
\end{equation}

Finally for $\Gamma_u(U^-) = P(U^-)$, $P(U^+) = \Gamma_s(U^+)$, there exists only one candidate minimizer which is given by Corollary \ref{Cor:monotone_1}, and then it is the optimal one and no other conditions can be applied.

The fact that, depending on the boundaries of the interval $I$ different conditions hold is due to the admissible perturbations: in the first case \eqref{Equa:nec2_gen} we can move the interval in both directions, while in the case \eqref{nec2-strong} we can move only in the positive direction, in the case \eqref{nec2+strong} in the negative and finally for $\beta = \beta^*$ no perturbation is admissible.

\end{remark}

As an immediate consequence of Proposition \ref{nec2} we get the following:

\begin{corollary}
\label{neccond2}
Let $\beta \ge \beta^*$ and $P$ optimal for $E(\beta)$. Then for all $U \in \partial \gls{IcalP} \cap (U_u,U_s)$ it holds
\begin{equation}
\label{Equa:short_nece}
\int_U^V \frac{Wf(W)-P^2(W)}{WP^2(W)} \, dW \geq 0, \quad V \in [0,1].
\end{equation}
In particular
\begin{equation}
\label{necint}
\int_{I} \frac{Uf(U)-P^2(U)}{UP^2(U)} \, dU = 0
\end{equation}
for all $I=(U^-, U^+) \subseteq \mathcal{I}(P)$ connected component s.t. $U_u(\beta) < U^- < U^+ < U_s(\beta)$.
\end{corollary}

\begin{proof}
We prove the statement for $V > U$, the other case being entirely similar.

Assume first that $U$ is the lower boundary of a connected component $(U^-,U^+)$: then we have that \eqref{Equa:short_nece} holds for $V \in (U^-,U^+)$, and the integral is $0$ when $V = U^+$. If there exists a sequence of connected components \gls{Aomega} of \gls{IcalP} converging from above to $U$, then if $V \in A_{\bar \varpi}$ write
\begin{equation*}
\begin{split}
\int_U^V \frac{Wf(W)-P^2(W)}{WP^2(W)} \, dW &= \bigg[ \sum_{A_\varpi \in (U,V)} \int_{A_\varpi} + \int_{(U,V) \cap A_{\bar \varpi}} \bigg] \frac{Wf(W)-P^2(W)}{WP^2(W)} \, dW \\
&= \int_{(U,V) \cap A_{\bar \varpi}} \frac{Wf(W)-P^2(W)}{WP^2(W)} \, dW \geq 0.
\end{split}
\end{equation*}
The second condition is obtained by using the two inequalities at $U^-$ and $U^+$.
\end{proof}

%
%

Now we analyze the first and last points where the control is applied, in other words the endpoints of the interval where the optimal $P_\beta(U)$ is strictly between $\Gamma_u,\Gamma_s$.

\begin{definition}
\label{Def:u_intersection_set}
Let $\beta\ge \beta^*$ and define the intersection set 
\begin{equation*}
\gls{Cu}(\beta) := 
\Big\{ U \in [U^*, 1): \ \Gamma_u(U, \beta) = P^*(U) \Big\}. 
\end{equation*}
For $U_u(\beta) > U^*$ and $\forall \bar U\in C_u(\beta)$ define \newglossaryentry{IuUbar}{name=\ensuremath{I_u^{\bar U}},description={stability functional for $C_u(\beta)$}}
\begin{align*}
\begin{array}{ccccl}
\gls{IuUbar} &:& [U_u(\beta),\bar U] &\to& \R \\
&& U &\mapsto& I_u^{\bar U}(U) := {\displaystyle \int_U^{\overline U} \frac{Wf(W)-\Gamma_u^2(W, \beta)}{W\Gamma_u^2(W, \beta)}\, dW}
\end{array}
\end{align*}
and \newglossaryentry{Chatubeta}{name=\ensuremath{\hat C_u(\beta)},description={stable set of $C_u(\beta)$}}
\begin{equation*}
\gls{Chatubeta} := \Big\{ \bar U \in C_u(\beta): \ I_u^{\bar U}(U) \le 0 \ \forall U \in [U_u(\beta), \bar U] \Big\}.
\end{equation*}
Finally, being $\hat{C}_u(\beta)$ compact with $U_u(\beta) \in \hat C_u(\beta)$, we can define \newglossaryentry{Uhatubeta}{name=\ensuremath{\hat U_u(\beta)},description={last stable point of $\hat C_u(\beta)$}}
\begin{equation*}
\gls{Uhatubeta} := 
\max \hat{C}_u(\beta). 
\end{equation*}
See Fig. \ref{Fig:unstab_last}.
\end{definition}

\begin{figure}
\resizebox{.75\textwidth}{!}{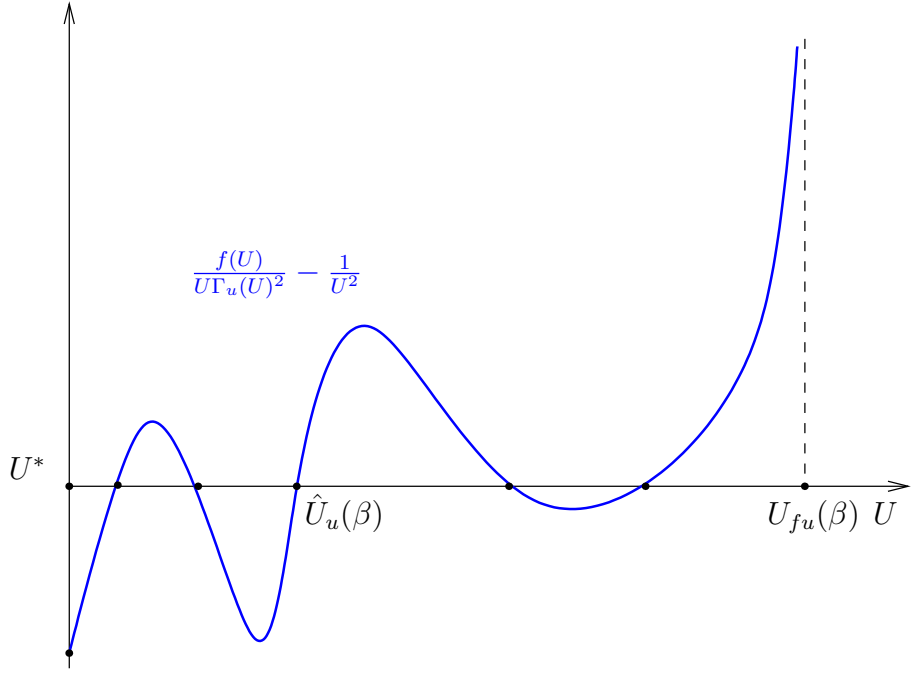}
\caption{The construction of Definition \ref{Def:u_intersection_set}: the points where $\frac{f(U) U - \Gamma_u(U)^2}{U^2 \Gamma_u(U)^2} = 0$ correspond to the set $C_u(\beta)$; the function $I_u^{\bar U}(U)$ is its integral in $(U,\bar U)$; the point $\hat U(\beta)$ is the last point $\bar U$ for which the integral is not positive for all $U \in C_u(\beta) \cap [U^*,\bar U]$.}
\label{Fig:unstab_last}
\end{figure}

A similar definition can be done for $\Gamma_s(\beta)$.

\begin{definition}
\label{Def:s_intersection_set}
Let $\beta\ge \beta^*$ and define the intersection set 
\begin{equation*}
\gls{Cs}(\beta) := \Big\{ U \in (U^*, 1): \Gamma_s(U, \beta)=P^*(U) \Big\}.
\end{equation*}
\newglossaryentry{IsUbar}{name=\ensuremath{I_s^{\bar U}},description={stability functional for $C_s(\beta)$}}
\begin{align*}
\begin{array}{ccccl}
\gls{IsUbar} &:& [\bar U,U_s(\beta)] &\to& \R \\
&& U &\mapsto& I_s^{\bar U}(U) := {\displaystyle \int_{\bar U}^{U} \frac{Wf(W)-\Gamma_u^2(W, \beta)}{W\Gamma_u^2(W, \beta)}\, dW}
\end{array}
\end{align*}
\newglossaryentry{Chatsbeta}{name=\ensuremath{\hat C_s(\beta)},description={stable set of $C_s(\beta)$}}
\begin{equation*}
\gls{Chatsbeta} := \Big\{ \bar U \in C_s(\beta): \ I_s^{\bar U}(U) \le 0 \ \forall U \in [\bar U,U_s(\beta)] \Big\},
\end{equation*}
\newglossaryentry{Uhatsbeta}{name=\ensuremath{\hat U_s(\beta)},description={first stable point of $\hat C_s(\beta)$}}
\begin{equation*}
\gls{Uhatsbeta} := 
\min \hat{C}_s(\beta). 
\end{equation*}
%
%
\end{definition}

\begin{proposition}
\label{Prop:optimal_Gamma_us}
Let $P \in \tilde{\mathcal{A}}(\beta)$ optimal for $E(\beta)$. Then
\begin{enumerate}
\item \label{Point1:optimal_Gamma} if $I=(U_u(\beta), U^+(P))\subseteq \mathcal{I}(P) \cap (U_u(\beta),U_s(\beta))$ connected component, necessarily $U^+(P) = \hat U_u(\beta)$; 
\item \label{Point2:optimal_Gamma} moreover
\begin{equation*}
\int_{\hat U_u(\beta)}^{\tilde U} \frac{Uf(U)-P^2(U)}{UP^2(U)}\, dU\ge 0 \quad \forall \tilde U\ge \hat U_u(\beta).
\end{equation*}
\end{enumerate}

Similarly, for $\Gamma_s$:
\begin{enumerate}\setcounter{enumi}{2}
\item \label{Point3:optimal_Gamma} if $I=(U^-(P), U_s(\beta))\subseteq \mathcal{I}(P) \cap (U_u(\beta),U_s(\beta))$ connected component, necessarily $U^-(P) = \hat U_s(\beta)$;
\item \label{Point4:optimal_Gamma} moreover 
\begin{equation*}
\int_{\tilde U}^{\hat U_s(\beta)} \frac{Uf(U)-P^2(U)}{UP^2(U)}\, dU\le 0 \quad \forall \tilde U \le \hat U_s(\beta).
\end{equation*}
\end{enumerate}
\end{proposition}

\begin{proof}
We prove the statement only for $\Gamma_u$, the other case being completely analogous.

\medskip

\noindent{\it Proof of Point \eqref{Point1:optimal_Gamma}.} First we show that $P(U) = \Gamma_u(\beta, U)$ for all $U \in [U_u(\beta), \hat U_u(\beta)]$: since $\Gamma_u$ solves \eqref{Equa:Gammau_eq0} with $\alpha = 0$, this is equivalent to say that $P(\hat U_u) = \Gamma_u(\hat U_u,\beta) = P^*(\hat U_u)$. If $\hat U_u(\beta)=U_u(\beta)$, there is nothing to prove: therefore we can assume $\hat U_u(\beta)>U_u(\beta)>U^*$.

If $\Gamma_u(\hat U_u,\beta) = P^*(\hat U_u) < P(\hat U_u)$, then there exists a connected component $I = (U^-,U^+) \subset \mathcal I(P)$ such that $\hat U_u \in I$: then by \eqref{Equa:nec2-} we have that
\begin{equation*}
\int_{U^-}^{\hat U_u} \frac{U f(U) - P^2}{U P^2} dU \geq 0.
\end{equation*}
On the other hand using that $P > \Gamma_u$ in $(U^-,\hat U_u)$ and that $\hat U_u \in \hat C_u(\beta)$ we have
\begin{equation*}
\int_{U^-}^{\hat U_u(\beta)} \frac{Uf(U)-P^2(U)}{UP^2(U)}\, dU < \int_{U^-}^{\hat U_u(\beta)} \frac{Uf(U)-\Gamma_u^2(U,\beta )}{U\Gamma_u^2(U, \beta)}\, dU \le 0,
\end{equation*}
a contradiction.

If $I = (0,U^+)$ is a connected component of $\mathcal I(P)$ with $U^+ > \hat U$, then by \eqref{Equa:nec2+} we have that
\begin{equation*}
\int_{U}^{U^+} \frac{Uf(U)-P^2(U)}{UP^2(U)}\, dU \leq 0 \quad \forall U \leq U^+,
\end{equation*}
and being $\alpha = 0$ in $(0,U^+)$, we also have that $\Gamma_u(U^+,\beta) = P^*(U^+)$: then $U^+ \in \hat C_u(\beta)$, a contradiction. Thus Point \eqref{Point1:optimal_Gamma} holds. 

\medskip

\noindent{\it Proof of Point \eqref{Point2:optimal_Gamma}.} If \gls{Aomega} are the connected components of $\mathcal I(P)$, we can write
\begin{equation*}
\begin{split}
\int_{\hat U_u}^{\tilde U} \frac{Uf(U)-P^2(U)}{UP^2(U)}\, dU &= \sum_{A_\varpi \subset (\hat U_u,\tilde U)} \int_{I_\varpi} \frac{Uf(U)-P^2(U)}{UP^2(U)}\, dU + \int_{U^-_{\bar \varpi}}^{\tilde U} \frac{Uf(U)-P^2(U)}{UP^2(U)}\, dU,
\end{split}
\end{equation*}
if $\tilde U \in A_{\tilde \varpi}$ for some $\tilde \varpi$, otherwise the last integral is not present. Using \eqref{necint} and \eqref{Equa:nec2-} (if $\tilde U < \hat U_s$) or the definition of $\hat U_s$ (if $\tilde U > \hat U_s$) and using Point \eqref{Point3:optimal_Gamma} we obtain Point \eqref{Point2:optimal_Gamma}.
%
%
\end{proof}

\subsection{Uniqueness of the minimizing profile}
\label{Ss:uniqe_min}

Finally, 
we prove that at most one candidate minimizer can satisfy Proposition \ref{nec2}: not only this does imply uniqueness for the optimization problem, but also that such conditions are indeed sufficient in \gls{Acaltildebeta}.

\begin{theorem}
\label{suff}
Let $\beta \ge \beta^*$ and $P_1, P_2 \in \tilde{\mathcal{A}}(\beta)$ both satisfying the following necessary conditions:
\begin{enumerate}
\item \label{Point1:ther_uniq} $P_i(U)=\Gamma_u(U, \beta)$ for all $U\in[U_u(\beta), \hat U_u(\beta)]$;
\item \label{Point2:ther_uniq} $P_i(U)=\Gamma_s(U, \beta)$ for all $U\in[\hat U_s(\beta), U_s(\beta)]$;
\item \label{Point3:ther_uniq} if $I=(U^-, U^+) \subseteq \mathcal{I}(P_i)$ is a connected component such that $(\hat U_u(\beta), \hat U_s(\beta))\cap I \neq \emptyset$, then $\hat U_u(\beta)< U^- < U^+ < \hat U_s(\beta)$;
\item \label{Point4:ther_uniq} if $I=(U^-, U^+)\subseteq \mathcal{I}(P)$ is a connected component with $\hat U_u(\beta) < U^- < U^+ < \hat U_s(\beta)$, then
\begin{equation*}
\int_{U^-}^{\tilde U} \frac{Uf(U)-P^2(U)}{UP^2(U)} \, dU\ge0 \quad \forall \tilde U \in I
\end{equation*}
and
\begin{equation*}
\int_{\tilde U}^{U^+} \frac{Uf(U)-P^2(U)}{UP^2(U)} \, dU\le 0 \quad \forall \tilde U\in I.
\end{equation*}
\end{enumerate}
Then $P_1=P_2$.
\end{theorem}

Observe here that
\begin{itemize}
\item Condition \eqref{Point1:ther_uniq} above corresponds to Point (\ref{Point1:optimal_Gamma}) of Proposition \ref{Prop:optimal_Gamma_us},
\item Condition \eqref{Point2:ther_uniq} above corresponds to Point (\ref{Point3:optimal_Gamma}) of Proposition \ref{Prop:optimal_Gamma_us},
\item Condition \eqref{Point3:ther_uniq} above corresponds to Points (\ref{Point1:optimal_Gamma},\ref{Point3:optimal_Gamma}) of Proposition \ref{Prop:optimal_Gamma_us},
\item Condition \eqref{Point4:ther_uniq} above corresponds to Equations \eqref{Equa:nec2_gen}.
\end{itemize}


\begin{proof}
The proof is straightforward if we assume $P_1\le P_2$. Indeed, by contradiction, let $P_1\neq P_2$: by the first two conditions, they can differ only in $(\hat U_u(\beta),\hat U_s(\beta))$, i.e. there exists $U\in [\hat U_u(\beta), \hat U_s(\beta)]$ such that $P_1(U)<P_2(U)$. By continuity of $P_1,P_2$ it follows that
\begin{equation*}
\int_{\hat U_u(\beta)}^{\hat U_s(\beta)} \frac{Uf(U)-P_1^2(U)}{UP_1^2(U)} \, dU > \int_{\hat U_u(\beta)}^{\hat U_s(\beta)} \frac{Uf(U)-P_2^2(U)}{UP_2^2(U)} \, dU,
\end{equation*}
which is a contradiction because both integrals must be $0$ by Assumption \eqref{Point4:ther_uniq} above: indeed, following for example the proof of Point \eqref{Point2:optimal_Gamma} of Proposition \ref{Prop:optimal_Gamma_us}, we obtain for $P_1$ that 
\begin{equation*}
\int_{\hat U_u}^{\hat U_s} \frac{Uf(U)-P_1^2(U)}{UP_1^2(U)} \, dU = \sum_{\varpi} \int_{A_\varpi} \frac{Uf(U)-P_1^2(U)}{UP_1^2(U)} \, dU = 0,
\end{equation*}
where $\gls{Aomega}$ are the connected components of $\mathcal I(P_1) \cap (\hat U_u,\hat U_s)$.

For the general case, we define
\begin{equation*}
P^M:=\max\{P,\tilde{P}\} \quad \text{and} \quad P^m:=\min\{P, \tilde{P}\}.
\end{equation*}
If we prove that $P^M$ and $P^m$ satisfy the assumptions of the theorem, then by the argument above $P^M = P^m$ and then $P = \tilde P$.

Clearly $P^M$ and $P^m$ are both still candidate minimizers with controls, respectively,
\begin{equation*}
\alpha_M = \begin{cases}
\alpha & \text{in } \{P>\tilde{P}\}, \\
\tilde{\alpha} & \text{in } \{P\le\tilde{P}\},
\end{cases}
\quad \text{and} \quad
\alpha_m = \begin{cases}
\alpha & \text{in } \{P\le\tilde{P}\}, \\
\tilde{\alpha} & \text{in } \{P>\tilde{P}\}.
\end{cases}
\end{equation*}

Moreover $P^M,P^m$ satisfy the first two assumptions of the statement above, because $P = \tilde P = \Gamma_u$ in $[0,\hat U(\beta)]$ and $P = \tilde P = \Gamma_s$ in $[\hat U_s(\beta),1]$.

Note also that
\begin{equation*}
\Gamma_u(U,\beta) < P(U),\tilde P(U) < \Gamma_s(U,\beta) \quad \forall U \in (\hat U_u,\hat U_s)
\end{equation*}
by the third assumption in the statement. Hence also $P^m,P^M$ satisfy the same constraints above. Indeed, if $I = (U^-,U^+) \subset \mathcal I(P^M) \cap (\hat U_u,\hat U_s)$, then $\hat U_u(\beta) < U^- < U^+ < \hat U_s(\beta)$ because otherwise $P^M < \min\{P,\tilde P\}$ in every right neighborhood of $U^-$. A similar reasoning holds for $P^m$ and for left neghborhood of $\hat U_s(\beta)$. Thus Assumption \eqref{Point3:ther_uniq} of the statement is verified.

It remains to verify Assumption \eqref{Point4:ther_uniq}. 
%
%
Fix $I =(U^-, U^+) \subseteq \mathcal{I}(P^M)$ connected component. Assume w.l.o.g. that $P^M(U^-) = P(U^-)$: being $P$ an admissible trajectory, necessarily 
$P^M=P$ in $I$. In particular $P$ solves \eqref{functeq} in a right neighborhood of $U^-$, and $I$ is contained in some component $I_P \subset \mathcal I(P)$. Assume that $U^-$ belongs to the interior of $I_P$. Since
\begin{equation*}
\mathcal O(P^M) \cap (U^- - \delta,U^-) \not= \emptyset \quad \forall \delta > 0,
\end{equation*}
then the control is acting on $P^M$, which gives if $U^- - \delta \in I_P$
\begin{equation*}
P^M(U^- - \delta) = P^M(U^-) - \int_{U^- - \delta}^{U^-} \frac{dP^M}{dV} dV < P(U^-) - \int_{U^- - \delta}^{U^-} \frac{dP}{dV} dV = P(U^- - \delta),
\end{equation*}
a contradiction with the definition of $P^M = \max\{P,\tilde P\}$. Hence $U^-$ is the infimum of $I_P$, which implies
\begin{equation*}
\int_{U^-}^{\tilde U} \frac{Uf(U)-P_M(U)^2}{UP_M(U)^2}\, dU = \int_{U^-}^{\tilde U} \frac{Uf(U)-P(U)^2}{UP(U)^2}\, dU\ge0 \quad \forall \tilde U \in (U^-,U^+).
\end{equation*}
A similar argument can be done for the second integral of Assumption \eqref{Point4:ther_uniq} above. This concludes the proof.
\end{proof}

Observing that $P \in \tilde{\mathcal A}(\beta)$ and \eqref{Equa:nec2_gen_pro} imply the assumptions of Theorem \ref{suff}, we can restate it as follows.

\begin{corollary}
\label{unique}
There exists a unique profile $P \in \tilde{\mathcal A}(\beta)$ satisfying Conditions \eqref{Equa:nec2_gen_pro}, hence the optimal profile is unique. 
\end{corollary}

\section{Structure of the control set $\mathcal O(\beta)$ as a function of $\beta$} 
\label{S:regular_calO}

The aim of this section is to study the dependence of the minimizer $P(\beta)$ on $\beta$: we will show that if $f \in \gls{Tfrak}$ and $\beta \in \gls{Tcal}_f$, then the structure of the free set \gls{IcalP} is the union of finitely many intervals whose endpoints depend smoothly on $\beta$. We will also consider the dependence w.r.t. the source $f$, giving explicit formulas for the Fre\'echet derivatives of the endpoints of the components of \gls{IcalP}.

Recall that $\gls{Pbeta} \in \tilde{\mathcal{A}}(\beta)$ is the unique optimal profile for $E(\beta)$, and we will denote by \newglossaryentry{alphabeta}{name=\ensuremath{\alpha_\beta},description={control corresponding to the unique optimal profile $P(\beta)$}} $\gls{alphabeta}$ its associated control and we write \newglossaryentry{KcalPopt}{name=\ensuremath{\mathcal K(\beta)},description={critical set, closed set where the optimal profile $P_\beta$ is equal to $P^*$}}
\begin{equation}
\label{Equa:KOIbeta}
\gls{KcalPopt} = \mathcal{K}(P_\beta,\beta), \quad \gls{OcalPopt} = \mathcal{O}(P_\beta,\beta), \quad \gls{IcalPopt} = \mathcal{I}(P_\beta,\beta).
\end{equation}
Moreover, for $\beta_1, \beta_2 \ge \beta^*$, for simplicity, we denote by $P_i := P_{\beta_i}$ and $\alpha_i:=\alpha_{\beta_i}$, $i=1,2$, the optimal profiles for the speed $\beta_1,\beta_2$ respectively. If we need to consider the dependence w.r.t. the source $f$, we will write $P_\beta(f,U)$, $\alpha_\beta(f)$, etc.

\subsection{Constrained minimization problem}
\label{Ss:costra_mini}

In this subsection, we study the minimization problem for the effort function among profiles restricted to a connected component $I$ of $\mathcal I(\beta_1)$ when $\beta_1$ is increased to $\beta_2$: more precisely, we will show that the unique optimal solution coincides with the restriction of the optimal solution $P_{\beta_2}$ to $I$.

Let $\beta^*\le \beta_1<\beta_2$ and $I = (U^-, U^+) \subseteq \mathcal{I}(\beta_1) \cap (\hat U_u(\beta_1),\hat U_s(\beta_1))$ be a connected component. Consider the \emph{constrained problem}
\begin{equation}
\label{local}
\begin{split}
E(\beta_2;I):=\inf \bigg\{ \int_{U^-}^{U^+} \frac{1}{U} \left(P' +\frac{f(U)}{P} + \beta_2 \right) \, dU&; \ P:[U^-, U^+]\to \R^+ \text{ Lipschitz s.t.} \\
& \quad P(U^\pm)=P^*(U^\pm) \text{ and } P' +\frac{f(U)}{P} + \beta_2 \ge 0 \bigg\}.
\end{split}
\end{equation}
For all $P$ admissible in the definition of (\ref{local}), following Definition \ref{Def:free_control_set} set \newglossaryentry{KcalPI}{name=\ensuremath{\mathcal K(P;I)},description={critical set, closed set where the function $P$ is equal to $P^*$ for the constrained problem \eqref{local}}} \newglossaryentry{OcalPI}{name=\ensuremath{\mathcal O(P;I)},description={control set, open set where the control $\alpha$ is acting, i.e. $\alpha > 0$, for the constrained problem \eqref{local}}} \newglossaryentry{IcalPI}{name=\ensuremath{\mathcal I(P;I)},description={free set, open set where $\alpha = 0$ for the constrained problem \eqref{local}}}
\begin{align*}
\gls{KcalPI} &:= \big\{ U \in [U^-, U^+]: \ P(U)=P^*(U) \big\}, \\
\gls{OcalPI} &:= \inter \big( \mathcal{K}(P;I) \cap \big\{ U \in (U^-, U^+): F(U)+\beta_2>0 \big\} \big), \\
\gls{IcalPI} &:= \inter \big( I \setminus \mathcal{O}(P;I) \big).
\end{align*}
The following lemma shows that \eqref{local} is equivalent to the restriction of the global problem in $(0,1)$ to $I$.

\begin{lemma}
\label{Lem:restr_beta_i}
There exists a unique minimizer \newglossaryentry{PImin}{name=\ensuremath{P_I},description={minimizer of the optimality problem restricted to an interval}} $\gls{PImin}$ for (\ref{local}) satisfying
\begin{equation}
\label{soluzz}
{\alpha}_I(U) U:= P_I' + \frac{f(U)}{P_I}+\beta_2 = 0 \text{ in } \mathcal{K}(P_I, I)^C
\end{equation}
and for all connected component $J=(W^-, W^+) \subseteq \mathcal{I}(P; I)$
\begin{subequations}
\label{Equa:nec2-local_tot}
\begin{equation}
\label{nec2-local}
\int_{W^-}^{\tilde{W}} \frac{Uf(U)-P^2(U)}{UP^2(U)} \, dU\ge0 \quad \forall \tilde{W}\in J,
\end{equation}
\begin{equation}
\label{nec2+local}
\int_{\tilde{W}}^{W^+} \frac{Uf(U)-P^2(U)}{UP^2(U)} \, dU\le 0 \quad \forall \tilde{W}\in J.
\end{equation}
\end{subequations}
Moreover, there exists $\delta>0$ such that
\begin{align*}
P_I(U)=P^*(U) \quad \forall U \in (U^-, U^-+\delta)\cup(U^+-\delta, U^+).
\end{align*}
\end{lemma}

\begin{figure}
\resizebox{.70\textwidth}{!}{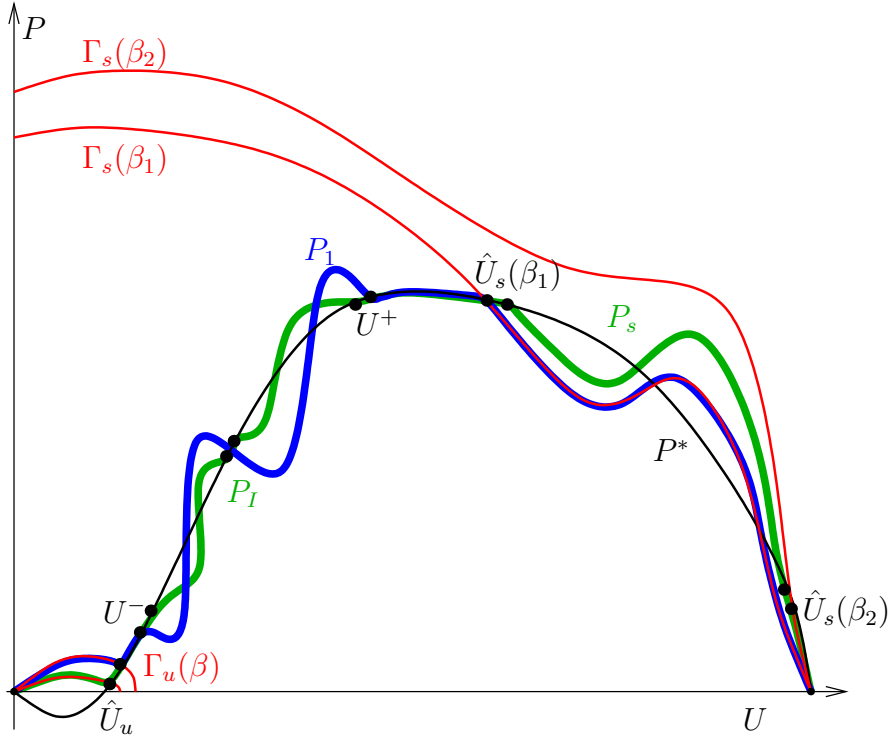}
\caption{The solutions of the minimum problems of Lemma \ref{Lem:restr_beta_i} and Remark \ref{Rem:beta12_manifold}. The interval $(U-,U^+) \subset \mathcal I(\beta_1)$ is split into two connected components of $\mathcal I(P,(U^-,U^+))$, and the interval $(\hat U_s(\beta_1),\hat U_s(\beta_2))$ into a new component of $\mathcal I(\beta_2)$. The black dots mark the boundary points of $\mathcal I(\beta_1)$ and $\mathcal I(\beta_2)$.}
\label{Fig:regul_O_beta}
\end{figure}

See Fig. \ref{Fig:regul_O_beta}.

\begin{proof}
The proof is given in several steps.

\medskip

\noindent{\it Step 0: the minimum is $< \infty$.} First note that we are minimizing over a non-empty family of profiles because $P_1$ is admissible for $\beta_2$ with control $\frac{\beta_2-\beta_1}{U}$ and satisfies the endpoint conditions.

\medskip

\noindent{\it Step 1: the minimum is attained.} For each $P$ admissible for the minimizing \eqref{local}, we can repeat the same construction as the one in Theorem \ref{betterprofile} to obtain a cheaper trajectory $\tilde{P}$ solving Equation (\ref{functeq}) outside $\mathcal{K}(\tilde{P};I).$ In particular, as in Theorem \ref{exist}, we deduce that the localized problem (\ref{local}) admits an optimal profile $P_I$ with associated control ${\alpha}_I$ satisfying (\ref{soluzz}).

\medskip

\noindent{\it Step 2: Equations (\ref{Equa:nec2-local_tot}) hold.} Fix $J=(W^-, W^+) \subseteq \mathcal{I}(P_I;I)$ connected component. If $U^-<W^-<W^+<U^+$, the inequalities (\ref{nec2+local}) and (\ref{nec2-local}) can be obtained considering the same perturbations as in Proposition \ref{nec2}.

Assume $W^-=U^-$: by the optimality of $P_1$ and Proposition \ref{nec2},
\begin{equation*}
\int_{W^-}^{\tilde{W}} \frac{Uf(U)-P_1^2(U)}{UP_1^2(U)} \, DU \ge 0 \quad \forall \tilde W \in I.
\end{equation*}
Moreover, $P_1$ is a supersolution of \eqref{functeq} with respect to $\beta_2$ in $I$, whereas $P_I$ is a solution, and they both equal $P^*$ at $U^-$. We deduce that $P_1 \ge P_I$ in $J$ and, thus,
\begin{equation*}
\int_{W^-}^{\tilde{W}}\frac{Uf(U)-P_I^2(U)}{UP_I^2(U)}\ge \int_{W^-}^{\tilde{W}}\frac{Uf(U)-P_1^2(U)}{UP_1^2(U)}\ge 0.
\end{equation*}
Similarly we prove (\ref{nec2+local}).

\medskip

\noindent{\it Step 3: $\{P_I \not= P^*\} \Subset (U^-,U^+)$.} It remains to show that $U^-$ admits a right neighborhood where $P_I=P^*$: the same argument can be used for $U^+$.

Since $F(U^-)+\beta_2>0$, by the same argument exploited in the proof of Lemma \ref{costexpression}, $U^-$ cannot be the accumulation point of extrema of connected components of $\mathcal{I}(P_I;I)$. Therefore, if we assume that $U^-$ can be approximated from the right via points where $P_I\neq P^*$, necessarily there exists $J\subseteq \mathcal{I}(P_I;I)$ connected component such that $U^-=\inf J$.

By the monotonicity with respect to $\beta$ of solutions to \eqref{functeq}, in $J \cap I$ there holds $P_1>P_I$, whereas $P_1(U^+)=P_I(U^+)=P^*(U^+)$. Define
\begin{align*}
\tilde U := \min \Big\{ U \in (U^-, U^+]: \ P_1(U)=P_I(U) \Big\},
\end{align*}
where we observed that the minimum exists being $P_I < P_1$ in a left neighborhood of $U^-$. Note that $P_1(\tilde U)=P_I(\tilde U)=P^*(\tilde U)$, because the control must act on $P_I$ in order to reach $P_1$. By Proposition \ref{Prop:optimal_Gamma_us} 
\begin{align}
\label{cont}
\int_{U^-}^{\tilde U} \frac{Uf(U)-P_I^2(U)}{UP_I^2(U)} \, dU > \int_{U^-}^{\tilde U} \frac{Uf(U)-P_1^2(U)}{UP_1^2(U)} \, dU \ge 0.
\end{align}

Observe that if $F(\tilde U) + \beta_2 \leq 0$, then $F(\tilde U) + \beta_1 < 0$, so that $P_1 < P^*$ in a left neighborhood of $\tilde U$, and then $P_I(\tilde U) < P^*(\tilde U)$. Hence $F(\tilde U) + \beta_2 > 0$. 
%
Setting
\begin{equation*}
W^+ := \inf \left \{ U<\tilde U: \ P_I\equiv P^* \text{ and } F+\beta_2 >0 \text { in } (U, \tilde U)\right\},
\end{equation*}
by optimality of $P_I$ we can conclude that $P_I = P^*$ in $(W^+,\tilde U)$.

We can perform the same perturbation as the one presented in the proof of Proposition \ref{nec2}, to obtain
\begin{equation*}
\int_{U^-}^{W^+}\frac{f(U)U-P_I^2(U)}{UP_I^2(U)}\, dU \le 0,
\end{equation*}
and thus, since $P_I=P^*$ in $(W^+, \tilde U)$,
\begin{equation*}
\int_{U^-}^{\tilde U}\frac{f(U)U-P_I^2(U)}{UP_I^2(U)}\, dU= \int_{U^-}^{W^+}\frac{f(U)U-P_I^2(U)}{UP_I^2(U)}\, dU+ \int_{W^+}^{\tilde U}\frac{f(U)U-P_I^2(U)}{UP_I^2(U)}\, dU\le 0,
\end{equation*}
which contradicts (\ref{cont}). 
%
\end{proof}

\begin{remark}
\label{Rem:beta12_manifold}
With the same arguments as the ones employed in the previous proof, we can also study what happens in the regions
\begin{equation*}
(\gls{Uubeta},\gls{Uhatubeta}) \quad \text{and} \quad (\gls{Uhatsbeta},\gls{Usbeta}).
\end{equation*}
Observe first that by the definition of $\hat U_u(\beta)$ and the fact that $\beta \mapsto \Gamma_u(\beta)$ is monotone decreasing, then
\begin{equation*}
\beta \mapsto \hat U_u(\beta) \quad \text{is monotone decreasing}.
\end{equation*}
Similarly $\beta \mapsto \hat U_s(\beta)$ is monotone increasing.

Next, the constrained problem
\begin{equation*}
\begin{split}
E_u(\beta_2) := \inf \bigg\{ \int_{\hat U_u(\beta_2)}^{\hat U_u(\beta_1)} \frac{1}{U} \left(P' +\frac{f(U)}{P} + \beta \right) \, dU&; \ P:[U_u(\beta_2),U_u(\beta_1)]\to \R^+ \text{ Lipschitz s.t.} \\
& \quad P(U_u(\beta_1))=P^*(U_u(\beta_1)), P' +\frac{f(U)}{P} + \beta_2\ge 0 \bigg\},
\end{split}
\end{equation*}
admits a unique minimizer \newglossaryentry{PuU}{name=\ensuremath{P_u(U)},description={boundary local minimizer for $\Gamma_u$}} \gls{PuU} satisfying the first-order necessary conditions and $P_u=P^*$ in a left neighborhood of $U_u(\beta_1)$.

Similarly, there exists a unique optimal profile \newglossaryentry{PsU}{name=\ensuremath{P_s(U)},description={boundary local minimizer for $\Gamma_s$}} \gls{PsU} for
\begin{equation*}
\begin{split}
E_s(\beta_2):=\inf \bigg\{ \int_{\hat U_s(\beta_1)}^{\hat U_s(\beta_2)} \frac{1}{U} \left(P' +\frac{f(U)}{P} + \beta \right) \, dU &; \ P:[U_s(\beta_1),U_s(\beta_2)]\to \R^+ \text{ Lipschitz s.t.} \\
& \quad P(U_s(\beta_1))=P^*(U_s(\beta_1)), P' +\frac{f(U)}{P} + \beta_2 \ge 0 \bigg\},
\end{split}
\end{equation*}
satisfying the first-order necessary conditions and with $P_s=P^*$ in a right neighborhood of $U_s(\beta_1)$.
\end{remark}

\begin{corollary}
\label{monotone}
It holds $\mathcal{O}(\beta_1)\subseteq \mathcal{O}(\beta_2)$ for all $\beta^* \le \beta_1<\beta_2$ and the effort function $\beta \mapsto E(\beta)$ is strictly increasing.
\end{corollary}


\begin{proof}
Set
\begin{align*}
P(\beta_2,U) := \begin{cases}
\Gamma_u(\beta_2,U) & U \in (0,\hat U_u(\beta_2), \\
P_u(U) & U \in [\hat U_u(\beta_2),\hat U_u(\beta_1)], \\
P_I(U) & U \in I \subseteq \mathcal{I}(\beta_1) \cap (\hat U_u(\beta_1),\hat U_s(\beta_1)) \text{ connected component}, \\
P^*(U) & U \in [\hat U_u(\beta_1),\hat U_s(\beta_1)] \setminus \mathcal{I}(\beta_1), \\
P_s(\beta_1,U) & U \in [\hat U_s(\beta_1),\hat U_s(\beta_2)], \\
\Gamma_s(\beta_2,U) & U \in (\hat U_s(\beta_2),1),
\end{cases}
\end{align*}
where $P_I$ is given in Lemma \ref{Lem:restr_beta_i} and $P_u$, $P_s$ are defined in Remark \ref{Rem:beta12_manifold}. The function $P(\beta_2,U)$ is well defined and Lipschitz continuous, it solves
\begin{equation*}
P' + \frac{f(U)}{P} + \beta_2 = 0 \quad \text{when} \ U \in \{P(\beta_2) \neq P^*\},
\end{equation*}
and satisfies the necessary conditions (\ref{Equa:nec2_gen_pro}) by construction. By Corollary \ref{unique}, $P(\beta_2) = P_2$: in particular
$$
\mathcal{O}(\beta_2)=\mathcal{O}(P(\beta_2))\supseteq \mathcal{O}(\beta_1).
$$

By the above inclusion and since $\mathcal{O}(\beta)= \emptyset$ if and only if $\beta=\beta^*$,
\begin{align*}
E(\beta_2)=&\int_{\mathcal{O}(\beta_2)} \frac{F(U) +\beta_2}{U} \, dU \\
&= \int_{\mathcal{O}(\beta_1)} \frac{F(U) +\beta_1}{U} \, dU + \int_{\mathcal{O}(\beta_1)} \frac{\beta_2- \beta_1}{U} \, dU + \int_{\mathcal{O}(\beta_2) \setminus \mathcal O(\beta_1)} \frac{F(U) +\beta_2}{U} \, dU 
> E(\beta_1). \qedhere
\end{align*}
\end{proof}

\begin{remark}
\label{Rem:tree_structure}
For $\beta = \beta^*$ it holds $\mathcal I(\beta) = (0,1)$, since no control is needed. For $\beta_1 \leq \beta_2$ the above results give 
\begin{equation*}
\mathcal I(\beta_2) \subset \mathcal I(\beta_1).
\end{equation*}
Indeed, each connected component of $\mathcal I(\beta_2)$ is strictly contained in a connected component of $\mathcal I(\beta_1)$. the number of connected components can only increase because it can happen that the solution $P_I$ of Lemma \ref{Lem:restr_beta_i} has $\mathcal I(P_I;I)$ not connected or because $P_u$ or $P_s$ are not equal to $P^*$. We will see later that this splitting has consequences on the regularity of $E(\beta)$.
\end{remark}

\subsection{Continuous dependence of $P_\beta$}
\label{Ss:contr_dep_Pbeta}

We conclude this section by proving, in the transversal case $f \in \gls{Tfrak}$, that outside finitely many speeds $\beta$ the number of connected components \gls{Aomega} of $\mathcal I(\beta,f)$ is constant, and the boundary points of each connected component \gls{Aomega} depend smoothly on $\beta,f$. The finitely many speeds that must be excluded are the ones for which some intervals $A_\varpi$ split because of the minimization problem \eqref{local}. This can only happen finitely many times for $f \in \mathfrak T$.

\begin{proposition}
\label{Prop:reguls_Ibeta}
Assume that $f \in \mathfrak T$. Then there is a finite set $\gls{Ncal2} = \mathcal N_2(f)$ such that for all $\beta \notin \mathcal N_2$ the following holds:
\begin{enumerate}
\item the number of connected components of $\mathcal I(\beta,f)$ is finite and constant;
\item the set
\begin{equation*}
\bigcup_{f \in \mathfrak T} \mathcal N_{2}(f)^c \times \{f\} \subseteq \gls{Dcal} = \Big\{ (\beta,f) \in \R \times C^2([0,1]), \beta > \beta^*(f), f \ \text{satisfying Assumption \eqref{H1}} \Big\}
\end{equation*}
is open and dense in the product topology;
\item the boundary points of each connected component of $\mathcal I(\beta,f)$ are smooth functions of $\beta,f$ in the Fr\'echet sense: this means that they admit a Fr\'echet derivative.
\end{enumerate}
\end{proposition}

\begin{proof}
Recall that if $f \in \mathfrak{T}$, then there is a finite set of speeds $\gls{Ncal0}(f) \cup \gls{Ncal1}(f)$ such that for all $\beta \notin \mathcal N_0 \cup \mathcal N_1$ the points \gls{Uibeta}, $i=1,\dots,2N-1$, where $F(U) + \beta = 0$ are finite, the derivative $F'(U_i) \not= 0$ and the crossing points $\gls{Cs} = \{\Gamma_s = P^*\}$, $\gls{Cu} = \{\Gamma_u = P^*\}$ are finite and transversal to $P^*$. Recall also that this set of $(\beta,f)$ is open and dense in \gls{Dcal} by Theorem \ref{Theo:transv_generic}.

\medskip

\noindent{\it Step 1: continuous dependence of $\hat U_u,\hat U_s$.} Consider the function $\beta \mapsto \hat U_u(\beta,f)$: by definition, it holds
\begin{equation*}
\int_U^{\hat U_u(\beta,f)} \frac{V f(V) - \Gamma_u(\beta,f,V)^2}{V \Gamma_u(\beta,f,V)^2} dV \leq 0, \quad \forall U < \hat U_u(\beta,f),
\end{equation*}
and for all $\bar U \in \{ \Gamma_u(\beta,f,U) = P^*(U), U > \hat U_u(\beta,f)\}$
\begin{equation*}
\int_U^{\bar U} \frac{V f(V) - \Gamma_u(\beta,f,V)^2}{V \Gamma_u(\beta,f,V)^2} dV > 0, \quad \text{for some} \ U < \bar U \in C_u(\beta,f).
\end{equation*}
The second condition is open in $(\beta,f)$ because by $f \in \mathfrak T$ then $\{\Gamma_u(\beta,f,U) = P^*(U), U > \hat U_u(\beta,f)\}$ is finite, so that we need to study the first. Observing that by transversality of the crossing $\Gamma_u > P^*$ for $(\hat U - \delta,\hat U)$ for some $\delta \ll 1$, and that
\begin{equation*}
\frac{\partial}{\partial_\beta} \int_U^{\hat U_u(\beta,f)} \frac{V f(V) - \Gamma_u(\beta,f,V)^2}{V \Gamma_u(\beta,f,V)^2} dV = \int_{\bar U}^{\hat U_u(\beta,f)} - \frac{2f(V)}{\Gamma_u(\beta,f,V)^3} \partial_\beta \Gamma_u(\beta,f,V)dV > 0,
\end{equation*}
it follows that the speeds $\beta$'s for which
$$
\max_{U < \hat U_u} \bigg\{ \int_U^{\hat U_u(\beta,f)} \frac{V f(V) - \Gamma_u(\beta,f,V)^2}{V \Gamma_u(\beta,f,V)^2} dV \bigg\}
$$
becomes positive are isolated. Hence there exists an open and dense set of $\beta$'s such that
\begin{equation}
\label{Equa:strict_transve}
\int_U^{\hat U_u(\beta,f)} \frac{V f(V) - \Gamma_u(\beta,f,v)^2}{V \Gamma_u(\beta,f,v)^2} dV > 0, \quad \forall U < \hat U_u(\beta,f),
\end{equation}
and this implies that in every connected component of this set $\hat U_u(\beta,f)$ depends continuously on $(\beta,f)$, because of Remark \ref{Rem:beta12_manifold}.
Moreover, if $U_{u,i}(\beta)$ are the point of $C_u(\beta)$ ordered according to $i$, then the function
$$
\beta \mapsto i(\beta) \quad \text{such that} \quad \hat U_u(\beta) = U_{u,i(\beta)}(\beta)
$$
is decreasing, so that it can only jump down finitely many times, being $C_u$ finite: this is also a consequence that the transition points are isolated.

The same reasoning applies to $\hat U_s(\beta)$, obtaining that outside a finite number of speeds it holds
\begin{equation}
\label{Equa:strict_transves}
\int_{\hat U_s(\beta,f)}^U \frac{V f(V) - \Gamma_s(\beta,f,v)^2}{V \Gamma_s(\beta,f,v)^2} dV < 0, \quad \forall U > \hat U_s(\beta,f),
\end{equation}
and then $\beta,f \mapsto \hat U_s(\beta,f)$ depends continuously on $\beta,f$.

\medskip

\noindent{\it Step 2: continuous dependence of \gls{IcalPopt}.} We use the following

\begin{lemma}
\label{Lem:not_bdr_tr}
If $U$ is a point such that $F(U) + \beta = 0$, $F'(U) \not= 0$, then it cannot be in $\partial \gls{IcalPopt}$. 
\end{lemma}

\begin{proof}
We prove the lemma for $F'(U) > 0$.

First, $U_i$ cannot be the initial point of a connected component, because $(U_i-\delta,U_i) \subset \gls{KcalPopt}^c$ being $F + \beta < 0$ there. If $U_i$ is the endpoints of a connected component of $\mathcal I(\beta)$, then $P(U) < P^*(U)$ in $(U_i - \delta,U_i)$ and thus
\begin{equation*}
\int_{U_i - \delta}^{U_i} \frac{U f(U) - P^2}{U P^2} dU > 0,
\end{equation*}
contradicting \eqref{Equa:nec2+}.
\end{proof}

The fact that $F(U) + \beta > 0$ is consists of finitely disjoint segments implies that $\mathcal O(\beta)$ has finitely many connected components, so the same happens for $\mathcal I(\beta)$. By the previous lemma, the infimum and the maximum of a connected component of $\mathcal I(\beta)$ belong to an interval where $F(U) + \beta > 0$: in particular, any connected component of $\mathcal I(\beta)$ must contain an interval where $F(U) + \beta < 0$, since the trajectory $P$ has to cross $P^*$ somewhere.

Hence the length of a connected component of $\mathcal I(\beta)$ is uniformly positive in every compact subset of $\R \setminus \mathcal N_0 \cup \mathcal N_1$, and thus the number of components can only increase because of Lemma \ref{Lem:restr_beta_i}. This can only happen a finite number of times in each connected component of $\R \setminus \mathcal N_0 \cup \mathcal N_1$.

As for the $\Gamma_u$-case, if $(U_i^-,U_i^+)$ is a connected component of $I(\beta,f)$, the set where
\begin{equation}
\label{Equa:strict_comI}
\int_{U_i^-}^{U} \frac{U f(U) - P^2}{U P^2} dU > 0 \quad \forall U \in (U^-_i,U^+_i)
\end{equation}
is open and dense: in particular, there exists an open and dense set of $\beta$'s such that the above inequality holds strictly for all connected components of $I(\beta,f)$. By the analysis of Lemma \ref{Lem:restr_beta_i}, no splitting of $I(\beta,f)$ can occur if \eqref{Equa:strict_comI} holds for all $i$.

\medskip

\noindent{\it Step 3: definition of set of stability for \gls{IcalPopt}.} Let \newglossaryentry{Ncal2}{name=\ensuremath{\mathcal N_{2}},description={the set where the number of components of $I(\beta)$ varies for $f \in \mathfrak T$}} \gls{Ncal2} be the finite set
\begin{equation*}
\gls{Ncal2} = \gls{Ncal0} \cup \gls{Ncal1} \cup \Big\{ \beta \ \text{such that \eqref{Equa:strict_transve}, \eqref{Equa:strict_transves} and \eqref{Equa:strict_comI} hold} \Big\}.
\end{equation*}
The complement of this set satisfies the first and second part of the statement.

\medskip

\noindent{\it Step 4: regularity of the boundary points of connected components of $\mathcal I(\beta,f)$ w.r.t. $\beta,f$.} The smooth dependence of $\hat U_u(\beta,f),\hat U_s(\beta,f)$ follows from the transversal interaction of $\Gamma_u(\beta,f),\Gamma_s(\beta,f)$ with $P^*$ and their smooth dependence w.r.t. $\beta,f$, Theorem \ref{manifoldbound} and Corollary \ref{Cor:Frechet_Gamma}.

Let $(U^-(\beta,f),U^+(\beta,f)) \in \mathcal I(\beta,f)$ be a connected component for $\beta \notin \mathcal N_2$. The optimality conditions give that necessarily
\begin{equation}
\label{Equa:perturb_betaini}
\int_{U^-(\beta,f)}^{U^+(\beta,f)} \frac{U f(U) - P^2}{U P^2} dU = 0.
\end{equation}
The perturbation of a trajectory when perturbing the data
$$
\beta,f,P(U^-(\beta,f)) \mapsto \beta + \delta \beta, \ f + \delta f, \ P(U^-(\beta,f)) + \delta P
$$
satisfies
\begin{equation*}
\frac{d}{dU} \delta P = \frac{f(U)}{P^2} \delta P - \frac{\delta f(U)}{P} - \delta \beta.
\end{equation*}
Hence, using the explicit form of the solution
\begin{equation*}
\delta P(U) = \delta P e^{\int_{U^-(\beta,f)}^U \frac{f(W)}{P^2(W)} dW} - \int_{U^-(\beta,f)}^U \bigg( \frac{\delta f(V)}{P(V)} 
+ \delta \beta \bigg) e^{\int_{V}^U \frac{f(W)}{P^2(W)} dW} dV, \quad \delta P(U^-(\beta,f)) = \delta P,
\end{equation*}
and recalling that
\begin{equation*}
\frac{U^\pm(\beta,f) f(U^\pm(\beta,f)) - P^2(U^\pm(\beta,f))}{U^\pm(\beta,f) P^2(U^\pm(\beta,f))} = 0,
\end{equation*}
by differentiating \eqref{Equa:perturb_betaini} we obtain the relation
\begin{equation*}
\begin{split}
&\int_{U^-(\beta,f)}^{U^+(\beta,f)} \frac{\delta f(U)}{P^2(U)} dU +  \delta P \int_{U^-(\beta,f)}^{U^+(\beta,f)} - \frac{2 f(U)}{P^3(U)} e^{\int_{U^-}^U \frac{f(W)}{P^2(W)} dW} dU \\
&\quad - \int_{U^-(\beta,f)}^{U^+(\beta,f)} - \frac{2 f(U)}{P^3(U)} \bigg( \int_{U^-(\beta,f)}^U \bigg( \frac{\delta f(V)}{P(V)} + \delta \beta \bigg) e^{\int_{V}^U \frac{f(W)}{P^2(W)} dW} dV \bigg) dU = 0,
\end{split}
\end{equation*}
\begin{equation}
\label{Equa:deltaPpm}
\begin{split}
\delta P &= \bigg( \int_{U^-(\beta,f)}^{U^+(\beta,f)} \frac{2 f(U)}{P^3(U)} e^{\int_{U^-}^U \frac{f(W)}{P^2(W)} dW} dU \bigg)^{-1} \\ & \quad \quad \cdot \bigg[ \int_{U^-(\beta,f)}^{U^+(\beta,f)} \frac{\delta f(U)}{P^2} + \frac{2 f(U)}{P^3(U)} \bigg( \int_{U^-(\beta,f)}^U \bigg( \frac{\delta f(V)}{P(V)} + \delta \beta \bigg) e^{\int_{V}^U \frac{f(W)}{P^2(W)} dW} dV \bigg) dU \bigg].
\end{split}
\end{equation}
From $\delta P$ one recovers the derivatives $\partial_\beta U^\pm(\beta,f)$ by differentiating $P^*(U)$:
\begin{equation*}
\frac{d}{dU} P^*(f,U^\pm(\beta,f)) \delta U^\pm(\beta,f) + \frac{U^\pm(\beta,f) \delta f(U^\pm(\beta,f))}{2 \sqrt{f(U^\pm(\beta,f))}} = \delta P(U^\pm(\beta,f)) + \bigg( - \frac{f(U^\pm(\beta,f))}{P^*(U^\pm(\beta,f))} - \beta \bigg) \delta U^\pm(\beta,f),
\end{equation*}
\begin{equation}
\label{Equa:deltaUpm}
\begin{split}
\delta U^\pm(\beta,f) &= \bigg( \frac{d}{dU} P^*(U^\pm(\beta,f)) + \frac{f(U^\pm(\beta,f))}{P^*(U^\pm(\beta,f))} + \beta \bigg)^{-1} \bigg( \delta P(U^\pm(\beta,f)) - \frac{U^\pm(\beta,f) \delta f(U^\pm(\beta,f))}{2 \sqrt{f(U^\pm(\beta,f))}} \bigg) \\
&= \frac{\delta P(U^\pm(\beta,f)) - \frac{U^\pm(\beta,f) \delta f(U^\pm(\beta,f))}{2 \sqrt{f(U^\pm(\beta,f))}}}{F(\beta,U^\pm(\beta,f)) + \beta}.
\end{split}
\end{equation}
Due to transversality it holds
$$
\frac{d}{dU} P^*(U^\pm(\beta,f)) + \frac{f(U^\pm(\beta,f))}{P^*(U^\pm(\beta,f))} + \beta = F(\beta,U^\pm(\beta,f)) + \beta > 0,
$$
so that the above expression is meaningful.

This concludes the proof of the regularity of the boundary points of $\mathcal I(\beta,f)$, being the expressions at the r.h.s. of \eqref{Equa:deltaPpm}, \eqref{Equa:deltaUpm} continuously dependent on the entries $\beta,f$ and the optimal solution $P$. It is also clear that the derivatives \eqref{Equa:deltaUpm} define a linear continuous operator on $\delta \beta,\delta f$, depending uniformly on \gls{Pbeta}, and then the functions are Fr\'echet differentiable.
%
%
%
%
%
\end{proof}

\begin{remark}
\label{Rem:split_calI}
We remark that the condition that no splitting occurs in the components of $\mathcal I(\beta,f)$ is fundamental to obtain the regularity of $\beta \mapsto U^\pm(\beta,f),\hat U_u(\beta,f),\hat U_s(\beta,f)$, which in turn yields regularity of $E(\beta,f)$ in that set, as we will prove later. If at $\beta$ a splitting occurs, we will see that $E(\beta,f)$ cannot be in general $C^2$, yielding additional non-regular points outside $\beta^{**}$.
\end{remark}

\section{Lipschitz regularity of the cost function $E(\beta,f)$ w.r.t. $\beta$ and the source $f$}
\label{S:regula_E_beta_f}

In this section we begin the investigation of the regularity of $E$: we will obtain that $f,\beta \mapsto E(\beta,f)$ is Lipschitz in  $(\beta^*,+\infty)$, and monotone increasing. We will also study the dependence w.r.t. the source $f$, showing its monotonicity and local Lipschitz regularity. The Lipschitz dependence w.r.t. $f$ allows to study the effort function $E$ when $f \in \gls{Tfrak}$ is a polynomial and $\beta \in \gls{Tcal}$: more precisely we can suppose to be in the conditions of Proposition \ref{Prop:reguls_Ibeta}, so that the boundary points of \gls{IcalPopt} are smooth functions of $\beta,f$. A first application of this approximation method is used to prove the monotonicity w.r.t. $f$ in Section \ref{Ss:monotoni_Frech}, and in the next section we will obtain more precise regularity estimates. 

We start with the Lipschitz regularity analysis of $\beta \mapsto E(\beta)$.

\begin{proposition}
\label{Prop:E_Lipschitz_beta}
The functional $\beta \mapsto E(\beta)$ is monotone increasing and Lipschitz continuous with Lipschitz constant bounded by $\frac{1}{U^*} + \ln( \frac{1}{U^*})$. Moreover
\begin{equation*}
\lim_{\beta \to \infty} \bigg[ E(\beta) - \ln \bigg( \frac{1}{U^*} \bigg) \beta \bigg] = \int_{U^*}^U \frac{F(U)}{U} dU = \int_{U^*}^1 \frac{2 \sqrt{f(U)}}{U^{3/2}} dU > 0.
\end{equation*}
\end{proposition}

\begin{proof}
The monotonicity is proved in Corollary \ref{monotone}.

Let $\beta^*\le \beta_1<\beta_2$ and define an admissible profile $\tilde \gamma$ for $\beta_2$ by concatenating (see Fig. \ref{Fig:dep_beta_f} left)
\begin{itemize}
\item the unstable manifold $\Gamma_u(\beta_2)$ in the interval $[0,U^*]$, 
\item provided $U_u(\beta_1) > U^*$, the upwards segment $s$ joining $\Gamma_u(\beta_2,U^*)$ to $\Gamma_u(\beta_1,U^*)$, 
\item the unstable manifold $\Gamma_u(\beta_1)$ restricted to $[U^*,U_u(\beta_1)]$, which we denote by $\Gamma_1^l$,
\item the optimal profile for $\beta_1$ restricted to $[U_u(\beta_1),U_s(\beta_1)]$, denoted as $\gamma_1$,
\item the stable manifold $\Gamma_s(\beta_1)$ from $(U_s(\beta_1), P^*(U_s(\beta_1)))$ to $(U_s(\beta_2),\Gamma_s(\beta_1, U_s(\beta_2)))$, whose restriction we denote by $\Gamma_1^r$,
\item the upwards segment $s'$ from $(U_s(\beta_2), \Gamma_s(\beta_1,U_s(\beta_2)))$ to $(U_s(\beta_2), \Gamma_s(\beta_2, U_s(\beta_2)))$,
\item the stable manifold corresponding to $\beta_2$ starting from $(U_s(\beta_2), P^*(U_s(\beta_2))$ up to $U=1$.
\end{itemize}
The profile is Lipschitz continuous and corresponds to a non-negative control because
\begin{equation*}
\beta_1 < \beta_2, \quad \Gamma_u(\beta_2,U^*) < \Gamma_u(\beta_1, U^*_u(\beta_2)), \quad \Gamma_s(\beta_2, U_s(\beta_2)) > \Gamma_s(\beta_1, U_s(\beta_2)). 
\end{equation*}
We denote by $\varpi_2$ the one-form associated to $\beta_2$, i.e.
\begin{equation*}
\varpi_2:= \left(\frac{f(U)}{PU}+\frac{\beta_2}{U} \right)\, dU +\frac{1}{U}\, dP
\end{equation*}
and note that
\begin{align*}
E(\beta_2) &\le \int_{s} \varpi_2+\int_{\Gamma_1^l}\varpi_2+\int_{\gamma_1} \varpi_2 + \int_{\Gamma_1^r}\varpi_2+\int_{s'}\varpi_2 \\
&= \frac{\Gamma_u(\beta_1,U^*)-\Gamma_u(\beta_2, U^*)}{U^*} + \int_{U^*}^{U_u(\beta_1)} \frac{\beta_2 - \beta_1}{U}\, dU + E(\beta_1) + \int_{\gamma_1} \frac{\beta_2-\beta_1}{U}\, dU \\
& \quad + \int_{U_s(\beta_1)}^{U_s(\beta_2)}\frac{\beta_2 - \beta_1}{U}\, dU + \frac{\Gamma_s(\beta_2,U_s(\beta_2))-\Gamma_s(\beta_1, U_s(\beta_2))}{U_s(\beta_2)},
\end{align*}
which implies
\begin{align*}
E(\beta_2)-E(\beta_1) &\le \frac{\Gamma_u(\beta_1,U^*) - \Gamma_u(\beta_2, U^*)}{U^*} \\
& \quad + \int_{U^*}^1 \frac{\beta_2-\beta_1}{U}\, dU + \frac{\Gamma_s(\beta_2,U_s(\beta_2))-\Gamma_s(\beta_1, U_s(\beta_2))}{U_s(\beta_2)}  \\
&\le \bigg( 1 + \frac{1 - U_s(\beta_2)}{U_s(\beta_2)} \bigg) (\beta_2 - \beta_1) + \ln \bigg( \frac{1}{U^*} \bigg)  (\beta_2-\beta_1) \\
&\leq \bigg[ \frac{1}{U^*} + \ln \bigg( \frac{1}{U^*} \bigg) \bigg] ( \beta_2 - \beta_1 ),
\end{align*}
where we have used \eqref{Equa:Gammau_beta}, \eqref{Equa:Gammas_beta} of Theorem \ref{manifoldbound}. 

Finally, as $\beta \to \infty$ the optimal profile approaches the function $P^*$ in $(U^*,1)$, so that the cost is computed as
\begin{equation*}
E(\beta) \underset{\beta \to \infty}{\sim} \int_{U^*}^1 \frac{F(u) + \beta}{U} dU = \beta \ln \bigg( \frac{1}{U^*} \bigg) + \int_{U^*}^1 \frac{F(U)}{U} dU.
\end{equation*}
Integrating by parts and using the definition $P^* = \sqrt{U f(U)}$
\begin{equation*}
\int_{U^*}^1 \frac{F(U)}{U} dU = \int_{U^*}^1 \frac{1}{U} \frac{dP^*}{dU} + \frac{f(U)}{U P^*} dU = \int_{U^*}^1 \frac{2 P^*}{U^2} dU = \int_{U^*}^1 \frac{\sqrt{f}}{U^{\frac{3}{2}}} dU. \qedhere
\end{equation*}
\end{proof}
%
%
%
%

\begin{figure}
\begin{subfigure}{.475\textwidth}
\resizebox{\textwidth}{!}{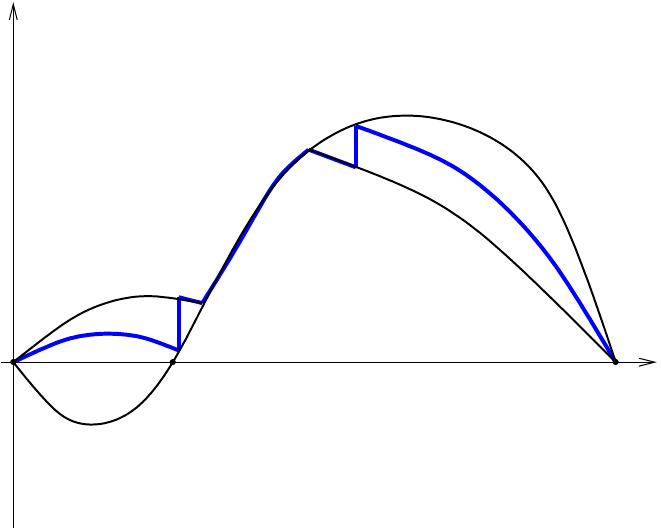}
\end{subfigure}
\hfill
\begin{subfigure}{.475\textwidth}
\resizebox{\textwidth}{!}{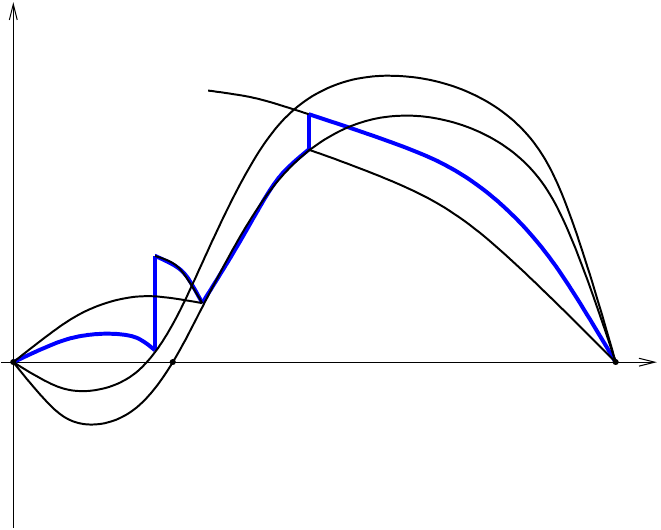}
\end{subfigure}
\caption{The comparison path $\tilde \gamma$ for the dependence of the effort $E$ w.r.t. $\beta$ (left) and $f$ (right) when $\beta < \beta^{**}(f_0)$.}
\label{Fig:dep_beta_f}
\end{figure}

The next two propositions address the Lipschitz dependence w.r.t. the source $f$. The proofs are more involved and the estimates are only local, with coefficients depending on $f_0,f_1$. 

\begin{proposition}
\label{Prop:depe_E_f}
Let $f_0,f_1$ be two source functions satisfying Assumptions \eqref{H1} and such that $f_0 \leq f_1$, and denote the corresponding optimal solutions with $P_0(U),P_1(U)$ respectively. Let \gls{rbar}, \gls{betaunder} be the constants chosen in Theorem \ref{manifoldbound} for the function $f_1$. Then there is a constant \newglossaryentry{Df}{name=\ensuremath{C_{f_0,f_1}},description={Lipschitz constant for the dependence of the $E$ on the source $f$}} $\gls{Df}$ depending only on
\begin{equation*}
\max \Big\{ \gls{Cfbeta}, f = (1-\varpi) f_0 + \varpi f_1, \varpi \in [0,1] \Big\}, \quad 
U_u(f_1), U_u(f_0), 
\end{equation*}
\begin{equation*}
\|f_0\|_{C^2}, \|f_1\|_{C^2}, \quad \beta, \quad \frac{df_0}{dU}(U^*(\beta,f_0)),\frac{df_1}{dU}(U^*(\beta,f_1)), \quad \int_{U^*(f_0)}^{U_s(f_0)} \frac{1}{U P(f_0,U)} dU,
\end{equation*}
such that
\begin{equation*}
E(\beta,f_1) \le E(\beta,f_0) + C_{f_0,f_1} \|f_1 - f_0\|_{C^1([0,1])}.
\end{equation*}
\end{proposition}

\begin{proof}
%
%
\noindent{\it Study of the manifolds $\Gamma_u,\Gamma_s$.} Consider the curve $f_\varpi = (1 - \varpi) f_0 + \varpi f_1$ with $f_0 \leq f_1$. Recall that by Equation \eqref{Equa:estimat_f_uGG} of Theorem \ref{manifoldbound} we have the estimates
\begin{equation*}
0 \leq - \partial_\varpi \Gamma_u(f_\varpi,\bar U) \leq C(f_\varpi) \bar U \|f_1 - f_0\|_{C^1},
\end{equation*}
for all $\bar U \leq U_u(f_\varpi)$. Hence, setting
\begin{equation*}
\bar C = \max \big\{ \gls{Cfbeta}, f = (1 - \varpi) f_0 + \varpi f_1, \varpi \in [0,1] \big\}
\end{equation*}
we obtain
\begin{equation}
\label{Equa:Gamma_u_Lip}
0 \leq \frac{\Gamma_u(f_0,U_u(f_1)) - \Gamma_u(f_1,U_u(f_1))}{U_u(f_1)} \leq \bar C \|f_1 - f_0\|_{C^1([0,1])}.
\end{equation}
Note that $U_u(f_1) \leq U_u((1-\varpi) f_0 + \varpi f_1)$ for all $\varpi \in [0,1]$ because $f_0 \leq f_1$, so that we can apply Theorem \ref{manifoldbound}.

In a similar way, using Equation \eqref{Equa:estimat_f_sGG} of Theorem \ref{manifoldbound} we obtain for all $U \in (U^*(f_0),1)$
\begin{equation}
\label{Equa:Gamma_s_Lip}
0 \leq \frac{\Gamma_s(f_0,U) - \Gamma_s(f_1,U)}{1-U} \leq \bar C \|f_1 - f_0\|_{C^1([0,1])}.
\end{equation}
%
%

\medskip

\noindent{\it Comparison of trajectories.} Let $\underline{P}(U)$ be the trajectory solving
\begin{equation*}
\frac{d\underline{P}}{dU} = - \frac{f_1(U)}{\underline{P}} - \beta, \quad \underline{P}(U_u(f_0)) = \Gamma_u(f_0,U_u(f_0)).
\end{equation*}
Since $\beta_1 \leq \beta_2$, it follows that
\begin{equation*}
\Gamma_u(f_1,U) \leq \Gamma_u(f_0,U) \leq \underline{P}(U), \quad U \in (U_u(f_0),U_u(f_1)).
\end{equation*}
We state that there is a constant $C$, depending only on $\|f_1\|_{C^2}$, $\beta$ and $\frac{df_0}{dU}(U^*(\beta,f_0)),\frac{df_1}{dU}(U^*(\beta,f_1))$ such that
\begin{equation}
\label{Equa:underP_fcomp}
\underline{P}(U_u(f_1)) - \Gamma_u(f_1,U_u(f_1)) \leq C \big( \|f_1 - f_0\|_{C^1} + (U^*(f_0) - U^*(f_1)) \big).
\end{equation}

If $0 \leq \beta \leq \gls{betaunder}(f_0)$, then the difference can be estimated by using the ODE \eqref{Equa:funct_eq_UP} for $\Gamma_u(f_0),\underline{P}$:
\begin{equation*}
\begin{split}
\underline{P}(U_u(f_1)) - \Gamma_u(f_1,U_u(f_1)) &= \Gamma_u(f_0,U_u(f_1)) - \Gamma_u(f_1,U_u(f_1)) + \int_{U_u(f_1)}^{U_u(f_0)} \frac{f_1(V) - f_0(V)}{\underline{P}(V)} dV \\
\big[ (\ref{Equa:Gamma_u_Lip},\ref{Equa:Gamma_s_Lip}) \big] \quad &\leq \frac{\bar C}{U_s(f_0)} \|f_1 - f_0\|_{C^1} + \frac{\|f_1 - f_0\|_{C^1}}{\Gamma_u(f_0,U_u(f_0))},
\end{split}
\end{equation*}
where we observed that in the region of interest $\underline{P}$ is decreasing because $\beta,f_1 \geq 0$.

If $\beta < 0$, the quantities $\min\{ \underline{P}\}$, $\Gamma_u(f_0,U_u(f_0))$ are comparable being both uniformly positive.

Thus the only case to be considered is when the r.h.s. of \eqref{Equa:underP_fcomp} tends to $0$ and $\beta > \gls{betaunder}(f_0)$. This means that
$$
(U_u(f_1),U_u(f_0)) \subset \bigg( U^*(f_0) - \frac{\gls{rbar}}{2},U^*(f_0) + \frac{\gls{rbar}}{2} \bigg) \cap \bigg( U^*(f_1) - \frac{\gls{rbar}}{2},U^*(f_1) + \frac{\gls{rbar}}{2} \bigg).
$$
We recall that \gls{rbar} depends only on $\|f\|_{C^2}$, so that it can be chosen uniformly for $f_0,f_1$. In this case the linear analysis of Proposition \ref{Prop:blowupsecond} gives that the trajectories of $\underline{P}$ are well approximated by
\begin{equation*}
\big| \underline{P}(U) - \lambda_+(f_1,U^*(f_1)) (U - U^*) \big| \sim \big| \underline{P}(U) - \lambda_-(f_1,U^*(f_1)) (U - U^*) \big|^{\frac{|\lambda_+(f_1,U^*(f_1))|}{|\lambda_-(f_1,U^*(f_1))|}}.
\end{equation*}
If $\underline{\hat P}$ is the trajectory starting from $U^*(f_0)$, then using $\underline{\hat P} \leq \underline{P}$ and $f_1((U^*(f_1),U^*(f_0)) \leq 0$, one obtains
\begin{equation*}
\begin{split}
\underline{P}(U^*(f_1)) - \underline{P}(U^*(f_0)) &\leq \underline{\hat P}(U^*(f_1)) - \underline{\hat P}(U^*(f_0)) = \underline{\hat P}(U^*(f_1)) \\
&\sim (U^*(f_0) - U^*(f_1)) \Bigg( \frac{|\lambda_-(f_1,U^*(f_1))|^{|\lambda_-(f_1,U^*(f_1))|}}{|\lambda_+(f_1,U^*(f_1))|^{|\lambda_+(f_1,U^*(f_1))|}} \Bigg)^{\frac{1}{|\lambda_-(f_1,U^*(f_1))|-|\lambda_+(f_1,U^*(f_1))|}},
\end{split}
\end{equation*}
where in the last line we used the correct expression of $\underline{\hat P}$ in the linear case.
We conclude
\begin{equation*}
\begin{split}
\underline{P}(U_u(f_1)) - \Gamma_u(f_0,U_u(f_1)) &\leq \underline{P}(U^*(f_1)) - \Gamma_u(f_0,U^*(f_0)) \\ 
\big[ \Gamma_u(f_0) \leq \underline{P} \big] \quad &\leq \underline{P}(U^*(f_1)) - \underline{P}(U^*(f_0)) + \int_{U^*(f_0)}^{U_u(f_0)} \frac{f_1(V) - f_0(V)}{\Gamma_u(f_0,V)} dV \\
&\lesssim U^*(f_0) - U^*(f_1) + \|f_1 - f_0\|_{C^0} \int_{U^*(f_0)}^{U_u(f_0)} \frac{dV}{\Gamma_u(f_0,V)} \\
&\sim U^*(f_0) - U^*(f_1) + \|f_1 - f_0\|_{C^0} \Gamma_u(f_0,U_u(f_0)).
\end{split}
\end{equation*}
where in the first line we used that $\Gamma_u$ is decreasing in this interval.

\medskip

\noindent{\it Comparison curve and computation of the cost.} We consider the curve $\tilde \gamma$ obtained by concatenating the following $5$ curves (see Fig. \ref{Fig:dep_beta_f} right):
\begin{itemize}
\item the unstable manifold $\Gamma_u(f_1,U)$ up to $U_u(f_1)$;
\item the vertical segment from $\Gamma_u(f_1,U_u(f_1))$ to $\underline{P}(U_u(f_1))$;
\item the curve $\underline{P}(U)$ for $U \in (U_u(f_1),U_u(f_0))$;
\item the optimal profile $P_\beta(f_0,U)$ for $f_0$ from $U_u(f_0)$ until the point $(U_s(f_0),\Gamma_s(f_0,U_s(f_0)))$,
\item the upwards segment from $(U_s(f_0),\Gamma_s(f_0,U_s(f_0))$ to $(U_s(f_0),\Gamma_s(f_1,U_s(f_0))$,
\item the stable manifold $\Gamma_s(f_1)$ starting from $(U_s(f_0)$. 
\end{itemize}
It is fairly easy to see that $\tilde \gamma$ is admissible for $f_1$, because $f_0 \leq f_1$. We obtain 
\begin{align*}
E(\beta,f_1) &\le \int_{\tilde \gamma} \left( \frac{f_1(U)}{PU} + \frac{\beta}{U} \right) \, dU  + \left( \frac{1}{U} \right) \, dP \\
&= \frac{\underline{P}(U_u(f_1)) - \Gamma_u(f_1,U_u(f_1))}{U_u(f_1)} + \int_{U_u(f_0)}^{U_s(f_0)} \left( \frac{f_1(U)}{U P_\beta(f_0,U)} + \frac{\beta}{U} \right) \, dU \\
& \quad + \frac{\Gamma_s(f_1,U_s(f_0))-\Gamma_s(f_0,U_s(f_0))}{U_s(f_0)} \\
\big[ (\ref{Equa:Gamma_s_Lip}),(\ref{Equa:underP_fcomp}) \big] \quad &\le E(f_0,\beta) + \int_{\gamma(f_0) \cap \{U \in [U_u(f_1),U_s(f_0)]\}} \frac{f_1(U)-f_0(U)}{P_\beta(f_0,U) U} \, dU \\
& \quad + C \big( U^*(f_0) - U^*(f_1) + \|f_1 - f_0\|_{C^1} \big).
\end{align*}
The second term is estimated by first noticing that from the asymptotic behavior of $P(f_0)$ near $U^*(f_0)$ (consequence of \eqref{Equa:limi_F} and Lemma \ref{lowerbound}) 
\begin{equation*}
P_\beta(f_0,U) \geq \frac{P^*(f_0,U)}{C(\|f\|_{C^2})}, \quad U \in (U_u(f_0),U_s(f_0)),
\end{equation*}
so that for some constant depending only on $f_0$:
\begin{equation*}
\begin{split}
\int_{\gamma(f_0) \in \{U_u(f_1) \le U \le U_s(f_0)\}} \frac{f_1(U)-f_0(U)}{U P_\beta(f_0,U)} \, dU &\le \mathcal O(1) \frac{\|f_1 - f_0\|_{C^0}}{U^*(f_0)} \int_{U^*(f_0)}^{1} \frac{1}{P^*(f_0,U)} dU \\
&= \mathcal O(1) \|f_1 - f_0\|_{C^1}. 
\end{split}
\end{equation*}
which gives the statement, because from $f'_i(U^*_i) \not =0$, $i=1,2$, it follows that
\begin{equation*}
U^*(f_0) - U^*(f_1) \lesssim \|f_1 - f_0\|_{C^0}. \qedhere
\end{equation*}
\end{proof}

\begin{proposition}
\label{Prop:Lipschitz_below}
Let $f_0 \leq f_1$ be admissible sources satisfying Assumption \eqref{H1}. There exists a constant, which we denote again by $C_{f_0,f_1}$, depending only on
\begin{equation*}
\beta, \quad \|f_1\|_{C^2}, \|f_0\|_{C^2}, \quad U^*(\beta,f_1), \quad \frac{d}{dU} f_1(U^*(\beta,f_1)), 
\end{equation*}
such that
\begin{equation*}
E(\beta,f_0) \leq E(\beta,f_1) + C_{f_0,f_1} \|f_1 - f_0\|_{C^1}.
\end{equation*}
\end{proposition}

\begin{figure}
\resizebox{.75\textwidth}{!}{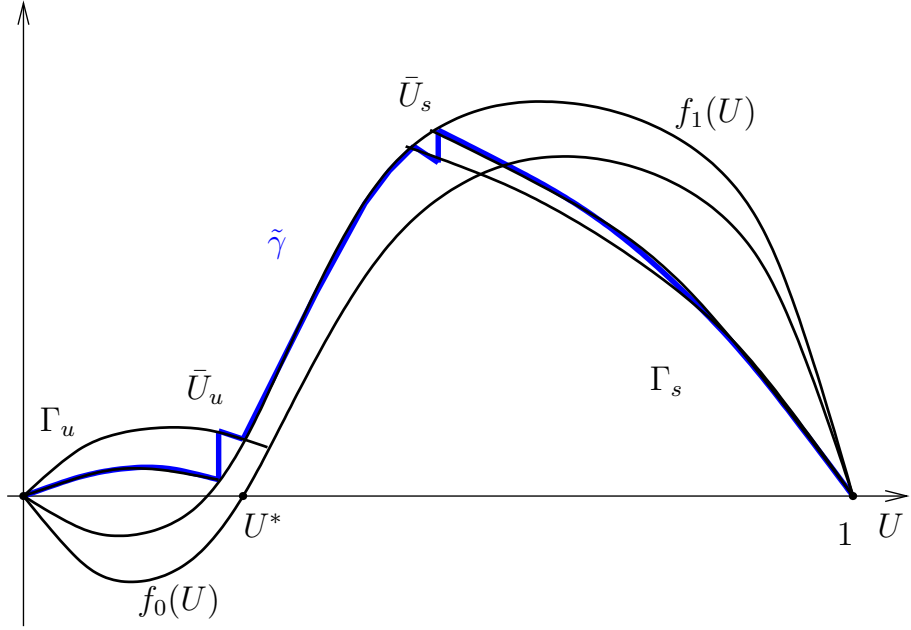}
\caption{The contruction of the curve $\tilde \gamma$ of Step 2 of the proof of Proposition \ref{Prop:Lipschitz_below}.}
\label{Fig:dep_beta_f_2}
\end{figure}

\begin{proof}
Let $P_\beta(f_1)$ be the optimal solution for $f_1$ with optimal control $\alpha(f_1,U)$. 

\medskip

\noindent{\it Step 2: $f_0$-admissibility for larger $\beta$.} In order to choose the speed which makes an interval of the optimal trajectory $P_\beta(f_1)$ admissible with the source $f_0$, we first define the curve
\begin{equation}
\label{Equa:compar_cc}
\bar P(U) = \begin{cases}
P^*(f_1,U) & U \leq U_u(\beta,f_1) \ \text{or} \ U \geq U_s(\beta,f_1), \\
P_\beta(f_1,U) & U \in (U_u(\beta,f_1),U_s(\beta,f_1)).
\end{cases}
\end{equation}
We estimate how much we have to increase the speed $\beta$ in order to make this curve admissible in $(U_u(\beta,f_1),U_s(\beta,f_1))$: it must hold
\begin{equation*}
0 \leq \frac{d\bar P}{dU} + \frac{f_0(U)}{\bar P(U)} + \beta' = \begin{cases}
\frac{U f_1(U) + \frac{d}{dU} f_1(U) + 2 U f_0(U)}{2 \sqrt{U f_1(U)}} + \beta' & \bar P(U) = P^*(f_1,U), \\
\frac{f_0(U) - f_1(U)}{\bar P(U)} + \beta' - \beta & \bar P(U) \not= P^*(f_1,U).
\end{cases}
\end{equation*}
Recalling that by Point \ref{Point4:maxback} of Proposition \ref{Mbt} the minimal value of $\bar P(U)$ when $\bar P \not= P^*(f_1)$ is strictly positive by a constant depending only on $f_1$ (for $\beta$ fixed), the second case is easily estimated by
\begin{equation*}
\beta' - \beta \leq C_{f_0,f_1} \|f_1 - f_0\|_{C^0}.
\end{equation*}
In the regions where $P(f_1) = P^*(f_1)$ we have
\begin{equation*}
\begin{split}
\frac{d}{dU} P(f_1,U) + \frac{f_0(U)}{P(f_1,U)} + \beta &= \frac{U f_1(U) + \frac{d}{dU} f_1(U) + 2 U f_0(U)}{2 \sqrt{U f_1(U)}} + \beta. 
\end{split}
\end{equation*}
Since $f_1$ satisfies Assumption \eqref{H1} and similarly to \eqref{Equa:bound_integr}, it follows that
\begin{equation*}
\frac{U f_1(U) + \frac{d}{dU} f_1(U) + 2 U f_0(U)}{2 \sqrt{U f_1(U)}} \underset{U \searrow U^*(f_1)}{\sim} \frac{1}{\sqrt{U - U^*(f_1)}},
\end{equation*}
\begin{equation*}
\frac{U f_1(U) + \frac{d}{dU} f_1(U) + 2 U f_0(U)}{2 \sqrt{U f_1(U)}} \underset{U \nearrow 1}{\sim} - \frac{1}{\sqrt{1 - U}},
\end{equation*}
which means that there is a compact interval $\bar I = [\underline{U}(-), 1] \subset (U^*(f_1),1]$ such that
\begin{equation*}
\frac{d}{dU} P(f_1,U) + \frac{f_0(U)}{P(f_1,U)} + \beta \geq 0 \quad U \in (U^*(f_1),U_s(f_1)] \setminus \bar I.
\end{equation*}

In the interval $\bar I$ we can use the ODE for $P(f_1)$ obtaining
\begin{equation*}
- \frac{d}{dU} P(f_1,U) - \frac{f_0(U)}{P(f_1,U)} - \beta = \frac{f_1(U) - f_0(U)}{P(f_1,U)} - \alpha(f_1,U) \leq \frac{\|f_1 - f_0\|_{C^0}}{{\displaystyle \min_{[\underline U(-),U_s(f_1)] \cap \bar I} P(f_1,U)}}.
\end{equation*}
Hence in both cases we obtain the estimate
\begin{equation}
\label{Equa:beta_Lipf}
\beta' - \beta \leq C_{f_0,f_1} \|f_1 - f_0\|_{C^0},
\end{equation}
where the constant depends on $\|f_0\|_{C^2}$, $\|f_1\|_{C^2}$ and the local behavior of $f_1$ near $U = U^*(\beta,f_1)$.

\medskip

\noindent{\it Step 2: definition of the comparison curve and computation of the cost.} Define $\bar U_u(\beta')$ ($\bar U_s(\beta')$) as the first (last) intersection point of $\Gamma_u(\beta',f_0,U)$ ($\Gamma_u(\beta',f_0,U)$) with the curve \eqref{Equa:compar_cc}. Similarly define $\bar U_u(\beta)$ ($\bar U_s(\beta)$) as the first (last) intersection point of $\Gamma_u(\beta,f_0,U)$ ($\Gamma_u(\beta,f_0,U)$) with the curve \eqref{Equa:compar_cc}: note that in the latter case by Theorem \ref{manifoldbound}
\begin{equation*}
U_u(\beta,f_1) \leq U_u(\beta,f_0) \leq U_s(\beta,f_0) \leq U_s(\beta,f_0).
\end{equation*}

Define the comparison curve $\tilde \gamma$ as the concatenation of:
\begin{itemize}
\item the manifold $\Gamma_u(\beta',f_0,U)$ up to $\bar U_u(\beta')$;
\item the vertical segment $[\Gamma_u(\beta',f_0,\bar U_u(\beta')),\Gamma_u(\beta,f_0,\bar U_u(\beta'))]$;
\item the manifold $\Gamma_u(\beta,f_0,U)$ up to $\bar U_u(\beta)$;
\item the solution $P_\beta(f_1,U)$ for $U \in (\bar U_u(\beta),\bar U_s(\beta))$;
\item the manifold $\Gamma_s(\beta,f_0,U)$ up to $\bar U_s(\beta')$;
\item the vertical segment $[\Gamma_s(\beta,f_0,\bar U_s(\beta')),\Gamma_s(\beta',f_0,\bar U_s(\beta'))]$;
\item the manifold $\Gamma_s(\beta',f_0,U)$ up to $1$.
\end{itemize}
Each arc of the above curve is admissible for the source $f_0$ with the speed $\beta'$: the solution $P_\beta(f_1,U)$, $U \in (\bar U_u(\beta),\bar U_s(\beta))$, by the choice of $\beta'$ made in Step 1, the segments are "going up" and finally
\begin{equation*}
\frac{d}{dU} \Gamma_u(\beta,f_0,U) + \frac{f_0(U)}{\Gamma_u(\beta,f_0)} + \beta' = \beta' - \beta \geq 0.
\end{equation*}

The effort is then computed as
\begin{equation*}
\begin{split}
E(\beta',f_0) &\leq \frac{\Gamma_u(\beta',f_0,\bar U_u(\beta')) - \Gamma_u(\beta,f_0,\bar U_u(\beta'))}{\bar U_u(\beta')} + \int_{\bar U_u(\beta')}^{\bar U_u(\beta)} \frac{\beta' - \beta}{U} dU \\
& \quad + \int_{\bar U_u(\beta)}^{\bar U_u(\beta')} \bigg( \frac{dP_\beta(f_1)}{dU} + f_0(U) + \beta' \bigg) \frac{dU}{U} + \int_{\bar U_s(\beta)}^{\bar U_s(\beta')} \frac{\beta' - \beta}{U} dU \\
& \quad + \frac{\Gamma_s(\beta,f_0,\bar U_s(\beta')) - \Gamma_s(\beta',f_0,\bar U_s(\beta'))}{\bar U_s(\beta')}.
\end{split}
\end{equation*}
Using Theorem \ref{manifoldbound} we have that 
\begin{equation*}
\big| \Gamma_u(\beta,f_0,\bar U_u(\beta')) - \Gamma_u(\beta',f_0,\bar U_u(\beta)) \big| \leq C_{f_0,f_1} (\beta' - \beta) \underset{(\ref{Equa:beta_Lipf})}{\leq} C_{f_0,f_1} \|f_1 - f_0\|_{C^1},
\end{equation*}
and similarly 
\begin{equation*}
\big| \Gamma_s(\beta',f_0,\bar U_s(\beta')) - \Gamma_s(\beta,f_1,\bar U_s(\beta)) \big| \leq C_{f_0,f_1} \|f_1 - f_0\|_{C^1}.
\end{equation*}
Moreover
\begin{equation*}
\int_{\bar U_u(\beta')}^{\bar U_u(\beta)} \frac{\beta' - \beta}{U} dU + \int_{\bar U_s(\beta)}^{\bar U_s(\beta')} \frac{\beta' - \beta}{U} dU \leq (\beta' - \beta) \ln((U^*(f_1))^{-1}) \leq C_{f_0,f_1} \|f_1 - f_0\|_{C^1}.
\end{equation*}
Finally
\begin{equation*}
\begin{split}
\int_{\bar U_u(\beta)}^{\bar U_u(\beta')} \bigg( \frac{dP_\beta(f_1)}{dU} + f_0(U) + \beta' \bigg) \frac{dU}{U} &= \int_{\bar U_u(\beta)}^{\bar U_u(\beta')} \bigg( \frac{dP_\beta(f_1)}{dU} + f_1(U) + \beta \bigg) \frac{dU}{U} \\
& \quad + \int_{\bar U_u(\beta)}^{\bar U_s(\beta)} \frac{f_0 - f_1 + \beta' - \beta}{U} dU \\
&\leq E(\beta,f_1) + C_{f_0,f_1} \|f_1 - f_1\|_{C_0}.
\end{split}
\end{equation*}
Adding all terms we obtain by Proposition \ref{Prop:E_Lipschitz_beta}
\begin{equation*}
E(\beta,f_0) \leq E(\beta',f_0) \leq E(\beta,f_1) + C_{f_0,f_1} \|f_1 - f_0\|_{C^1}. \qedhere
\end{equation*}
\end{proof}

%

Combining the previous three propositions together, we obtain the following

\begin{corollary}
\label{Cor:Lipsc_f}
The map $\beta,f \mapsto E(\beta,f)$ is locally Lipschitz in the set 
\begin{equation*}
\gls{Dcal} = \Big\{ (\beta,f), f \ \text{satisfies Assumption \eqref{H1}}, \ \beta > \beta^*(f) \Big\} \subset \R \times C^2([0,1]).
\end{equation*}
\end{corollary}

\subsection{Monotonicity and explicit computation of the Fr\'echet derivative for $\beta \not= \beta^{**}$}
\label{Ss:monotoni_Frech}

Having proved the local Lipschitz regularity of $E(\beta,f)$, we can study the effort function $E(\beta,f)$ when $f$ is a polynomial, and then pass to the limit. We first apply this idea in the next proposition to prove that $f \mapsto E(\beta,f)$ is monotone increasing. The same analysis for the $\beta$-dependence will be done in the next section.

\begin{proposition}
\label{Prop:monotone_f}
The map $f \mapsto E(\beta, f)$ is monotone increasing: if $f_0 \leq f_1$ then $E(\beta,f_0) \leq E(\beta,f_1)$ for $\beta > \beta^*(f_0)$.
\end{proposition}

\begin{proof}
When $f = \sum_{i=1}^k a_i U^i$ is a polynomial, by Proposition \ref{Prop:reguls_Ibeta} there is a closed set of measure $0$ in the space $(\beta,a_1,\dots,a_k)$ such that outside this set the dependence
$$
(\beta,f) \mapsto \hat U_u(\beta,f),U_i^-(\beta,f), U^+_i(\beta,f), \hat U_s(\beta,f)
$$
is smooth, being $\hat U_s,U^\pm_i,\hat U_s$ the finitely many boundary points of $\gls{IcalP}$. Consider a segment
$$
[0,1] \ni \varpi \mapsto (1 - \varpi) f_0 + \varpi f_1
$$
inside this open set.

\medskip

\noindent{\it Step 1: computation of the first derivative.} The set $\mathcal I(\varpi)$ is given by
\begin{equation*}
\mathcal I(\varpi) = (0,\hat U_u(\varpi)) \cup \bigcup_{i=1}^N (U_i^-(\varpi),U_i^+(\varpi)) \cup (\hat U_s(\varpi),1),
\end{equation*}
with the function $\varpi \mapsto \hat U_u(\varpi),U_i^\pm(\varpi),\hat U^+_s(\varpi)$ smoothly depending on $\varpi$.

The effort is then computed by
\begin{equation*}
E(\varpi) = \bigg[ \int_{\hat U_u(\varpi)}^{U_1^-(\varpi)} + \sum_i \int_{U_i^+(\varpi)}^{U_{i+1}^-(\varpi)} + \int_{U_N^+(\varpi)}^{\hat U_s(\varpi)} \bigg] \frac{F(\varpi,U) + \beta}{U} dU
\end{equation*}
with
\begin{equation*}
F(\varpi,U) = \frac{d}{dU} P^*(\varpi,U) + \frac{f(\varpi,U)}{P^*(\varpi,U)}.
\end{equation*}

Taking derivatives w.r.t. $\varpi$ we get
\begin{equation}
\label{Equa:E_deltaf}
\begin{split}
\frac{dE}{d\varpi} &= \bigg[ \int_{\hat U_u(\varpi)}^{U_1^-(\varpi)} + \sum_i \int_{U_i^+(\varpi)}^{U_{i+1}^-(\varpi)} + \int_{U_N^+(\varpi)}^{\hat U_s(\varpi)} \bigg] \frac{1}{U} \frac{\partial}{\partial \varpi} F(\varpi,U) dU \\
& \quad - \frac{F(\varpi,\hat U_u) + \beta}{\hat U_u} \frac{\partial \hat U_u}{\partial \varpi} + \frac{F(\varpi,\hat U_s) + \beta}{\hat U_s} \frac{\partial \hat U_s}{\partial \varpi} \\
& \quad + \sum_i \bigg( \frac{F(\varpi,U_i^-) + \beta}{U_i^-} \frac{\partial U_i^-}{\partial \varpi} - \frac{F(\varpi,U_i^+) + \beta}{U_i^+} \frac{\partial U_i^+}{\partial \varpi} \bigg).
\end{split}
\end{equation}
The derivative of $\hat U_u,\hat U_s$ are computed by solving the ODE
\begin{equation}
\label{Equa:Gamma_u_deltaf}
\frac{d}{dU} \partial_\varpi \Gamma_u = \frac{f(U)}{\Gamma_u^2} \partial_\varpi \Gamma_u - \frac{\partial_\varpi f}{\Gamma_u}, \quad \partial_\varpi \Gamma_u(\varpi,0) = \partial_\varpi \lambda_\beta(\varpi,U=0),
\end{equation}
\begin{equation}
\label{Equa:Gamma_s_deltaf}
\frac{d}{dU} \partial_\varpi \Gamma_s = \frac{f(U)}{\Gamma_u^2} \partial_\varpi \Gamma_s - \frac{\partial_\varpi f}{\Gamma_s}, \quad \partial_\varpi \Gamma_u(\varpi,1) = \partial_\varpi \lambda_\beta(\varpi,U=1),
\end{equation}
\begin{equation*}
\begin{split}
\frac{\partial}{\partial \varpi} P^*(\hat U_u) + \partial_\varpi \hat U_u \frac{d}{dU} P^*(\hat U_u) = \partial \varpi \Gamma_u(\hat U_u) + \partial_\varpi \hat U_u \frac{d}{dU} \Gamma_u(\hat U_u), \\
\frac{\partial}{\partial \varpi} P^*(\hat U_s) + \partial_\varpi \hat U_s \frac{d}{dU} P^*(\hat U_s) = \partial \varpi \Gamma_u(\hat U_s) + \partial_\varpi \hat U_s \frac{d}{dU} \Gamma_u(\hat U_s),
\end{split}
\end{equation*}
while in each interval $I_i = (U^-(\varpi),U^+(\varpi))$ the perturbations are computed in order to preserve the necessary condition for optimality: denoting by $P_i$ the solution $P_\beta$ restricted to $I_i$,
\begin{equation}
\label{Equa:Pi_deltaf_1}
\frac{d}{dU} \partial_\varpi P_i = \frac{f(U)}{P_i^2} \partial_\varpi P_i - \frac{\partial_\varpi f}{P_i},
\end{equation}
\begin{equation*}
\frac{d}{d\varpi} \int_{U_i^-}^{U_i^+} \frac{U f(U) - P_i^2}{U P_i^2} dU = \int_{U_i^-}^{U_i^+} \frac{\partial_\varpi f(U)}{P_i^2} + \frac{2 f(U) \partial_\varpi P_i}{P_i^3} dU = 0,
\end{equation*}
\begin{equation*}
\frac{\partial}{\partial \varpi} P^*(U^\pm_i) + \partial_\varpi U_i^\pm \frac{d}{dU} P^*(U^\pm_i) = \partial \varpi P_i(U^\pm_i) + \partial_\varpi U_i^\pm \frac{d}{dU} P_i(U^\pm_i)
\end{equation*}

Integrating by parts
\begin{equation*}
\begin{split}
&\bigg[ \int_{\hat U_u(\varpi)}^{U_1^-(\varpi)} + \sum_i \int_{U_i^+(\varpi)}^{U_{i+1}^-(\varpi)} + \int_{U_N^+(\varpi)}^{\hat U_s(\varpi)} \bigg] \frac{1}{U} \frac{\partial}{\partial \varpi} F(\varpi,U) dU \\
&= \bigg[ \int_{\hat U_u(\varpi)}^{U_1^-(\varpi)} + \sum_i \int_{U_i^+(\varpi)}^{U_{i+1}^-(\varpi)} + \int_{U_N^+(\varpi)}^{\hat U_s(\varpi)} \bigg] \frac{1}{U} \frac{\partial}{\partial \varpi} \bigg( \frac{d}{dU} P^* + \frac{f(U)}{P^*} \bigg) dU \\
&= \bigg[ \int_{\hat U_u(\varpi)}^{U_1^-(\varpi)} + \sum_i \int_{U_i^+(\varpi)}^{U_{i+1}^-(\varpi)} + \int_{U_N^+(\varpi)}^{\hat U_s(\varpi)} \bigg] \frac{\partial}{\partial \varpi} \bigg( \frac{P^*}{U^2} + \frac{f(U)}{U P^*} \bigg) dU \\
& \quad - \frac{\partial_\varpi P^*(\hat U_u(\varpi))}{\hat U_u(\varpi)} + \frac{\partial_\varpi P^*(\hat U_s(\varpi))}{\hat U_s(\varpi)} + \sum_i \bigg( \frac{\partial_\varpi P^*(U_i^-(\varpi))}{U_i^-(\varpi)} - \frac{\partial_\varpi P^*(U_i^+(\varpi))}{U_i^+(\varpi)} \bigg).
\end{split}
\end{equation*}
Substituting into \eqref{Equa:E_deltaf}
\begin{equation*}
\begin{split}
\frac{dE}{d\varpi} &= \bigg[ \int_{\hat U_u(\varpi)}^{U_1^-(\varpi)} + \sum_i \int_{U_i^+(\varpi)}^{U_{i+1}^-(\varpi)} + \int_{U_N^+(\varpi)}^{\hat U_s(\varpi)} \bigg] \frac{\partial}{\partial \varpi} \bigg( \frac{P^*}{U^2} + \frac{f(U)}{U P^*} \bigg) dU \\
& \quad - \frac{\partial_\varpi \Gamma_u(\hat U_u(\varpi))}{\hat U_u(\varpi)} + \frac{\partial_\varpi \Gamma_s(\hat U_s(\varpi))}{\hat U_s(\varpi)} + \sum_i \bigg( \frac{\partial_\varpi P_i(U_i^-(\varpi))}{U_i^-(\varpi)} - \frac{\partial_\varpi P_i(U_i^+(\varpi))}{U_i^+(\varpi)} \bigg).
\end{split}
\end{equation*}

\medskip

\noindent{\it Step 2: evaluation of the sign of $\frac{dE}{d\varpi}$.} Using \eqref{Equa:Gamma_u_deltaf}, \eqref{Equa:Gamma_s_deltaf} we obtain
\begin{equation*}
\partial_\varpi \Gamma_u(\hat U_u(\varpi)) \leq 0, \quad \partial_\varpi \Gamma_s(\hat U_s(\varpi)) \geq 0.
\end{equation*}
Concerning the term $\frac{\partial_\varpi P_i(U_i^-(\varpi))}{U_i^-(\varpi)} - \frac{\partial_\varpi P_i(U_i^+(\varpi))}{U_i^+(\varpi)}$, we can rewrite \eqref{Equa:Pi_deltaf_1} as
\begin{equation*}
\frac{d}{dU} \frac{\partial_\varpi P_i}{U} = \frac{U f(U) - P_i^2}{U P_i^2} \frac{\partial_\varpi P_i}{U} - \frac{\partial_\varpi f}{U P_i}.
\end{equation*}
Defining
\begin{align*}
K(U_1,U_2) := \int_{U_1}^{U_2} \frac{U f(U) - P(U)^2}{U P(U)^2} \, dU
\end{align*}
and using that $K(U^-_i(\varpi),U^+_i(\varpi)) = 0$ by optimality, we obtain
\begin{equation*}
\begin{split}
\frac{\partial_\varpi P_i(U^+_i(\varpi))}{U^+_i(\varpi)} &= \frac{\partial_\varpi P_i(U^-_i(\varpi))}{U^-_i(\varpi)} e^{K(U^-_i,U^+_i)} - \int_{U^-_i(\varpi)}^{U^+_i(\varpi)} \frac{\partial_\varpi f}{U P_i} e^{K(V,U^+_i(\varpi))} dU \\
&= \frac{\partial_\varpi P_i(U^-_i(\varpi))}{U^-_i(\varpi)} - \int_{U^-_i(\varpi)}^{U^+_i(\varpi)} \frac{\partial_\varpi f}{U P_i} e^{K(V,U^+_i(\varpi))} dU.
\end{split}
\end{equation*}
We then obtain the following expression for $\partial_\varpi E$:
\begin{equation*}
\begin{split}
\frac{dE}{d\varpi} &= \bigg[ \int_{\hat U_u(\varpi)}^{U_1^-(\varpi)} + \sum_i \int_{U_i^+(\varpi)}^{U_{i+1}^-(\varpi)} + \int_{U_N^+(\varpi)}^{\hat U_s(\varpi)} \bigg] \frac{\partial}{\partial \varpi} \bigg( \frac{P^*}{U^2} + \frac{f(U)}{U P^*} \bigg) dU \\
& \quad - \frac{\partial_\varpi \Gamma_u(\hat U_u(\varpi))}{\hat U_u(\varpi)} + \frac{\partial_\varpi \Gamma_s(\hat U_s(\varpi))}{\hat U_s(\varpi)} + \sum_i \int_{U^-_i(\varpi)}^{U^+_i(\varpi)} \frac{\partial_\varpi f}{U P_i} e^{K(V,U^+_i(\varpi))} dU,
\end{split}
\end{equation*}
Using
\begin{equation*}
\frac{\partial}{\partial \varpi} \frac{P^*}{U^2} = \frac{\partial}{\partial \varpi} \frac{f(U)}{U P^*} = \frac{\partial}{\partial \varpi} \frac{\sqrt{f(U)}}{U\sqrt{U}} = \frac{\partial_\varpi f}{2 U \sqrt{Uf}},
\end{equation*}
we finally get
\begin{equation}
\label{Equa:deriv_expli_sour}
\begin{split}
\frac{dE}{d\varpi} &= \bigg[ \int_{\hat U_u(\varpi)}^{U_1^-(\varpi)} + \sum_i \int_{U_i^+(\varpi)}^{U_{i+1}^-(\varpi)} + \int_{U_N^+(\varpi)}^{\hat U_s(\varpi)} \bigg] \frac{\partial_\varpi f(U)}{U P^*(U)} dU \\
& \quad - \frac{\partial_\varpi \Gamma_u(\hat U_u(\varpi))}{\hat U_u(\varpi)} + \frac{\partial_\varpi \Gamma_s(\hat U_s(\varpi))}{\hat U_s(\varpi)} + \sum_i \int_{U^-_i(\varpi)}^{U^+_i(\varpi)} \frac{\partial_\varpi f(U)}{U P_i(U)} e^{K(V,U^+_i(\varpi))} dU. 
\end{split}
\end{equation}
where $P(\varpi,U)$ is the optimal solution for $f(\varpi,U)$.

The above expression is positive because $\partial_\varpi f \geq 0$.

\medskip

\noindent{\it Step 3: monotonicity.} We have proved that, outside a closed Lebesgue-negligible set, the map $(a_1,\dots,a_k) \mapsto E(f)$ is monotone w.r.t. the directions in the nonempty the cone
\begin{equation*}
\sum_{i = 1}^k a_i U^i \geq 0, \quad \sum_{i=1}^k a_i = 0.
\end{equation*}
It is fairly easy to see that this implies that $E(f)$ is monotone increasing when $f$ is increasing.

Let now $f_0 \leq f_1$ two admissible fluxes with $f_0 \leq f_1$. Then there are polynomial approximations $f_0^n, f_1^n$ such that $f_0^n \leq f_1^n$, which gives $E(f_0^n) \leq E(f_1^n)$. Using the continuity of Corollary \ref{Cor:Lipsc_f} and passing to the limit we conclude that $f \mapsto E(f)$ is monotone.
\end{proof}
%
%
%

\begin{remark}
\label{Rem:deriva_f}
Note that the formula \eqref{Equa:deriv_expli_sour} is meaningful also for $f \notin \mathfrak T$: indeed Proposition \ref{Prop:reguls_Ibeta} allows to repeat the same computations outside a closed set with empty interior. 
In particular, we obtain that 
$\frac{d}{d\varpi} E$ is given by \eqref{Equa:deriv_expli_sour}. For any fixed $U \in \clos(\mathcal O(\beta,f))$, we can observe that by optimality conditions \eqref{Equa:nec2_gen}
\begin{equation*}
\int_V^U \frac{W F(W) - P(W)^2}{W P^2(W)} dW = \begin{cases}
0 & V \in \clos(\mathcal O(\beta,f)), \\
K(V,U^+) & V < U, V \in (U^-,U^+) \ \text{connected component of} \ \mathcal I(\beta,f), \\
- K(U^-,V) & V > U, V \in (U^-,U^+) \ \text{connected component of} \ \mathcal I(\beta,f).
\end{cases}
\end{equation*}
Hence \eqref{Equa:deriv_expli_sour} can be rewritten as ($P(U)$ is the optimal solution for the source $f$)
\begin{equation*}
\frac{dE}{d\varpi} = \bigg[ \int_U^1 - \int_0^U \bigg] \frac{\partial_\varpi f(V)}{V P(V)} e^{K(V,U)} dV, \quad U \in \mathcal O(\beta,f).
\end{equation*}
We observe now that if $U \in \mathcal O(\beta,f)$, then the same happens for $f'$ sufficiently close to $f$ in $C^2$ (otherwise using $P_\beta(f') \to P_\beta(f)$ would give $F(U) + \beta = 0$, a contradiction, see also the general case part of the proof of Theorem \ref{Theo:derivative_E}). Being the operator
\begin{equation*}
C^2 \ni g \mapsto \bigg[ \int_U^1 - \int_0^U \bigg] \frac{\partial_\varpi f(V)}{V P_\beta(V)} e^{K(V,U)} dV
\end{equation*}
bounded w.r.t. the $C^2$-norm and continuous w.r.t. $P_\beta$ for $\beta \not= \beta^{**}$, it is the expression of the Fr\'echet derivative of $E$ at $f$ away from $\beta^{**}$. However, because of Proposition \ref{Prop:Lipechit_U_u}, the only regularity across $\beta = \beta^{**}(f)$ is Lipschitz.

We will see in the next section that a similar expression can be given for $\frac{dE}{d\beta}$: in this case Proposition \ref{Prop:Lipechit_U_u} gives the $C^1$-regularity.
\end{remark}

\section{Differentiability properties of $E(\beta)$}
\label{S:diffe_E_beta}

In this section we address the differentiability properties of $\beta \mapsto E(\beta)$, showing that the effort function is of class $C^1((\beta^*,\infty))$. We will also address the regularity outside the point $\beta^{**}$, showing that in general it is only $C^{1,1}_\loc((\beta^*,\infty) \setminus \{\beta^{**}\})$. Crossing $\beta = \beta^{**}$, the second derivative in general does not exists.

We will also answer the convexity problem of $E(\beta)$, a key assumption in the $\Gamma$-convergence problem considered in \cite{BressanChiri23,bressancetraro}. By constructing 2 examples, we will shows that in general $E(\beta)$ is not convex, also in the case of the cubic polynomial.

\subsection{First derivative of $E(\beta)$}
\label{Ss:first_der_Ebeta}

Let $\gls{Pbeta}(U) = P_\beta(f,U)$ be the optimal solution for $\beta \in [\beta^*,+\infty)$ and let \newglossaryentry{Ppartialu}{name=\ensuremath{\partial P_u},description={evolution of the perturbation along the optimal solution starting from $U = 0$}} \newglossaryentry{Ppartials}{name=\ensuremath{\partial P_s},description={evolution of the perturbation along the optimal solution starting from $1$}} $\gls{Ppartialu} = \partial P_u(\beta,U)$ and $\gls{Ppartials} = \partial P_s(\beta,U)$ denote respectively the solutions to
\begin{subequations}
\label{Equa:delta_P}
\begin{align}
\label{Equa:delta_Pu}
\begin{cases}
\frac{d}{dU}\left( \frac{\partial P_u(U)}{U} \right) = \frac{Uf(U)-P(U)^2}{UP(U)^2} \left(\frac{\partial P_u(U)}{U} \right) - \frac{1}{U}, \\ 
\partial P_u(U_u(\beta)) = \partial_\beta \Gamma_u(\beta,U_u(\beta)),
\end{cases}
\end{align}
\begin{align}
\label{Equa:delta_Ps}
\begin{cases}
\frac{d}{dU}\left( \frac{\partial P_s(U)}{U} \right) = \frac{Uf(U)-P(U)^2}{UP(U)^2} \left(\frac{\partial P_s(U)}{U} \right) - \frac{1}{U}, \\ 
\partial P_s(U_s(\beta)) = \partial_\beta \Gamma_s(\beta,U_s(\beta)).
\end{cases}
\end{align}
\end{subequations}
The quantities $\partial_\beta \Gamma_u,\partial_\beta \Gamma_s$ are the derivative of $\Gamma_u,\Gamma_s$ given by Theorem \ref{manifoldbound} and \ref{Cor:Frechet_Gamma}:
\begin{equation*}
\begin{split}
\partial_\beta \Gamma_u(\beta,U) &= - \int_0^U e^{\int_V^U \frac{f(W)}{\Gamma_u(\beta,f,W)^2} dW} dV,
\end{split}
\end{equation*}
\begin{equation*}
\begin{split}
\partial_\beta \Gamma_s(\beta,U) &= \int_U^1 e^{-\int_U^V \frac{f(W)}{\Gamma_s(\beta,f,W)^2} dW} dV.
\end{split}
\end{equation*}

Define the kernel \newglossaryentry{Kkernel}{name=\ensuremath{K(\beta,U_1,U_2)},description={Kernel for the linearized ODE along the optimal solution}}
\begin{align}
\label{k}
\gls{Kkernel} := \int_{U_1}^{U_2} \frac{U f(U) - P_\beta(U)^2}{U P_\beta(U)^2} \, dU, \quad U_u(\beta) \leq U_1,U_2 \leq U_s(\beta). 
\end{align}
Since $P_\beta(U) = U f(U)$ for $U_u(\beta) \leq U \leq \bar U$ by Point \eqref{Point4:maxback} of Proposition \ref{Mbt}, we conclude that $K(\beta,V,U)$ is well defined for $U_u(\beta) < V,U < U_s(\beta)$. By the necessary condition for optimality (see Proposition \ref{nec2}),
\begin{equation*}
K(\beta, \hat U_u(\beta), U) = \begin{cases}
0 & U \in \overline{\mathcal{O}(\beta)}, \\
K(\beta, U^-, U) & U \in (U^-, U^+)\subseteq \mathcal{I}(\beta) \text{ connected component},
\end{cases}
\end{equation*}
\begin{equation*}
K(\beta, U, \hat U_s(\beta))=\begin{cases}
0 &  U \in \overline{\mathcal{O}(\beta)}, \\
K(\beta, U, U^+) & U \in (U^-, U^+) \subseteq \mathcal{I}(\beta) \text{ connected component}.
\end{cases}
\end{equation*}
In particular, we obtain
\begin{equation}
\label{Equa:kernel_canc}
K(\beta, \hat U_u(\beta), U)+K(\beta, U, \hat U_s(\beta)) = 0 \quad \forall U  \in [\hat U_u(\beta), \hat U_s(\beta)].
\end{equation}

%

\begin{remark}
\label{Rem:extention_K}
By means of the computations of Section \ref{Ss:stab_sti_Gamma_us}, one can extend the definition of $K$ to $V \in [0,1]$. However, having already computed the differentiability properties of $\Gamma_u,\Gamma_s$, we need only to study the perturbations of optimal solutions in the interval $(U_u(\beta,f),U_s(\beta,f))$.
\end{remark}

\begin{lemma}
\label{Lem:unique}
The unique solutions to \eqref{Equa:delta_P} are given by the formulas
\begin{subequations}
\label{Equa:sol_der_betaf}
\begin{equation}
\label{Equa:sol_der_betaf1}
\begin{split}
\frac{\delta P_u(U)}{U} &= \frac{1}{U_u(\beta)} \partial_\beta \Gamma_u(\beta,U_u(\beta)) e^{K(\beta,U_u(\beta),U)} - \int_{U_u(\beta)}^U \frac{ e^{K(\beta,V,U)}}{V} \, dV \\
&= - \int_{0}^U \frac{e^{K(\beta,V,U)}}{V} \, dV, 
\end{split}
\end{equation}
\begin{equation}
\label{Equa:sol_der_betaf2}
\begin{split}
\frac{\delta P_s(U)}{U} &= \frac{1}{\bar U_s(\beta)} \partial_\beta \Gamma_u(\beta,\bar U) e^{- K(\beta,U,U_s(\beta))} + \int_U^{U_s(\beta)} \frac{e^{- K(\beta,U,V)}}{V} \, dV \\
&= \int_{U}^1 \frac{e^{-K(\beta,U,V)}}{V} \, dV, 
\end{split}
\end{equation}
\end{subequations}
and depend continuously on $\beta,f$.
\end{lemma}

\begin{proof}
We prove only the continuity of the solution $\delta P_u$, the analysis of $\delta P_s$ being completely similar. The equivalence of the formulas follows by observing that from Step 1 of Theorem \ref{manifoldbound} we have
\begin{equation*}
\frac{\partial_\beta \Gamma_u(\beta,U)}{U} = \int_0^U \frac{e^{K(V,U)}}{V} dV.
\end{equation*}

%

\medskip

\noindent{\it Step 1: continuity for $\beta < \beta^{**}(f)$.} First we observe that $\underline{\beta}$ can be chosen uniformly in a neighborhood of $f$. In this case we use the formula for every $0 < \bar U < U^*(f)$
\begin{equation*}
\begin{split}
\frac{\delta P_u(U)}{U} &= \frac{1}{\bar U} \partial_\beta \Gamma_u(\beta,\bar U) e^{K(\bar U,U)} - \int_{\bar U}^U \frac{1}{V} e^{K(V,U)} \, dV,
\end{split}
\end{equation*}
The dependence $\beta,f \mapsto \partial_\beta \Gamma(\beta,f,U)$ is already considered in Corollary \ref{Cor:Frechet_Gamma}, where it is shown that it is continuous. Using the uniform convergence of the optimal solution $P_\beta(f,U)$, it follows that also the second term is continuous.

\medskip

\noindent{\it Step 2: case $\beta > \beta^{**}(f)$.} In this case the solution formula is
\begin{equation*}
\begin{split}
\frac{\delta P_u(U)}{U} &= - \int_{U^*(f)}^U \frac{1}{V} e^{K(V,U)} \, dV.
\end{split}
\end{equation*}
The only difference w.r.t. the previous case is the analysis in a neighborhood of $U^*(f)$. By Proposition \ref{Mbt}, we can find a uniform neighborhood where the optimal solution $P$ is uniformly greater than $0$ for $U^*(f) + \underline{r} < U < U_s(f)$ and it coincides with $P^*(f)$ in $[U^*(f), U^*(f) + \underline{r}]$. Assuming w.l.o.g. $U > U^*(f) + \underline{r}$ and writing
\begin{equation*}
\begin{split}
\int_{U^*(f)}^U \frac{1}{V} e^{K(V,U)} \, dV = \bigg[ \int_{U^*(f)}^{U^*(f) + \underline{r}} + \int_{U^*(f) + \underline{r}}^U \bigg] \frac{1}{V} e^{K(V,U)} \, dV = \ln \bigg( \frac{U^* + \underline{r}}{U^*} \bigg) + \int_{U^*(f) + \underline{r}}^U \frac{1}{V} e^{K(V,U)} \, dV,
\end{split}
\end{equation*}
one sees that $\delta P_u$ is continuous w.r.t. $\beta,f$. \\ 

\medskip

\noindent{\it Step 3: $\beta \to \beta(f)$.} Assume now that $\beta$ is sufficiently close to $\beta^*(f)$: this means that $U^*(f) < U_u(\beta,f) < U^*(f) + \underline{r}$. Proposition \ref{Prop:Lipechit_U_u} gives that
\begin{equation*}
\partial_\beta \Gamma(\beta,U_u(\beta)) \underset{\beta \nearrow \beta^{**}}{\to} 0,
\end{equation*}
so that \eqref{Equa:sol_der_betaf1} is continuous.
\end{proof}

\begin{theorem}
\label{Theo:derivative_E}
The effort function $\beta \mapsto E(\beta)$ is of class $C^1((\beta^*,+\infty))$, and its first derivative is equal to the following expression:
\begin{equation}
\label{Equa:explit_deri}
\frac{d}{d\beta} E(\beta) = \begin{cases}
\frac{\partial P_s(U,\beta) - \partial P_u(U, \beta)}{U} & U \in \clos(\gls{OcalPopt}), \\
\frac{\partial P_s(U,\beta) - \partial P_u(U,\beta)}{U} e^{K(\beta,U,U^+)} & U \in (U^-,U^+) \in \gls{IcalPopt} \ \text{connected component}.
\end{cases} 
\end{equation}
Moreover, if $f_n \to f$ in $C^2([0,1])$, $f_n,f$ admissible sources, then $E_n(\beta) = E(\beta,f_n)$ converges to $E(\beta)$ in $C^1_\loc(\beta^*(f),+\infty)$.
\end{theorem}

\begin{proof}
The idea of the proof is to show that the formula \eqref{Equa:explit_deri} holds in the case $f$ polynomial, and then pass to the limit by means of the convergence of $\partial_\beta E(\beta,f)$. Using \eqref{Equa:kernel_canc} and \eqref{Equa:sol_der_betaf} we get for $U \in \overline{\mathcal{O}(\beta)}$,
\begin{equation}
\label{sum}
\begin{split}
\frac{\partial P_s(U, \beta) - \partial P_u(U, \beta)}{U} &= \frac{1}{U_s(\beta)} \partial_\beta \Gamma_s(\beta,U_s(\beta)) - \frac{1}{U_u(\beta)} \partial_\beta \Gamma_u(\beta,U_u(\beta)) \\
& \quad + \int_{U_u(\beta)}^{U} \frac{e^{K(\beta,W,U_u(\beta))}}{W} \, dW + \int_{U}^{U_s(\beta)} \frac{e^{-K(\beta,U,V)}}{V} dV \\
&= \int_0^1 \frac{e^{K(\beta,W,U)}}{W} \, dW.
\end{split}
\end{equation}
If $U \in (U^-,U^+)$, then
\begin{equation*}
\int_{0}^{1} \frac{e^{K(\beta,V,U^+}}{V} dV = \bigg( \int_{0}^{U} \frac{e^{K(\beta,V,U)}}{V} dV \bigg) e^{K(U,U^+)} + \bigg( \int_U^1 \frac{e^{K(\beta,V,U)}}{V} dV \bigg) e^{K(\beta,U,U^+)},
\end{equation*}
which is the second expression of \eqref{Equa:explit_deri}. Hence we need to prove the formula for just one point $U \in \gls{OcalPopt}$, where $F(U) + \beta > 0$ and $P_\beta(U) = P^*(U)$.

\medskip

\noindent{\it Derivative for $f$ polynomial.} We have shown in Proposition \ref{Prop:E_Lipschitz_beta} that $\beta \mapsto E(\beta)$ is Lipschitz and by Proposition \ref{Prop:reguls_Ibeta} the region where the derivative where the derivative may not exist has Lebesgue measure $0$ in the set
\begin{equation*}
\bigg\{ (\beta,a_1,\dots,a_k), \ f = \sum_i a_i U^i \ \text{satisfies Assumption \eqref{H1}}, \ \beta > \beta^*(f) \bigg\}.
\end{equation*}
Being the functions $\delta P_u,\delta P_s$ continuous w.r.t. $\beta,f$, it is enough to prove that the derivative formula holds a.e.

In the set given by Proposition \ref{Prop:reguls_Ibeta}, we can write
$$
\mathcal{I}(\beta) = (U_u(\beta),\hat U_u(\beta)) \cup \bigcup_{i=1}^N (U_i^-(\beta), U_i^+(\beta)) \cup (\hat U_s(\beta),U_s(\beta))
$$
with all the extremals of the intervals depending smoothly on $\beta$, so that 
\begin{align}
\label{Equa:E0_Ei}
E(\beta) = E_0(\beta) - \sum_{i}^N E_i(\beta),
\end{align}
with \newglossaryentry{E0cost}{name=\ensuremath{E_0(\beta)},description={cost if the optimal solution is allowed to be on $P^*$}} \newglossaryentry{Eicost}{name=\ensuremath{E_i(\beta)},description={cost correction in the connected components of $\mathcal I(\beta)$}}
\begin{equation*}
\gls{E0cost} := \int_{\hat U_u(\beta)}^{\hat U_s(\beta)} \frac{F(U)+\beta}{U}\,dU, \quad
\gls{Eicost} := \int_{U_i^-(\beta)}^{U_i^+(\beta)} \frac{F(U)+\beta}{U}\,dU, \ i=1,\dots,N.
\end{equation*}

The derivative of the term $E_0$ is
\begin{align}
\label{der0}
\partial_\beta E_0(\beta) = \frac{F(\hat U_u(\beta)) + \beta}{\hat U_u(\beta)} \frac{d \hat U_u(\beta)}{d\beta} - \frac{F(\hat U_s(\beta)) + \beta}{\hat U_s(\beta)} \frac{d\hat U_s(\beta)}{d\beta} + \int_{\hat U_u(\beta)}^{\hat U_u(\beta)} \frac{1}{U}\, dU.
\end{align}
First, we derive with respect to $\beta$ the identities
\begin{equation*}
\Gamma_u(\beta,\hat U_u(\beta)) = P^*(\hat U_u(\beta)) \quad \text{and} \quad \Gamma_s(\beta,\hat U_s(\beta)) = P^*(\hat U_s(\beta)),
\end{equation*}
and deduce
\begin{equation*}
\frac{\partial \Gamma_u(\hat U_u,\beta)}{\partial \beta} + \bigg( - \frac{f(\hat U_u)}{\Gamma_u(\hat U_u,\beta)} - \beta \bigg) \frac{d\hat U_u(\beta)}{d\beta} = \frac{dP^*(\hat U_u(\beta))}{dU} \frac{d\hat U_u(\beta)}{d\beta},
\end{equation*}
\begin{equation}
\label{Equa:Gammaus_Fbeta}
\frac{F(\hat U_u(\beta)) + \beta}{\hat U_u(\beta)} \frac{d\hat U_u(\beta)}{d\beta} = \frac{1}{\hat U_u(\beta)} \frac{\partial \Gamma_u}{\partial \beta}(\beta,\hat U_u(\beta)). 
\end{equation}
The same computation for $\Gamma_s(\beta,\hat U_s(\beta))$ gives
\begin{equation*}
\frac{F(\hat U_s(\beta)) + \beta}{\hat U_s(\beta)} \frac{d\hat U_s(\beta)}{d\beta} = \frac{1}{\hat U_s(\beta)} \frac{\partial \Gamma_s}{\partial \beta}(\beta,\hat U_s(\beta)). 
\end{equation*}

Combining the previous identities with the properties of $K$ and with \eqref{sum} and \eqref{der0}, we obtain the expression
\begin{equation}
\label{e0p}
\begin{split}
\partial_\beta E_0(\beta) &= \frac{\partial P_s(U, \beta) - \partial P_u(U, \beta)}{U} - \int_{\hat U_u(\beta)}^{\hat U_s(\beta)} \frac{1}{U} \left( 1 - e^{-K(\beta,U_0^+(\beta), U)} \right) \,dU \\
&= \frac{\partial P_s(U, \beta) - \partial P_u(U, \beta)}{U} - \sum_{i=1}^N \int_{U_i^-(\beta)}^{U_i^+(\beta)} \frac{1}{U} \left(1-e^{-K(\beta,U_i^-(\beta), U)}\right) \,dU.
\end{split}
\end{equation}

It remains to compute $E_i'(\beta)$, $i=1,...,N$. To such end, denote by $P_i=P_i(U, \beta)$ the solution to
\begin{equation}
\label{pi}
P_i'(U,\beta) + \frac{f(U)}{P_i(U, \beta)} + \beta = 0, \quad P_i(U_i^-(\beta), \beta)=P^*(U_i^-(\beta)).
\end{equation}
An integration by parts, combined with the endpoint conditions $P_i(U^\pm(\beta)) = P^*(U^\pm(\beta))$ as in Proposition \ref{improv} yields
\begin{equation*}
E_i(\beta) = \int_{U_i^-(\beta)}^{U_i^+(\beta)} \int_{P^*(U)}^{P_i(U, \beta)} \frac{Uf(U)-P^2}{U^2P^2} \,dU\,dP,
\end{equation*}
and, thus, being $\frac{Uf(U)-P^2}{U^2P^2} = 0$ at the extremals,
\begin{align}
\label{Equa:E_i'_comp}
E_i'(\beta) = \int_{U_i^-(\beta)}^{U_i^+(\beta)} \frac{\partial P_i}{\partial \beta}(U, \beta) \frac{Uf(U)-P_i^2(U, \beta)}{P_i^2(U, \beta) U^2} \, dU.
\end{align}

Equation \eqref{pi} can be derived with respect to the parameter $\beta$ to obtain
\begin{equation*}
\frac{d}{dU} \left(\frac{\partial P_i}{\partial \beta}(U, \beta)\right) - \frac{f(U)}{P_i^2(U, \beta)} \frac{\partial P_i(U, \beta)}{\partial \beta} + 1 = 0,
\end{equation*}
and, thus,
\begin{equation}
\label{dpi}
\frac{d}{dU} \left(\frac{1}{U}\frac{\partial P_i}{\partial \beta}(U, \beta)\right) = \frac{Uf(U)-P_i^2(U, \beta)}{U P_i^2(U, \beta)} \frac{1}{U}\frac{\partial P_i(U, \beta)}{\partial \beta} - \frac{1}{U}.
\end{equation}
Thus, we get the new expression
\begin{equation*}
\begin{split}
E_i'(\beta) &= \int _{U_i^-(\beta)}^{U_i^+(\beta)} \frac{d}{dU}\left(\frac{1}{U}\frac{\partial P_i}{\partial \beta}(U, \beta)\right) +\frac{1}{U} \, dU \\
&= \frac{1}{U_i^+}\frac{\partial P_i}{\partial \beta}(U_i^+, \beta)-\frac{1}{U_i^-}\frac{\partial P_i}{\partial \beta}(U_i^-, \beta) +\int_{U_i^-(\beta)}
^{U_i^+(\beta)} \frac{1}{U}\, dU.
\end{split}
\end{equation*}
Moreover, the solution to equation (\ref{dpi}) can be explicitly computed via the kernel $K$ as defined in \eqref{k}:
\begin{equation*}
\frac{1}{U} \frac{\partial P_i}{\partial \beta}(U, \beta) = \left( \frac{1}{U_i^-(\beta)} \frac{\partial P_i}{\partial \beta}(U_i^-(\beta), \beta) - \int_{U_i^-(\beta)}^{U} \frac{1}{W} e^{-K(\beta, U_i^-(\beta), W)}\, dW\right)e^{K(\beta,U_i^-(\beta),U)}.
\end{equation*}
In particular, once again by the optimality condition of Lemma \ref{Lem:restr_beta_i},
\begin{equation*}
\frac{1}{U_i^+(\beta)} \frac{\partial P_i}{\partial \beta}(U_i^+(\beta), \beta) = \frac{1}{U_i^-(\beta)}\frac{\partial P_i}{\partial \beta} (U_i^-(\beta), \beta) - \int_{U_i^-(\beta)}^{U_i^+(\beta)} \frac{1}{U} e^{-K(\beta, U_i^-(\beta), U)}\, dU.
\end{equation*}
Combining all the previous identity, we are left with the expression
\begin{align*}
E_i'(\beta) = \int_{U_i^-(\beta)}^{U_i^+(\beta)} \frac{1}{U} \left(1-e^{-K(\beta, U_i^-(\beta), U)}\right)\, dU, \quad \forall i=1,...,N,
\end{align*}
which, together with \eqref{e0p}, complete the proof.
%
%

\medskip

\noindent{\it General case.} Let $f_n$ polynomial converging in $C^2$ to the source $f$, and $\beta_1,\beta_2 \in (\beta^*(f),+\infty)$. By the continuity of the manifolds $\Gamma_u(\beta,f),\Gamma_s(\beta,f)$, we can assume w.l.o.g. that $\beta^*(f_n) < \beta_1 < \beta_2 < +\infty$. We can also take a point $U \in \gls{OcalPopt}$ such that
$$
\inf_n \{U_u(\beta_2,f)\} < \inf_n \{U_u(\beta_1,f_n)\} < U < \sup_n \{U_s(\beta_1,f)\} < \sup_n \{U_s(\beta_2,f)\},
$$
again by the continuity of the points $f \mapsto U_u(\beta,f),U_s(\beta,f)$ and the monotonicity w.r.t. $\beta$. Since $U \in \gls{OcalPopt}$, then definitely $U \in \gls{OcalPopt}_n$: indeed, $F_n(U) + \beta > C > 0$ definitely and then if $U \notin \gls{OcalPopt}_n$ there is a left or right neighborhood of $U$ where the left or right derivative $\frac{dP_n}{dU} \not= \frac{dP^*}{dU}$, and thus also in the limit $\frac{dP}{dU} \not= \frac{dP^*}{dU}$, yielding a contradiction.

The formula for the derivative in the polynomial case gives
\begin{equation*}
E(\beta_2,f_n) - E(\beta_1,f_n) = \int_{\beta_1}^{\beta_2} \partial_\beta E(\beta,f_n) d\beta = \int_{\beta_1}^{\beta_2} \frac{\delta P_s(\beta,f_n,U) - \delta P_u(\beta,f_n,U)}{U} d\beta.
\end{equation*}
The Lipschitz continuity of $\beta,f \mapsto \delta P_u(\beta,f,U), \delta P_s(\beta,f,U)$ and of $E(\beta,f)$ implies that we can pass to the limit as $f_n \to f$ obtaining
\begin{equation*}
E(\beta_2,f) - E(\beta_1,f) = \int_{\beta_1}^{\beta_2} \frac{\delta P_s(\beta,f,U) - \delta P_u(\beta,f,U)}{U} d\beta.
\end{equation*}
Being the integrand a continuous function of $\beta$, we conclude that the statement holds.
%
%
\end{proof}

\subsection{Second derivative and convexity}
\label{Ss:second_der_convex}

In the next theorem we study the second derivative of $E$ w.r.t. $\beta$: we will see that for $\beta \nearrow \beta^{**}$ it may happen that the second derivative diverges, and in any case in general that $\partial_\beta E(\beta)$ is only locally Lipschitz in $\{(\beta^*,+\infty) \setminus \beta^{**}\}$.

\begin{theorem}[Second derivative]
\label{2der}
The cost function $\beta \mapsto E(\beta)$ belongs to $C^{1,1}_\loc$ in the open set $(\beta^*,+\infty) \setminus \{\beta^{**}\}$. Moreover, for $\beta,f$ as in Proposition \ref{Prop:reguls_Ibeta}, the function $E(\beta)$ admits a second derivative whose explicit expression is given by
\begin{equation}
\label{Equa:second_der_E_beta}
\begin{split}
\partial^2_\beta E(\beta) &= \frac{1}{\hat U_s(\beta)} \frac{\partial^2 \Gamma_s}{\partial \beta^2}(\beta,\hat U_s(\beta)) - \frac{1}{\hat U_u(\beta)}\frac{\partial ^2 \Gamma_u}{\partial \beta^2}(\beta,\hat U_u(\beta)) \\
& \quad + \sum_{i=1}^N \int_{U_i^-(\beta)}^{U_i^+(\beta)} \frac{2f(W)}{WP_i^3(W, \beta)} \left(\frac{\partial P_i}{\partial \beta}(W, \beta)\right)^2\, e^{-K(\beta, U_i^-(\beta), W)} dW,
\end{split}
\end{equation}
being $P_i$ the solution to \eqref{functeq} in the $i$-th connected component of \gls{IcalPopt}.
\end{theorem}

In particular, away from $\beta^{**}$, the Lipschitz regularity of $\partial_\beta E$ is a consequence of the fact that some connected component of \gls{IcalPopt} are splitting: an explicit example is in Section \ref{Ss:exampl_C11}.

\begin{proof}
%
%
%
\noindent {\it Step 1: $\beta,f$ as in Proposition \ref{Prop:reguls_Ibeta}.} In this case $\partial \gls{IcalPopt}$ is finite and depends smoothly on $\beta$: 
\begin{equation*}
\mathcal{I}(\beta) = (0,\hat U_u(\beta)) \cup (\hat U_s(\beta),1) \cup \bigcup_{i=1}^N (U_i^-(\beta), U_i^+(\beta))
\end{equation*}
and denote by $P_i = P_i(U,\beta)$ the solution to the non controlled equation \eqref{functeq} satisfying $P_i(U_i^\pm(\beta)) = P^*(U^\pm(\beta))$. Recall also \eqref{Equa:E0_Ei}:
\begin{equation*}
E(\beta) = \gls{E0cost} - \sum_{i=1}^N \gls{Eicost}.
\end{equation*}
From \eqref{Equa:E_i'_comp} we have
\begin{equation*}
\partial_\beta E_i(\beta) = \int_{U_i^-(\beta)}^{U_i^+(\beta)} \frac{\partial P_i}{\partial \beta}(U, \beta) \frac{Uf(U)-P_i^2(U, \beta)}{P_i^2(U, \beta) U^2} \, dU 
\end{equation*}
and thus differentiating
\begin{equation*}
\begin{split}
\partial^2_\beta E_i(\beta) &= \int_{U_i^-(\beta)}^{U_i^+(\beta)} \frac{\partial}{\partial \beta} \bigg[ \frac{\partial P_i}{\partial \beta}(U, \beta) \frac{Uf(U)-P_i^2(U, \beta)}{P_i^2(U, \beta) U^2} \bigg] dU \\
&= \int_{U_i^-(\beta)}^{U_i^+(\beta)} \frac{d}{dU} \left( \frac{1}{U}\frac{\partial^2 P_i}{\partial \beta^2}(U,\beta) \right)\, dU = \frac{1}{U^+_i} \frac{\partial^2 P_i}{\partial \beta^2}(U^+_i,\beta) - \frac{1}{U^-_i} \frac{\partial^2 P_i}{\partial \beta^2}(U^-_i,\beta),
\end{split}
\end{equation*}
where we have used the ODE for the second derivative $\partial_\beta^2 P_i$
\begin{equation}
\label{d2pi}
\begin{split}
\frac{d}{dU} \bigg( \frac{\partial^2_\beta P_i}{U} \bigg) &= \frac{\partial}{\partial \beta} \bigg[ \frac{\partial P_i}{\partial \beta}(U, \beta) \frac{Uf(U)-P_i^2(U, \beta)}{P_i^2(U, \beta) U^2} \bigg] \\
&= \frac{Uf(U)-P_i^2(U, \beta)}{UP_i^2(U, \beta)}\frac{1}{U}\frac{\partial^2 P_i}{\partial \beta^2}(U, \beta) - \frac{2f(U)}{UP_i^3(U,\beta)} \left( \frac{\partial P_i}{\partial \beta}(U,\beta) \right)^2,
\end{split}
\end{equation}
obtained by differentiating \eqref{dpi}.

The solution to \eqref{d2pi} can be written by means of the kernel $K(\beta,U_1,U_2)$ (defined in Equation \eqref{k}) as
\begin{align*}
\frac{1}{U} \frac{\partial^2 P_i}{\partial \beta^2}(U,\beta) &= \frac{1}{U_i^-(\beta)} \frac{\partial^2 P_i}{\partial \beta^2}(U_i^-(\beta),\beta) e^{K(\beta, U_i^-(\beta), U)} \\
& \quad - \int_{U_i^-(\beta)}^U \frac{2f(W)}{WP_i^3(W,\beta)} \left( \frac{\partial P_i}{\partial \beta}(W, \beta) \right)^2\, e^{K(\beta, U_i^-(\beta),U) - K(\beta,U_i^-(\beta),W)} \, dW.
\end{align*}
In particular, since $K(\beta, U_i^-(\beta), U_i^+(\beta))=0$,
\begin{equation*}
\begin{split}
\partial^2_\beta E_i(\beta) &= \frac{1}{U^+_i} \frac{\partial^2 P_i}{\partial \beta^2}(U^+_i,\beta) - \frac{1}{U^-_i} \frac{\partial^2 P_i}{\partial \beta^2}(U^-_i,\beta) \\
&= - \int_{U_i^-(\beta)}^{U_i^+(\beta)} \frac{2f(W)}{WP_i^3(W, \beta)} \left(\frac{\partial P_i}{\partial \beta}(W, \beta)\right)^2\, e^{-K(\beta, U_i^-(\beta), W)}dW \le 0.
\end{split}
\end{equation*}

Next we compute the second variation of $\Gamma_u(\beta,\hat U_u(\beta))$, $\Gamma_s(\beta,\hat U_s(\beta))$: using the ODE for perturbation
\begin{equation*}
\frac{d}{dU} \bigg( \frac{\partial_\beta \Gamma_u}{U} \bigg) = \frac{U f(U) - \Gamma_u}{U \Gamma_u^2} \frac{\partial_\beta \Gamma_u}{U_u(\beta)} -\frac{1}{U}
\end{equation*}
we conclude that
\begin{equation*}
\begin{split}
\frac{d}{dU} \bigg( \frac{1}{U}\frac{\partial \Gamma_u}{\partial \beta} \bigg)(\beta,\hat U_u(\beta)) &= \frac{(P^*(\hat U_u(\beta)))^2 - \Gamma_u(\beta,\hat U_u(\beta))^2}{\hat U_u(\beta) \Gamma_u(\beta,\hat U_u(\beta))^2} \frac{1}{\hat U_u(\beta)} \frac{\partial \Gamma_u}{\partial \beta}(\beta,\hat U_u(\beta)) - \frac{1}{\hat U_u(\beta)} \\
&= - \frac{1}{\hat U_u(\beta)}.
\end{split}
\end{equation*}
With a similar computation,
\begin{equation*}
\frac{d}{dU} \left( \frac{1}{\hat U_s(\beta)} \frac{\partial \Gamma_s}{\partial \beta} \right)(\beta,\hat U_s(\beta)) = - \frac{1}{\hat U_s(\beta)}.
\end{equation*}
We can then compute from \eqref{der0} and \eqref{Equa:Gammaus_Fbeta}
\begin{align*}
\partial^2_\beta E_0(\beta) 
&= - \frac{d}{d\beta} \left( \frac{1}{\hat U_u(\beta)} \frac{\partial \Gamma_u}{\partial \beta} (\beta,\hat U_u(\beta)) \right) + \frac{d}{d\beta} \left( \frac{1}{\hat U_s(\beta)} \frac{\partial \Gamma_s}{\partial \beta} (\beta,\hat U_s(\beta)) \right) + \frac{\hat U_s'(\beta)}{\hat U_s(\beta)} - \frac{\hat U_u'(\beta)}{\hat U_u(\beta)} \\
&= - \frac{d}{dU} \left( \frac{1}{U} \frac{\partial \Gamma_u}{\partial \beta} \right)(\beta,\hat U_u(\beta)) \hat U_u'(\beta) - \frac{1}{\hat U_u(\beta)} \frac{\partial^2 \Gamma_u}{\partial \beta^2}(\beta,\hat U_u(\beta)) \\
& \quad + \frac{d}{dU} \left( \frac{1}{U} \frac{\partial \Gamma_s}{\partial \beta} \right)(\beta,\hat U_s(\beta)) \hat U_s'(\beta) + \frac{1}{\hat U_s(\beta)} \frac{\partial^2 \Gamma_s}{\partial \beta^2}(\beta,\hat U_s(\beta)) \\
& \quad + \frac{\hat U_s'(\beta)}{\hat U_s(\beta)} - \frac{\hat U_u'(\beta)}{\hat U_u(\beta)} \\
&= \frac{1}{\hat U_s(\beta)} \frac{\partial^2 \Gamma_s}{\partial \beta^2}(\beta,\hat U_s(\beta)) - \frac{1}{\hat U_u(\beta)}\frac{\partial ^2 \Gamma_u}{\partial \beta^2}(\beta,\hat U_u(\beta)).
\end{align*}

We conclude that for $\beta,f$ the second derivative $\partial^2_\beta E(\beta)$ has the explicit expression
\begin{equation*}
\begin{split}
\partial^2_\beta E(\beta) &= \frac{1}{\hat U_s(\beta)} \frac{\partial^2 \Gamma_s}{\partial \beta^2}(\beta,\hat U_s(\beta)) - \frac{1}{\hat U_u(\beta)}\frac{\partial ^2 \Gamma_u}{\partial \beta^2}(\beta,\hat U_u(\beta)) \\
& \quad + \sum_{i=1}^N \int_{U_i^-(\beta)}^{U_i^+(\beta)} \frac{2f(W)}{WP_i^3(W, \beta)} \left(\frac{\partial P_i}{\partial \beta}(W, \beta)\right)^2\, e^{-K(\beta, U_i^-(\beta), W)} dW.
\end{split}
\end{equation*}

\medskip

\noindent{\it Step 2: uniform bound on the second derivative and the polynomial case.} We observe that the above expressions are uniformly bounded when $\beta$ is away from $\beta^{**}$, because of the estimates of Theorem \ref{Theo:second_der_Gammaus} and that $P_i > 0$ uniformly by Lemma \ref{lowerbound}. For $f \in \gls{Tfrak}$ polynomial, we know that there is a finite number of speeds $\beta$ such that $\partial^2_\beta E(\beta,f)$ does not exists: the above estimate on the second derivative thus implies that for $\beta \mapsto \partial_\beta E(\beta,f)$ is Lipschitz for $f \in \gls{Tfrak}$ polynomial. Being \gls{Tfrak} of full measure in the set of polynomials, by approximation we deduce that $\partial_\beta E(\beta,f)$ is Lipschitz for all $f$ polynomial, and the passing to the limit $\partial_\beta E_\beta(\beta,f)$ is also Lipschitz for generic $f$ and $\beta \not= \beta^{**}$.
\end{proof}

We conclude this section with the following corollary of Proposition \ref{Prop:blowupsecond}, whose proof is just passing to the limit \eqref{Equa:second_der_E_beta}: the first part of the statement is a direct consequence of Formula \eqref{Equa:second_der_E_beta}, since for $\beta > \beta^{**}$ the negative term $- \frac{\partial^2_\beta \Gamma_u}{U_u}$ is not present.

\begin{corollary}
\label{Cor:blowup_E''}
The function \gls{Ebeta} is convex for $\beta > \beta^{**}$. Under the assumption \eqref{Equa:cond_blow}, it holds
\begin{equation*}
\lim_{\beta \nearrow \beta^{**}} \partial^2_\beta E(\beta) = - \infty.
\end{equation*}
In particular $E(\beta)$ is not convex in the interval $(\beta^*,\beta^{**})$.  
\end{corollary}

\section{Examples where $E$ is neither convex nor $C^2$.}
\label{S:counter_convex_C11}

In this section we show that the results obtained are sharp: the cost $E(\beta)$ is in general not convex nor smooth. The first example is an explicit source $f$ for which we can compute the second derivative near $\beta = 0$, and if $f$ is highly positively predominant, then $\partial^2_\beta E(\beta=0) < 0$.

\subsection{$E(\beta)$ not convex}
\label{Sss:E_not_convex}

As a first step, we compute the explicit formulas for $\Gamma_u,\Gamma_s$ and their derivatives w.r.t. $\beta$ in the case $\beta = 0$: this case is easier because the ODE is now Hamiltonian.

\subsubsection{Formulas for $\Gamma_u,\Gamma_s$ in the case $\beta = 0$}
\label{Ss:formulas_beta0}

Fix $\beta=0$ and define the primitive function
\begin{equation*}
G(U) := - \int_0^U f(V) dV.
\end{equation*}
Note that, clearly, $G(U)> 0$ for all $U \in [0, U^*]$. Moreover, by assumption
\begin{align*}
G(1) = - \int_0^1 f(V) dV < 0,
\end{align*}
we can find $\bar U \in (U^*,1)$ such that $G(\bar U)=0$ by the intermediate value theorem. Finally, since
\begin{equation*}
G'=-f \le 0 \quad \text{in} \ [U^*,1],
\end{equation*}
we also deduce
\begin{align*}
G(U)\ge G(1) \quad \forall U \in [U^*,1].
\end{align*}
In this case, Equation \eqref{functeq} is an Hamiltonian system and can be rewritten as
\begin{equation*}
\frac{d}{dU} \left(\frac{P^2(U)}{2}\right) = -f(U) = G'(U), 
\end{equation*}
which implies, by the previous observations,
\begin{equation}
\label{unstable0}
\Gamma_u(0,U) = \sqrt{2G(U)} \quad \forall U \in [0,\bar U],
\end{equation}
\begin{equation}
\label{stable0}
\Gamma_s(0,U) = \sqrt{2 \left( G(U)-G(1) \right)} \quad \forall U\in [0,1].
\end{equation}
Moreover, for $\beta=0$, 
\begin{align}
\label{delta}
\frac{\partial}{\partial U} \left( \Gamma_u(\beta, U) \frac{\partial \Gamma_u}{\partial \beta}(\beta, U) \right) = - \Gamma_u(\beta,U),
\end{align}
and the same for $\Gamma_s$. By the boundary conditions
\begin{align*}
\Gamma_u(0,0) = \Gamma_s(1,0)=0,
\end{align*}
we obtain the explicit formulas
\begin{equation}
\label{eqmeno}
\frac{\partial \Gamma_u}{\partial \beta}(0, U)=-\frac{1}{\Gamma_u(0, U)}\int_0^U \Gamma_u(0, U') \, dU'= -\frac{1}{\sqrt{G(U)} }\int_0^U \sqrt{G(U')} dU',
\end{equation}
\begin{equation*}
\frac{\partial \Gamma_s }{\partial \beta}(0, U) = -\frac{1}{\Gamma_s(0,U)} \int_1^U \Gamma_s(0, U') \, dU'= \frac{1}{\sqrt{G(U) - G(1)}} \int_U^1 \sqrt{G(U') - G(1)} dU'.
\end{equation*}
Similarly, the equation for $\partial^2_\beta \Gamma_u$ is for $\beta = 0$ 
\begin{equation}
\label{deltas}
\frac{\partial}{\partial U}
\left(\Gamma_u(0, U) \frac{\partial ^2 \Gamma_u }{\partial \beta^2} (0, U) \right) = - 2 \left[ \frac{\partial}{\partial U} \left(\frac{\partial \Gamma_u}{\partial \beta} \right)(0, U) + 1 \right] \frac{\partial \Gamma_u}{\partial \beta}(0, U),
\end{equation}
and the same for $\Gamma_s$. Thus,
\begin{equation}
\label{comp}
\frac{\partial ^2 \Gamma_u }{\partial \beta^2} (0, U)=\frac{1}{\Gamma_u(0, U)}   \left[-\left(\frac{\partial \Gamma_u}{\partial\beta}(0,U)\right)^2 -2 \int_0^U  \frac{\partial \Gamma_u}{\partial \beta}(U',0) \, dU' \right],
\end{equation}
\begin{equation*}
\frac{\partial ^2\Gamma_s}{\partial \beta^2}(U,0)=\frac{1}{\Gamma_s(0,U)}    \left[-\left(\frac{\partial \Gamma_s}{\partial\beta}(U,0)\right)^2 +2 \int_U^1  \frac{\partial \Gamma_s}{\partial \beta}(U',0) \, dU' \right].
\end{equation*}
We remark that for all the above computation to be meaningful, we only need $f$ to be continuous: by approximation the same non-convexity result holds for any smooth source sufficiently close in $C^0$ to the $f$ we construct below.

\subsubsection{Definition of the source $f$}
\label{Sss:explicitc_f}

Let $K \ge 1$ and $U^* \in (0, \frac{1}{2})$ and define the function
\begin{align}\label{Eq:nonconvsource}
f(U)=f_K(U) := \begin{cases}
U(U-U*) & 0\le U \le U^*, \\
K(U-U^*) & U^* \le U \le  \frac{1}{2}, \\
K(1-2U^*)(1-U) & \frac{1}{2}\le U \le 1.
\end{cases}
\end{align}
We compute
\begin{equation*}
G(U) = \begin{cases}
U^*\frac{U^2}{2}-\frac{U^3}{3} & 0\le U \le U^* \\
G(U^*)-\frac{K}{2}\left(U-U^*\right)^2 & U^* \le U \le  \frac{1}{2} \\
G(\frac{1}{2}) +\frac{K}{2}(1-2U^*) [ (U-1)^2 - \frac{1}{4} ] & \frac{1}{2}\le U \le 1,
\end{cases}
\end{equation*}
with $G(U^*)=\frac{1}{6}(U^*)^3$ and $G(\frac{1}{2})=G(U^*) -\frac{K}{2}(\frac{1}{2}-U^*)^2.$ Note, in particular, that condition $U^* <\frac{1}{2}$ ensures
\begin{equation*}
G(1) =-\int_0^1 f(U) \,dU =\frac{1}{6}\left(U^*\right)^3-\frac{K}{2} (1-U^*) \bigg( \frac{1}{2} - U^* \bigg) < 0 \quad \text{ for } K \gg 1,
\end{equation*}
that is $f$ is positively predominant.
\subsubsection{Estimates on the manifold $\Gamma_u$}
\label{Sss:couter_esimt_Gammau}

For $0\le U\le U^*$, by (\ref{eqmeno}), we can directly compute
\begin{align*}
\frac{\partial\Gamma_u}{\partial\beta}(0,U) &= -\frac{1}{U\sqrt{1-\frac{2U}{3U^*}}} \int_0^U U' \sqrt{1-\frac{2U'}{3U^*}} \, dU' \\
&= - \frac{(\frac{3U^*}{2})^2}{U \sqrt{1 - \frac{2U}{3U^*}}} \bigg[ \frac{2}{3} \bigg( 1 - \bigg( 1 - \frac{2U}{3U^*} \bigg)^\frac{3}{2} \bigg) - \frac{2}{5} \bigg( 1 - \bigg( 1 - \frac{2U}{3U^*} \bigg)^\frac{5}{2} \bigg) \bigg] \\
&= - \frac{\frac{2 U^*}{5}}{\sqrt{1-\frac{2U}{3U^*}} ( 1 + \sqrt{1-\frac{2U}{3U^*}} )} + \frac{U^*}{5} - \frac{2U}{5}.
\end{align*}
In particular,
\begin{equation*}
\frac{\partial \Gamma_u}{\partial\beta}(0,U^*) = - \bigg( \frac{3\sqrt{3}-2}{5} \bigg) U^*.
\end{equation*}
It follows that, for $0\le U \le U^*$,
\begin{align*}
\int_0^{U} \frac{\partial\Gamma_u}{\partial\beta}(0,U') \, dU' &= - \frac{2U^*}{5} \int_0^U (-3U^*) \frac{d}{dU'} \bigg[ \ln \bigg( 1 + \sqrt{1 - \frac{2U'}{3U^*}} \bigg) \bigg] \, dU' + \frac{U^*U}{5} - \frac{U^2}{5} \\
&= - \frac{6 (U^*)^2}{5} \ln \bigg( \frac{2}{1 + \sqrt{1 - \frac{2U}{3U^*}}} \bigg) + \frac{U^* U}{5} - \frac{U^2}{5},
\end{align*}
and, consequently,
\begin{equation*}
\int_0^{U^*} \frac{\partial \Gamma_u}{\partial\beta}(0, U') \, dU' = - \frac{6 (U^*)^2}{5} \log{(3-\sqrt{3})}.
\end{equation*}
By identity (\ref{comp}), we conclude
\begin{align*}
\frac{\partial ^2 \Gamma_u}{\partial \beta^2}(0,U^*) &= \frac{1}{\sqrt{2G(U^*)}} \bigg( - \bigg( \frac{3\sqrt{3} - 2}{5} \bigg)^2 (U^*)^2 + \frac{12}{5} \ln(3-\sqrt{3}) (U^*)^2 \bigg) \\
&= \frac{\sqrt{3U^*}}{5} \bigg( 12 \ln(3-\sqrt{3}) - \frac{31-12\sqrt{3}}{5} \bigg) > 0.
\end{align*}

Thanks to formulas (\ref{unstable0}) and (\ref{stable0}), we can explicitly compute the unique intersections with the critical curve $P^*$
\begin{equation*}
U_u(0) = \frac{3}{4} U^* + \sqrt{ \bigg( \frac{U^*}{4} \bigg)^2 + \frac{G(U^*)}{K}} \quad \text{and} \quad U_s(0) = \frac{1}{2}.
\end{equation*}
Note, in particular, that, by the inequality
\begin{equation*}
\sqrt{1+x}-1 \le \frac{x}{2} \quad \forall x\ge 0,
\end{equation*}
we obtain
\begin{equation}
\label{dist}
0 \le U_u(0) - U^* = - \frac{U^*}{4} + \sqrt{\bigg( \frac{U^*}{4} \bigg)^2 + \frac{G(U^*)}{K}} \le \frac{(U^*)^2}{3K}.
\end{equation}

For every $U \in [U^*, \bar U]$ we can write
\begin{equation*}
\Gamma_u(0,U) = \Gamma_u(0, U^*) - R_1(U),
\end{equation*}
with
\begin{equation*}
0 \le R_1(U):= \sqrt{2G(U^*)}- \sqrt{2G(U^*) -K(U-U^*)^2} \le \sqrt{K}(U-U^*).
\end{equation*}
In particular
\begin{equation*}
R_1(U_u) \leq \frac{(U^*)^2}{3 \sqrt{K}}.
\end{equation*}

Moreover, by \eqref{delta}, there holds
\begin{equation*}
\frac{\partial \Gamma_u}{\partial \beta}(0,U)=\frac{\partial \Gamma_u}{\partial \beta}(0,U^*)+ R_2(U)
\end{equation*}
with
\begin{align*}
R_2(U) := - (U - U^*) + \frac{1}{\Gamma_u(0, U^*)} \int_{U^*}^U R_1(U') \, dU' + \frac{\partial \Gamma_u}{\partial \beta}(0,U) \frac{R_1(U)}{\Gamma_u(0, U^*)},
\end{align*}
and, thus, for all $U^* \le U \le U_u(0) \le \bar U$ and $K$ sufficiently large, by \eqref{dist},
\begin{align*}
\left|R_2(U)\right|&\le (U-U^*)+ \frac{1}{2\Gamma_u(0, U^*)}\sqrt{K} (U-U^*)^2 + \frac{C}{\Gamma_u(0, U^*)}\sqrt{K}(U-U^*) \le \frac{C}{\sqrt{K}},
\end{align*}
For simplicity, throughout the remainder of this subsection, $C$ is a numeric constant whose value only depends on $U^*$ and may vary from time to time

Similarly, by (\ref{deltas}), we can write
\begin{align*}
\Gamma_u(0,U) \frac{\partial ^2 \Gamma_u}{\partial \beta^2}(0, U)=\Gamma_u(0,U^*)\frac{\partial ^2 \Gamma_u}{\partial \beta^2}(0,U^*)+ \frac{\partial \Gamma_u}{\partial \beta}(0, U^*)^2 -\frac{\partial \Gamma_u}{\partial \beta}(0,U)^2 -2\int_{U^*}^U \frac{\partial \Gamma_u}{\partial \beta} (0, U') \, dU',
\end{align*}
which implies
\begin{equation*}
\frac{\partial ^2 \Gamma_u}{\partial \beta^2}(0,U)= \frac{\partial ^2 \Gamma_u }{\partial \beta^2}(0, U^*)+R_3(U),
\end{equation*}
with
\begin{align*}
R_3(U) :&= \frac{1}{\Gamma_u(0,U^*)}\left[\left( \frac{\partial \Gamma_u}{\partial \beta}(0,U^*)\right)^2 -\left(\frac{\partial \Gamma_u}{\partial \beta}(0,U)\right)^2\right] \\
& \quad -\frac{2}{\Gamma_u(0, U)}\int_{U^*}^U \frac{\partial \Gamma_u}{\partial \beta}(0, U') \, dU'\\
& \quad - \frac{\partial ^2\Gamma_u}{\partial \beta^2}(0,U^*)\left(\frac{\Gamma_u(0, U^*)}{\Gamma_u(0,U)}-1 \right)
\end{align*}
and, therefore, also
\begin{equation*}
\left|R_3(U)\right| \le \frac{C}{\sqrt{K}}, \quad \forall U^*\le U \le U_u(0) \text{ and } K \gg 1.
\end{equation*}

In particular, we conclude that, for $K$ sufficiently large,
\begin{equation*}
\frac{\partial ^2 \Gamma_u }{\partial \beta^2}(0,U_u(0)) \ge \frac{1}{2}\frac{\partial ^2 \Gamma_u}{\partial \beta^2}(0,U^*) > 0.
\end{equation*}

\subsubsection{Estimates on the manifold $\Gamma_s$}
\label{Sss:estim_couteGammas}

Conversely, equations (\ref{delta}) and (\ref{deltas}) for $\Gamma_s$ can be directly integrated to find
\begin{equation*}
\Gamma_s(0,U) = \sqrt{K(1-2U^*)}(1-U), \quad \frac{\partial \Gamma_s}{\partial \beta}(0,U) = \frac{1-U}{2}, \quad \frac{\partial ^2\Gamma_s}{\partial \beta^2}(0,U)= \frac{1-U}{4\sqrt{K(1-2U^*)}}.
\end{equation*}

\subsubsection{Computing the second derivative of $E(\beta)$ for $\beta = 0$ and non-convexity of $E$}
\label{Sss:comput_second}

In particular,
\begin{equation*}
\frac{\partial^2\Gamma_s}{\partial \beta^2}(0, U_s(0))=\frac{1}{8\sqrt{K(1-2U^*)}} \to 0 \text{ as } K\to +\infty.
\end{equation*}
We therefore conclude that it is possible to pick $K$ sufficiently large so that
\begin{equation*}
\partial^2_\beta E(0) = \frac{1}{U_s(0)} \frac{\partial^2 \Gamma_s}{\partial \beta^2}(0,U_s(0)) - \frac{1}{U_u(0)}\frac{\partial^2 \Gamma_u}{\partial \beta^2}(0,U_u(0))<0.
\end{equation*}


\subsection{Non-convexity and anisotropic perimeter}
\label{Ss:anisotr_nonconvex}

In \cite{BressanChiri23} the following assumptions are made on the effort function $E(\beta)$:
\begin{enumerate}
\item $\beta \mapsto E(\beta)$ is convex,
\item $E(\beta) - \beta E'(\beta) \geq 0 \quad \forall \beta\ge0.$
\end{enumerate}
Because of the existence of the asymptote
\begin{equation*}
\ln \bigg( \frac{1}{U^*} \bigg) \beta + \int_{U^*}^1 \frac{2 \sqrt{f(U)}}{U^{\frac{3}{2}}} dU \quad \text{with} \ \int_{U^*}^1 \frac{2 \sqrt{f(U)}}{U^{\frac{3}{2}}} dU > 0
\end{equation*}
by Proposition \ref{Prop:E_Lipschitz_beta}, the first assumption implies the second. Indeed,  if we assume convexity a priori, in particular $\beta \mapsto E'(\beta)$ is increasing, so that the asymptotic behavior of $E$ implies
\begin{equation*}
E'(\beta)\le\ln \bigg( \frac{1}{U^*} \bigg) \quad \forall \beta \in \R.
\end{equation*}
In particular, the function $\beta \mapsto E(\beta)-\beta \ln \big( \frac{1}{U^*} \big)$ is decreasing, so that, for all $\beta \ge0$,
\begin{equation}
\label{Eq:aboveas}
E(\beta)-\beta E'(\beta) \ge E(\beta) -\ln \bigg( \frac{1}{U^*} \bigg) \beta \ge \lim_{\beta \to \infty} \bigg[ E(\beta) - \ln \bigg( \frac{1}{U^*} \bigg) \beta \bigg] = \int_{U^*}^1 \frac{2 \sqrt{f(U)}}{U^{3/2}} dU >0.
\end{equation}
Conversely, similar computations show here that for the cubic source $f(U) = U (U - U^*) (1 - U)$ there are $U^*$ such that it holds
\begin{equation*}
E(\beta = 0) < \int_{U^*}^1 \frac{2 \sqrt{f(U)}}{U^{\frac{3}{2}}} dU.
\end{equation*}
Hence $E(\beta)$ is not convex even for the cubic source. For this purpose, just consider $U^* = \frac{1}{2}$, so that $E(\beta = 0) = 0$, and observe that
\begin{equation*}
\int_{\frac{1}{2}}^1 \frac{2 f(U)}{U^{\frac{3}{2}}} dU > 0.
\end{equation*}
Therefore, by the Lipschitz dependence of Corollary \ref{Cor:Lipsc_f}, for $U^*$ close to $\frac{1}{2}$ and $\beta$ close to $0$, the effort $E(\beta)$ satisfies
\begin{equation*}
E(\beta) < \ln \bigg( \frac{1}{U^*} \bigg) \beta + \int_{U^*}^1 \frac{2 \sqrt{f(U)}}{U^{\frac{3}{2}}} dU,
\end{equation*}
which implies non-convexity (by contradicting \eqref{Eq:aboveas}).

\subsection{Example of $E$ with only $C^{1,1}$-regularity}
\label{Ss:exampl_C11}

In the previous analysis, a key part of the computation of the derivative is the fact that $U^\pm_i(\beta)$ depends smoothly on $\beta$: this is not the case when a splitting occurs in one of the intervals. We will show that, in this situation, the second derivative generally does not exist.

We consider $f \in \gls{Tfrak}$ and we assume for simplicity that for $\beta < \bar \beta$
\begin{equation*}
I(\beta) = (U_1^-(\beta),U_1^+(\beta)),
\end{equation*}
whereas for $\beta > \bar{\beta}$
\begin{equation*}
I(\beta)=(U_2^-(\beta), U_2^+(\beta)) \cup (U_3^-(\beta), U_3^+(\beta)).
\end{equation*}
This situation is in accordance with Lemma \ref{Lem:restr_beta_i}. We will compute the right and left limits of $\partial^2_\beta E$ at $\bar \beta$. For $i=1,2,3$, denote by $P_i=P_i(U, \beta)$ the solution to
\begin{equation*}
\begin{cases}
P_i'(\beta) + \frac{f(U)}{P_i} + \beta=0, \quad U \in (U_i^-(\beta), U_i^+(\beta))\\
P_i(U_i^-(\beta)) = P^*(U_i^-(\beta)).
\end{cases}
\end{equation*}

For $\beta < \bar \beta$, the equation of the perturbation is
\begin{equation*}
\frac{d}{dU} \delta P_1 = \frac{f(U)}{P^2} \delta P_1 - 1.
\end{equation*}
The initial condition $\delta P_1$ at some point $U \in (U^-_1,U^+_1)$ is computed in order to maintain
\begin{equation*}
\int_{U^-_1}^{U^+_1} \frac{U f(U) - P^2}{U P^2} dU = 0,
\end{equation*}
which becomes
\begin{equation}
\label{Equa:require_cc}
\int_{U^-_1}^{U^+_1} \frac{2 f(U) \delta P_1}{P^3} dU = 0.
\end{equation}
Note that this requires
\begin{equation}
\label{Equa:sing_deltaP}
\delta P_1(U_1^+) < 0 \quad \text{and} \quad \delta P_1(U^-_1) > 0,
\end{equation}
otherwise the solution would be strictly positive or negative in the interval, respectively.

For $\beta > \bar \beta$ the interval $(U^-_1,U^+_1)$ splits into two intervals and then the initial data for them is similarly computed: if $\delta P_2$, $\delta P_3$ are the solutions in the intervals $(U^-_2,U^+_2),(U^-_3,U^+_3)$, then
\begin{equation*}
\int_{U^-_2}^{U^+_2} \frac{2 f(U) \delta P_2}{P^3} dU = \int_{U^-_3}^{U^+_3} \frac{2 f(U) \delta P_3}{P^3} dU = 0.
\end{equation*}
In particular by \eqref{Equa:sing_deltaP} we have
\begin{equation}
\label{Equa:sign_deltaP_23}
\delta P_2(U_2^+) < 0 \quad \text{and} \quad \delta P_3(U^-_3) > 0.
\end{equation}

For $\beta = \bar \beta$, the optimal profile $P(U)$ is the same for both cases, the difference being in the initial data: in particular, we can write $\delta P_1$ in terms of $\delta P_2, \delta P_3$ and the solution of the homogeneous equation as
\begin{equation*}
\delta P_1(\bar \beta,U) = \begin{cases}
a e^{\int_{U_m}^U \frac{f(W)}{W^2} dW} + \delta P_2(\bar \eta,U) & U \in (U^-_1,U_m], \\
b e^{\int_{U_m}^U \frac{f(W)}{W^2} dW} + \delta P_3(\bar \beta,U) & U \in (U_m,U_1^+),
\end{cases}
\end{equation*}
with
\begin{equation*}
a + \delta P_2(U_m) = b + \delta P_3(U_m),
\end{equation*}
Imposing \eqref{Equa:require_cc} we obtain also
\begin{equation*}
a \int_{U_1^-}^{U_m} \frac{2 f(W)}{P^3} e^{\int_{U_m}^W \frac{f}{P^2} dV} dW + b \int_{U^m}^{U^+_1} \frac{2 f(W)}{P^3} e^{\int_{U_m}^W \frac{f}{P^2} dV} dW = 0.
\end{equation*}
We observe also that $\delta P_2(U_m) < 0$ and $\delta P_3(U_m) > 0$ by \eqref{Equa:sign_deltaP_23}, so that being the linear system above invertible we conclude that
\begin{equation*}
a = \frac{\int_{U^m}^{U^+_1} \frac{2 f(W)}{P^3} e^{\int_{U_m}^W \frac{f}{P^2} dV} dW}{\int_{U^-_1}^{U^+_1} \frac{2 f(W)}{P^3} e^{\int_{U_m}^W \frac{f}{P^2} dV} dW} \big( \delta P_3(U_m) - \delta P_2(U_m) \big) > 0,
\end{equation*}
\begin{equation*}
b = - \frac{\int_{U_1^-}^{U_m} \frac{2 f(W)}{P^3} e^{\int_{U_m}^W \frac{f}{P^2} dV} dW}{\int_{U^-_1}^{U^+_1} \frac{2 f(W)}{P^3} e^{\int_{U_m}^W \frac{f}{P^2} dV} dW} \big( \delta P_3(U_m) - \delta P_2(U_m) \big) < 0.
\end{equation*}

We now compute the second derivative of $E$ at $\bar \beta-$ using the above expression:
\begin{equation*}
\begin{split}
\partial^2_\beta E(\bar \beta) &= \int_{U_1^-(\bar \beta)}^{U_1^+(\bar \beta)} \frac{2f(W)}{W P(\bar \beta,W)^3} \bigg( \frac{\partial P_1}{\partial \beta}(W,\bar \beta) \bigg)^2 e^{-K(\bar \beta, U_1^-(\beta), W)} dW \\
&= \int_{U^-_1}^{U^+_1} \frac{2 f(W) e^{-K(\bar \beta, U_1^-, W)}}{W P(\bar \beta,W)^3} \Big( \big( a e^{\int_{U_m}^U \frac{f(W)}{W^2} dW} + \delta P_2 \big) \rest_{(U^-_1,U_m)} + \big( b e^{\int_{U_m}^U \frac{f(W)}{W^2} dW} + \delta P_3 \big) \rest_{(U_m,U^+_1)} \Big)^2 dW \\
&= \int_{U^-_1}^{U_m} \frac{2 f(W) e^{-K(\bar \beta, U_1^-, W)}}{W P(\bar \beta,W)^3} \Big( a e^{\int_{U_m}^U \frac{f(W)}{W^2} dW} + \delta P_2(W) \Big)^2 dW \\
& \quad + \int_{U_m}^{U^+_1} \frac{2 f(W) e^{-K(\bar \beta, U_1^-, W)}}{W P(\bar \beta,W)^3} \Big( b e^{\int_{U_m}^U \frac{f(W)}{W^2} dW} + \delta P_3(W) \Big)^2 dW.
\end{split}
\end{equation*}
Observing that from the necessary conditions for optimality
\begin{equation*}
K(\bar \beta, U_1^-, U_m) = K(\bar \beta, U_m, U^+_1) = 0,
\end{equation*}
we obtain the expression
\begin{equation*}
\begin{split}
\partial^2_\beta E(\bar \beta) &= \int_{U^-_1}^{U_m} \frac{2 f(W) e^{- K(\bar \beta,U_1^-,W)}}{W P(\bar \beta,W)^3} \delta P_2(W)^2 dW + 2 a \int_{U^-_1}^{U_m} \frac{2 f(W) e^{- \int_{U_m}^{W} \frac{V f - P^2}{V P^2} dV}}{W P(\bar \beta,W)^3} e^{\int_{U_m}^W \frac{f}{V^2} dV} \delta P_2(W) dW \\
& \quad + a^2 \int_{U^-_1}^{U_m} \frac{2 f(W) e^{- \int_{U_m}^{W} \frac{V f - P^2}{V P^2} dV}}{W P(\bar \beta,W)^3} e^{2 \int_{U_m}^W \frac{f}{V^2} dV} dW + \int_{U_m}^{U^+_1} \frac{2 f(W) e^{-K(\bar \beta,U_m,W)}}{W P(\bar \beta,W)^3} \delta P_3(W)^2 dW \\
& \quad + 2 \int_{U_m}^{U^+_1} \frac{2 f(W) e^{- \int_{U_m}^{W} \frac{V f - P^2}{V P^2} dV}}{W P(\bar \beta,W)^3} b e^{\int_{U_m}^W \frac{f}{V^2} dV} \delta P_3(W) dW + \int_{U_m}^{U^+_1} \frac{2 f(W) e^{- \int_{U_m}^{W} \frac{V f - P^2}{V P^2} dV}}{W P(\bar \beta,W)^3} b^2 e^{2 \int_{U_m}^W \frac{f}{V^2} dV} dW.
\end{split}
\end{equation*}
We compute now the various terms: using \eqref{Equa:KOIbeta}
\begin{equation*}
\begin{split}
& 2 a \int_{U^-_1}^{U_m} \frac{2 f(W) e^{- \int_{U_m}^{W} \frac{V f - P^2}{V P^2} dV}}{W P(\bar \beta,W)^3} e^{\int_{U_m}^W \frac{f}{V^2} dV} \delta P_2(W) dW \\
&= 2 a \int_{U^-_1}^{U_m} \frac{2 f(W) e^{\int_{U_m}^{W} \frac{1}{V} dV}}{W P(\bar \beta,W)^3} \delta P_2(W) dW = \frac{2 a}{U_m} \int_{U^-_1}^{U_m} \frac{2 f(W)}{W P(\bar \beta,W)^3} \delta P_2(W) dW = 0,
\end{split}
\end{equation*}
\begin{equation*}
\begin{split}
&a^2 \int_{U^-_1}^{U_m} \frac{2 f(W) e^{- \int_{U_m}^{W} \frac{V f - P^2}{V P^2} dV}}{W P(\bar \beta,W)^3} e^{2 \int_{U_m}^W \frac{f}{V^2} dV} dW = \frac{a^2}{U_m} \int_{U^-_1}^{U_m} \frac{2 f(W)}{P(\bar \beta,W)^3} e^{\int_{U_m}^W \frac{f}{V^2} dV} dW,
\end{split}
\end{equation*}
\begin{equation*}
\begin{split}
&2 \int_{U_m}^{U^+_1} \frac{2 f(W) e^{- \int_{U_m}^{W} \frac{V f - P^2}{V P^2} dV}}{W P(\bar \beta,W)^3} b e^{\int_{U_m}^W \frac{f}{V^2} dV} \delta P_3(W) dW = 2 \int_{U_m}^{U^+_1} \frac{2 f(W)}{W P(\bar \beta,W)^3} b \frac{W}{U_m} \delta P_3(W) dW = 0,
\end{split}
\end{equation*}
\begin{equation*}
\begin{split}
&\int_{U_m}^{U^+_1} \frac{2 f(W) e^{- \int_{U_m}^{W} \frac{V f - P^2}{V P^2} dV}}{W P(\bar \beta,W)^3} b^2 e^{2 \int_{U_m}^W \frac{f}{V^2} dV} dW = \frac{b^2}{U^+_1} \int_{U_m}^{U^+_1} \frac{2 f(W)}{P(\bar \beta,W)^3} e^{\int_{U_m}^W \frac{f(V)}{V^2} dV} dW.
\end{split}
\end{equation*}
Hence we obtain
\begin{equation*}
\begin{split}
\partial^2_\beta E(\bar \beta-) &= \int_{U^-_1}^{U_m} \frac{2 f(W) e^{-K(\bar \beta,U_1^-, W)}}{W P(\bar \beta,W)^3} \delta P_2(W)^2 dW + \int_{U_m}^{U^+_1} \frac{2 f(W) e^{-K(\bar \beta, U_m, W)}}{W P(\bar \beta,W)^3} \delta P_3(W)^2 dW \\
& \quad + \frac{a^2}{U_m} \int_{U^-_1}^{U_m} \frac{2 f(W)}{P(\bar \beta,W)^3} e^{\int_{U_m}^W \frac{f}{V^2} dV} dW + \frac{b^2}{U_m} \int_{U_m}^{U^+_1} \frac{2 f(W)}{P(\bar \beta,W)^3} e^{\int_{U_m}^W \frac{f(V)}{V^2} dV} dW.
\end{split}
\end{equation*}
We thus obtain substituting the expressions of $a,b$
\begin{equation*}
\begin{split}
\partial^2_\beta E(\bar \beta-) - \partial^2_\beta E(\bar \beta+) &= \frac{a^2}{U^-_1} \int_{U^-_1}^{U_m} \frac{2 f(W)}{P(\bar \beta,W)^3} e^{\int_{U^-_1}^W \frac{f}{V^2} dV} dW + \frac{b^2}{U^+_1} \int_{U_m}^{U^+_1} \frac{2 f(W)}{P(\bar \beta,W)^3} e^{\int_{U^+_1}^W \frac{f(V)}{V^2} dV} dW \\
&= \frac{ \big( \int_{U^m}^{U^+_1} \frac{2 f(W)}{P^3} e^{\int_{U_m}^W \frac{f}{P^2} dV} dW \big) \big( \int_{U^-_1}^{U_m} \frac{2 f(W)}{P^3} e^{\int_{U_m}^W \frac{f}{P^2} dV} dW \big)}{\int_{U^-_1}^{U^+_1} \frac{2 f(W)}{P^3} e^{\int_{U_m}^W \frac{f}{P^2} dV} dW} \big( \delta P_3(U_m) - \delta P_2(U_m) \big)^2 > 0.
%
\end{split}
\end{equation*}
Hence there is a discontinuity in $\partial^2_\beta E(\beta)$.

\newpage

\appendix

\printglossaries

\newpage

\printbibliography
\end{document}

%% file: result_2.5.pdf_t
\begin{picture}(0,0)%
\includegraphics{result_2.5.pdf}%
\end{picture}%
\setlength{\unitlength}{4144sp}%
\begin{picture}(4929,2994)(3544,-4483)
\put(6571,-2491){\makebox(0,0)[lb]{\smash{\fontsize{12}{14.4}\usefont{T1}{ptm}{m}{n}{\color[rgb]{0,0,0}$\Gamma^s$}%
}}}
\put(3646,-3436){\makebox(0,0)[lb]{\smash{\fontsize{12}{14.4}\usefont{T1}{ptm}{m}{n}{\color[rgb]{0,0,0}$u^*$}%
}}}
\put(8146,-3436){\makebox(0,0)[lb]{\smash{\fontsize{12}{14.4}\usefont{T1}{ptm}{m}{n}{\color[rgb]{0,0,0}$y$}%
}}}
\put(3691,-1771){\makebox(0,0)[lb]{\smash{\fontsize{12}{14.4}\usefont{T1}{ptm}{m}{n}{\color[rgb]{0,0,0}$P - P^*$}%
}}}
\put(4051,-3976){\makebox(0,0)[lb]{\smash{\fontsize{12}{14.4}\usefont{T1}{ptm}{m}{n}{\color[rgb]{0,0,0}$\Gamma^u$}%
}}}
\put(8056,-3076){\makebox(0,0)[lb]{\smash{\fontsize{12}{14.4}\usefont{T1}{ptm}{m}{n}{\color[rgb]{0,0,0}$1$}%
}}}
\end{picture}%

%% file: fU_Gammasu.pdf_t
\begin{picture}(0,0)%
\includegraphics{fU_Gammasu.pdf}%
\end{picture}%
\setlength{\unitlength}{4144sp}%
\begin{picture}(10509,4129)(529,-3593)
\put(1891,-2491){\makebox(0,0)[lb]{\smash{\fontsize{12}{14.4}\usefont{T1}{ptm}{m}{n}{\color[rgb]{0,0,0}$U^*$}%
}}}
\put(5131,-2491){\makebox(0,0)[lb]{\smash{\fontsize{12}{14.4}\usefont{T1}{ptm}{m}{n}{\color[rgb]{0,0,0}$1$}%
}}}
\put(10711,-3121){\makebox(0,0)[lb]{\smash{\fontsize{12}{14.4}\usefont{T1}{ptm}{m}{n}{\color[rgb]{0,0,0}$1$}%
}}}
\put(8641,-286){\makebox(0,0)[lb]{\smash{\fontsize{12}{14.4}\usefont{T1}{ptm}{m}{n}{\color[rgb]{0,0,0}$\Gamma_s(\beta) = \Gamma_s(\beta,U)$}%
}}}
\put(8371,-2446){\makebox(0,0)[lb]{\smash{\fontsize{12}{14.4}\usefont{T1}{ptm}{m}{n}{\color[rgb]{0,0,0}$\Gamma_u(\beta) = \Gamma_u(\beta,U)$}%
}}}
\put(7201,-1366){\makebox(0,0)[lb]{\smash{\fontsize{12}{14.4}\usefont{T1}{ptm}{m}{n}{\color[rgb]{0,0,0}$\Gamma_u(\beta^*) = \Gamma_s(\beta^*)$}%
}}}
\put(7201,-3076){\makebox(0,0)[lb]{\smash{\fontsize{12}{14.4}\usefont{T1}{ptm}{m}{n}{\color[rgb]{0,0,0}$U^*$}%
}}}
\put(6166,344){\makebox(0,0)[lb]{\smash{\fontsize{12}{14.4}\usefont{T1}{ptm}{m}{n}{\color[rgb]{0,0,0}$P$}%
}}}
\put(5356,-2446){\makebox(0,0)[lb]{\smash{\fontsize{12}{14.4}\usefont{T1}{ptm}{m}{n}{\color[rgb]{0,0,0}$U$}%
}}}
\put(10846,-3076){\makebox(0,0)[lb]{\smash{\fontsize{12}{14.4}\usefont{T1}{ptm}{m}{n}{\color[rgb]{0,0,0}$U$}%
}}}
\put(6976,-2581){\makebox(0,0)[lb]{\smash{\fontsize{12}{14.4}\usefont{T1}{ptm}{m}{n}{\color[rgb]{0,0,0}$\Gamma_u(\beta^{**})$}%
}}}
\put(8416,-3031){\makebox(0,0)[lb]{\smash{\fontsize{12}{14.4}\usefont{T1}{ptm}{m}{n}{\color[rgb]{0,0,0}$U_{uf}(\beta)$}%
}}}
\put(3421,-151){\makebox(0,0)[lb]{\smash{\fontsize{12}{14.4}\usefont{T1}{ptm}{m}{n}{\color[rgb]{0,0,0}$f(U)$}%
}}}
\end{picture}%

%% file: nearUstar.pdf_t
\begin{picture}(0,0)%
\includegraphics{nearUstar.pdf}%
\end{picture}%
\setlength{\unitlength}{4144sp}%
\begin{picture}(6282,5575)(3586,-7154)
\put(9516,-5110){\makebox(0,0)[lb]{\smash{\fontsize{12}{14.4}\usefont{T1}{ptm}{m}{n}{\color[rgb]{0,0,0}$U$}%
}}}
\put(6848,-1912){\makebox(0,0)[lb]{\smash{\fontsize{12}{14.4}\usefont{T1}{ptm}{m}{n}{\color[rgb]{0,0,0}$P$}%
}}}
\put(6794,-4764){\makebox(0,0)[lb]{\smash{\fontsize{12}{14.4}\usefont{T1}{ptm}{m}{n}{\color[rgb]{0,0,0}$U^*$}%
}}}
\put(5006,-5066){\makebox(0,0)[lb]{\smash{\fontsize{12}{14.4}\usefont{T1}{ptm}{m}{n}{\color[rgb]{0,0,0}$U^* - \underline{r}$}%
}}}
\put(6855,-4093){\makebox(0,0)[lb]{\smash{\fontsize{12}{14.4}\usefont{T1}{ptm}{m}{n}{\color[rgb]{0,0,0}$\Gamma_u(U^*)$}%
}}}
\put(5986,-2761){\makebox(0,0)[lb]{\smash{\fontsize{12}{14.4}\usefont{T1}{ptm}{m}{n}{\color[rgb]{0,0,1}$\Gamma_u(U)$}%
}}}
\put(8821,-5731){\makebox(0,0)[lb]{\smash{\fontsize{12}{14.4}\usefont{T1}{ptm}{m}{n}{\color[rgb]{0,0,0}$\sim \lambda_+(\beta,U^*)(U - U^*)$}%
}}}
\put(3601,-2356){\makebox(0,0)[lb]{\smash{\fontsize{12}{14.4}\usefont{T1}{ptm}{m}{n}{\color[rgb]{0,0,0}$\Gamma_u(U^*) + M^*_-(\beta,U^*)$}%
}}}
\put(8056,-7081){\makebox(0,0)[lb]{\smash{\fontsize{12}{14.4}\usefont{T1}{ptm}{m}{n}{\color[rgb]{0,0,0}$M^*_-(\beta,U) \sim \lambda_-(\beta,U^*) (U - U^*)$}%
}}}
\end{picture}%

%% file: optproof_1.pdf_t
\begin{picture}(0,0)%
\includegraphics{optproof_1.pdf}%
\end{picture}%
\setlength{\unitlength}{4144sp}%
\begin{picture}(10284,5845)(1744,-6614)
\put(11611,-6541){\makebox(0,0)[lb]{\smash{\fontsize{12}{14.4}\usefont{T1}{ptm}{m}{n}{\color[rgb]{0,0,0}$U$}%
}}}
\put(10936,-6451){\makebox(0,0)[lb]{\smash{\fontsize{12}{14.4}\usefont{T1}{ptm}{m}{n}{\color[rgb]{0,0,0}$1$}%
}}}
\put(2071,-6451){\makebox(0,0)[lb]{\smash{\fontsize{12}{14.4}\usefont{T1}{ptm}{m}{n}{\color[rgb]{0,0,0}$0$}%
}}}
\put(2161,-1051){\makebox(0,0)[lb]{\smash{\fontsize{12}{14.4}\usefont{T1}{ptm}{m}{n}{\color[rgb]{0,0,0}$P$}%
}}}
\put(9361,-1501){\makebox(0,0)[lb]{\smash{\fontsize{12}{14.4}\usefont{T1}{ptm}{m}{n}{\color[rgb]{0,0,0}$P^*(U)$}%
}}}
\put(7201,-4786){\makebox(0,0)[lb]{\smash{\fontsize{12}{14.4}\usefont{T1}{ptm}{m}{n}{\color[rgb]{0,0,0}$P^2<Uf(U)$}%
}}}
\put(2701,-3526){\makebox(0,0)[lb]{\smash{\fontsize{12}{14.4}\usefont{T1}{ptm}{m}{n}{\color[rgb]{0,0,0}$P^2 > Uf(U)$}%
}}}
\put(2566,-5191){\makebox(0,0)[lb]{\smash{\fontsize{12}{14.4}\usefont{T1}{ptm}{m}{n}{\color[rgb]{0,0,0}$\Gamma_u(\beta,U)$}%
}}}
\put(7471,-2176){\makebox(0,0)[lb]{\smash{\fontsize{12}{14.4}\usefont{T1}{ptm}{m}{n}{\color[rgb]{0,0,0}$\Gamma_s(\beta,U)$}%
}}}
\put(5986,-6451){\makebox(0,0)[lb]{\smash{\fontsize{12}{14.4}\usefont{T1}{ptm}{m}{n}{\color[rgb]{0,0,0}$U_s(\beta)$}%
}}}
\put(3376,-6451){\makebox(0,0)[lb]{\smash{\fontsize{12}{14.4}\usefont{T1}{ptm}{m}{n}{\color[rgb]{0,0,0}$U^*$}%
}}}
\put(3871,-6451){\makebox(0,0)[lb]{\smash{\fontsize{12}{14.4}\usefont{T1}{ptm}{m}{n}{\color[rgb]{0,0,0}$U_u(\beta)$}%
}}}
\end{picture}%

%% file: diagram_Gamma_F.pdf_t
\begin{picture}(0,0)%
\includegraphics{diagram_Gamma_F.pdf}%
\end{picture}%
\setlength{\unitlength}{4144sp}%
\begin{picture}(10512,6340)(301,-5804)
\put(721,-1906){\makebox(0,0)[lb]{\smash{\fontsize{12}{14.4}\usefont{T1}{ptm}{m}{n}{\color[rgb]{0,0,1}$\Gamma_u = P^*$}%
}}}
\put(3421,-1726){\makebox(0,0)[lb]{\smash{\fontsize{12}{14.4}\usefont{T1}{ptm}{m}{n}{\color[rgb]{0,.69,0}$\Gamma_s=P^*$}%
}}}
\put(5086,-2716){\makebox(0,0)[lb]{\smash{\fontsize{12}{14.4}\usefont{T1}{ptm}{m}{n}{\color[rgb]{0,0,0}$1$}%
}}}
\put(5266,-3031){\makebox(0,0)[lb]{\smash{\fontsize{12}{14.4}\usefont{T1}{ptm}{m}{n}{\color[rgb]{0,0,0}$U$}%
}}}
\put(10486,-5686){\makebox(0,0)[lb]{\smash{\fontsize{12}{14.4}\usefont{T1}{ptm}{m}{n}{\color[rgb]{0,0,0}$U$}%
}}}
\put(10621,-5281){\makebox(0,0)[lb]{\smash{\fontsize{12}{14.4}\usefont{T1}{ptm}{m}{n}{\color[rgb]{0,0,0}$1$}%
}}}
\put(6256,-5731){\makebox(0,0)[lb]{\smash{\fontsize{12}{14.4}\usefont{T1}{ptm}{m}{n}{\color[rgb]{0,0,0}$U^*$}%
}}}
\put(9361,-781){\makebox(0,0)[lb]{\smash{\fontsize{12}{14.4}\usefont{T1}{ptm}{m}{n}{\color[rgb]{0,0,0}$f(U)$}%
}}}
\put(9406,-1996){\makebox(0,0)[lb]{\smash{\fontsize{12}{14.4}\usefont{T1}{ptm}{m}{n}{\color[rgb]{0,.69,0}$\Gamma_s$}%
}}}
\put(7111,-3976){\makebox(0,0)[lb]{\smash{\fontsize{12}{14.4}\usefont{T1}{ptm}{m}{n}{\color[rgb]{0,0,1}$\Gamma_u$}%
}}}
\put(6616,-4921){\makebox(0,0)[lb]{\smash{\fontsize{12}{14.4}\usefont{T1}{ptm}{m}{n}{\color[rgb]{1,0,0}$F(U)+\beta=0$}%
}}}
\put(4996,-2041){\makebox(0,0)[lb]{\smash{\fontsize{12}{14.4}\usefont{T1}{ptm}{m}{n}{\color[rgb]{1,0,0}$-F(U)$}%
}}}
\put(316,-3616){\makebox(0,0)[lb]{\smash{\fontsize{12}{14.4}\usefont{T1}{ptm}{m}{n}{\color[rgb]{0,0,0}$\beta^*$}%
}}}
\put(316,-2986){\makebox(0,0)[lb]{\smash{\fontsize{12}{14.4}\usefont{T1}{ptm}{m}{n}{\color[rgb]{0,0,0}$U^*$}%
}}}
\put(316,-1231){\makebox(0,0)[lb]{\smash{\fontsize{12}{14.4}\usefont{T1}{ptm}{m}{n}{\color[rgb]{0,0,0}$\beta^{**}$}%
}}}
\put(406,209){\makebox(0,0)[lb]{\smash{\fontsize{12}{14.4}\usefont{T1}{ptm}{m}{n}{\color[rgb]{0,0,0}$\beta$}%
}}}
\put(1891,-1051){\makebox(0,0)[lb]{\smash{\fontsize{12}{14.4}\usefont{T1}{ptm}{m}{n}{\color[rgb]{0,0,0}The vector field is enetering $\{P \leq P^*\}$}%
}}}
\put(1981,-5011){\makebox(0,0)[lb]{\smash{\fontsize{12}{14.4}\usefont{T1}{ptm}{m}{n}{\color[rgb]{0,0,0}The vector field is exiting $\{P \leq P^*\}$}%
}}}
\end{picture}%

%% file: cubic_hyperbol_perp_1.pdf_t
\begin{picture}(0,0)%
\includegraphics{cubic_hyperbol_perp_1.pdf}%
\end{picture}%
\setlength{\unitlength}{4144sp}%
\begin{picture}(5019,5305)(5794,-5804)
\put(10486,-5686){\makebox(0,0)[lb]{\smash{\fontsize{12}{14.4}\usefont{T1}{ptm}{m}{n}{\color[rgb]{0,0,0}$U$}%
}}}
\put(10621,-5281){\makebox(0,0)[lb]{\smash{\fontsize{12}{14.4}\usefont{T1}{ptm}{m}{n}{\color[rgb]{0,0,0}$1$}%
}}}
\put(6256,-5731){\makebox(0,0)[lb]{\smash{\fontsize{12}{14.4}\usefont{T1}{ptm}{m}{n}{\color[rgb]{0,0,0}$U^*$}%
}}}
\put(7066,-2896){\makebox(0,0)[lb]{\smash{\fontsize{12}{14.4}\usefont{T1}{ptm}{m}{n}{\color[rgb]{0,0,0}$f(U)$}%
}}}
\put(9046,-2131){\makebox(0,0)[lb]{\smash{\fontsize{12}{14.4}\usefont{T1}{ptm}{m}{n}{\color[rgb]{0,0,1}$f_\kappa(U)$}%
}}}
\put(7786,-5731){\makebox(0,0)[lb]{\smash{\fontsize{12}{14.4}\usefont{T1}{ptm}{m}{n}{\color[rgb]{0,0,0}$U_1$}%
}}}
\put(8191,-5731){\makebox(0,0)[lb]{\smash{\fontsize{12}{14.4}\usefont{T1}{ptm}{m}{n}{\color[rgb]{0,0,0}$U^-_I$}%
}}}
\put(8686,-5731){\makebox(0,0)[lb]{\smash{\fontsize{12}{14.4}\usefont{T1}{ptm}{m}{n}{\color[rgb]{0,0,0}$U^+_I$}%
}}}
\put(9631,-5731){\makebox(0,0)[lb]{\smash{\fontsize{12}{14.4}\usefont{T1}{ptm}{m}{n}{\color[rgb]{0,0,0}$U_2$}%
}}}
\end{picture}%

%% file: rreplAcal.pdf_t
\begin{picture}(0,0)%
\includegraphics{rreplAcal.pdf}%
\end{picture}%
\setlength{\unitlength}{4144sp}%
\begin{picture}(5019,5305)(5794,-5804)
\put(10486,-5686){\makebox(0,0)[lb]{\smash{\fontsize{12}{14.4}\usefont{T1}{ptm}{m}{n}{\color[rgb]{0,0,0}$U$}%
}}}
\put(10621,-5281){\makebox(0,0)[lb]{\smash{\fontsize{12}{14.4}\usefont{T1}{ptm}{m}{n}{\color[rgb]{0,0,0}$1$}%
}}}
\put(6256,-5731){\makebox(0,0)[lb]{\smash{\fontsize{12}{14.4}\usefont{T1}{ptm}{m}{n}{\color[rgb]{0,0,0}$U^*$}%
}}}
\put(6751,-4336){\makebox(0,0)[lb]{\smash{\fontsize{12}{14.4}\usefont{T1}{ptm}{m}{n}{\color[rgb]{0,0,0}$U(s^-)$}%
}}}
\put(7471,-2311){\makebox(0,0)[lb]{\smash{\fontsize{12}{14.4}\usefont{T1}{ptm}{m}{n}{\color[rgb]{0,0,0}$U(s^+)$}%
}}}
\put(6391,-1276){\makebox(0,0)[lb]{\smash{\fontsize{12}{14.4}\usefont{T1}{ptm}{m}{n}{\color[rgb]{1,0,0}$\Gamma_s(\beta)$}%
}}}
\put(6616,-5326){\makebox(0,0)[lb]{\smash{\fontsize{12}{14.4}\usefont{T1}{ptm}{m}{n}{\color[rgb]{1,0,0}$\Gamma_u(\beta)$}%
}}}
\put(6256,-4786){\makebox(0,0)[lb]{\smash{\fontsize{12}{14.4}\usefont{T1}{ptm}{m}{n}{\color[rgb]{0,0,0}$U_u(\beta)$}%
}}}
\put(8056,-1546){\makebox(0,0)[lb]{\smash{\fontsize{12}{14.4}\usefont{T1}{ptm}{m}{n}{\color[rgb]{0,0,0}$U_s(\beta)$}%
}}}
\put(9766,-1321){\makebox(0,0)[lb]{\smash{\fontsize{12}{14.4}\usefont{T1}{ptm}{m}{n}{\color[rgb]{0,0,0}$f(U)$}%
}}}
\put(6661,-2401){\makebox(0,0)[lb]{\smash{\fontsize{12}{14.4}\usefont{T1}{ptm}{m}{n}{\color[rgb]{0,0,1}$\gamma$}%
}}}
\put(7156,-3031){\makebox(0,0)[lb]{\smash{\fontsize{12}{14.4}\usefont{T1}{ptm}{m}{n}{\color[rgb]{.53,.81,1}$\gamma'$}%
}}}
\end{picture}%

%% file: KOIdef.pdf_t
\begin{picture}(0,0)%
\includegraphics{KOIdef.pdf}%
\end{picture}%
\setlength{\unitlength}{4144sp}%
\begin{picture}(5019,6366)(5794,-6865)
\put(6256,-5731){\makebox(0,0)[lb]{\smash{\fontsize{12}{14.4}\usefont{T1}{ptm}{m}{n}{\color[rgb]{0,0,0}$U^*$}%
}}}
\put(6391,-1276){\makebox(0,0)[lb]{\smash{\fontsize{12}{14.4}\usefont{T1}{ptm}{m}{n}{\color[rgb]{1,0,0}$\Gamma_s(\beta)$}%
}}}
\put(6616,-5326){\makebox(0,0)[lb]{\smash{\fontsize{12}{14.4}\usefont{T1}{ptm}{m}{n}{\color[rgb]{1,0,0}$\Gamma_u(\beta)$}%
}}}
\put(9766,-1321){\makebox(0,0)[lb]{\smash{\fontsize{12}{14.4}\usefont{T1}{ptm}{m}{n}{\color[rgb]{0,0,0}$f(U)$}%
}}}
\put(6117,-5000){\makebox(0,0)[lb]{\smash{\fontsize{12}{14.4}\usefont{T1}{ptm}{m}{n}{\color[rgb]{0,0,0}$U_u(\beta)$}%
}}}
\put(8128,-5977){\makebox(0,0)[lb]{\smash{\fontsize{12}{14.4}\usefont{T1}{ptm}{m}{n}{\color[rgb]{0,0,0}$\mathcal K(P,\beta)$}%
}}}
\put(8108,-6378){\makebox(0,0)[lb]{\smash{\fontsize{12}{14.4}\usefont{T1}{ptm}{m}{n}{\color[rgb]{0,0,0}$\mathcal O(P,\beta)$}%
}}}
\put(8120,-6689){\makebox(0,0)[lb]{\smash{\fontsize{12}{14.4}\usefont{T1}{ptm}{m}{n}{\color[rgb]{0,0,0}$\mathcal I(P,\beta)$}%
}}}
\put(9847,-5347){\makebox(0,0)[lb]{\smash{\fontsize{12}{14.4}\usefont{T1}{ptm}{m}{n}{\color[rgb]{0,0,0}$U$}%
}}}
\put(10633,-5657){\makebox(0,0)[lb]{\smash{\fontsize{12}{14.4}\usefont{T1}{ptm}{m}{n}{\color[rgb]{0,0,0}$1$}%
}}}
\put(6572,-2970){\makebox(0,0)[lb]{\smash{\fontsize{12}{14.4}\usefont{T1}{ptm}{m}{n}{\color[rgb]{0,0,1}$P(U)$}%
}}}
\put(8074,-1542){\makebox(0,0)[lb]{\smash{\fontsize{12}{14.4}\usefont{T1}{ptm}{m}{n}{\color[rgb]{0,0,0}$U_s(\beta)$}%
}}}
\end{picture}%

%% file: linera_pert_1.pdf_t
\begin{picture}(0,0)%
\includegraphics{linera_pert_1.pdf}%
\end{picture}%
\setlength{\unitlength}{4144sp}%
\begin{picture}(5019,5305)(5794,-5804)
\put(6256,-5731){\makebox(0,0)[lb]{\smash{\fontsize{12}{14.4}\usefont{T1}{ptm}{m}{n}{\color[rgb]{0,0,0}$U^*$}%
}}}
\put(6616,-5326){\makebox(0,0)[lb]{\smash{\fontsize{12}{14.4}\usefont{T1}{ptm}{m}{n}{\color[rgb]{1,0,0}$\Gamma_u(\beta)$}%
}}}
\put(9847,-5347){\makebox(0,0)[lb]{\smash{\fontsize{12}{14.4}\usefont{T1}{ptm}{m}{n}{\color[rgb]{0,0,0}$U$}%
}}}
\put(10633,-5657){\makebox(0,0)[lb]{\smash{\fontsize{12}{14.4}\usefont{T1}{ptm}{m}{n}{\color[rgb]{0,0,0}$1$}%
}}}
\put(6266,-4802){\makebox(0,0)[lb]{\smash{\fontsize{12}{14.4}\usefont{T1}{ptm}{m}{n}{\color[rgb]{0,0,0}$U^-$}%
}}}
\put(9271,-781){\makebox(0,0)[lb]{\smash{\fontsize{12}{14.4}\usefont{T1}{ptm}{m}{n}{\color[rgb]{1,0,0}$\Gamma_s(\beta)$}%
}}}
\put(8956,-1411){\makebox(0,0)[lb]{\smash{\fontsize{12}{14.4}\usefont{T1}{ptm}{m}{n}{\color[rgb]{0,0,0}$\tilde U$}%
}}}
\put(9496,-1231){\makebox(0,0)[lb]{\smash{\fontsize{12}{14.4}\usefont{T1}{ptm}{m}{n}{\color[rgb]{0,0,0}$U^+$}%
}}}
\put(10351,-2716){\makebox(0,0)[lb]{\smash{\fontsize{12}{14.4}\usefont{T1}{ptm}{m}{n}{\color[rgb]{0,0,0}$f(U)$}%
}}}
\put(7831,-3121){\makebox(0,0)[lb]{\smash{\fontsize{12}{14.4}\usefont{T1}{ptm}{m}{n}{\color[rgb]{1,.5,.5}$P_n(U)$}%
}}}
\put(6166,-5056){\makebox(0,0)[lb]{\smash{\fontsize{12}{14.4}\usefont{T1}{ptm}{m}{n}{\color[rgb]{0,0,0}$U_n$}%
}}}
\put(7471,-961){\makebox(0,0)[lb]{\smash{\fontsize{12}{14.4}\usefont{T1}{ptm}{m}{n}{\color[rgb]{0,0,1}$P(U)$}%
}}}
\end{picture}%

%% file: unstab_stable_last.pdf_t
\begin{picture}(0,0)%
\includegraphics{unstab_stable_last.pdf}%
\end{picture}%
\setlength{\unitlength}{4144sp}%
\begin{picture}(5382,3984)(256,-4303)
\put(271,-3166){\makebox(0,0)[lb]{\smash{\fontsize{12}{14.4}\usefont{T1}{ptm}{m}{n}{\color[rgb]{0,0,0}$U^*$}%
}}}
\put(5401,-3436){\makebox(0,0)[lb]{\smash{\fontsize{12}{14.4}\usefont{T1}{ptm}{m}{n}{\color[rgb]{0,0,0}$U$}%
}}}
\put(4771,-3436){\makebox(0,0)[lb]{\smash{\fontsize{12}{14.4}\usefont{T1}{ptm}{m}{n}{\color[rgb]{0,0,0}$U_{fu}(\beta)$}%
}}}
\put(2026,-3436){\makebox(0,0)[lb]{\smash{\fontsize{12}{14.4}\usefont{T1}{ptm}{m}{n}{\color[rgb]{0,0,0}$\hat U_u(\beta)$}%
}}}
\put(1351,-1996){\makebox(0,0)[lb]{\smash{\fontsize{12}{14.4}\usefont{T1}{ptm}{m}{n}{\color[rgb]{0,0,1}$\frac{f(U)}{U\Gamma_u(U)^2} - \frac{1}{U^2}$}%
}}}
\end{picture}%

%% file: regul_O_beta.pdf_t
\begin{picture}(0,0)%
\includegraphics{regul_O_beta.pdf}%
\end{picture}%
\setlength{\unitlength}{4144sp}%
\begin{picture}(5019,4405)(5794,-5714)
\put(6616,-5326){\makebox(0,0)[lb]{\smash{\fontsize{12}{14.4}\usefont{T1}{ptm}{m}{n}{\color[rgb]{1,0,0}$\Gamma_u(\beta)$}%
}}}
\put(6256,-2311){\makebox(0,0)[lb]{\smash{\fontsize{12}{14.4}\usefont{T1}{ptm}{m}{n}{\color[rgb]{1,0,0}$\Gamma_s(\beta_1)$}%
}}}
\put(6256,-1681){\makebox(0,0)[lb]{\smash{\fontsize{12}{14.4}\usefont{T1}{ptm}{m}{n}{\color[rgb]{1,0,0}$\Gamma_s(\beta_2)$}%
}}}
\put(7876,-3301){\makebox(0,0)[lb]{\smash{\fontsize{12}{14.4}\usefont{T1}{ptm}{m}{n}{\color[rgb]{0,0,0}$U^+$}%
}}}
\put(7111,-4291){\makebox(0,0)[lb]{\smash{\fontsize{12}{14.4}\usefont{T1}{ptm}{m}{n}{\color[rgb]{0,.69,0}$P_I$}%
}}}
\put(5896,-1546){\makebox(0,0)[lb]{\smash{\fontsize{12}{14.4}\usefont{T1}{ptm}{m}{n}{\color[rgb]{0,0,0}$P$}%
}}}
\put(9361,-3256){\makebox(0,0)[lb]{\smash{\fontsize{12}{14.4}\usefont{T1}{ptm}{m}{n}{\color[rgb]{0,.69,0}$P_s$}%
}}}
\put(10171,-5641){\makebox(0,0)[lb]{\smash{\fontsize{12}{14.4}\usefont{T1}{ptm}{m}{n}{\color[rgb]{0,0,0}$U$}%
}}}
\put(6353,-5641){\makebox(0,0)[lb]{\smash{\fontsize{12}{14.4}\usefont{T1}{ptm}{m}{n}{\color[rgb]{0,0,0}$\hat U_u$}%
}}}
\put(10525,-4992){\makebox(0,0)[lb]{\smash{\fontsize{12}{14.4}\usefont{T1}{ptm}{m}{n}{\color[rgb]{0,0,0}$\hat U_s(\beta_2)$}%
}}}
\put(7569,-2858){\makebox(0,0)[lb]{\smash{\fontsize{12}{14.4}\usefont{T1}{ptm}{m}{n}{\color[rgb]{0,0,1}$P_1$}%
}}}
\put(6378,-5009){\makebox(0,0)[lb]{\smash{\fontsize{12}{14.4}\usefont{T1}{ptm}{m}{n}{\color[rgb]{0,0,0}$U^-$}%
}}}
\put(8578,-2994){\makebox(0,0)[lb]{\smash{\fontsize{12}{14.4}\usefont{T1}{ptm}{m}{n}{\color[rgb]{0,0,0}$\hat U_s(\beta_1)$}%
}}}
\put(9637,-4051){\makebox(0,0)[lb]{\smash{\fontsize{12}{14.4}\usefont{T1}{ptm}{m}{n}{\color[rgb]{0,0,0}$P^*$}%
}}}
\end{picture}%

%% file: pert_beta.pdf_t
\begin{picture}(0,0)%
\includegraphics{pert_beta.pdf}%
\end{picture}%
\setlength{\unitlength}{4144sp}%
\begin{picture}(5019,4029)(529,-3493)
\put(1891,-2491){\makebox(0,0)[lb]{\smash{\fontsize{12}{14.4}\usefont{T1}{ptm}{m}{n}{\color[rgb]{0,0,0}$U^*$}%
}}}
\put(5131,-2491){\makebox(0,0)[lb]{\smash{\fontsize{12}{14.4}\usefont{T1}{ptm}{m}{n}{\color[rgb]{0,0,0}$1$}%
}}}
\put(5356,-2446){\makebox(0,0)[lb]{\smash{\fontsize{12}{14.4}\usefont{T1}{ptm}{m}{n}{\color[rgb]{0,0,0}$U$}%
}}}
\put(946,-1681){\makebox(0,0)[lb]{\smash{\fontsize{12}{14.4}\usefont{T1}{ptm}{m}{n}{\color[rgb]{0,0,0}$\Gamma_u$}%
}}}
\put(2521,-1231){\makebox(0,0)[lb]{\smash{\fontsize{12}{14.4}\usefont{T1}{ptm}{m}{n}{\color[rgb]{0,0,1}$\tilde \gamma$}%
}}}
\put(4096,-1636){\makebox(0,0)[lb]{\smash{\fontsize{12}{14.4}\usefont{T1}{ptm}{m}{n}{\color[rgb]{0,0,0}$\Gamma_s$}%
}}}
\put(3421,-151){\makebox(0,0)[lb]{\smash{\fontsize{12}{14.4}\usefont{T1}{ptm}{m}{n}{\color[rgb]{0,0,0}$f(U)$}%
}}}
\end{picture}%

%% file: pert_f.pdf_t
\begin{picture}(0,0)%
\includegraphics{pert_f.pdf}%
\end{picture}%
\setlength{\unitlength}{4144sp}%
\begin{picture}(5019,4029)(529,-3493)
\put(1891,-2491){\makebox(0,0)[lb]{\smash{\fontsize{12}{14.4}\usefont{T1}{ptm}{m}{n}{\color[rgb]{0,0,0}$U^*$}%
}}}
\put(5131,-2491){\makebox(0,0)[lb]{\smash{\fontsize{12}{14.4}\usefont{T1}{ptm}{m}{n}{\color[rgb]{0,0,0}$1$}%
}}}
\put(5356,-2446){\makebox(0,0)[lb]{\smash{\fontsize{12}{14.4}\usefont{T1}{ptm}{m}{n}{\color[rgb]{0,0,0}$U$}%
}}}
\put(946,-1681){\makebox(0,0)[lb]{\smash{\fontsize{12}{14.4}\usefont{T1}{ptm}{m}{n}{\color[rgb]{0,0,0}$\Gamma_u$}%
}}}
\put(2521,-1231){\makebox(0,0)[lb]{\smash{\fontsize{12}{14.4}\usefont{T1}{ptm}{m}{n}{\color[rgb]{0,0,1}$\tilde \gamma$}%
}}}
\put(4096,-1636){\makebox(0,0)[lb]{\smash{\fontsize{12}{14.4}\usefont{T1}{ptm}{m}{n}{\color[rgb]{0,0,0}$\Gamma_s$}%
}}}
\put(3511, 74){\makebox(0,0)[lb]{\smash{\fontsize{12}{14.4}\usefont{T1}{ptm}{m}{n}{\color[rgb]{0,0,0}$f_1(U)$}%
}}}
\put(1261,-2851){\makebox(0,0)[lb]{\smash{\fontsize{12}{14.4}\usefont{T1}{ptm}{m}{n}{\color[rgb]{0,0,0}$f_0(U)$}%
}}}
\put(1711,-1321){\makebox(0,0)[lb]{\smash{\fontsize{12}{14.4}\usefont{T1}{ptm}{m}{n}{\color[rgb]{0,0,0}$\underline{P}(U)$}%
}}}
\end{picture}%

%% file: pert_f_3.pdf_t
\begin{picture}(0,0)%
\includegraphics{pert_f_3.pdf}%
\end{picture}%
\setlength{\unitlength}{4144sp}%
\begin{picture}(5019,3489)(529,-2953)
\put(5131,-2491){\makebox(0,0)[lb]{\smash{\fontsize{12}{14.4}\usefont{T1}{ptm}{m}{n}{\color[rgb]{0,0,0}$1$}%
}}}
\put(5356,-2446){\makebox(0,0)[lb]{\smash{\fontsize{12}{14.4}\usefont{T1}{ptm}{m}{n}{\color[rgb]{0,0,0}$U$}%
}}}
\put(4096,-1636){\makebox(0,0)[lb]{\smash{\fontsize{12}{14.4}\usefont{T1}{ptm}{m}{n}{\color[rgb]{0,0,0}$\Gamma_s$}%
}}}
\put(1261,-2851){\makebox(0,0)[lb]{\smash{\fontsize{12}{14.4}\usefont{T1}{ptm}{m}{n}{\color[rgb]{0,0,0}$f_0(U)$}%
}}}
\put(2701,-61){\makebox(0,0)[lb]{\smash{\fontsize{12}{14.4}\usefont{T1}{ptm}{m}{n}{\color[rgb]{0,0,0}$\bar U_s$}%
}}}
\put(721,-1861){\makebox(0,0)[lb]{\smash{\fontsize{12}{14.4}\usefont{T1}{ptm}{m}{n}{\color[rgb]{0,0,0}$\Gamma_u$}%
}}}
\put(1846,-2446){\makebox(0,0)[lb]{\smash{\fontsize{12}{14.4}\usefont{T1}{ptm}{m}{n}{\color[rgb]{0,0,0}$U^*$}%
}}}
\put(4231,-151){\makebox(0,0)[lb]{\smash{\fontsize{12}{14.4}\usefont{T1}{ptm}{m}{n}{\color[rgb]{0,0,0}$f_1(U)$}%
}}}
\put(1531,-1681){\makebox(0,0)[lb]{\smash{\fontsize{12}{14.4}\usefont{T1}{ptm}{m}{n}{\color[rgb]{0,0,0}$\bar U_u$}%
}}}
\put(1981,-871){\makebox(0,0)[lb]{\smash{\fontsize{12}{14.4}\usefont{T1}{ptm}{m}{n}{\color[rgb]{0,0,1}$\tilde \gamma$}%
}}}
\end{picture}%